\documentclass[11pt,reqno]{amsart}
\usepackage{amsmath,amsbsy,amssymb,amscd,amsfonts,latexsym,amstext,delarray, amsmath,color,caption}
\usepackage{tikz,lscape}
\usetikzlibrary{matrix,arrows}
\usepackage{graphicx}
\usepackage{multicol}
\usepackage{scrextend}
\usepackage{mathtools}
\usepackage{amssymb} 
\usepackage{calligra}
\usepackage{calrsfs}
\usepackage{eucal}
\usepackage{calrsfs}
\usepackage{epsfig}
\usepackage{xcolor,import}
\usepackage{transparent}
\usepackage{tikz-cd}
\usepackage{epstopdf}
\usepackage{subcaption}
\usepackage{enumitem}
\usepackage{float}
\numberwithin{equation}{section}
\numberwithin{figure}{section}
\usepackage{cleveref}
\usepackage{mathrsfs}

\DeclareMathAlphabet{\pazocal}{OMS}{zplm}{m}{n}

\input xy

\xyoption{all}
    \advance \topmargin by -\headheight
    \advance \topmargin by -\headsep

    \evensidemargin \oddsidemargin
\makeatletter
\newcommand{\leqnomode}{\tagsleft@true\let\veqno\@@leqno}
\newcommand{\reqnomode}{\tagsleft@false\let\veqno\@@eqno}
\makeatother

\newtheorem {theorem}    {Theorem}[section]
\newtheorem*{main-theorem}{Theorem A}
\newtheorem*{main-corollary}{Corollary B}

\theoremstyle{definition}
\newtheorem {lemma}      [theorem]    {Lemma}

\theoremstyle{definition}

\theoremstyle{definition}

\newcommand{\gen}[1]{\langle #1 \rangle}

\def\TM{{\rm TM}}

\def\a{\alpha}                
\def\b{\beta}

\def\eps{\varepsilon}

\def\a{\alpha}
\def\b{\beta}

\def\N{\mathbb{N}}     
\def\lab{{\text{Lab}}}

\def\vertexradius{.1}
\def\vertex(#1){\fill (#1) circle (\vertexradius)}

\begin{document}

\title{\bf Malnormal subgroups of finitely presented groups}
\maketitle
\begin{center}

Francis Wagner

\end{center}

\bigskip

\begin{center}

\textbf{Abstract}

\end{center}

\begin{addmargin}[2em]{2em}



The following refinement of the Higman embedding theorem is proved: A finitely generated group $R$ is recursively presented if and only if there exists a quasi-isometric malnormal embedding of $R$ into a finitely presented group $H$ such that the image of the embedding enjoys the congruence extension property.  Moreover, it is shown that the finitely presented group $H$ can be constructed to have decidable Word Problem if and only if the Word Problem for $R$ is decidable, yielding a refinement of a theorem of Clapham. Finally, given a countable group $G$ and a computable function $\ell:G\to\N$ satisfying some necessary requirements, it is proved that there exists a malnormal embedding of $G$ into a finitely presented group $H$ such that the restriction of $|\cdot|_H$ to $G$ is equivalent to $\ell$, producing a refinement of a theorem of Ol'shanskii.

\end{addmargin}

\bigskip


\section{Introduction}

The renowned Higman embedding theorem gives a remarkable connection between an algebraic property of groups (being isomorphic to a subgroup of a finitely presented group) and an intrinsically algorithmic one (possessing a recursive presentation).  The main result of this article is the following refinement of this seminal result:

\begin{main-theorem} \label{thm-main}

A finitely generated group $G$ is recursively presented if and only if there exists an embedding $G\hookrightarrow H$ satisfying all of the following properties:

\begin{enumerate}

\item $H$ is finitely presented

\item $G$ is a malnormal subgroup of $H$

\item $G$ satisfies the congruence extension property

\item The restriction of $|\cdot|_H$ to $G$ is equivalent to $|\cdot|_G$.

\item The Word Problem for $H$ is decidable if and only if the Word Problem for $G$ is decidable.
%
%
%
%
%

\end{enumerate}

\end{main-theorem}

An interesting feature of the above statement is the disparity between the nature of the conditions: (1)-(3) can be interpreted in purely algebraic terms, (4) has a geometric essence relating to subgroup distortion, and (5) is inherently algorithmic.  

Theorem A may be applied to a wider class of groups if one omits the finitely generated hypothesis:

\begin{main-corollary}

Given a countable group $G$ and a computable length function $\ell:G\to\N$, there exists an embedding $G\hookrightarrow H$ such that:

\begin{enumerate}

\item $H$ is finitely presented

\item $G$ is a malnormal subgroup of $H$

\item $G$ satisfies the congruence extension property

\item The restriction of $|\cdot|_H$ to $G$ is equivalent to $\ell$.

\end{enumerate}

\end{main-corollary}

The proof of these statements proceed with the introduction and development of a generalization of the computational model of $S$-machines.  $S$-machines were invented by Sapir in the 1990s as a means to construct finitely presented groups with desirable geometric properties.  The generalized model introduced here is called `noisy $S$-machines', hinting at the way a computation introduces `noise' in the tape words.  This property is the key to producing a malnormal embedding, but with care can be tamed to contain the manner in which it alters the geometry of a group constructed through $S$-machines.  As such, noisy $S$-machines can be viewed as a model that combines the benefits of Aanderaa's groups associated to Turing machines \cite{Aanderaa} with Sapir's $S$-machines.



\subsection{Background} \

A countable group is said to be \textit{recursively presented} if it admits a presentation $\gen{X\mid\pazocal{R}}$ such that there exists an `effective algorithm' to list the elements of $\pazocal{R}$, {\frenchspacing i.e. such that} $\pazocal{R}$ is a recursively enumerable subset of the set of words $(X\cup X^{-1})^*$.  It is clear from this definition that finitely presented groups are recursively presented, and it is not difficult to see that this property descends to the finitely generated subgroups of such groups.  

In his celebrated embedding theorem, G. Higman \cite{Higman} showed the converse, exhibiting an embedding of an arbitrary finitely generated recursively presented group into a finitely presented group.  As constructions of Higman, Neumann, and Neumann \cite{HNN} show that an arbitrary recursively presented group can be embedded into a two-generated recursively presented group, it thus follows that all recursively presented groups embed in a finitely presented group.

Higman's embedding theorem has inspired many works investigating what properties of the group may be preserved under such an embedding.  For just a few examples (chosen for immediate relevance to this article):

\begin{itemize}

\item Clapham demonstrated \cite{Clapham} a refinement of the embedding which preserves the decidability of the Word Problem.

\item Ol'shanskii demonstrated \cite{O97} a refinement of the embedding which is bi-Lipschitz (and so quasi-isometric).

\item Birget, Rips, Ol'shanskii, and Sapir demonstrated \cite{BORS} a refinement of the embedding where the (non-deterministic) complexity of the Word Problem of the embedded group is polynomially equivalent to the Dehn function of the finitely presented group.

\item Sapir demonstrated \cite{S14} a refinement of the embedding that preserves the asphericity of the group.

\item Ol'shanskii and Sapir demonstrated \cite{OSconj} a refinement of the embedding which preserves the decidability of the Conjugacy problem.

\end{itemize}

Other notable refinements can be found in \cite{Baumslag}, \cite{Birget}, \cite{Boone-Higman}, \cite{O20}, and \cite{Valiev}.  

In this article, we investigate a refinement along similar lines to those above, adding the condition of `malnormality'.

A subgroup $G\leq H$ is called \textit{malnormal}, denoted $G\leq_{mal}H$, if for all $h\in H$ with $h\notin G$, $(h^{-1}Gh)\cap G=\{1\}$. Naturally, an embedding $\iota:G\hookrightarrow H$ is called \textit{malnormal} if $\iota(G)$ is a malnormal subgroup of $H$.

The notion of malnormality has origins with B. Baumslag \cite{BaumslagMal} as a condition on an amalgam so that any two-generated subgroup of the amalgamated free product is free.  Since this introduction, malnormal subgroups have been found to arise naturally in several important settings, {\frenchspacing e.g. Fuchsian groups, Kleinian groups, one-relator groups, and groups possessing forms of negative curvature}  (see \cite{dlHMal} for detailed discussions).  Malnormal subgroups of hyperbolic groups have been studied deeply (for example in \cite{BMR}, \cite{BridsonMal}), and whether malnormality implies quasiconvexity in such groups is a well-known open problem \cite{BMS}.

An example that is important for the purposes of this article arises in the setting of relatively hyperbolic groups, where a peripheral subgroup of a torsion free relatively hyperbolic group is malnormal (see \cite{OsinRelHyp} and the Appendix of \cite{dlHMal}).  Malnormality can thus be seen as a generalization of the property of being hyperbolic relative to a subgroup, leading naturally to the question of what groups can be realized as malnormal subgroups of others.

Sapir posed a version of this question in \cite{S11} (see Remark 5.23), attributed to D. Osin, asking whether there is a malnormal refinement of the Higman embedding theorem.  While it is suggested in this Remark that it is `quite possible' that the methods employed in that setting may produce a positive answer, the question was left open.  Further details are discussed in \cite{Overflow}.  

Theorem A and Corollary B resolve this question of Osin, with conditions (1) and (2) indicating an analogue of the Higman embedding theorem which adds the condition of malnormality.  As suggested by the rest of these statements, though, much more can be said about the embeddings constructed for this purpose.  For example, condition (5) allows one to view Theorem A as a malnormal refinement of the theorem of Clapham mentioned above.

%
%
%
%



Recall that for a group $G$ with a finite generating set $X$, the \textit{word length} of an element $g\in G$ with respect to $X$, denoted $|g|_X$, is the minimal number of letters from $X\cup X^{-1}$ needed to produce a word that represents $g$. This defines a function $|\cdot|_X:G\to\N$ given by $g\mapsto|g|_X$.

In general, for a countable group $G$, two functions $\ell_1,\ell_2:G\to\N$ are said to be \textit{equivalent}, denoted $\ell_1\simeq\ell_2$, if there exist positive constants $c_1,c_2$ such that for all $g\in G$, $$c_1\ell_1(g)\leq \ell_2(g)\leq c_2\ell_1(g)$$
It is easy to see that the relation $\simeq$ is indeed an equivalence relation on functions $G\to\N$.  What's more, it is immediate that $|\cdot|_X\simeq|\cdot|_Y$ for any finite generating sets $X$ and $Y$ of a group.  As a result, if $G$ is finitely generated, then the notation $|\cdot|_G$ may be used to denote any function in the equivalence class of $|\cdot|_X$ for some fixed finite generating set $X$.

With this terminology in hand, condition (4) of Theorem A indicates that the malnormal Higman embedding is bi-Lipschitz, and so a quasi-isometric embedding.  Hence, Theorem A can be seen as a refinement of the result of Ol'shanskii in \cite{O97} mentioned above.

%
%
%

However, more care is needed when considering countable groups which are not finitely generated.  To this end, for any countable group $G$, a function $\ell:G\to\N$ is called a \textit{length function} if the following conditions hold:

\begin{enumerate}[label=(D{\arabic*})]

\item $\ell(g)=0$ if and only if $g=1$;

\item $\ell(g)=\ell(g^{-1})$ for all $g\in G$;

\item $\ell(gh)\leq\ell(g)+\ell(h)$ for all $g,h\in G$;

\item There exists $a>0$ such that $\#\{g\in G\mid\ell(g)\leq r\}\leq a^r$ for all $r\in\N$.

\end{enumerate}

Note that if $G$ is a subgroup of a finitely generated group $H$, then the restriction of $|\cdot|_H$ to $G$ is indeed a length function.

Conversely, for any countable group $G$ and any length function $\ell$, Ol'shanskii exhibited in \cite{O97.2} an embedding of $G$ into a finitely generated group $S$ such that the restriction of $|\cdot|_S$ to $G$ is equivalent to $\ell$, with $S$ recursively presented if $\ell$ is a computable function.  Hence, Corollary B can be viewed as a refinement of this result.


While this statement applies to any countable recursively presented group, its power can be underscored by considering what it says for finitely generated groups: If $G$ is a finitely generated group, then any `reasonable' function can be realized as the distortion of $G$ as a malnormal subgroup of a finitely presented group.  As an example, for any computable $\a\in[1,\infty)$, there exists a finitely presented group $G_\a$ containing a malnormal cyclic subgroup with distortion $n^\a$.



Finally, we study the congruence extension property, a characteristic of subgroups first introduced by Ol'shanskii in \cite{OSQ}.  A subgroup $G$ of a group $H$ satisfies the \textit{congruence extension property} (CEP) if for any epimorphism $\eps:G\to G_1$, there exists an epimorphism $\bar{\eps}:H\to H_1$ for some group $H_1$ containing $G_1$ as a subgroup and such that the restriction of $\bar{\eps}$ to $G$ is $\eps$. In this case, we write $G\leq_{CEP}H$ and say that $G$ is a \textit{CEP-subgroup} of $H$ or that $G$ is \textit{CEP-embedded} in $H$.

With this terminology, Theorem A and Corollary B can also be seen as refinements of the Higman embedding theorem which are CEP-embeddings.  Note that this is an innovation of this article, as this is the first such refinement of the Higman embedding theorem.

\medskip

\subsection{Proof of Corollary B} \

As its name indicates, Corollary B follows immediately from Theorem A and observations related to previous constructions, namely that of Ol'shanskii in \cite{O97.2}.

As discussed above, if $\ell:G\to\N$ is a computable length function on the countable group $G$, then Ol'shanskii constructs in \cite{O97.2} an embedding $\rho_{G,\ell}:G\hookrightarrow S$ such that $S$ is a finitely generated (indeed $2$-generated) recursively presented group and the restriction of $|\cdot|_S$ to $G$ is equivalent to $\ell$.  Let $\tau_S:S\hookrightarrow H$ be the embedding of $S$ into a finitely presented group $H$ as given by Theorem A.  The composition $\gamma=\tau_S\circ\rho_{G,\ell}$ thus produces an embedding of $G$ into a finitely presented group.

Now, by construction, the restriction of $|\cdot|_H$ to (the image of) $G$ is equivalent to the restriction of $|\cdot|_S$ to $G$, and so is equivalent to $\ell$. 

Further, note that it is shown in \cite{O16} (see Corollary 1.7) that $\rho_{G,\ell}$ is itself a malnormal embedding.  Hence, since $\leq_{mal}$ is evidently a transitive relation, $\gamma$ is itself a malnormal embedding.

Finally, note the following two reformulations of the definition of the congruence extension property which may be convenient for understanding and checking this condition:

\begin{enumerate}[label=(C\arabic*)]

\item $G$ is a CEP-subgroup of $H$ if and only if for any normal subgroup $N\triangleleft G$, there exists a normal subgroup $M\triangleleft H$ such that $M\cap G=N$

\item $G$ is a CEP-subgroup of $H$ if and only if for any subset $S\subseteq G$, $G\cap\gen{\gen{S}}^H=\gen{\gen{S}}^G$ (where $\gen{\gen{T}}^K$ denotes the normal closure of a subset $T$ of a group $K$).

\end{enumerate}

It is clear from (C1) that any retract of a group is a CEP-subgroup and that $\leq_{CEP}$ is a transitive relation.  But $\rho_{G,\ell}$ is itself a CEP-embedding (see Section 2.5 of \cite{OS01} for further discussion), and so $\gamma$ is sufficient for Corollary B.

\medskip

\subsection{Approach} \

The construction of the finitely presented groups of interest is through \textit{$S$-machines} (see Section 4.1 for a full definition of $S$-machine).

The $S$-machine was first defined by Sapir, Birget, and Rips in \cite{SBR} as a computational model carefully tailored to produce finitely presented groups with desired algebraic and geometric properties.  These groups arise from their associated $S$-machine in a canonical way, with defining relations that yield a group structure that `simulates' the computational structure of the machine (see \Cref{sec-the-groups} or \cite{W} for further discussion).

Using a computational model to construct a group satisfying desired properties is a fundamental technique for many seminal results in algorithmic group theory.  Indeed, this is the general approach to several classical solutions to the Higman embedding theorem (see for example the construction of Aanderaa in \cite{Aanderaa} or the exposition of Rotman in \cite{Rotman}).  

In a very rough sense (see \cite{OS01} for a detailed discussion), $S$-machines are novel in their ability to produce groups whose geometric and algorithmic properties are informed by the machine, while also crucially providing a robust computational structure (see \cite{SBR} for full details or \Cref{sec-M_2} for a cursory discussion).

%
%
%

For a concrete example of this point, the groups associated to non-deterministic Turing machines given in \cite{Aanderaa} are defined by `Baumslag-Solitar-type' relations, necessitating the groups to have at least exponential Dehn function (see Section 14); on the other hand, the commutator relations inherent to the presentations associated to an $S$-machine allow for these groups to have Dehn functions as low as quadratic, a point exploited in \cite{W}, \cite{O18}, \cite{OS19}, \cite{OS22}, and others.


However, these commutator relations seem to naturally preclude the use of $S$-machines for the constructions of the main theorems of this article, as they necessitate group elements that provide a counterexample to the malnormality of any embedded subgroup.  

Indeed, this obstacle necessitates a generalization of the computational structure, defined here in \Cref{sec-generalized-S-machine}.  This adaptation, simply termed \textit{generalized $S$-machines}, allows a computation to apply a predetermined automorphism of a free group to the words written on the tapes.  When these automorphisms take a particular form, this essentially combines the theory of $S$-machines with the `Baumslag-Solitar-type' relations found in sources like \cite{Aanderaa}.  In this case, the machine is termed a \textit{noisy $S$-machine}.  The result is an associated group whose structure is not very different from that of \cite{S11}, but that is fully defined in terms of a computational model, allowing more effective study through the associated machine.


This generalization is employed in one particular `step' of the main machine (see \Cref{sec-M_1}), particularly the only one that involves the letters which correspond to the image of the embedding.  Introducing this `noise' into the relational structure is enough to ensure the malnormality of the given embeddings (see \Cref{sec-annular-diagrams}).

As indicated above, the introduction of these `Baumslag-Solitar-type' relations means a `loss of control' on the Dehn function of the associated groups.  Accordingly, much less care is taken in this article in finding upper bounds on the area of circular diagrams over these presentations when compared to the detailed arguments made in previous sources (e.g in \cite{W}, \cite{O18}, and \cite{OS19}); some computable upper bound is necessary, though, for the proof of condition (5) in Theorem A.

With that said, the resemblance to the setting of $S$-machines allows for a similar treatment here.  For example, we again use the notion of \textit{$a$-cells}, first introduced by the author in \cite{W}, to study the embedding.  Hence, despite the loss of control on the Dehn function, other algebraic and geometric properties can be proved 
through similar means to those of previous settings, particularly \cite{W}.  Of course, the new types of relations in this generalization also introduce several new obstacles to just about every argument; for example, compare \Cref{sec-transposition} and \Cref{sec-distortion-diagrams} to their analogues in \cite{W}.

\medskip


\subsection{Organization of the article} \

What follows is a brief outline of the contents of this article.

\Cref{preliminaries} functions mainly to recall the definition of a diagram over the presentation of a group, the fundamental tool for the arguments of Section 3 and Sections 7-12.

In Section 3, we construct an initial embedding of a finitely generated recursively presented group into another such group which possesses some convenient combinatorial qualities (see for example \Cref{SQ lengths}).  This embedding is also shown to satisfy several key properties which reduce the proof of Theorem A to demonstrating an embedding of a group with these qualities.

Sections 4-6 serve to study the main computational structures of this construction.  Section 4 recalls the definition of $S$-machines and introduces the notion of generalized $S$-machines.  Several auxiliary generalized $S$-machines are then constructed and studied in Section 5, culminating with the construction and study of the noisy $S$-machine $\textbf{M}^\pazocal{L}$ in Section 6.

Several group presentations associated to a generalized $S$-machine are introduced in Section 7, with these relational structures arising in an analogous manner to that employed in \cite{W}.  The section culminates with an investigation of diagrams over these presentations, demonstrating properties shared by the presentations associated to any generalized $S$-machine.

In Sections 8-12, we study the group presentations associated to the noisy $S$-machine $\textbf{M}^\pazocal{L}$, using the properties established in Section 6 to verify sufficient conditions to ensure that the corresponding groups are suitable for the proof of Theorem A.

The final sections provide the proofs of the conditions comprising Theorem A, pulling together the group properties verified in Sections 7-12 and the initial embedding of Section 3.  Note, however, that the proofs presented in these sections seem to require different setups, an observation which may seem peculiar since our goal appears to be to prove the existence of a single embedding.  A major reason for this variation is the decidability of the Word Problem for the initial group $G$: If it is undecidable, then an obvious construction suffices, as (5) is immediately satisfied; if it is decidable, on the other hand, we must make some minor alterations to ensure that the finitely presented group it embeds into shares this property.  However, even the proof of condition (3) benefits from flexibility in the setup.  Because of this, many of the constructions outlined in Sections 5-12 are done so generally, allowing for the necessary pivots in the proofs that follow.

\medskip

\textbf{Acknowledgements.} The author expresses his deep gratitude to Alexander Ol'shanskii for his suggestions on this work.  The author is also thankful for the comments and advice of Mark Sapir and Denis Osin.  Finally, the author would like to thank Jingying Huang, Arman Darbinyan, and Bogdan Chornomaz for their helpful discussions.

\bigskip

\section{Preliminaries} \label{preliminaries}

\subsection{Diagrams over presentations} \label{sec-Diagrams} \

A vital tool for many of the arguments to come is the concept of \textit{van Kampen diagrams} over group presentations, a notion introduced by its namesake in 1933 \cite{v-K}. It is assumed that the reader is intimately acquainted with this concept, but some of the most important definitions are summarized below; for further reference, see \cite{O}, \cite{Lyndon-Schupp}, and \cite{S14}.

Let $G$ be a group with presentation $\gen{\pazocal{A}\mid\pazocal{R}}$. Suppose $\Delta$ is an oriented 2-complex homeomorphic to a disk equipped with a \textit{labelling function}, i.e a function $\lab:E(\Delta)\to\pazocal{A}\cup\pazocal{A}^{-1}\cup\{1\}$ which satisfies $\lab(\textbf{e}^{-1})=\lab(\textbf{e})^{-1}$ for any edge $\textbf{e}\in E(\Delta)$ (with, of course, $1^{-1}\equiv1$). The label of a path in $\Delta$ is defined in the obvious way, i.e $\lab(\textbf{e}_1\dots\textbf{e}_n)\equiv\lab(\textbf{e}_1)\dots\lab(\textbf{e}_n)$ (where `$\equiv$' denotes `visual' letter-for-letter equality). For any edge $\textbf{e}$ in $\Delta$, $\textbf{e}$ is called a $0$-edge if $\lab(\textbf{e})=1$; otherwise, $\textbf{e}$ is called a \textit{positive edge}.

Suppose that for each cell $\Pi$ of $\Delta$, one of the following is true:

\begin{enumerate}[label=({\arabic*})]

\item omitting the label of any zero edges, $\lab(\partial\Pi)$ is visually equal to a cyclic permutation of $R^{\pm1}$ for some $R\in\pazocal{R}$

\item $\partial\Pi$ consists of $0$-edges and exactly two positive edges $\textbf{e}$ and $\textbf{f}$, with $\lab(\textbf{e})\equiv\lab(\textbf{f})^{-1}$

\item $\partial\Pi$ consists only of $0$-edges. 

\end{enumerate}

\begin{figure}
\centering
\subcaptionbox{Positive cell corresponding to the relator $R=aba^{-1}b^{-1}$.}[0.33\textwidth]{\includegraphics[height=1.9in]{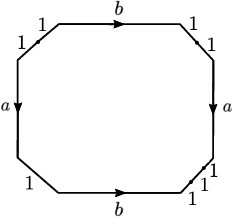}}\hfill
\subcaptionbox{$0$-cell of type (2), $a\in\pazocal{A}$.}[0.33\textwidth]{\includegraphics[height=1.9in]{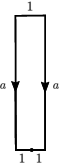}}\hfill
\subcaptionbox{$0$-cell of type (3).}[0.33\textwidth]{\includegraphics[height=1.9in]{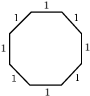}}\hfill
\caption{Cells in van Kampen diagrams}
\end{figure}

Then $\Delta$ is called a \textit{van Kampen diagram} (or simply a \textit{circular diagram}) over the presentation $\gen{\pazocal{A}\mid\pazocal{R}}$. The cells satisfying condition (1) above are called \textit{positive cells}, while the others are called \textit{0-cells}.

For any 0-cell of type (2), the positive edges $\textbf{e}$ and $\textbf{f}$ are called \textit{immediately adjacent}. In any diagram, two positive edges $\textbf{e}$ and $\textbf{f}$ are said to be \textit{adjacent} if there exists a sequence of edges $\textbf{e}=\textbf{e}_1,\textbf{e}_2,\dots,\textbf{e}_{k+1}=\textbf{f}$ such that $\textbf{e}_i$ and $\textbf{e}_{i+1}$ are immediately adjacent for $i=1,\dots,k$.

It is easy to see that the contour, $\partial\Delta$, of a circular diagram $\Delta$ has label equal to the identity in $G$. Conversely, van Kampen's Lemma (Lemma 11.1 of \cite{O}) ensures that a word $W$ over $\pazocal{A}^{\pm1}$ represents the identity of $G$ if and only if there exists a circular diagram $\Delta$ over the presentation $\gen{\pazocal{A}\mid\pazocal{R}}$ with $\lab(\partial\Delta)\equiv W$.

The \textit{area} of a diagram $\Delta$, denoted $\text{Area}(\Delta)$, is the number of positive cells it contains.  Further, the area of a word $W$ satisfying $W=1$ in $G$ is the minimal area of a circular diagram $\Delta$ satisfying $\lab(\partial\Delta)\equiv W$.

A \textit{0-refinement} of a diagram $\Delta$ is a diagram $\Delta'$ with homeomorphic underlying map obtained from $\Delta$ by the insertion/deletion of 0-edges and/or 0-cells. Note that a 0-refinement has the same area as the diagram from which it arises.

Let $\Delta$ be a circular diagram over $\gen{\pazocal{A}\mid\pazocal{R}}$ and $\Pi_1$, $\Pi_2$ be two positive cells in $\Delta$. Suppose there exists a simple path $\textbf{t}$ in $\Delta$ between the vertices $O_1,O_2$ of $\partial\Pi_1,\partial\Pi_2$, respectively, such that:

\begin{itemize}

\item $\lab(\textbf{t})=1$ in $F(\pazocal{A})$ (that is, the free group with basis $\pazocal{A}$), and 

\item $\lab(\partial\Pi_1)$ read starting at $O_1$ and $\lab(\partial\Pi_2)$ read starting at $O_2$ are mutually inverse 

\end{itemize}

Then $\Pi_1$ and $\Pi_2$ are called \textit{cancellable} in $\Delta$.

\begin{figure}[H]
\centering
\includegraphics[scale=1.25]{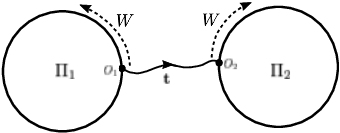}
\caption{Cancellable cells}
\label{cancellable}
\end{figure}

This terminology is justified by the ability to `remove' $\Pi_1$ and $\Pi_2$ from $\Delta$ through $0$-refinement, yielding a circular diagram $\Delta'$ satisfying $\lab(\partial\Delta')\equiv\lab(\partial\Delta)$ and $\text{Area}(\Delta')<\text{Area}(\Delta)$.


Naturally, a circular diagram is called \textit{reduced} if it has no pair of cancellable cells. By simply removing pairs of cancellable cells, any circular diagram over a presentation can be made reduced without affecting its contour label. This immediately implies a strengthened version of van Kampen's lemma: A word $W$ over $\pazocal{A}$ represents the identity in $G$ if and only if there exists a reduced circular diagram $\Delta$ over the presentation with $\lab(\partial\Delta)\equiv W$.

A \textit{Schupp diagram} (or simply \textit{annular diagram}) over the presentation $\gen{\pazocal{A}\mid\pazocal{R}}$ is defined in the analogous way, changing only that the underlying map is homeomorphic to an annulus rather than a disk.  Pairs of \textit{cancellable cells} in an annular diagram are defined in exactly the same way as for circular diagrams, again justified by the ability to use $0$-refinement to remove them without affecting the contour labels.

It is then an immediate consequence of van Kampen's lemma that two words $W$ and $V$ are conjugate in $G$ if and only if there exists a reduced annular diagram $\Delta$ with contour components $\textbf{p}$ and $\textbf{q}$ satisfying $\lab(\textbf{p})\equiv W$ and $\lab(\textbf{q})\equiv V^{-1}$.

A subdiagram of a diagram over a presentation is defined in the natural way, inheriting the labelling function from that of the diagram. However, it is convenient to restrict the terminology by assuming that subdiagrams are always circular, even if the original diagram is annular.

\begin{figure}[H]
\centering
\includegraphics[scale=0.9]{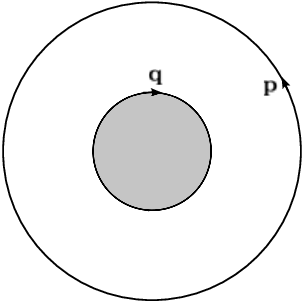}
\caption{Annular diagram}
\end{figure}

\medskip


\subsection{Parameters} \label{sec-parameters} \

The arguments spanning the rest of this paper are reliant on the \textit{highest parameter principle}, the dual to the \textit{lowest parameter principle} introduced in \cite{O}. For this, we introduce the relation $<<$ on parameters defined as follows.

If $\a_1,\a_2,\dots,\a_n$ are parameters with $\a_1<<\a_2<<\dots<<\a_n$, then for all $2\leq i\leq n$, it is understood that $\a_1,\dots,\a_{i-1}$ are assigned prior to the assignment of $\a_i$ and that the assignment of $\a_i$ is dependent on the assignment of its predecessors. The resulting inequalities are then understood as `$\a_i\geq$ (any expression used henceforth involving $\a_1,\dots,\a_{i-1}$)'.

Specifically, the assignment of parameters we use here is:
\begin{align*}
N<<C<<c_0<<L<<c_1<<\delta^{-1}<<K
\end{align*}

\medskip


\section{Initial Embedding} \label{sec-initial-embedding}

The first step toward addressing Theorem A is to `expand' a general recursive presentation with finite generating set, producing another such presentation into which the original presentation embeds.  Crucially, this new presentation will be shown (using diagrammatic arguments that resemble those of \cite{OSQ}) to satisfy key properties that are vital to later combinatorial calculations.



Let $Q$ be a finitely generated recursively presented group.  Let $\gen{Y\mid\pazocal{S}}$ be a presentation of $Q$ with $|Y|<\infty$ and assume that $\pazocal{S}$ satisfies the following three conditions:


\begin{enumerate}[label=(R\arabic*)]

\item $\pazocal{S}\subseteq Y^*$, i.e $\pazocal{S}$ is a set of \textit{positive} words in $Y$ in that each word is comprised entirely of letters from $Y$ (and not $Y^{-1}$)

\item $\pazocal{S}$ is a recursive subset of $Y^*$

\item The trivial word is not an element of $\pazocal{S}$

\end{enumerate}



As will be discussed further in Sections 13 and 14, these three conditions are not restrictive.

Setting $Y=\{y_1,\dots,y_m\}$, for all $i\in\{1,\dots,m\}$ let $Y_{C,i}=\{a_{1,i},\dots,a_{C,i}\}$, where $C$ is the parameter listed in Section 2.3, and let $Y_C=\cup Y_{C,i}$.

Further, for all $i$ let $A_i=a_{1,i}\dots a_{C,i}\in F(Y_{C,i})$. Then, for all $r=r(y_1,\dots,y_m)\in\pazocal{S}$, define the word $r_C=r(A_1,\dots,A_m)\in F(Y_C)$. For example, if $r=y_1y_m$, then $r_C=a_{1,1}\dots a_{C,1}a_{1,m}\dots a_{C,m}$. 

Letting $\pazocal{S}_C=\{r_C:r\in\pazocal{S}\}$, it follows that $\pazocal{S}_C$ is a set of positive words over $Y_C$ which is evidently recursive. Hence, $\gen{Y_C\mid\pazocal{S}_C}$ is a recursive presentation of a group $Q_C$ with $|Y_C|<\infty$.

Let $F$ be the subgroup of $F(Y_C)$ generated by $\pazocal{D}=\{A_1,\dots,A_m\}$. Since every letter of $Y_C$ appears once and only once in an element of $\pazocal{D}$, no cancellation occurs when forming products over $\pazocal{D}^{\pm1}$. Hence, $\pazocal{D}$ is a basis for the free subgroup $F$ of $F(Y_C)$.

For any normal subgroup $N\triangleleft F$, let $T_N$ be the set of non-trivial cyclically reduced words over $\pazocal{D}\cup\pazocal{D}^{-1}$ which are elements of $N$.  Note that since these words are cyclically reduced as words over $\pazocal{D}\cup\pazocal{D}^{-1}$, they are cyclically reduced as words over $Y_C\cup Y_C^{-1}$.  Further, note that $T_N$ is closed under taking inverses.  A diagram $\Delta$ over the presentation $\gen{Y_C\mid T_N}$ is then called a \textit{$T_N$-diagram}.

Let $L_N=\gen{\gen{N}}^{F(Y_C)}$, i.e the normal closure of $N$ in $F(Y_C)$. Then for words $W$ and $V$ over $Y_C^{\pm1}$, van Kampen's lemma implies the following statements:

\begin{itemize}

\item $W\in L_N$ if and only if there exists a circular $T_N$-diagram satisfying $\lab(\partial\Delta)\equiv W$.

\item $W$ and $V$ are conjugate in $F(Y_C)/L_N$ if and only if there exists an annular $T_N$-diagram with contour components $\textbf{p}$ and $\textbf{q}$ satisfying $\lab(\textbf{p})\equiv W$ and $\lab(\textbf{q})\equiv V^{-1}$.

\end{itemize}

For any positive cell $\Pi$ in a $T_N$-diagram $\Delta$, there exists a decomposition $\partial\Pi=\textbf{p}_1\dots \textbf{p}_s$ such that $\textbf{p}_i$ is labelled by $A_{j(i)}^{\eps_i}$ for some $j(i)\in\{1,\dots,m\}$ and $\eps_i\in\{\pm1\}$. In this case, $\textbf{p}_i$ is called an \textit{$F$-subpath} of $\partial\Pi$ and the vertices $(\textbf{p}_i)_-$ and $(\textbf{p}_i)_+$ are called \textit{entire vertices} of $\partial\Pi$.

If a subpath $\textbf{q}$ of a boundary component of $\Delta$ is labelled by an element of $F$, then the \textit{$F$-subpaths} and the \textit{entire vertices} of $\textbf{q}$ are defined analogously.

Suppose there exist positive cells $\Pi_1$ and $\Pi_2$ (perhaps $\Pi_1=\Pi_2$) in a $T_N$-diagram $\Delta$ and a path $\textbf{t}$ such that $\textbf{t}_-$ is an entire vertex of $\partial\Pi_1$, $\textbf{t}_+$ is an entire vertex of $\partial\Pi_2$, and $\lab(\textbf{t})$ is trivial in $F(Y_C)$. Then $\Pi_1$ and $\Pi_2$ are said to be \textit{compatible}. 

In this case, viewing $\lab(\partial\Pi_1)$ as starting at $\textbf{t}_-$ and $\lab(\partial\Pi_2)$ as starting at $\textbf{t}_+$, the label of the loop $\textbf{t}(\partial\Pi_2)^{-1}\textbf{t}^{-1}(\partial\Pi_1)^{-1}$ is freely conjugate to an element of $T_N$. So, if $\Pi_1\neq\Pi_2$, then after $0$-refinement one may excise from $\Delta$ a (circular) subdiagram containing the cells $\Pi_1$ and $\Pi_2$ and paste in its place a circular $T_N$-diagram consisting of exactly one positive cell, yielding a $T_N$-diagram with the same contour labels and area less by one.

If a subpath $\textbf{q}$ of a component of $\partial\Delta$ is labelled by an element of $F$, then the compatibility of $\textbf{q}$ and a positive cell $\Pi$ is defined analogously. In this case, one may use $0$-refinement to remove $\Pi$ from $\Delta$, obtaining a $T_N$-diagram $\Delta'$ with area less by one. However, for $\textbf{q}'$ the subpath of the component of $\partial\Delta'$ arising from $\textbf{q}$, $\lab(\textbf{q})$ and $\lab(\textbf{q}')$ are not equal in $F$. As $\lab(\partial\Pi)\in T_N$, though, $\lab(\textbf{q})$ and $\lab(\textbf{q}')$ do represent the same element of $F/N$.

A $T_N$-diagram $\Delta$ containing no pair of compatible cells is called \textit{$T_N$-reduced}. Note that any $T_N$-diagram can be made $T_N$-reduced by simply iterating the process of replacing pairs of compatible cells with single positive cells. Hence, the statements above can be refined in the following ways:

\begin{itemize}

\item $W\in L_N$ if and only if there exists a circular $T_N$-reduced diagram satisfying $\lab(\partial\Delta)\equiv W$.

\item $W$ and $V$ are conjugate in $F(Y_C)/L_N$ if and only if there exists an annular $T_N$-reduced diagram with contour components $\textbf{p}$ and $\textbf{q}$ satisfying $\lab(\textbf{p})\equiv W$ and $\lab(\textbf{q})\equiv V^{-1}$.

\end{itemize}

Suppose there exist positive cells $\Pi_1$ and $\Pi_2$ in a $T_N$-diagram $\Delta$ and edges $\textbf{e}_1\in\partial\Pi_1$ and $\textbf{e}_2\in\partial\Pi_2$ such that the labels of $\textbf{e}_1$ and $\textbf{e}_2$ are mutually inverse. Further, suppose there exists a path $\textbf{s}$ in $\Delta$ with $\textbf{s}_-=(\textbf{e}_1)_-$ and $\textbf{s}_+=(\textbf{e}_2)_+$ and so that $\lab(\textbf{s})$ is freely trivial. Then since any letter of $Y_C$ appears once and only once in any element of $\pazocal{D}$, the $F$-subpaths $\textbf{p}_1,\textbf{p}_2$ of $\partial\Pi_1,\partial\Pi_2$ containing $\textbf{e}_1,\textbf{e}_2$, respectively, must have inverse labels. Letting $\textbf{p}_1'$ be the initial subpath of $\textbf{p}_1$ with $(\textbf{p}_1')_+=(\textbf{e}_1)_-$ and $\textbf{p}_2'$ be the terminal subpath of $\textbf{p}_2$ with $(\textbf{p}_2')_-=(\textbf{e}_2)_+$, it follows that $\textbf{p}_1'\textbf{s}\textbf{p}_2'$ is a path between entire vertices of these cells with freely trivial label. Hence, $\Pi_1$ and $\Pi_2$ are compatible.

Thus, a $T_N$-reduced diagram is necessarily reduced. Moreover, for any $T_N$-reduced diagram $\Delta$ and any pair of distinct positive cells $\Pi_1$ and $\Pi_2$ in $\Delta$, if $\textbf{e}_1\in\partial\Pi_1$ and $\textbf{e}_2\in\partial\Pi_2$, then $\textbf{e}_1$ cannot be adjacent $\textbf{e}_2^{-1}$.

Similarly, adjacency implies the compatibility between a positive cell $\Pi$ and an appropriate subpath $\textbf{q}$ of a contour component of a $T_N$-reduced diagram $\Delta$.

\begin{lemma}[Compare to Theorem 2 of \cite{OSQ}] \label{SQ}

For any normal subgroup $N\triangleleft F$, $L_N\triangleleft F(Y_C)$ and satisfies $L_N\cap F=N$.

\end{lemma}

\begin{proof}

By definition, $L_N\triangleleft F(Y_C)$ and contains $N$.

Supposing $L_N\cap F\neq N$, there exists a circular $T_N$-reduced diagram $\Delta$ of minimal area satisfying $\lab(\partial\Delta)\in(L_N\cap F)\setminus N$.

As the label of the contour of any circular diagram with zero area is freely trivial, $\Delta$ must contain at least one positive cell.

Suppose there exists a positive cell $\Pi$ in $\Delta$ that is self-compatible. Then the entire vertices $o,o'$ of $\partial\Pi$ defining this compatibility partition $\partial\Pi=\textbf{q}_1\textbf{q}_2$ such that $\lab(\textbf{q}_i)\in F$. Letting $\textbf{t}$ be a path from $o$ to $o'$ such that $\lab(\textbf{t})=1$ in $F(Y_C)$, after 0-refinement the loop $\textbf{t}\textbf{q}_1$ can be assumed to bound a (circular) subdiagram $\Delta_1$ not containing $\Pi$ satisfying $\lab(\partial\Delta_1)=\lab(\textbf{q}_1)$ in $F(Y_C)$. As $\lab(\textbf{q}_1)$ cannot be trivial in $F(Y_C)$, $\Delta_1$ must contain at least one positive cell; further, $\lab(\partial\Delta_1)\in L_N$ by van Kampen's Lemma, so that the inductive hypothesis implies $\lab(\textbf{q}_1)\in N$. Since $\lab(\partial\Pi)\in N$, this also means that $\lab(\textbf{q}_2)\in N$. Letting $\Delta_2$ be the subdiagram bounded by the loop $\textbf{t}\textbf{q}_2^{-1}$, it follows that $\lab(\partial\Delta_2)=_{F(Y_C)}\lab(\textbf{q}_2)^{-1}\in N$. As a result, $\lab(\partial\Delta_2)$ is freely conjugate to an element of $T_N$, so that one may excise $\Delta_2$ from $\Delta$ and paste in its place a diagram containing exactly one positive cell. But this produces a circular $T_N$-diagram with the same contour label as $\Delta$ and strictly lesser area, contradicting the minimality of $\Delta$.

Hence, for any positive cell $\Pi$ of $\Delta$, every edge of $\partial\Pi$ is adjacent a boundary edge of $\Delta$. As a result, any positive cell is compatible with $\partial\Delta$, so that we may remove such a cell to produce a diagram $\Delta'$ with area less by one and such that $\lab(\partial\Delta)=\lab(\partial\Delta')$ in $F/N$. But then $\lab(\partial\Delta')\in(L_N\cap F)\setminus N$, so that the minimality of $\Delta$ is again contradicted.

\end{proof}

\begin{lemma} \label{SQ lengths 0}

Let $N\triangleleft F$.  If $W\in L_N$ is a non-trivial cyclically reduced word over $Y_C\cup Y_C^{-1}$, then there exists a subword of a cyclic permutation of $W$ which is a cyclic permutation of an element of $T_N$.  In particular, $|W|_{Y_C}\geq C$.


\end{lemma}

\begin{proof}

As noted above, $W\in L_N$ if and only if there exists a circular $T_N$-reduced diagram $\Delta$ such that $\lab(\partial\Delta)\equiv W$. Choose such a diagram $\Delta$ with minimal area.

As $W\neq1$ in $F(Y_C)$, the area of $\Delta$ must be at least 1.  Further, as in the proof of Lemma \ref{SQ}, no positive cell of $\Delta$ can be self-compatible.  Hence, for any positive cell $\Pi$ of $\Delta$, every edge of $\partial\Pi$ is adjacent to a boundary edge of $\Delta$.  

Thus, the diagram $\Delta$ as a map satisfies the small-cancellation condition $C'(0)$ (see Chapter 5 of \cite{Lyndon-Schupp}), so that Grindlinger's Lemma implies $\Delta$ contains a cell $\Pi$ such that the edges adjacent to $\partial\Pi$ form a subpath $\textbf{q}$ of $\partial\Delta$. 

Let $W'$ be the cyclic permutation of $W$ obtained from reading $\lab(\partial\Delta)$ starting at $\textbf{q}_-$.  Further, let $V\in T_N$ such that $\lab(\partial\Pi)\equiv V$.  Then, the subword $\lab(\textbf{q})$ of $W'$ is a cyclic permutation of $V$.  In particular, $|W|_{Y_C}\geq|\lab(\textbf{q})|_{Y_C}=C|V|_\pazocal{D}\geq C$.


\end{proof}

Note that if $N=\gen{\gen{\pazocal{S}_C}}^F$, then $L_N=\gen{\gen{\pazocal{S}_C}}^{F(Y_C)}$. Hence, Lemma \ref{SQ} implies: $$\gen{\gen{\pazocal{S}_C}}^{F(Y_C)}\cap F=\gen{\gen{\pazocal{S}_C}}^F$$
Letting $K=F\gen{\gen{\pazocal{S}_C}}^{F(Y_C)}\leq F(Y_C)$, it then follows that $$K/\gen{\gen{\pazocal{S}_C}}^{F(Y_C)}\cong F/\gen{\gen{\pazocal{S}_C}}^F\cong\gen{\pazocal{D}\mid\pazocal{S}_C}$$
By the theorem of von Dyck (Theorem 4.5 of \cite{O}), the map $Y\to\gen{\pazocal{D}\mid\pazocal{S}_C}$ defined by $y_i\mapsto A_i$ extends to an isomorphism $Q\to\gen{\pazocal{D}\mid\pazocal{S}_C}$. Hence, for $\tilde{Q}=K/\gen{\gen{\pazocal{S}_C}}^{F(Y_C)}\leq Q_C$, there exists an isomorphism $\varphi:Q\to \tilde{Q}$ given by $\varphi(y_i)=A_i$ for all $i$.


The following is then an immediate consequence of \Cref{SQ lengths 0}:

\begin{lemma} \label{SQ lengths}

Let $W$ be a non-trivial cyclically reduced word over $Y_C\cup Y_C^{-1}$ which represents the identity in $Q_C$.  Then there exists a non-trivial subword of a cyclic permutation of $W$ which is a cyclic permutation of an element of $\gen{\gen{\pazocal{S}_C}}^F$.  In particular, $|W|_{Y_C}\geq C$.

\end{lemma}

\begin{proof}

Set $N=\gen{\gen{\pazocal{S}_C}}^{F}$.  Then, since $L_N=\gen{\gen{\pazocal{S}_C}}^{F(Y_C)}$, $W$ represents the identity in $Q_C$ if and only if $W\in L_N$.  Hence, the statement follows from \Cref{SQ lengths 0}.

\end{proof}

\begin{lemma} \label{SQ quasi}

For all $w\in Q$, $|\varphi(w)|_{Y_C}=C|w|_Y$.

\end{lemma}

\begin{proof}

Let $k=|w|_Y$ and set $N=\gen{\gen{\pazocal{R}_C}}^F$.

Then, there exist $i_1,\dots,i_k\in\{1,\dots,m\}$ and $\eps_1,\dots,\eps_k\in\{\pm1\}$ such that $w=y_{i_1}^{\eps_1}\dots y_{i_k}^{\eps_k}$ in $Q$. So, $A_{i_1}^{\eps_1}\dots A_{i_k}^{\eps_k}=\varphi(w)$ in $Q_C$. Hence, $|\varphi(w)|_{Y_C}\leq Ck$.

Conversely, for any $f\in F$ such that $\varphi(w)=f$ in $\tilde{Q}\leq Q_C$, then for $j_1,\dots,j_\ell\in\{1,\dots,m\}$ and $\delta_1,\dots,\delta_\ell\in\{\pm1\}$ such that $f\equiv A_{j_1}^{\delta_1}\dots A_{j_\ell}^{\delta_\ell}$, then $y_{j_1}^{\delta_1}\dots y_{j_\ell}^{\delta_\ell}=w$. As a result, $\ell\geq k$ and so $|f|_{Y_C}\geq Ck$.

Now let $v$ be an arbitrary reduced word over $Y_C^{\pm1}$ such that $v=\varphi(w)$ in $Q_C$. Then for $f\in F$, $f=\varphi(w)$ in $Q_C$ if and only if $v^{-1}f\in L_N$, i.e if and only if there exists a circular $T_N$-reduced diagram $\Delta$ satisfying $\lab(\partial\Delta)\equiv v^{-1}f$. Choose such an $f\in F$ and corresponding diagram $\Delta$ such that $\Delta$ has minimal area. Then, partition $\partial\Delta=\textbf{p}\textbf{q}$ such that $\lab(\textbf{p})\equiv v^{-1}$ and $\lab(\textbf{q})\equiv f$.

If any positive cell of $\Delta$ is compatible with $\textbf{q}$, then $0$-refinement allows us to remove this cell to yield a circular $T_N$-reduced diagram $\Delta'$ satisfying $\lab(\partial\Delta')\equiv v^{-1}f'$ with $f'\in F$. But then $\text{Area}(\Delta')=\text{Area}(\Delta)-1$, contradicting the minimality of $\Delta$.

As a result, every edge of $\textbf{q}$ must be adjacent another boundary edge of $\Delta$. 

Suppose there exists a subpath $\textbf{e}_1\textbf{q}'\textbf{e}_2$ of $\textbf{q}$ such that $\textbf{e}_1$ and $\textbf{e}_2^{-1}$ are adjacent edges.  Then, let $\textbf{t}$ be a path consisting entirely of $0$-edges such that $\textbf{t}_-=(\textbf{e}_1)_-$ and $\textbf{t}_+=(\textbf{e}_2)_+$.  Using $0$-refinement, we may then assume that $\textbf{e}_1\textbf{q}'\textbf{e}_2$ and $\textbf{t}$ bound a subdiagram $\Gamma$.  As no edge on the boundary of a positive cell can be adjacent an edge of $\partial\Gamma$, $\Gamma$ must be a circular diagram over the free group $F(Y_C)$, and so $\lab(\partial\Gamma)$ is freely trivial.  But then $\lab(\textbf{e}_1\textbf{q}'\textbf{e}_2)$ is freely trivial, contradicting the assumption that $\lab(\textbf{q})$ is reduced.

Hence, every edge of $\textbf{q}$ is adjacent an edge of $\textbf{p}^{-1}$.  In particular, $|v|_{Y_C}\geq|f|_{Y_C}\geq Ck$.

\end{proof}

\begin{lemma} \label{SQ malnormal}

$\tilde{Q}\leq_{mal}Q_C$.

\end{lemma}

\begin{proof}

Let $N=\gen{\gen{\pazocal{S}_C}}^F$.

Supposing the statement is false, there exists an annular $T_N$-reduced diagram $\Delta$ with contour components $\textbf{p}$ and $\textbf{q}$ such that $\lab(\textbf{p}),\lab(\textbf{q})\in F\setminus N$ and there exists a path $\textbf{t}$ in $\Delta$ such that $\textbf{t}_-$ is an entire vertex of $\textbf{p}$, $\textbf{t}_+$ is an entire vertex of $\textbf{q}$, and $\lab(\textbf{t})L_N\notin\tilde{Q}$. Choose such a diagram $\Delta$ with minimal area.

If any positive cell $\Pi$ of $\Delta$ is compatible with $\textbf{p}$, then $\Pi$ may be removed to yield an annular $T_N$-reduced diagram $\Delta'$ with contour components $\textbf{p}'$ and $\textbf{q}$ satisfying $\lab(\textbf{p}')=\lab(\textbf{p})$ in $\tilde{Q}$. Further, using $0$-refinement, the path $\textbf{t}$ can be assumed to be undisturbed by this procedure, so that there exists a path $\textbf{t}'$ in $\Delta'$  between entire vertices of $\textbf{p}'$ and $\textbf{q}$ satisfying $\lab(\textbf{t}')\equiv\lab(\textbf{t})$. But then the existence of $\Delta'$ contradicts the minimality of $\Delta$.

Similarly, no positive cell of $\Delta$ can be compatible with $\textbf{q}$. 

In particular, since $\lab(\textbf{p})$ and $\lab(\textbf{q})$ are reduced words, then as in the proof of \Cref{SQ quasi} every edge of $\textbf{p}$ must be adjacent an edge of $\textbf{q}^{-1}$ and vice versa. As a result, there exists a path $\textbf{s}_1$ in $\Delta$ such that $(\textbf{s}_1)_-=\textbf{t}_-$, $(\textbf{s}_1)_+$ is an entire vertex of $\textbf{q}$, and $\lab(\textbf{s}_1)$ is freely trivial. 

Let $\textbf{s}_2$ be the subpath of $\textbf{q}$ such that $(\textbf{s}_2)_-=(\textbf{s}_1)_+$ and $(\textbf{s}_2)_+=\textbf{t}_+$. Then since the initial and terminal vertices of $\textbf{s}_2$ are entire, $\lab(\textbf{s}_2)\in F$.  

So, setting $\textbf{s}=\textbf{s}_1\textbf{s}_2$, $\lab(\textbf{s})\in F$ with $\textbf{s}_-=\textbf{t}_-$ and $\textbf{s}_+=\textbf{t}_+$.  Hence, there exists an integer $\ell$ such that $\lab(\textbf{t})=\lab(\textbf{p})^\ell\lab(\textbf{s})$ in $Q_C$ (see Lemma 11.4 of \cite{O}).

But $\lab(\textbf{p})^\ell\lab(\textbf{s})\in F\leq K$, contradicting the assumption that $\lab(\textbf{t})L_N\notin\tilde{Q}$.

\end{proof}

\bigskip


\section{Rewriting Systems}

\subsection{$S$-Machines} \label{sec-S machines} \

There are many equivalent interpretations of $S$-machines (for example, see \cite{S06} and \cite{CW}). Following the conventions of \cite{BORS}, \cite{O18}, \cite{OS01}, \cite{OSconj}, \cite{OS06}, \cite{OS19}, \cite{SBR}, \cite{W}, and others, we describe them here as rewriting systems for words over group alphabets. 

Let $(Y,Q)$ be a pair of finite sets with $Q=\sqcup_{i=0}^N Q_i$ and $Y=\sqcup_{i=1}^N Y_i$ for some positive integer $N$. For convenience of notation, set $Y_0=Y_{N+1}=\emptyset$ in this setting.

The elements of $Q\cup Q^{-1}$ are called \textit{state letters} or \textit{$q$-letters}, while those of $Y\cup Y^{-1}$ are \textit{tape letters} or \textit{$a$-letters}. The sets $Q_i$ and $Y_i$ are called the \textit{parts} of $Q$ and $Y$, respectively. Note that the parts of the state letters are typically represented by capital letters, while their elements are represented by lowercase.

For any reduced word $W\in F(Y\cup Q)$, define its \textit{$a$-length} $|W|_a$ as the number of $a$-letters that comprise it. The $q$-length of $W$ is defined similarly and is denoted $|W|_q$.

The \textit{language of admissible words} of $(Y,Q)$ is the collection of reduced words which are of the form $W\equiv q_0^{\eps_0}u_1q_1^{\eps_1}\dots u_kq_k^{\eps_k}$ where $q_i\in Q$, $\eps_i\in\{\pm1\}$, and each subword $q_{i-1}^{\eps_{i-1}}u_iq_i^{\eps_i}$ either:

\begin{enumerate}[label=({\arabic*})]

\item belongs to $(Q_{j(i)-1}F(Y_{j(i)})Q_{j(i)})^{\pm1}$;

\item has the form $quq^{-1}$ for $q\in Q_{j(i)}$ and $u\in F(Y_{j(i)+1})$; or

\item has the form $q^{-1}uq$ for $q\in Q_{j(i)}$ and $u\in F(Y_{j(i)})$

\end{enumerate}

In this case, the \textit{base} of $W$ is $\text{base}(W)\equiv Q_{j(0)}^{\eps_0}Q_{j(1)}^{\eps_1}\dots Q_{j(k)}^{\eps_k}$, where these letters are merely representatives of their corresponding parts, and $u_i$ is called the $Q_{j(i)}^{\eps_i}Q_{j(i+1)}^{\eps_{i+1}}$-sector of $W$. Note that the base of an admissible word $W$ need not be a reduced word over the corresponding symbols and that $W$ is permitted to have many sectors of the same name (for example, $W$ may contain many $Q_0Q_1$-sectors). 

The base $Q_0Q_1\dots Q_N$ is called the \textit{standard base} of $(Y,Q)$. An admissible word with the standard base is called a \textit{configuration}. 

Now, let $Q(\theta)$ be a subset of $Q$ such that $Q(\theta)\cap Q_i$ is a singleton for each $i$. If $Q(\theta)\cap Q_i=\{q_i\}$, then to $q_i$ there is an associated word $u_iq_i'v_{i+1}$ where $q_i'\in Q_i$, $u_i\in F(Y_i)$, and $v_{i+1}\in F(Y_{i+1})$.

Further, let $Y(\theta)=\sqcup Y_j(\theta)$ be some subset of $Y$ such that $Y_j(\theta)\subseteq Y_j$. For each $j$, $Y_j(\theta)$ is called the \textit{domain} of $\theta$ in the corresponding sector of the standard base. 

In this case, $\theta$ is called an \textit{$S$-rule} of $(Y,Q)$ and is denoted 
$$\theta=[q_0\to u_0q_0'v_1, \ q_1\to u_1q_1'v_2, \dots, q_N\to u_Nq_N'v_{N+1}]$$ 
Note that this notation does not fully specify the rule, as the domain $Y(\theta)$ is not included. 

Suppose $W$ is an admissible word with all its state letters contained in $Q(\theta)\cup Q(\theta)^{-1}$ and all its tape letters contained in $Y(\theta)\cup Y(\theta)^{-1}$. Then, $W$ is said to be \textit{$\theta$-admissible} and $W\cdot\theta$ is defined to be the admissible word resulting from the simultaneously: 

\begin{itemize}

\item for all $q_i\in Q(\theta)$, replacing every occurrence of $q_i^{\pm1}$ in $W$ with $(u_iq_i'v_{i+1})^{\pm1}$, and

\item reducing/trimming the resulting word so that it is again admissible.

\end{itemize}

An important note to stress is that the application of an $S$-rule results in a reduced word, i.e reduction is not a separate step in the process.

If the $i$-th part of $\theta$ is $U_i\to V_i$ and $Y_{i+1}(\theta)=\emptyset$, then this part of the rule is denoted $U_i\xrightarrow{\ell}V_i$ and $\theta$ is said to \textit{lock} the $Q_{i}Q_{i+1}$-sector of the standard base. 


Note that every $S$-rule $\theta$ has a natural inverse, namely the $S$-rule $\theta^{-1}$ with $Y(\theta^{-1})=Y(\theta)$ and $\theta^{-1}=[q_0'\to u_0^{-1}q_0v_1^{-1},\dots,q_N'\to u_N^{-1}q_Nv_{N+1}^{-1}]$. Observe that $W$ is $\theta$-admissible, then $W\cdot\theta$ is $\theta^{-1}$-admissible with $(W\cdot\theta)\cdot\theta^{-1}\equiv W$.

An \textit{$S$-machine} $\textbf{S}$ with hardware $(Y,Q)$ is defined to be a rewriting system whose software is a finite symmetric set of $S$-rules $\Theta(\textbf{S})=\Theta$, i.e so that $\theta\in\Theta$ if and only if $\theta^{-1}\in\Theta$.  It is convenient to then partition $\Theta$ into two disjoint sets, $\Theta^+$ and $\Theta^-$, such that $\theta\in\Theta^+$ if and only if $\theta^{-1}\in\Theta^-$. The elements of $\Theta^+$ are called the \textit{positive rules} and those of $\Theta^-$ the \textit{negative rules}.

For $t\geq0$, suppose $W_0,\dots,W_t$ are admissible words with the same base such that there exist $\theta_1,\dots,\theta_t\in\Theta$ satisfying $W_{i-1}\cdot\theta_i\equiv W_i$ for all $1\leq i\leq t$. Then the sequence of applications of rules $\pazocal{C}:W_0\to\dots\to W_t$ is called a \textit{computation} of \textit{length} or \textit{time} $t$ of $\textbf{S}$. The word $H=\theta_1\dots\theta_t$ is called the \textit{history} of $\pazocal{C}$ and the notation $W_t\equiv W_0\cdot H$ is used to represent the computation.

A computation is called \textit{reduced} if its history is a reduced word over $\Theta$. Note that every computation can be made reduced without changing the initial and final admissible words of the computation by simply removing consecutive mutually inverse rules.

Typically, it is assumed that each part of the state letters contains two (perhaps the same) fixed elements, called the \textit{start} and \textit{end} state letters. A configuration is called a \textit{start} (or \textit{end}) configuration if all its state letters are start (or end) letters.

A \textit{recognizing} $S$-machine is one with specified sectors called the \textit{input sectors}. If a start configuration has all sectors empty except for the input sectors, then it is called an \textit{input configuration} and its projection onto $Y\cup Y^{-1}$ (i.e its image under the map that sends each state letter to $1$ and each tape letter to itself) is called its \textit{input}. The end configuration with every sector empty is called the \textit{accept configuration}.

A configuration $W$ is \textit{accepted} by a recognizing $S$-machine if there is an \textit{accepting computation}, i.e a computation with initial configuration $W$ and final configuration the accept configuration. If $W$ is an accepted input configuration with input $u$, then $u$ is also said to be \textit{accepted}.

If the configuration $W$ is accepted by the $S$-machine $\textbf{S}$, then $T(W)$ is the minimal time of its accepting computations. For a recognizing $S$-machine \textbf{S}, its \textit{time function} is $$T_{\textbf{S}}(n)=\max\{T(W): W\text{ is an accepted input configuration of } \textbf{S}, \ |W|_a\leq n\}$$



If two recognizing $S$-machines have the same language of accepted words and $\Theta$-equivalent time functions, then they are said to be \textit{equivalent}.

The following simplifies how one approaches the rules of a recognizing $S$-machine.

\begin{lemma}[Lemma 2.1 of \cite{O18}] \label{simplify rules}

Every recognizing $S$-machine $\textbf{S}$ is equivalent to a recognizing $S$-machine such that for every part $q_i\to u_iq_i'v_{i+1}$ of every rule, $\|u_i\|+\|v_{i+1}\|\leq1$.

%
%
%
%

\end{lemma}

Through the rest of our discussion of computational models, we will often use copies of words over disjoint alphabets. To be precise, let $A$ and $B$ be disjoint alphabets, $W\equiv a_1^{\eps_1}\dots a_k^{\eps_k}$ with $a_i\in A$ and $\eps_i\in\{\pm1\}$, and $\varphi:\{a_1,\dots,a_k\}\to B$ be an injection. Then the \textit{copy} of $W$ over the alphabet $B$ formed by $\varphi$ is the word $W'\equiv\varphi(a_1)^{\eps_1}\dots\varphi(a_k)^{\eps_k}$. Typically, the injection defining the copy will be contextually clear.

Alternatively, a copy of an alphabet $A$ is a disjoint alphabet $A'$ which is in one-to-one correspondence with $A$. For a word over $A$, its copy over $A'$ is defined by the (fixed) correspondence between the alphabets.

\medskip

\subsection{Generalized $S$-machines} \label{sec-generalized-S-machine} \

We now introduce a modification to the definition of $S$-machines, permitting the rewriting to also take place within the tapes. The motivation of this alteration will be made clear by the definitions of the associated groups.

Let $(Y,Q)$ be a pair of finite sets with $Y=\sqcup_{i=1}^N Y_i$ and $Q=\sqcup_{i=0}^N Q_i$.  As in the definition of $S$-rule, let $Q(\theta)$ be a subset of $Q$ with $Q(\theta)\cap Q_i=\{q_i\}$ for each $i$. In this environment, however, for each $i$, in place of a domain $Y_i(\theta)$ we assign two finite subsets $X_i(\theta),Z_i(\theta)\subseteq F(Y_i)$ which form bases of free subgroups of $F(Y_i)$ and such that there exists a bijection $f_{\theta,i}:X_i(\theta)\to Z_i(\theta)$ extending to an isomorphism $\widetilde{f}_{\theta,i}:\gen{X_i(\theta)}\to\gen{Z_i(\theta)}$.  Finally, let $u_i,v_i\in \gen{Z_i(\theta)}$ for all $i$.

Then the \textit{generalized $S$-rule} $\theta$ is defined as the rewriting rule denoted
$$\theta=[q_0\to q_0'v_1, \ q_1\to u_1q_1'v_2, \ \dots, \  q_N\to u_Nq_N']$$
Note that, similar to how the notation for $S$-rules does not specify the domain of a rule, this notation does not capture the subsets $X_i(\theta)$ and $Z_i(\theta)$ or the bijection $f_{\theta,i}$.

Let $w\in\gen{X_i(\theta)}$ for some $i\in\{1,\dots,N\}$.  Since $X_i(\theta)$ forms a basis for the corresponding subgroup of $F(Y_i)$, there must exist $x_1,\dots,x_l\in X_i(\theta)$ and $\delta_1,\dots,\delta_l\in\{\pm1\}$ such that $w=_{F(Y_i)} x_1^{\delta_1}\dots x_k^{\delta_l}$.  In this case, we define the \textit{$\theta$-length of $w$}, $l_\theta(w)$, to be the value $l$, i.e $l_\theta(w)=|w|_{X_i(\theta)}$.  Note that if $X_i(\theta)\subseteq Y_i$, then $l_\theta(w)=\|w\|$.

Now suppose $W\equiv p_0w_1p_1\dots p_{k-1}w_kp_k$ is an admissible word where for each $i$, $p_i\in Q(\theta)\cup Q(\theta)^{-1}$ and $w_i\in\gen{X_{j(i)}(\theta)}$ for some $j(i)\in\{1,\dots,N\}$. Then $W$ is said to be \textit{$\theta$-admissible}, with the application $W\cdot\theta$ taken as the admissible word resulting from simultaneously doing the following:

\begin{itemize}

\item for all $i$, replace $w_i$ with $\widetilde{f}_{\theta,j(i)}(w_i)$,

\item for all $i$, replace $p_i$ with $(u_jq_j'v_{j+1})^{\pm1}$ where $p_i=q_j^{\pm1}$, and

\item reduce/trim the resulting word so that it is again admissible.

\end{itemize}

The \textit{$\theta$-length} of the $\theta$-admissible word $W$ is taken to be $l_\theta(W)=(k+1)+\sum_{i=1}^k l_\theta(w_i)$.  As above, note that if $X_{j(i)}(\theta)\subseteq Y_{j(i)}$ for all $i$, then $l_\theta(W)$ is simply $\|W\|$.

As with $S$-rules, it is important to stress that the application of a generalized $S$-rule immediately results in an admissible word. 

Also, note that if $X_i(\theta)=Z_i(\theta)\subseteq Y_i$ and $f_{\theta,i}$ is the identity map for each $i$, then $\theta$ can be regarded as an $S$-rule with domain $Y_i(\theta)=X_i(\theta)$.  Thus, every $S$-rule can be viewed as a generalized $S$-rule.

We then extend the definition of locked sector to say that the generalized $S$-rule $\theta$ \textit{locks} the $Q_iQ_{i+1}$-sector of the standard base if $X_i(\theta)=Z_i(\theta)=\emptyset$ (and so $\widetilde{f}_{\theta,i}$ is the identity map on the trivial group). As with $S$-rules, this is denoted by $q_i\xrightarrow{\ell}q_i'$ in the definition of $\theta$.

Next, we define the inverse $\theta^{-1}$ of the generalized $S$-rule $\theta$. For this, let  $Q(\theta^{-1})$ be the subset of $Q$ with $Q(\theta^{-1})\cap Q_i=\{q_i'\}$ for all $i$, let $X_i(\theta^{-1})=Z_i(\theta)$ and $Z_i(\theta^{-1})=X_i(\theta)$ for all $i$, and let $f_{\theta^{-1},i}=f_{\theta,i}^{-1}$ for all $i$.  Note then that for any $z_i\in\gen{Z_i(\theta)}$, $\widetilde{f}_{\theta^{-1},i}(z_i)\in\gen{X_i(\theta)}=\gen{Z_i(\theta^{-1})}$.  Thus, we may define:
$$\theta^{-1}=[q_0'\to q_0\widetilde{f}_{\theta^{-1},1}(v_1^{-1}), \ q_1'\to \widetilde{f}_{\theta^{-1},1}(u_1^{-1})q_1\widetilde{f}_{\theta^{-1},2}(v_2^{-1}), \dots, \ q_N'\to\widetilde{f}_{\theta^{-1},N}(u_N^{-1})q_N]$$
Note that inversion is indeed an involutional operation on generalized $S$-rules, i.e $(\theta^{-1})^{-1}=\theta$.

Following the definition of $S$-machine, a \textit{generalized $S$-machine} is then a triple $(Y,Q,\Theta)$ where $\Theta=\Theta^+\sqcup\Theta^-$ is a symmetric set of generalized $S$-rules.

The following statement is critical to our study of computations of generalized $S$-machines.

\begin{lemma} \label{inverse rules}

Suppose $\theta=[q_0\to q_0'v_1, \ q_1\to u_1q_1'v_2, \ \dots, \  q_N\to u_Nq_N']$ is a generalized $S$-rule and $W$ is $\theta$-admissible. Then $W\cdot\theta$ is $\theta^{-1}$-admissible with $(W\cdot\theta)\cdot\theta^{-1}\equiv W$.

\end{lemma}

\begin{proof}

It suffices to prove this for admissible words with two-letter base.

Say the base of $W$ is $Q_{i-1}Q_i$, so that $W\equiv q_{i-1}w_iq_i$ for some $w_i\in \gen{X_i(\theta)}$. Then it follows that $W\cdot\theta=q_{i-1}' v_i \widetilde{f}_{\theta,i}(w_i) u_iq_i'$. 

Since $v_i\widetilde{f}_{\theta,i}(w_i)u_i\in \gen{Z_i(\theta)}=\gen{X_i(\theta^{-1})}$, $W\cdot\theta$ is $\theta^{-1}$-admissible with 
\begin{align*}
(W\cdot\theta)\cdot\theta^{-1}&=(q_{i-1}\widetilde{f}_{\theta^{-1},i}(v_i^{-1}))\cdot \widetilde{f}_{\theta^{-1},i}(v_i\widetilde{f}_{\theta,i}(w_i)u_i)\cdot (\widetilde{f}_{\theta^{-1},i}(u_i^{-1})q_i) \\
&=q_{i-1}\widetilde{f}_{\theta^{-1},i}(v_i)^{-1}\cdot\widetilde{f}_{\theta^{-1},i}(v_i)w\widetilde{f}_{\theta^{-1},i}(u_i)\cdot\widetilde{f}_{\theta^{-1},i}(u_i)^{-1}q_i \\
&= q_{i-1}wq_i\equiv W
\end{align*}
The other cases, i.e where the base of $W$ is unreduced, are proved in a similar manner.

\end{proof}

It will prove useful in the sequel to consider a weakened version of computation in regards to generalized $S$-machines, called \textit{semi-computations}.

Given a reduced word $w\in F(Y_i)$, $w$ is said to be \textit{$\theta$-applicable} for $\theta\in\Theta$ if $w\in\gen{X_i(\theta)}$. Then, the application of $\theta$ to $w$ is taken to be $w\cdot\theta=\widetilde{f}_{\theta,i}(w)$.  

Note that an analogue of \Cref{inverse rules} in this setting is immediate by construction, i.e $w\cdot\theta$ is $\theta^{-1}$-applicable with $(w\cdot\theta)\cdot\theta^{-1}\equiv w$.  As $\widetilde{f}_{\theta,i}$ is an isomorphism, $l_{\theta^{-1}}(w\cdot\theta)=l_\theta(w)$.  

As with the concept of computation, this generalizes naturally to the notion of semi-computation: If $w_0,w_1,\dots,w_t\in F(Y_i)$ and $\theta_1,\dots,\theta_t\in\Theta$ such that $w_{i-1}\cdot\theta_i=w_i$ for $i=1,\dots,t$, then there is a corresponding \textit{semi-computation in the $Q_{i-1}Q_i$-sector}, denoted $\pazocal{S}:w_0\to w_1\to\dots\to w_t$. The history of $\pazocal{S}$ is defined to be the word $H\equiv\theta_1\dots\theta_t$ and $\pazocal{S}$ is called reduced if $H$ is a reduced word.

Note that semi-computations can be defined in the same way for $S$-machines; however, in that setting, semi-computations are merely constant sequences.

Moreover, for the trivial element $1\in F(Y_i)$ and $\theta\in\Theta$, $1\cdot\theta=1$.  So, for any semi-computation $\pazocal{S}:w_0\to\dots\to w_t$ in the $Q_{i-1}Q_i$-sector, $w_j=1$ for some $j$ if and only if $w_j=1$ for all $j$, in which case $\pazocal{S}$ is called a \textit{trivial semi-computation}.

Hence, the next statement is an immediate consequence of the definition of a locked sector:

\begin{lemma} \label{semi locked sectors}

Let $\pazocal{S}:w_0\to\dots\to w_t$ be a semi-computation in the $Q_{i-1}Q_i$-sector with history $H\equiv\theta_1\dots\theta_t$.  If there exists $j\in\{1,\dots,t\}$ such that $\theta_j$ locks the $Q_{i-1}Q_i$-sector, then $\pazocal{S}$ is a trivial semi-computation.

\end{lemma}

\medskip


\subsection{Noisy $S$-machines} \label{sec-noisy} \

We now introduce a specific class of generalized $S$-machines, termed \emph{noisy $S$-machines}, by restricting the possible non-identity bijections that can comprise the set of generalized $S$-rules.  As will be illustrated in \Cref{sec-bands and annuli}, this will facilitate the study of the diagrams over the associated groups, assuring the existence of a band-like structure resembling that used to study the groups constructed by classical $S$-machines.  Indeed, in practice all of our generalized $S$-machines will be noisy $S$-machines.

Consider a generalized $S$-machine $\pazocal{S}$ with hardware $(Y,Q)$, where $Q=\sqcup_{i=0}^N Q_i$ and $Y=\sqcup_{i=1}^N Y_i$, and software $\Theta=\Theta^+\sqcup\Theta^-$.  For each $1\leq i\leq N$, fix disjoint subsets $\pazocal{K}_i$, $\pazocal{M}_i$, $\pazocal{N}_i\subseteq Y_i$ along with a bijection $\varphi_i:\pazocal{K}_i\to\pazocal{M}_i$.

For each $1\leq i\leq N$, suppose that for each $\theta\in\Theta^+$ one of the following conditions hold:

\begin{enumerate}

\item $f_{\theta,i}={\rm{id}}_{Y_i(\theta)}$ for some subset $Y_i(\theta)\subseteq Y_i$

\item $f_{\theta,i}=\varphi_i$

\item The domain of $f_{\theta,i}$ is $\pazocal{M}_i\sqcup\pazocal{N}_i$ with:
\begin{itemize}

\item $f_{\theta,i}|_{\pazocal{N}_i}={\rm{id}}_{\pazocal{N}_i}$

\item for all $m_i\in\pazocal{M}_i$, there exists $v(\theta,m_i)\in F(\pazocal{N}_i)$ such that $f_{\theta,i}(m_i)=v(\theta,m_i)\cdot m_i$.

\end{itemize}

\end{enumerate}

Let $\Theta_i^+$ be the subset of $\Theta^+$ consisting of all generalized rules for which (3) holds and define the set $S_i=\{v(\theta,m_i)\mid\theta\in\Theta_i^+,~m_i\in\pazocal{M}_i\}$.  Then, $\pazocal{S}$ is called a \textit{noisy $S$-machine} if each $S_i$ forms a basis for a free subgroup of $F(\pazocal{N}_i)$.


Note that for $\theta\in\Theta_i^+$, letting $\pazocal{M}_{i,\theta}=\{v(\theta,m_i)\cdot m_i:m_i\in\pazocal{M}_i\}$, then $Z_i(\theta)=\pazocal{M}_{i,\theta}\sqcup\pazocal{N}_i$.  But then one can see that $Z_i(\theta)$ forms a basis for $F(\pazocal{M}_i\sqcup\pazocal{N}_i)$ via sequence of Tietze transformations, and so in particular $\gen{X_i(\theta)}=\gen{Z_i(\theta)}$.  Hence, if $\theta=[q_0\to q_0'v_1, \ q_1\to u_1q_1'v_2, \ \dots, \  q_N\to u_Nq_N']$, then $u_i$ and $v_i$ are permitted to be any reduced words over $\pazocal{M}_i\sqcup\pazocal{N}_i$.

Informally, the letters from the set $\pazocal{M}_i$ have some meaning while those of $\pazocal{N}_i$ are auxiliary letters introduced to control the relations of the associated group.  At each transition, the words $v(\theta,m_i)$ introduce a layer of `noise' next to the letters of $\pazocal{M}_i$, encoded with a marker of the rule that is being applied.  In the associated groups, this noise necessitates that single rules do not conjugate a word over $\pazocal{M}_i$ to another such word, while the condition that $S_i$ forms a basis for a free group enforces that no non-empty reduced string of rules does this.

Finally, observe that the rules $\theta$ which satisfy (1) in each sector correspond to $S$-rules in the traditional sense.  Hence, any $S$-machine can be viewed as a noisy $S$-machine.

Thus, as in the previous section, we adopt much of the terminology of $S$-machines for noisy $S$-machines.  For example, it is clear what is meant for such a rule to \textit{lock} a particular sector.

\bigskip


\section{Auxiliary Machines}

In this section, several machines are constructed with respect to some fixed finite non-empty alphabet $\pazocal{A}$ and recursive set $\pazocal{L}\subseteq\pazocal{A}^*$ of positive words over $\pazocal{A}$.  These sets are treated generally until Sections 13-14, as the proofs therein require different setups.

However, it is critical to note that the relevant contexts call for $\pazocal{A}$ to be taken to be an alphabet whose cardinality is bounded above by a linear function of $C$, justifying the parameter assignments in the sections that follow. 



\medskip

\subsection{The machine $\textbf{M}_1^\pazocal{A}$} \label{sec-M_1} \

The first machine in this construction is the noisy $S$-machine $\textbf{M}_1^\pazocal{A}$ that will assure the malnormality of the embedding of Theorem A.  Note the naming of the machine indicates that its makeup only depends on the alphabet $\pazocal{A}$ and not the specific language $\pazocal{L}$.

Let $\pazocal{A}_1$ and $\pazocal{A}_2$ be copies of the alphabet $\pazocal{A}$ given by the bijections $\varphi_i:\pazocal{A}\to\pazocal{A}_i$.  For simplicity, denote $\varphi_i(a)=a_i$ for each $a\in\pazocal{A}$.  Let $\widetilde{\varphi}_i:F(\pazocal{A})\to F(\pazocal{A}_i)$ be the isomorphism induced by $\varphi_i$.

Further, let $\pazocal{B}=\{b_1,b_2\}$ be a set of auxiliary letters.

Then, the hardware of $\textbf{M}_1^\pazocal{A}$ is $(Y_1^{\pazocal{A}}\sqcup Y_2^{\pazocal{A}},Q_0^{\pazocal{A}}\sqcup Q_1^{\pazocal{A}}\sqcup Q_2^{\pazocal{A}})$, where:

\begin{itemize}

\item $Q_i^{\pazocal{A}}=\{q_i\}$ for $i=0,1,2$

\item $Y_1^\pazocal{A}=\pazocal{A}\sqcup\pazocal{A}_1\sqcup\pazocal{B}$ and $Y_2^\pazocal{A}=\pazocal{A}_2$.

\end{itemize}

In terms of the terminology defining noisy $S$-machines in \Cref{sec-noisy}, the disjoint subsets $\pazocal{A}$, $\pazocal{A}_1$, and $\pazocal{B}$ function as $\pazocal{K}_1$, $\pazocal{M}_1$, and $\pazocal{N}_1$, respectively.  In the $Q_1^\pazocal{A}Q_2^\pazocal{A}$-sector, the defining disjoint subsets are taken to be empty, {\frenchspacing i.e. $\pazocal{K}_2=\pazocal{M}_2=\pazocal{N}_2=\emptyset$}.


Let $D_\pazocal{A}=4|(\pazocal{A}\sqcup\pazocal{B})\times\pazocal{A}|=4|\pazocal{A}|(|\pazocal{A}|+2)$ and fix a bijection $\eta_\pazocal{A}:(\pazocal{A}\sqcup\pazocal{B})\times\pazocal{A}\to\{1,\dots,D_\pazocal{A}/4\}$. Then, for $y\in\pazocal{A}\sqcup\pazocal{B}$ and $a\in\pazocal{A}$, define $v(y,a)=b_1^k(b_2b_1)^{D_\pazocal{A}-2k}b_2^k\in F(\pazocal{B})$ for $k=\eta_\pazocal{A}(y,a)$ . Note that for all $(y,a)\in(\pazocal{A}\sqcup\pazocal{B})\times\pazocal{A}$, $\|v(y,a)\|=D_\pazocal{A}$.

Let $S_\pazocal{A}=\{v(y,a):(y,a)\in(\pazocal{A}\sqcup\pazocal{B})\times\pazocal{A}\}$.  The next statement is clear by construction.


\begin{lemma} \label{free subgroup}

For $(y_1,a_1),(y_2,a_2)\in(\pazocal{A}\sqcup\pazocal{B})\times\pazocal{A}$ and $\eps_1,\eps_2\in\{\pm1\}$, either $(y_1,a_1)=(y_2,a_2)$ and $\eps_1=-\eps_2$ or less than $\frac{1}{4}D_\pazocal{A}$ letters of $v(y_j,a_j)$ are cancelled in the product $v(y_1,a_1)^{\eps_1}\cdot v(y_2,a_2)^{\eps_2}$. In particular, $S_\pazocal{A}$ is a basis for a (free) subgroup of $F(\pazocal{B})$.

\end{lemma}

In light of \Cref{free subgroup}, we may now define the software of $\textbf{M}_1^\pazocal{A}$ as follows:

The positive generalized $S$-rules of $\textbf{M}_1^\pazocal{A}$ correspond to the elements of $\pazocal{A}\sqcup\pazocal{B}$, with the rule corresponding to $y\in\pazocal{A}\sqcup\pazocal{B}$ denoted $\theta_y$ and given by:
\begin{itemize}

%

\item 
$\theta_{b_i}=[q_0\to q_0, \ q_1\to b_i^{-1}q_1, \ q_2\to q_2]$ for $i=1,2$.

\medskip

\item 
$\theta_a=[q_0\to q_0, \ q_1\to a_1^{-1}q_1a_2, \ q_2\to q_2]$ for $a\in\pazocal{A}$.

\end{itemize}

As is custom, though, this notation does not describe these rules' domains or how they act on the tape words.  Recall that per the definition of the positive generalized $S$-rules in a noisy $S$-machine outlined in \Cref{sec-noisy}, each rule takes one of three forms in each sector; for this machine, for every $y\in\pazocal{A}\sqcup\pazocal{B}$ this is given by:

\begin{itemize}

\item In the $Q_0^\pazocal{A}Q_1^\pazocal{A}$-sector, $\theta_y$ is of form (3), with $v(\theta_y,a_1)=v(y,a)$ for all $a\in\pazocal{A}$.

\item In the $Q_1^\pazocal{A}Q_2^\pazocal{A}$-sector, $\theta_y$ is of form (1), with $Y_2(\theta_y)=\pazocal{A}_2=Y_2^\pazocal{A}$.

\end{itemize}

Note that this first condition relating to the $Q_0^\pazocal{A}Q_1^\pazocal{A}$-sector seems to indicate that the rules have no interaction with the subset $\pazocal{A}$ of the tape alphabet $Y_1^\pazocal{A}$.  Indeed, this subset is superfluous to the construction of this machine and is included here simply to match the definition in \Cref{sec-noisy}.  However, $\pazocal{A}$ will come into play crucially in the next section through the construction of the main machines, and will indeed be the image of the generators in the embeddings constructed for the proof of Theorem A.

By construction, any admissible word of $\textbf{M}_1^\pazocal{A}$ is $\theta_y^{\pm1}$-admissible for every $y$ as long as it contains no tape letter from the subset $\pazocal{A}\subseteq Y_1^\pazocal{A}$ (in which case it is not $\theta_y^{\pm1}$-admissible for any $y$).  For simplicity, such an admissible word is called \emph{apt}.

For the sake of being explicit, observe that $$\theta_a^{-1}=[q_0\to q_0, \ q_1\to v(a,a)^{-1}a_1q_1a_2^{-1}, \ q_2\to q_2]$$



For any word $w\in F(\pazocal{A}_1\sqcup\pazocal{B})$, the \textit{$\pazocal{A}$-projection} of $w$, $\delta(w)$, is defined to be the (unreduced) word over $\pazocal{A}\cup\pazocal{A}^{-1}$ obtained from $w$ by removing any occurrence of $b_i^{\pm1}$ and applying $\varphi_1^{-1}$ to each letter in the remaining word.  The \textit{$\pazocal{A}$-length} of $w$ is then taken to be $|w|_\pazocal{A}=\|\delta(w)\|$.  Similarly, the \textit{$b$-length} is defined as $|w|_b=\|w\|-|w|_\pazocal{A}$.

For any configuration $W\equiv q_0w_1q_1w_2q_2$ of $\textbf{M}_1^\pazocal{A}$, the \textit{$\pazocal{A}$-projection} of $W$ is defined to be the freely reduced word $\eps(W)$ over $\pazocal{A}$ freely equal to $w_1'w_2'$, where $w_1'=\delta(w_1)$ and $w_2'=\widetilde{\varphi}_2^{-1}(w_2)$.

The following statement is an immediate consequence of the construction of the software of $\textbf{M}_1^\pazocal{A}$:

\begin{lemma} \label{projection argument}

For any $y\in\pazocal{A}\sqcup\pazocal{B}$ and any configuration $W$ of $\textbf{M}_1^\pazocal{A}$, $\eps(W)\equiv\eps(W\cdot\theta_y^{\pm1})$.

\end{lemma}

\begin{lemma} \label{M_1 start to end 1}

For $u,v\in F(\pazocal{A})$, suppose there exists a reduced computation of $\textbf{M}_1^\pazocal{A}$ in the standard base $\pazocal{C}:W_0\to\dots\to W_t$ such that $W_0\equiv q_0\widetilde{\varphi}_1(u)q_1q_2$ and $W_t\equiv q_0q_1\widetilde{\varphi}_2(v)q_2$. Then $u\equiv v$.

\end{lemma}

\begin{proof}

Noting that $\eps(W_0)\equiv u$ and $\eps(W_t)\equiv v$, the statement follows from Lemma \ref{projection argument}.

\end{proof}

Similarly, the next statement is a corollary to Lemma \ref{projection argument}:

\begin{lemma} \label{M_1 return to start}

For $u,v\in F(\pazocal{A})$, suppose there exists a reduced computation of $\textbf{M}_1^\pazocal{A}$ in the standard base $\pazocal{C}:W_0\to\dots\to W_t$ such that $W_0\equiv q_0\widetilde{\varphi}_1(u)q_1q_2$ and $W_t\equiv q_0\widetilde{\varphi}_1(v)q_1q_2$. Then $u\equiv v$.

\end{lemma}
%
%
%

Similarly, the following statement is an immediate consequence of the definition of the rules:

\begin{lemma} \label{M_1 difference}

Let $W$ be an apt admissible word of $\textbf{M}_1^\pazocal{A}$ with base $Q_0^{\pazocal{A}}Q_1^{\pazocal{A}}$.  For any $y\in\pazocal{A}\sqcup\pazocal{B}$ and $\eps\in\{\pm1\}$, $|W\cdot\theta_y^\eps|_a\leq c_0(|W|_a+1)$.

\end{lemma}

\begin{proof}

Let $W\equiv q_0wq_1$ with $w\equiv u_0x_1^{\delta_1}u_1x_2^{\delta_2}\dots u_{k-1}x_k^{\delta_k}u_k$ for $u_i\in F(\pazocal{B})$, $x_i\in\pazocal{A}_1$, and $\delta_i\in\{\pm1\}$.  Then, $W\cdot\theta_y^\eps\equiv q_0w'vq_1$ where:

\begin{itemize}

\item $w'=u_0(v(y,\widetilde{x}_1)x_1)^{\delta_1}u_1(v(y,\widetilde{x}_2)x_k)^{\delta_2}\dots u_{k-1}(v(y,\widetilde{x}_k)x_k)^{\delta_k}u_k$, where $\widetilde{x}_i=\varphi_1^{-1}(x_i)$

\item $v=y^{-\eps}$ if $y\in\pazocal{B}$

\item $v=y_1^{-1}$ if $y\in\pazocal{A}$ and $\eps=1$

\item $v=v(y,y)y_1$ if $y\in\pazocal{A}$ and $\eps=-1$

\end{itemize}

As a result,
\begin{align*}
|W\cdot\theta_y^\eps|_a&=\|w'v\|\leq\|w'\|+\|v\|\leq\|w\|+D_\pazocal{A}k+(D_\pazocal{A}+1) \\
&=(\|w\|+1)+D_\pazocal{A}(|w|_\pazocal{A}+1) \\
&\leq(D_\pazocal{A}+1)(|W|_a+1)
\end{align*}

As $|\pazocal{A}|$ is bounded above by a linear function of $C$, $D_\pazocal{A}$ is bounded above by a quadratic function of $C$.  As a result, the parameter assignment $c_0>>C$ given in \Cref{sec-parameters} can be interpreted as $c_0>>D_\pazocal{A}$, implying the statement.

\end{proof}

The next statement is immediate from the definition of $\pazocal{A}$-projection and the bijections defining the software of a noisy $S$-machine.

\begin{lemma} \label{semi-computation deltas}

For any $w\in F(\pazocal{A}_1\sqcup\pazocal{B})$ and $y\in\pazocal{A}\sqcup\pazocal{B}$, $\delta(w\cdot\theta_y^{\pm1})\equiv\delta(w)$.

\end{lemma}

%
%
%
%

Recall that an application of $\theta_y^{\pm1}$ multiplies the $Q_0^\pazocal{A}Q_1^\pazocal{A}$-sector on the right by a specified word.  Given the makeup of these words, the next statement is an immediate corollary of \Cref{semi-computation deltas}.

\begin{lemma} \label{one-rule delta}

Let $w\in F(\pazocal{A}_1\sqcup\pazocal{B})$, $W\equiv q_0wq_1$, $y\in\pazocal{A}\sqcup\pazocal{B}$, and $\eps\in\{\pm1\}$.  Set $w'\in F(\pazocal{A}_1\sqcup\pazocal{B})$ such that $W\cdot\theta_y^\eps\equiv q_0w'q_1$.

\begin{enumerate}[label=(\alph*)]

\item If $y\in\pazocal{B}$, then $\delta(w')\equiv\delta(w)$.

\item If $y\in\pazocal{A}$, then either $\delta(w')\equiv\delta(w)y^{-\eps}$ or $\delta(w')y^\eps\equiv\delta(w)$.

\end{enumerate}

\end{lemma}

%
%
%
%

\begin{lemma} \label{M_1 A-growth}

Let $\pazocal{C}:W_0\to\dots\to W_t$ be a reduced computation of $\textbf{M}_1^\pazocal{A}$ with base $Q_0^{\pazocal{A}}Q_1^{\pazocal{A}}$.  Set $W_i\equiv q_0w_iq_1$ for all $i$.  If $|w_0|_\pazocal{A}<|w_1|_\pazocal{A}$, then $|w_{i-1}|_\pazocal{A}\leq|w_i|_\pazocal{A}$ for all $1\leq i\leq t$.

\end{lemma}

\begin{proof}

Let $H\equiv\theta_1\dots\theta_t$ be the history of $\pazocal{C}$.

Assuming to the contrary, let $m\in\{2,\dots,t\}$ be the minimal index such that $|w_{m-1}|_\pazocal{A}>|w_m|_\pazocal{A}$.  Then, let $\ell\in\{1,\dots,m-1\}$ be the maximal index such that $|w_{\ell-1}|_\pazocal{A}<|w_\ell|_\pazocal{A}$.

By \Cref{one-rule delta}, there exists $a,a'\in\pazocal{A}$ and $\eps,\eps'\in\{\pm1\}$ such that $\theta_\ell=\theta_a^\eps$ and $\theta_m=\theta_{a'}^{\eps'}$ so that $\delta(w_\ell)\equiv\delta(w_{\ell-1})a^{-\eps}$ and $\delta(w_m)(a')^{\eps'}\equiv\delta(w_{m-1})$.

Further, the minimality of $m$ and the maximality of $\ell$ imply that $|w_\ell|_\pazocal{A}=\dots=|w_{m-1}|_\pazocal{A}$, so that \Cref{one-rule delta} implies that $\theta_i\in\{\theta_{b_1}^{\pm1},\theta_{b_2}^{\pm1}\}$ for all $\ell+1\leq i\leq m-1$.  So, $\delta(w_\ell)\equiv\dots\equiv\delta(w_{m-1})$.

As a result, $a=a'$ and $\eps=-\eps'$.  Hence, as $H$ is reduced, $m>\ell+1$.  

Let $y_{\ell+1},\dots,y_{m-1}\in\pazocal{B}$ and $\delta_{\ell+1},\dots,\delta_{m-1}\in\{\pm1\}$ such that $\theta_i=\theta_{y_i}^{\delta_i}$ for all $\ell+1\leq i\leq m-1$.  Letting $v\equiv y_{\ell+1}^{-\delta_{\ell+1}}\dots y_{m-1}^{-\delta_{m-1}}$, it follows that $v$ must be reduced.

Suppose $\eps=-1$.  Then, $w_\ell\equiv u_\ell a_1$ for some $u_\ell\in F(\pazocal{A}_1\sqcup\pazocal{B})$.  So, $w_{m-1}\equiv u_{m-1}a_1v$ for some $u_{m-1}\in F(\pazocal{A}_1\sqcup\pazocal{B})$.  But then $w_m=(u_{m-1}a_1\cdot\theta_a)va_1^{-1}$, so that $\delta(w_m)\equiv\delta(w_{m-1})a^{-1}$ since $v$ is non-trivial, contradicting the hypothesis for $m$.

Conversely, suppose $\eps=1$.  As above, this implies that $w_\ell\equiv u_\ell' a_1^{-1}$ for some $u_\ell'\in F(\pazocal{A}_1\sqcup\pazocal{B})$.  So, letting $z=\prod\limits_{i=\ell+1}^{m-1}v(y_i,a)$, it follows that $w_{m-1}= u_{m-1}'a_1^{-1}z^{-1}v$ for some $u_{m-1}'\in F(\pazocal{A}_1\sqcup\pazocal{B})$.

Since the product defining $z$ is reduced as an element of $\gen{S_\pazocal{A}}$, it follows from \Cref{free subgroup} that $\|z\|\geq\frac{1}{2}D_\pazocal{A}\|v\|$.  So, since $D_\pazocal{A}\geq4$, $\|z^{-1}v\|\geq\|z\|-\|v\|\geq\|v\|\geq1$.  In particular $z^{-1}v$ is a non-trivial element of $F(\pazocal{B})$.  

Hence, $w_m=(u_{m-1}'\cdot\theta_a^{-1})a_1^{-1}v(a,a)(z^{-1}v)v(a,a)^{-1}a_1$, so that $\delta(w_m)\equiv\delta(w_{m-1})a$, again yielding a contradiction.

\end{proof}

Suppose there exists a reduced computation $\pazocal{C}:W_0\to\dots\to W_t$ of $\textbf{M}_1^\pazocal{A}$ with base $Q_0^{\pazocal{A}}Q_1^{\pazocal{A}}$ such that for $W_i\equiv q_0w_iq_1$ for all $i$,
$$|w_0|_\pazocal{A}=\dots=|w_{t-1}|_\pazocal{A}=|w_t|_\pazocal{A}+1$$
Then, $w_0$ is called \textit{rear shiftable} and $\pazocal{C}$ is called a \textit{rear shift of $w_0$}.


Note that $|w_0|_\pazocal{A}\geq1$ for any rear shiftable word $w\in F(\pazocal{A}_1\sqcup\pazocal{B})$, and so there exist $x_1,\dots,x_k\in\pazocal{A}_1$, $\delta_1,\dots,\delta_k\in\{\pm1\}$, and $u_0,u_1,\dots,u_k\in F(\pazocal{B})$ such that $$w_0\equiv u_0x_1^{\delta_1}u_1x_2^{\delta_2}\dots u_{k-1}x_k^{\delta_k}u_k$$

\begin{lemma} \label{rear shift positive}


Let $w\in F(\pazocal{A}_1\sqcup\pazocal{B})$ and suppose $w\equiv u_0x_1^{\delta_1}u_1x_2^{\delta_2}\dots x_{k-1}^{\delta_{k-1}}u_{k-1}x_k^{\delta_k}u_k$ with $\delta_k=1$.  Then there exists a unique rear shift $\pazocal{C}_w:W_0\to\dots\to W_t$ of $w$.  In this case:


\begin{enumerate}[label=(\alph*)]

\item $t=\|u_k\|+1$

\item For all $i=1,\dots,t-1$, $W_i\equiv q_0w_iq_1$ where $w_i\equiv u_0^{(i)}x_1^{\delta_1}u_1^{(i)}x_2^{\delta_2}\dots x_{k-1}^{\delta_{k-1}}u_{k-1}^{(i)}x_ku_k^{(i)}$ such that $\|u_j^{(i)}\|\leq\|u_j\|+2D_\pazocal{A}i$ for $0\leq j\leq k-1$ and $\|u_k^{(i)}\|=\|u_k\|-i$

\item $W_t\equiv q_0w'q_1$ where $w'\equiv u_0'x_1^{\delta_1}u_1'x_2^{\delta_2}\dots x_{k-1}^{\delta_{k-1}}u_{k-1}'$ such that $\|u_j'\|\leq\|u_j\|+2D_\pazocal{A}t$ for all $0\leq j\leq k-1$.

\end{enumerate}

\end{lemma}

\begin{proof}


Let $a\in\pazocal{A}$ such that $x_k=a_1$.
Setting $m=\|u_k\|$, let $y_1,\dots,y_m\in\pazocal{B}$ and $\eps_1,\dots,\eps_m\in\{\pm1\}$ such that $u_k\equiv y_m^{\eps_m}\dots y_1^{\eps_1}$.  Then, for $i=1,\dots,m$, let $H_i=\theta_{y_1}^{\eps_1}\dots\theta_{y_i}^{\eps_i}$.  For $1\leq j\leq k$ and $1\leq i\leq m$, define $v_j^{(i)}=\prod_{\ell=1}^i v(y_\ell,\varphi_1^{-1}(x_j))^{\eps_\ell}$.

Then, for all $i=1,\dots,m$, $W_0\cdot H_i\equiv q_0w_iq_1$ where $$w_i=u_0(v_1^{(i)}x_1)^{\delta_1}u_1(v_2^{(i)}x_2)^{\delta_2}\dots(v_{k-1}^{(i)}x_{k-1})^{\delta_{k-1}}u_{k-1}v_k^{(i)}a_1u_k^{(i)}$$ such that $u_k^{(i)}=y_m^{\eps_m}\dots y_{i+1}^{\eps_{i+1}}$ (with $u_k^{(m)}=1$).  Hence, since $\|v_j^{(m)}\|\leq D_\pazocal{A}m$, the computation with history $H_m\theta_a$ is a rear shift of $w$ satisfying the given bounds.

Now, suppose $\pazocal{C}':W_0\equiv W_0'\to\dots\to W_{s+1}'$ is another rear shift of $w$.  By Lemmas \ref{one-rule delta} and \ref{M_1 A-growth}, there exists $z_1,\dots,z_s\in\pazocal{B}$ and $\nu_1,\dots,\nu_s\in\{\pm1\}$ such that the history $H'$ of $\pazocal{C}'$ satisfies $H'\equiv \theta_{z_1}^{\nu_1}\dots\theta_{z_s}^{\nu_s}\theta_a$.  Letting $H_0'$ be the prefix $\theta_{z_1}^{\nu_1}\dots\theta_{z_s}^{\nu_s}$ of $H'$, $w_0\equiv u_0x_1^{\delta_1}u_1x_2^{\delta_2}\dots u_{k-1}$ the prefix of $w$, $\bar{v}=\prod_{i=1}^s v(y_i',a)$, and $z\equiv z_1^{-\nu_1}\dots z_s^{-\nu_s}$, it follows that $W_s'\equiv W_0'\cdot H_0'=q_0(w_0\cdot H_0')\bar{v}a_1u_kzq_1$.

Then, $W_{s+1}'=q_0(w_0\cdot H')\bar{v}v(a,a)a_1u_kza_1^{-1}q_1$, so that $u_kz$ must be freely trivial since $a_1$ and $a_1^{-1}$ cancel by hypothesis.  Since $H'$ is reduced, $z$ is also reduced, and so $u_k\equiv z^{-1}$.  But then $H'\equiv H$.

\end{proof}

\begin{lemma} \label{rear shift negative}

Let $w\equiv u_0x_1^{\delta_1}u_1x_2^{\delta_2}\dots u_{k-1}x_k^{\delta_k}u_k$ be a rear shiftable word.  Then for any rear shift $\pazocal{C}:W_0\to\dots\to W_t$ of $w$:

\begin{enumerate}[label=(\alph*)]

    \item $t\leq\|u_k\|+1$

    \item For all $i=1,\dots,t-1$, $W_i\equiv q_0w_iq_1$ where $w_i\equiv u_0^{(i)}x_1^{\delta_1}u_1^{(i)}x_2^{\delta_2}\dots u_{k-1}^{(i)}x_k^{\delta_k}u_k^{(i)}$ such that $\|u_j^{(i)}\|\leq\|u_j\|+2D_\pazocal{A}i$ for $0\leq j\leq k$
    
    \item $W_t\equiv q_0w'q_1$ where $w'\equiv u_0'x_1^{\delta_1}u_1'x_2^{\delta_2}\dots x_{k-1}^{\delta_{k-1}}u_{k-1}'$ such that $\|u_j'\|\leq\|u_j\|+2D_\pazocal{A}t$ for all $0\leq j\leq k-1$

\end{enumerate}

%
%
%
%
%
%

\end{lemma}

\begin{proof}

By \Cref{rear shift positive}, it may be assumed that $\delta_k=-1$.  Let $w_0$ be the prefix of $w$ given by $w_0\equiv u_0x_1^{\delta_1}u_1x_2^{\delta_2}\dots x_{k-1}^{\delta_{k-1}}u_{k-1}$.  Further, let $a\in\pazocal{A}$ such that $x_k=a_1$.  So, $w\equiv w_0a_1^{-1}u_k$.


Let $\pazocal{C}:W_0\to\dots\to W_t$ be a rear shift of $w$ with history $H\equiv\theta_1\dots\theta_t$.  By Lemmas \ref{one-rule delta} and \ref{M_1 A-growth}, $\theta_t=\theta_a^{-1}$ and, for all $1\leq i\leq t-1$, there exists $y_i\in\pazocal{B}$ and $\eps_i\in\{\pm1\}$ such that $\theta_i=\theta_{y_i}^{\eps_i}$.

Then, letting $\bar{v}=\prod_{i=1}^{t-1}v(y_i,a)^{\eps_i}$ and $z=y_1^{-\eps_1}\dots y_{t-1}^{-\eps_{t-1}}$, it follows that 
$$W_t=q_0(w_0\cdot H)a_1^{-1}v(a,a)\bar{v}^{-1}u_kzv(a,a)^{-1}a_1q_1$$
So, since $a_1^{-1}$ and $a_1$ cancel by hypothesis, the word $v(a,a)\bar{v}^{-1}u_kzv(a,a)^{-1}$ must be freely trivial.

In particular, this implies that $u_kz$ must be freely equal to $\bar{v}$.  Since the product defining $\bar{v}$ is reduced as an element of $\gen{S_\pazocal{A}}$, \Cref{free subgroup} implies $\|u_kz\|=\|\bar{v}\|\geq\frac{1}{2}D_\pazocal{A}t$.  Hence, since $D_\pazocal{A}\geq4$, $\|u_k\|\geq\|u_kz\|-\|z\|\geq(\frac{1}{2}D_\pazocal{A}-1)t\geq t$.  The bound on $\|u_j^{(i)}\|$ and $\|u_j'\|$ follow from this bound on $t$ in much the same way as in the proof of \Cref{rear shift positive}.

\end{proof}


%
%
%
%
%
%
%
%
%
%

Similar to the previous definition, a word $w\in F(\pazocal{A}_1\sqcup\pazocal{B})$ is called \textit{shiftable} if there exists a reduced computation $\pazocal{C}:W_0\to\dots\to W_t$ of $\textbf{M}_1^\pazocal{A}$ with base $Q_0^{\pazocal{A}}Q_1^{\pazocal{A}}$ such that $W_0\equiv q_0wq_1$ and $W_t\equiv q_0q_1$.  Accordingly, the computation $\pazocal{C}$ in this case is called a \textit{shift of $w$}.

\begin{lemma} \label{shiftable zero}

Any word $w\in F(\pazocal{B})$ is shiftable.  Moreover, in this case there exists a unique shift $\pazocal{C}_w:W_0\to\dots\to W_t$ of $w$, which satisfies $t=\|w\|$.

\end{lemma}

\begin{proof}

By hypothesis, there exists $y_1,\dots,y_t\in\pazocal{B}$ and $\eps_1,\dots,\eps_t\in\{\pm1\}$ such that $w\equiv y_t^{\eps_t}\dots y_1^{\eps_1}$.

Then, letting $W_0\equiv q_0wq_1$, the computation $\pazocal{C}_w:W_0\to\dots\to W_t$ with history $\theta_{y_1}^{\eps_1}\dots\theta_{y_t}^{\eps_t}$ is a shift of $w$ with $W_i\equiv q_0y_t^{\eps_t}\dots y_{i+1}^{\eps_{i+1}}q_1$ for all $i$.

Now, suppose $\pazocal{C}':W_0\equiv W_0'\to\dots\to W_s'\equiv W_t$ is a shift of $w$.  Let $H'$ be the history of $\pazocal{C}'$.  By Lemmas \ref{one-rule delta} and \ref{M_1 A-growth}, there exist $z_1,\dots,z_s\in\pazocal{B}$ and $\nu_1,\dots,\nu_s\in\{\pm1\}$ such that $H'\equiv\theta_{z_1}^{\nu_1}\dots\theta_{z_s}^{\nu_s}$.  

Then, letting $z\equiv z_1^{-\nu_1}\dots z_s^{-\nu_s}$, it follows that $W_s'\equiv W_0\cdot H'=q_0wzq_1$.  Since $H'$ is reduced, $z$ must also be reduced.  So, $w\equiv z^{-1}$.  But then $H'=H$, and so $\pazocal{C}'=\pazocal{C}_w$.

\end{proof}

The following is an immediate consequence of Lemmas \ref{rear shift positive} and \ref{shiftable zero}:

\begin{lemma} \label{shiftable positive}

For any $w\in F(\pazocal{A}_1\sqcup\pazocal{B})$ such that $\delta(w)\in\pazocal{A}^*$, $w$ is shiftable.  Moreover, there exists a unique shift of $w$.

\end{lemma}

In particular, taking $w=1$ the following statement is an corollary:

\begin{lemma} \label{M_1 shift of 1}

Let $\pazocal{C}:W_0\to\dots\to W_t$ be a reduced computation of $\textbf{M}_1^\pazocal{A}$ with base $Q_0^\pazocal{A}Q_1^\pazocal{A}$.  If $W_0\equiv q_0q_1\equiv W_t$, then $t=0$.

\end{lemma}

%
%

\begin{lemma} \label{M_1 start to end 2}

For any $w\in\pazocal{A}^*$, there exists a (unique) reduced computation $\pazocal{C}:W_0\to\dots\to W_t$ of $\textbf{M}_1^\pazocal{A}$ in the standard base such that $W_0\equiv q_0\widetilde{\varphi}_1(w)q_1q_2$ and $W_t\equiv q_0q_1\widetilde{\varphi}_2(w)q_2$.

\end{lemma}

\begin{proof}

By \Cref{shiftable positive}, there exists a (unique) shift $\pazocal{D}:V_0\to\dots\to V_t$ of $\widetilde{\varphi}_1(w)$.  Let $H$ be the history of $\pazocal{D}$.

Since every apt configuration is $\theta$-admissible for any rule $\theta$ of $\textbf{M}_1^\pazocal{A}$, there exists a reduced computation $\pazocal{C}:W_0\to\dots\to W_t$ in the standard base with history $H$ such that $W_0\equiv q_0\widetilde{\varphi}_1(w)q_1q_2$.  Hence, the restriction of $\pazocal{C}$ to the base $Q_0^\pazocal{A}Q_1^\pazocal{A}$ is $\pazocal{D}$.  But then $\pazocal{C}$ must satisfy the hypotheses of \Cref{M_1 start to end 1}, so that $W_t\equiv q_0q_1\widetilde{\varphi}_2(w)q_2$.

The uniqueness of $\pazocal{C}$ is given by applying \Cref{shiftable positive} to the restriction of any such computation to the base $Q_0^\pazocal{A}Q_1^\pazocal{A}$.

\end{proof}

%
%
%
%
%
%
%

\begin{lemma} \label{M_1 shift}

Suppose $w\in F(\pazocal{A}_1\sqcup\pazocal{B})$ is shiftable.  Then there exists a unique shift of $w$ $\pazocal{C}_w:W_0\to\dots\to W_t$.  Moreover, $\pazocal{C}_w$ satisfies
$t\leq \|w\|+\|w\|c_0^{\|w\|}$.



\end{lemma}

\begin{proof}

First, let $\pazocal{C}_1:W_0\to\dots\to W_t$ and $\pazocal{C}_2:V_0\to\dots\to V_s$ be two shifts of $w$.  Letting $H_1$ and $H_2$ be the histories of these computations, respectively, there exists a (possibly unreduced) computation $\pazocal{D}:V_s\to\dots\to V_0\equiv W_0\to\dots\to W_t$ with history $H_2^{-1}H_1$.  Since $V_s\equiv q_0q_1\equiv W_t$, \Cref{M_1 shift of 1} implies the reduced version of $\pazocal{D}$ must be empty.  Thus, $H_1=H_2$, {\frenchspacing i.e. $\pazocal{C}_1=\pazocal{C}_2$}.

Now, let $\pazocal{C}_w:W_0\to\dots\to W_t$ be the shift of $w$.  By \Cref{shiftable zero}, it suffices to assume that $|w|_\pazocal{A}\geq1$.  So, $w\equiv u_0x_1^{\delta_1}u_1x_2^{\delta_2}\dots u_{k-1}x_k^{\delta_k}u_k$ for some $x_1,\dots,x_k\in\pazocal{A}_1$, $\delta_1,\dots,\delta_k\in\{\pm1\}$, and $u_0,u_1,\dots,u_k\in F(\pazocal{B})$.

Let $\pazocal{C}:W_0\to\dots\to W_t$ be a shift of $w$ and set $W_i\equiv q_0w_iq_1$ for all $i$.  By \Cref{M_1 A-growth}, there exist $1\leq t_1<\dots<t_k\leq t$ such that $|w_{t_j}|_\pazocal{A}=k-j$ while $|w_{t_j-1}|_\pazocal{A}=k-j+1$.  For completeness, set $t_0=0$ and $t_{k+1}=t$.

Then, for $j=1,\dots,k+1$, let $\pazocal{C}_j:W_{t_{j-1}}\to\dots\to W_{t_j}$ be the corresponding subcomputation of $\pazocal{C}$.

For each $1\leq j\leq k$, \Cref{one-rule delta} implies $w_{t_j}\equiv u_0^{(j)}x_1^{\delta_1}u_1^{(j)}x_2^{\delta_2}\dots u_{k-j-1}x_{k-j}^{\delta_{k-j}}u_{k-j}^{(j)}$ for some words $u_i^{(j)}\in F(\pazocal{B})$.  As above, set $u_i^{(0)}\equiv u_i$.  By construction, $w_{t_k}\equiv u_0^{(k)}\in F(\pazocal{B})$, so that \Cref{shiftable zero} yields the inequalities $t_{k+1}-t_k=\|u_0^{(k)}\|$ and $|W_i|_a\leq\|u_0^{(k)}\|$ for all $t_k\leq i\leq t$.

Further, for $1\leq j\leq k$, $\pazocal{C}_{j}$ is a rear shift of $w_{t_{j-1}}$.  \Cref{rear shift negative} then implies $t_{j}-t_{j-1}\leq\|u_{k-j+1}^{(j-1)}\|+1$ and $\|u_i^{(j)}\|\leq\|u_i^{(j-1)}\|+2D_\pazocal{A}(t_j-t_{j-1})\leq\|u_i^{(j-1)}\|+2D_\pazocal{A}(\|u_{k-j+1}^{(j-1)}\|+1)$ for all $0\leq i\leq k-j$.  Iterating, this second inequality yields:
\begin{equation} \label{u_i^j bound 1}
    \|u_i^{(j)}\|\leq\|u_i^{(0)}\|+2D_\pazocal{A}\sum_{\ell=0}^{j-1}\|u_{k-\ell}^{(\ell)}\|+2D_\pazocal{A}j
\end{equation}
for any $0\leq i\leq k-j$.  In particular, $\|u_{k-1}^{(1)}\|\leq\|u_{k-1}^{(0)}\|+2D_\pazocal{A}\|u_k^{(0)}\|+2D_\pazocal{A}$, so that:
\begin{equation} \label{u_i^j bound 2}
\sum_{\ell=0}^1\|u_{k-\ell}^{(\ell)}\|\leq\sum_{\ell=0}^1(2D_\pazocal{A}+1)^{1-\ell} \|u_{k-\ell}^{(0)}\|+(2D_\pazocal{A}+1)
\end{equation}
Now suppose $\sum\limits_{\ell=0}^{j-1}\|u_{k-\ell}^{(\ell)}\|\leq\sum\limits_{\ell=0}^{j-1}(2D_\pazocal{A}+1)^{j-1-\ell}\|u_{k-\ell}^{(0)}\|+(j-1)(2D_\pazocal{A}+1)^{j-1}$ for some $2\leq j\leq k$.  Then, using (\ref{u_i^j bound 1}) and (\ref{u_i^j bound 2}) and noting that $(2D_\pazocal{A}+1)^j\geq 2D_\pazocal{A}j$ for $D_\pazocal{A}\geq4$ and $j\geq2$:
\begin{align*}
    \sum_{\ell=0}^j\|u_{k-\ell}^{(\ell)}\|&\leq\|u_{k-j}^{(j)}\|+\sum_{\ell=0}^{j-1}\|u_{k-\ell}^{(\ell)}\|\leq\|u_{k-j}^{(0)}\|+(2D_\pazocal{A}+1)\sum_{\ell=0}^{j-1}\|u_{k-\ell}^{(\ell)}\|+2D_\pazocal{A}j \\
    &\leq\|u_{k-j}^{(0)}\|+\sum_{\ell=0}^{j-1}(2D_\pazocal{A}+1)^{j-\ell}\|u_{k-\ell}^{(0)}\|+(j-1)(2D_\pazocal{A}+1)^j+2D_\pazocal{A}j \\
    &\leq\sum_{\ell=0}^j (2D_\pazocal{A}+1)^{j-\ell}\|u_{k-\ell}^{(0)}\|+j(2D_\pazocal{A}+1)^j
\end{align*}
\allowdisplaybreaks
As a result, 
\begin{align*}
    t&=t_0+\sum_{j=1}^{k+1}(t_j-t_{j-1})\leq\sum_{j=1}^{k}(\|u_{k-j+1}^{(j-1)}\|+1)+\|u_0^{(k)}\|\leq k+\sum_{\ell=0}^k\|u_{k-\ell}^{(\ell)}\| \\
    &\leq k+\sum_{\ell=0}^{k}(2D_\pazocal{A}+1)^{k-\ell}\|u_{k-\ell}^{(0)}\|+k(2D_\pazocal{A}+1)^{k} \\
    &\leq k+(2D_\pazocal{A}+1)^{k}\left(\sum_{i=0}^k\|u_i^{(0)}\|+k\right) \\
    &\leq |w|_\pazocal{A}+(2D_\pazocal{A}+1)^{|w|_\pazocal{A}}(|w|_b+|w|_\pazocal{A})\leq \|w\|+\|w\|(2D_\pazocal{A}+1)^{\|w\|}
\end{align*}
Thus, the bound follows by the parameter choice $c_0>>D_\pazocal{A}$ arising from $c_0>>C$.

%

\end{proof}

Note that the upper bound on the length of the reduced computation given in Lemma \ref{M_1 shift} is not sharp.  For example, in the setting of \Cref{shiftable positive}, the factor of $\|w\|$ in the product $\|w\|(2D_\pazocal{A}+1)^{\|w\|}$ may be removed.  However, such improvements will prove moot for the purposes of this article.

Now, we study semi-computations of $\textbf{M}_1^\pazocal{A}$ in the $Q_0^\pazocal{A}Q_1^\pazocal{A}$-sector.  Many of the concepts and statements that follow can be adopted for any noisy $S$-machine, but for simplicity we restrict our attention here to $\textbf{M}_1^\pazocal{A}$, our prototype of a noisy $S$-machine.

The following statement is a vital first step in this and will be crucial to proving the malnormality of the embedding of Theorem A:

\begin{lemma} \label{M_1 semi-computation three A}

Let $\pazocal{S}:w_0\to\dots\to w_t$ be a reduced semi-computation of $\textbf{M}_1^\pazocal{A}$ in the $Q_0^{\pazocal{A}}Q_1^{\pazocal{A}}$-sector.  Suppose $w_0\equiv x_1^{\delta_1}x_2^{\delta_2}x_3^{\delta_3}$ for some $x_i\in\pazocal{A}_1$ and $\delta_i\in\{\pm1\}$.  Then, there exist $u_0,u_1,u_2,u_3\in F(\pazocal{B})$ such that:

\begin{enumerate}

\item $w_t\equiv u_0x_1^{\delta_1}u_1x_2^{\delta_2}u_2x_3^{\delta_3}u_3$

\item $\|u_0\|,\|u_3\|\leq D_\pazocal{A}t$

\item $\frac{1}{2}D_\pazocal{A}t\leq\|u_1\|+\|u_2\|\leq 3D_\pazocal{A}t$

\item $(u_1,u_2)$ uniquely determines the history of $\pazocal{S}$.

\end{enumerate}
    
\end{lemma}

\begin{proof}

Fix $y_i\in\pazocal{A}\sqcup\pazocal{B}$ and $\eps_i\in\{\pm1\}$ such that $H\equiv\theta_{y_1}^{\eps_1}\dots\theta_{y_t}^{\eps_t}$ is the history of $\pazocal{S}$.


Let $\tilde{x}_j=\varphi_1^{-1}(x_j)$ and set $v_j=\prod_{i=1}^t v(y_i,\tilde{x}_j)^{\eps_i}$.  Then, $w_t\equiv w_0\cdot H=(v_1x_1)^{\delta_1}(v_2x_2)^{\delta_2}(v_3x_3)^{\delta_3}$.

As $H$ is reduced, the product defining each $v_j$ is reduced as an element of $\gen{S_\pazocal{A}}$.  So, \Cref{free subgroup} implies that each $v_j$ uniquely determines the history of $H$ and $\frac{1}{2}D_\pazocal{A}t\leq\|v_j\|\leq D_\pazocal{A}t$.

In particular, $\|u_1\|+\|u_2\|\leq|w_t|_b\leq\|v_1\|+\|v_2\|+\|v_3\|=3D_\pazocal{A}t$, $\|u_0\|\leq\|v_1\|\leq D_\pazocal{A}t$, and $\|u_3\|\leq\|v_3\|\leq D_\pazocal{A}t$.

Suppose $\delta_2=1$.  If $\delta_1=1$, then $u_1=v_2$, and so the statement follows.  Conversely, if $\delta_1=-1$, then $u_1=v_1^{-1}v_2$.  But $x_1\neq x_2$, so that the product $\left( \prod_{i=1}^t v(y_i,\tilde{x}_1)^{\eps_i} \right)^{-1}
\left(\prod_{i=1}^t v(y_i,\tilde{x}_2)^{\eps_i}\right)$ defining $v_1^{-1}v_2$ is reduced as an element of $\gen{S_\pazocal{A}}$.  Hence, \Cref{free subgroup} implies that $\|u_1\|\geq D_\pazocal{A}t$ and $u_1$ uniquely determines $H$.

If $\delta_2=-1$, then the same arguments imply that $\|u_2\|\geq\frac{1}{2}D_\pazocal{A}t$ and $u_2$ uniquely determines $H$.
    
\end{proof}

The next statement is similar in nature to \Cref{M_1 semi-computation three A} and is proved in an analogous manner:

\begin{lemma} \label{M_1 semi-computation two A}

Let $\pazocal{S}:w_0\to\dots\to w_t$ be a reduced semi-computation of $\textbf{M}_1^\pazocal{A}$ in the $Q_0^\pazocal{A}Q_1^\pazocal{A}$-sector.  Suppose $w_0\equiv x_1^{\delta_1}x_2^{\delta_2}$ for some $x_i\in\pazocal{A}_1$ and $\delta_i\in\{\pm1\}$ such that $\delta_1\neq1$ or $\delta_2\neq-1$.  Then, there exist $u_0,u_1,u_2\in F(\pazocal{B})$ such that:

\begin{enumerate}

\item $w_t\equiv u_0x_1^{\delta_1}u_1x_2^{\delta_2}u_2$

\item $\|u_0\|,\|u_2\|\leq D_\pazocal{A}t$

\item $\frac{1}{2}D_\pazocal{A}t\leq\|u_1\|\leq2D_\pazocal{A}t$ 

\item $u_1$ uniquely determines the history of $\pazocal{S}$

\end{enumerate}

\end{lemma}

%
%
%
%
%
%

A word in $w\in F(\pazocal{A}_1\sqcup\pazocal{B})$ whose first and last letter is an element of $\pazocal{A}_1^{\pm1}$ is called \textit{compressed}.  For any word $w'\in F(\pazocal{A}_1\sqcup\pazocal{B})$ with $|w'|_\pazocal{A}\geq1$, the \textit{compression} $\mathscr{C}(w')$ is the maximal compressed subword of $w'$.



Let $\theta$ be a rule of $\textbf{M}_1^\pazocal{A}$ and $w\equiv x_1^{\delta_1}u_1x_2^{\delta_2}u_2\dots u_{k-1}x_k^{\delta_k}$ be a compressed word, {\frenchspacing i.e. with} $x_i\in\pazocal{A}_1$, $\delta_i\in\{\pm1\}$, and $u_i\in F(\pazocal{B})$.  \Cref{semi-computation deltas} then implies that $w\cdot\theta\equiv u_0'x_1^{\delta_1}u_1'x_2^{\delta_2}u_2'\dots u_{k-1}'x_k^{\delta_k}u_k'$ for some $u_i'\in F(\pazocal{B})$.  The \textit{compressed application} of $\theta$ to $w$ is then taken to be reduced word $$w*\theta\equiv\mathscr{C}(w\cdot\theta)\equiv x_1^{\delta_1}u_1'x_2^{\delta_2}u_2'\dots u_{k-1}'x_k^{\delta_k}$$
Note the resemblance between a compressed application of a rule and the standard setup of an application of a rule to an admissible word: The `compression' mimics the `trimming' that occurs in the latter to make the resulting word again admissible.

Accordingly, a \textit{compressed semi-computation} of $\textbf{M}_1^\pazocal{A}$ in the $Q_0^{\pazocal{A}}Q_1^{\pazocal{A}}$-sector is defined to be a sequence $\pazocal{S}_\mathscr{C}:w_0\to\dots\to w_t$ such that $w_0$ is compressed and $w_i\equiv w_{i-1}*\theta_i$ for some rule $\theta_i$.

Note that any semi-computation $\pazocal{S}:w_0\to\dots\to w_t$ of $\textbf{M}_1^\pazocal{A}$ in the $Q_0^{\pazocal{A}}Q_1^{\pazocal{A}}$-sector such that $|w_0|_\pazocal{A}\geq1$ can be associated to the compressed semi-computation $\pazocal{S}_\mathscr{C}:\mathscr{C}(w_0)\to\dots\to\mathscr{C}(w_t)$ whose history is the same as that of $\pazocal{S}$.  All terminology relating to semi-computations is carried over to compressed semi-computations.  

The following statement is then an immediate consequence of \Cref{M_1 semi-computation three A}:

\begin{lemma} \label{M_1 compressed semi-computation three A}

Let $\pazocal{S}_\mathscr{C}:w_0\to\dots\to w_t$ be a reduced compressed semi-computation of $\textbf{M}_1^\pazocal{A}$ in the $Q_0^{\pazocal{A}}Q_1^{\pazocal{A}}$-sector.  Suppose $w_0\equiv x_1^{\delta_1}x_2^{\delta_2}x_3^{\delta_3}$ for some $x_i\in\pazocal{A}_1$ and $\delta_i\in\{\pm1\}$.  Then, there exist $u_1,u_2\in F(\pazocal{B})$ such that:

\begin{enumerate}

\item $w_t\equiv x_1^{\delta_1}u_1x_2^{\delta_2}u_2x_3^{\delta_3}$

\item $\frac{1}{2}D_\pazocal{A}t\leq\|u_1\|+\|u_2\|\leq 3D_\pazocal{A}t$

\item the pair $(u_1,u_2)$ uniquely determines the history of $\pazocal{S}_\mathscr{C}$

\end{enumerate}

\end{lemma}

Similarly, the following statement is an immediate consequence of \Cref{M_1 semi-computation two A}:

\begin{lemma} \label{M_1 compressed semi-computation two A}

Let $\pazocal{S}_\mathscr{C}:w_0\to\dots\to w_t$ be a reduced compressed semi-computation of $\textbf{M}_1^\pazocal{A}$ in the $Q_0^{\pazocal{A}}Q_1^{\pazocal{A}}$-sector.  Suppose $w_0\equiv x_1^{\delta_1}x_2^{\delta_2}$ for some $x_i\in\pazocal{A}_1$ and $\delta_i\in\{\pm1\}$ such that $\delta_1\neq1$ or $\delta_2\neq-1$.  Then, there exist $u_1\in F(\pazocal{B})$ such that:

\begin{enumerate}

\item $w_t\equiv x_1^{\delta_1}u_1x_2^{\delta_2}$

\item $\frac{1}{2}D_\pazocal{A}t\leq\|u_1\|\leq 2D_\pazocal{A}$

\item $u_1$ uniquely determines the history of $\pazocal{S}_\mathscr{C}$

\end{enumerate}

\end{lemma}

Let $\Lambda_1^\pazocal{A}$ be a subset of $(\pazocal{A}_1\cup\pazocal{A}_1^{-1})^*$ consisting of cyclically reduced words of length at least $C$.  Then, define $\pazocal{E}_1(\Lambda_1^\pazocal{A})$ to be the set of all reduced words $w$ over $(\pazocal{A}_1\sqcup\pazocal{B})^{\pm1}$ for which there exists a semi-computation $\pazocal{S}:w_0\to\dots\to w_t$ of $\textbf{M}_1^\pazocal{A}$ in the $Q_0^{\pazocal{A}}Q_1^{\pazocal{A}}$-sector such that $w_0\equiv w$ and $w_t\in\Lambda_1^\pazocal{A}$.  In this case, $\pazocal{S}$ is said to \textit{$\Lambda_1^\pazocal{A}$-accept} $w$.



For any $\Lambda_1^\pazocal{A}$-accepting semi-computation, \Cref{semi-computation deltas} implies $\delta(w_i)\equiv\delta(w_t)\equiv\widetilde{\varphi}_1^{-1}(w_t)$ for all $i$.  Hence, the terminal word $w_t$ of any such $\Lambda_1^\pazocal{A}$-accepting semi-computation is uniquely determined by the word $w$.  In particular, $|w|_\pazocal{A}\geq C$, and so $w\equiv u_0x_1^{\delta_1}u_1x_2^{\delta_2}\dots u_{k-1}x_k^{\delta_k}u_k$ for some $x_i\in\pazocal{A}_1$, $\delta_i\in\{\pm1\}$, and $u_i\in F(\pazocal{B})$.


\begin{lemma} \label{M_1 Lambda semi-computations}

Let $w\in\pazocal{E}_1(\Lambda_1^\pazocal{A})$ and set $w\equiv u_0x_1^{\delta_1}u_1x_2^{\delta_2}\dots x_k^{\delta_k}u_k$ as above.  Then there exists a unique reduced semi-computation of $\textbf{M}_1^\pazocal{A}$ in the $Q_0^{\pazocal{A}}Q_1^{\pazocal{A}}$-sector $\pazocal{S}_1(w):w\equiv w_0\to\dots\to w_t$ which $\Lambda_1^\pazocal{A}$-accepts $w$.  In this case:

\begin{enumerate}

\item $\frac{1}{2}D_\pazocal{A}t\leq \|u_{i-1}\|+\|u_i\|\leq 3D_\pazocal{A}t$ for any $i\in\{2,\dots,k-1\}$

\item $\frac{1}{2}D_\pazocal{A}t\leq \|u_ku_0\|+\|u_j\|\leq 3D_\pazocal{A}t$ for any $j\in\{1,k-1\}$

\item $\|u_0\|,\|u_k\|\leq D_\pazocal{A}t$

\end{enumerate}

\end{lemma}

\begin{proof}

Let $\pazocal{S}:w\equiv w_0\to\dots\to w_t$ be a reduced semi-computation which $\Lambda_1^\pazocal{A}$-accepts $w$.  This implies $w_t\equiv x_1^{\delta_1}\dots x_k^{\delta_k}\in \Lambda_1^\pazocal{A}$, and so the definition of $\Lambda_1^\pazocal{A}$ yields $k\geq C\geq3$.

Let $H$ be the history of $\pazocal{S}$ and, for $i\in\{2,\dots,k-1\}$, let $\pazocal{S}_\mathscr{C}^{(i)}:v_0^{(i)}\to\dots\to v_t^{(i)}$ be the reduced compressed semi-computation with history $H^{-1}$ such that $v_0^{(i)}\equiv x_{i-1}^{\delta_{i-1}}x_i^{\delta_i}x_{i+1}^{\delta_{i+1}}$.  

Similarly, for $j\in\{1,k\}$, let $\pazocal{S}_\mathscr{C}^{(j)}:v_0^{(j)}\to\dots\to v_t^{(j)}$ be the reduced compressed semi-computation with history $H^{-1}$ such that $v_0^{(1)}\equiv x_k^{\delta_k}x_1^{\delta_1}x_2^{\delta_2}$ and $v_0^{(k)}\equiv x_{k-1}^{\delta_{k-1}}x_k^{\delta_k}x_1^{\delta_1}$.

For any $i$, applying \Cref{M_1 compressed semi-computation three A} to $\pazocal{S}_\mathscr{C}^{(i)}$ implies $v_t^{(i)}\equiv x_{i-1}^{\delta_{i-1}}y_ix_i^{\delta_i}z_ix_{i+1}^{\delta_{i+1}}$ for some $y_i,z_i\in F(\pazocal{B})$ such that $\frac{1}{2}D_\pazocal{A}t\leq\|y_i\|+\|z_i\|\leq3D_\pazocal{A}t$ and the pair $(y_i,z_i)$ uniquely determines $H^{-1}$.

By construction, $y_i=u_{i-1}$ and $z_i=u_i$ for $2\leq i\leq k$.  Hence, (1) holds and the semi-computation $\Lambda_1^\pazocal{A}$-accepting $\pazocal{S}$ is uniquely determined by $w$.

Further, it follows from construction that:

\begin{itemize}

\item $y_1=u_ku_0$ and $z_1=u_1$

\item $y_k=u_{k-1}$ and $z_k=u_ku_0$

\end{itemize}

Hence, (2) is implied by the application of \Cref{M_1 compressed semi-computation three A} to $\pazocal{S}_\mathscr{C}^{(1)}$ and $\pazocal{S}_\mathscr{C}^{(k)}$.


Finally, let $\pazocal{S}^{(2)}:v_0'\to\dots\to v_t'$ and $\pazocal{S}^{(k-1)}:v_0''\to\dots\to v_t''$ be the reduced semi-computations with history $H^{-1}$ such that $v_0'\equiv x_1^{\delta_1}x_2^{\delta_2}x_3^{\delta_3}$ and $v_0''\equiv x_{k-2}^{\delta_{k-2}}x_{k-1}^{\delta_{k-1}}x_k^{\delta_k}$.  In other words, these are the `uncompressed' semi-computations corresponding to $\pazocal{S}_\mathscr{C}^{(2)}$ and $\pazocal{S}_\mathscr{C}^{(k-1)}$, respectively.

Then, by construction there exist words $u_3',u_{k-3}''\in F(\pazocal{B})$ such that $v_t'\equiv u_0x_1^{\delta_1}u_1x_2^{\delta_2}u_2x_3^{\delta_3}u_3'$ and $v_t''\equiv u_{k-3}''x_{k-2}^{\delta_{k-2}}u_{k-2}x_{k-1}^{\delta_{k-1}}u_{k-1}x_k^{\delta_k}u_k$.  Thus, (3) follows by applying \Cref{M_1 semi-computation three A} to $\pazocal{S}^{(2)}$ and $\pazocal{S}^{(k-1)}$.

\end{proof}

\medskip

\subsection{The machine $\textbf{M}_2^\pazocal{L}$} \label{sec-M_2} \

As it is assumed that $\pazocal{L}$ is a recursive subset of $\pazocal{A}^*$, there exists a non-deterministic Turing machine $\pazocal{T}_\pazocal{L}$ with alphabet $\pazocal{A}$ that enumerates $\pazocal{L}$ (and another such machine that enumerates the complement of $\pazocal{L}$).


Let $\TM_\pazocal{L}$ be the time function of $\pazocal{T}_\pazocal{L}$, i.e $\TM_\pazocal{L}:\N\to\N$ is the non-decreasing function satisfying the condition that $\TM_\pazocal{L}(n)$ is the smallest number such that for all $w\in\pazocal{L}$ with $\|w\|\leq n$, $\pazocal{T}_\pazocal{L}$ computes $w$ by a finite sequence of $\leq\TM_\pazocal{L}(n)$ basic moves.  Note that since $\pazocal{L}$ is recursive, it may be assumed that $\TM_\pazocal{L}$ is a computable function.

A seminal result of Sapir, Birget, and Rips \cite{SBR} then produces the following auxiliary machine:

\begin{lemma}[Proposition 4.1 of \cite{SBR}] \label{M_2 language} 

There exists an $S$-machine $\textbf{M}_2^\pazocal{L}$ satisfying Lemma \ref{simplify rules} that `simulates' the Turing machine $\pazocal{T}_\pazocal{L}$ in the following sense:

\begin{enumerate}

\item The hardware of $\textbf{M}_2^\pazocal{L}$ is $(\sqcup_{i=1}^N X_i^{\pazocal{L}},\sqcup_{i=0}^N P_i^{\pazocal{L}})$, where $X_1^\pazocal{L}=\emptyset$, $X_2^\pazocal{L}=\pazocal{A}$, and the $P_1^\pazocal{L}P_2^\pazocal{L}$-sector is the only input sector

\item The language of accepted inputs is $\pazocal{L}$

\item For any accepted configuration $W$ satisfying $|W|_a\leq n$, there exists a computation of $\textbf{M}_2^\pazocal{L}$ which accepts $W$ and has length $\leq c_0\TM_\pazocal{L}(c_0n)^3+c_0n+c_0$

\end{enumerate}

\end{lemma}

In the terminology of \cite{SBR} and other settings, condition (3) may be summarized by saying that the `\textit{generalized time function}' of $\textbf{M}_2^\pazocal{L}$ is asymptotically bounded above by $\TM_\pazocal{L}^3$.

Further, observe that the constants $c_0$ and $N$ are listed amongst the parameters in \Cref{sec-parameters}.  In particular, $N$ can be taken to be as large as desired by simply adding sectors with empty tape alphabets to the standard base of $\textbf{M}_2^\pazocal{L}$.






Finally, note that the bounds given in Lemma \ref{M_2 language} may be improved: The statement of Proposition 4.1 in \cite{SBR} also gives upper bounds on the `space' and `area' functions of $\textbf{M}_2^\pazocal{L}$.  In fact, for any $\eps>0$, the main machine of \cite{CW} can be used to construct a machine in which the cubic exponent of $\TM_\pazocal{L}$ in the statement can be reduced to $1+\eps$.  However, such improvements are moot for the purposes of this article.

\medskip


\subsection{The machine $\textbf{M}_3^\pazocal{L}$} \label{sec-M_3} \

The next auxiliary machine is a \textit{composition} of the machines $\textbf{M}_1^\pazocal{A}$ and $\textbf{M}_2^\pazocal{L}$ in the sense described below (and as in the constructions of \cite{O18}, \cite{OSconj}, \cite{OS19}, \cite{W}, etc).  Informally, this is done by viewing the $S$-machine $\textbf{M}_2^\pazocal{L}$ as a noisy $S$-machine itself, then combining the machines through a `transition rule' which switches from $\textbf{M}_1^\pazocal{A}$ to the start state of $\textbf{M}_2^\pazocal{L}$.

To begin, define the sets $Y_i^\pazocal{L}$ for all $1\leq i\leq N$ as follows:

\begin{itemize}

\item $Y_1^\pazocal{L}=Y_1^\pazocal{A}\sqcup X_1^\pazocal{L}$

\item $Y_2^\pazocal{L}=Y_2^\pazocal{A}=\pazocal{A}_2$

\item $Y_i^\pazocal{L}=X_i^\pazocal{L}$ for all $i\geq3$.

\end{itemize}

Further, let $Q_j^\pazocal{A}=\{q_j\}$ for all $3\leq j\leq N$ and define $Q_i^\pazocal{L}=Q_i^\pazocal{A}\sqcup P_i^\pazocal{L}$ for all $0\leq i\leq N$.

The hardware of $\textbf{M}_3^\pazocal{L}$ is then taken to be $(\sqcup_{i=1}^N Y_i^\pazocal{L},\sqcup_{i=0}^N Q_i^\pazocal{L})$.  

The positive rules of $\textbf{M}_3^\pazocal{L}$, $\Theta^+(\textbf{M}_3^\pazocal{L})$, are defined as follows:

\begin{enumerate}[label=(\alph*)]

\item For any positive rule of $\textbf{M}_1^\pazocal{A}$, there is a corresponding positive rule of $\textbf{M}_3^\pazocal{L}$ which operates in exactly the same way as $\theta$ on the subword $Q_0^\pazocal{L}Q_1^\pazocal{L}Q_2^\pazocal{L}$ of the standard base and has the part $q_i\xrightarrow{\ell}q_i$ for all $3\leq i\leq N$.


\item $\sigma=[q_0\xrightarrow{\ell} p_0, \ q_1\to p_1, \ q_2\xrightarrow{\ell} p_2, \ \dots, \ q_{N-1}\xrightarrow{\ell}p_{N-1}, \ q_N\to p_N]$ where $p_i$ is the start letter of the part $P_i^\pazocal{L}$ of the state letters of $\textbf{M}_2^\pazocal{L}$.  Note that $\sigma$ is defined as an $S$-rule, with the domain $Y_2(\sigma)$ taken to be $\pazocal{A}_2$.


\item For every positive rule of $\textbf{M}_2^\pazocal{L}$, there exists a corresponding positive rule of $\textbf{M}_3^\pazocal{L}$ which operates in the analogous way, identifying the tape alphabet $\pazocal{A}_2$ with the input alphabet $\pazocal{A}$ of $\textbf{M}_2^\pazocal{L}$.

\end{enumerate}

The noisy $S$-machine $\textbf{M}_3^\pazocal{L}$ is made to be recognizing as follows:

\begin{itemize}

\item The $Q_0^\pazocal{L}Q_1^\pazocal{L}$-sector is taken to be the only input sector of $\textbf{M}_3^\pazocal{L}$.  

\item The letters of $Q_i^\pazocal{A}$ are the start letters of their corresponding parts

\item The end letters correspond to the end letters of $\textbf{M}_2^\pazocal{L}$ in $P_i^\pazocal{L}$.

\end{itemize}

By its construction, $\textbf{M}_3^\pazocal{L}$ can be viewed as the \textit{composition} of two `submachines', which are denoted $\textbf{M}_3^\pazocal{L}(1)$ and $\textbf{M}_3^\pazocal{L}(2)$ and given as follows:

\begin{enumerate}

\item The hardware of $\textbf{M}_3^\pazocal{L}(1)$ is $(\sqcup_{i=1}^N Y_i^\pazocal{A},\sqcup_{i=0}^N Q_i^\pazocal{A})$ (with $Y_i^\pazocal{A}=\emptyset$ for $i\geq3$) and its set of positive rules $\Theta^+(\textbf{M}_3^\pazocal{L}(1))$ consists are of all rules of the form (a) above.


\item The hardware of $\textbf{M}_3^\pazocal{L}(2)$ is $(\sqcup_{i=1}^N X_i^\pazocal{L},\sqcup_{i=0}^N P_i^\pazocal{L})$ (with $X_2^\pazocal{L}$ identified with $\pazocal{A}_2$) and its set of positive rules $\Theta^+(\textbf{M}_3^\pazocal{L}(2))$ consists of all rules of the form (c) above.

\end{enumerate}

With these definitions, $\textbf{M}_3^\pazocal{L}(j)$ is a noisy $S$-machine for $j=1,2,$ (in fact, $\textbf{M}_3^\pazocal{L}(2)$  is an $S$-machine), while $\textbf{M}_3^\pazocal{L}$ concatenates these machines with the \textit{transition rule} $\sigma$. 

Note that $\textbf{M}_3^\pazocal{L}(1)$ and $\textbf{M}_3^\pazocal{L}(2)$ can be identified with the machines $\textbf{M}_1^\pazocal{A}$ and $\textbf{M}_2^\pazocal{L}$, respectively, with the only major distinction being that several locked sectors are added to $\textbf{M}_1^\pazocal{A}$ to make the standard bases the same size.

The following statements are immediate consequences of this definition and Lemmas \ref{M_1 start to end 1} and \ref{M_1 start to end 2}:

\begin{lemma} \label{M_3 start to sigma 1}

Let $W_0\equiv q_0\widetilde{\varphi}_1(w)q_1q_2q_3\dots q_N$ for some $w\in F(\pazocal{A})$.  Suppose there exists a reduced computation $\pazocal{C}:W_0\to\dots\to W_t$ of $\textbf{M}_3^\pazocal{L}(1)$ in the standard base such that $W_t$ is $\sigma$-admissible.  Then $W_t\equiv q_0q_1\widetilde{\varphi}_2(w)q_2q_3\dots q_N$.

\end{lemma}

\begin{lemma} \label{M_3 start to sigma 2}

For any $w\in\pazocal{A}^*$, there exists a (unique) reduced computation $\pazocal{C}:W_0\to\dots\to W_t$ of $\textbf{M}_3^\pazocal{L}(1)$ in the standard base such that $W_0\equiv q_0\widetilde{\varphi}_1(w)q_1q_2q_3\dots q_N$ and $W_t$ is $\sigma$-admissible.

\end{lemma}

%
%
%

\begin{lemma} \label{return to sigma}

Suppose $\pazocal{C}:W_0\to\dots\to W_t$ is a reduced computation of $\textbf{M}_3^\pazocal{L}(1)$ in the standard base.  If $W_0$ and $W_t$ are both $\sigma$-admissible, then $t=0$.

\end{lemma}

\begin{proof}

The restriction $\pazocal{C}':W_0'\to\dots\to W_t'$ of $\pazocal{C}$ to the base $Q_0^\pazocal{L}Q_1^\pazocal{L}$ can be identified with a reduced computation of $\textbf{M}_1^\pazocal{A}$ in the base $Q_0^\pazocal{A}Q_1^\pazocal{A}$.  But if $W_0$ and $W_t$ are both $\sigma$-admissible, then $W_0'\equiv q_0q_1\equiv W_t'$, so that $\pazocal{C}'$ satisfies the hypotheses of \Cref{M_1 shift of 1}.

\end{proof}

The next statement follows as an immediate corollary to \Cref{return to sigma}:

\begin{lemma} \label{M_3(2) to M_3(2)}

Let $H$ be the history of a reduced computation $\pazocal{C}$ of $\textbf{M}_3^\pazocal{L}$.  Then, there exist $H_1,H_1'\in F(\Theta^+(\textbf{M}_3^\pazocal{L}(1)))$ and $H_2\in F(\Theta^+(\textbf{M}_3^\pazocal{L}(2)))$ such that $H$ is a subword of $H_1\sigma H_2\sigma^{-1}H_1'$.

\end{lemma}

%
%

\begin{lemma} \label{M_3 return to start}

Let $u,v\in F(\pazocal{A})$.  Suppose $\pazocal{C}:W_0\to\dots\to W_t$ is a reduced computation of $\textbf{M}_3^\pazocal{L}$ such that $W_0$ and $W_t$ are the input configurations with inputs $\widetilde{\varphi}_1(u)$ and $\widetilde{\varphi}_1(v)$, respectively.  Then $u\in\pazocal{L}$ if and only if $v\in\pazocal{L}$.

\end{lemma}

\begin{proof}

If $\pazocal{C}$ is a computation of $\textbf{M}_3^\pazocal{L}(1)$ then the restriction of $\pazocal{C}$ to the base $Q_0^\pazocal{L}Q_1^\pazocal{L}Q_2^\pazocal{L}$ can be identified with a computation of $\textbf{M}_1^\pazocal{A}$ in the standard base satisfying the hypotheses of \Cref{M_1 return to start}, so that $u\equiv v$.  Hence, it suffices to assume that $\pazocal{C}$ is not a computation of $\textbf{M}_3^\pazocal{L}(1)$.  

There thus exists a factorization $H\equiv H_1\sigma H_2\sigma^{-1}H_1'$ of the history of $\pazocal{C}$ as in \Cref{M_3(2) to M_3(2)}.  Note that since $\pazocal{C}$ is reduced, $H_2$ must be non-empty.


Let $\pazocal{C}_1:W_0\to\dots\to W_r$ be the subcomputation of $\pazocal{C}$ with history $H_1$.  Then, the restriction of $\pazocal{C}_1$ to the base $Q_0^\pazocal{L}Q_1^\pazocal{L}Q_2^\pazocal{L}$ can be identified with a computation of $\textbf{M}_1^\pazocal{A}$ satisfying the hypotheses of \Cref{M_1 start to end 1}, so that $W_r\equiv q_0q_1\widetilde{\varphi}_2(u)q_2q_3\dots q_N$.

Similarly, letting $\pazocal{C}_1':W_s\to\dots\to W_t$ be the subcomputation of $\pazocal{C}$ with history $H_1'$, the same argument (applied to the inverse computation $W_t\to\dots\to W_s$) implies $W_s\equiv q_0q_1\widetilde{\varphi}_2(v)q_2q_3\dots q_N$.

Hence, the subcomputation $\pazocal{C}_2:W_{r+1}\to\dots\to W_{s-1}$ with history $H_2$ can be identified with a computation of $\textbf{M}_2^\pazocal{L}$ between the input configurations whose inputs are $u$ and $v$.  But then this computation (or its inverse) can be concatenated with a computation of $\textbf{M}_2^\pazocal{L}$ accepting one of these inputs to produce an accepting computation of the other input, so that the statement follows from \Cref{M_2 language}.

\end{proof}

\begin{lemma} \label{M_3 language}

For $w\in F(\pazocal{A})$, the input $\widetilde{\varphi}_1(w)$ is accepted by $\textbf{M}_3^\pazocal{L}$ if and only if $w\in\pazocal{L}$.

\end{lemma}

\begin{proof}

First, suppose $\pazocal{C}:W_0\to\dots\to W_t$ is an accepting computation of the input configuration $W_0\equiv q_0\widetilde{\varphi}_1(w)q_1q_2q_3\dots q_N$.  As $W_0$ is a configuration of $\textbf{M}_3^\pazocal{L}(1)$, \Cref{M_3(2) to M_3(2)} implies there exists a factorization $H\equiv H_1\sigma H_2$ of the history of $\pazocal{C}$ such that:

\begin{itemize}

\item The subcomputation $\pazocal{C}_1:W_0\to\dots\to W_s$ with history $H_1$ is a computation of $\textbf{M}_3^\pazocal{L}(1)$

\item The subcomputation $\pazocal{C}_2:W_{s+1}\to\dots\to W_t$ with history $H_2$ is a computation of $\textbf{M}_3^\pazocal{L}(2)$

\end{itemize}

\Cref{M_3 start to sigma 1} then implies that $W_s\equiv q_0q_1\widetilde{\varphi}_2(w)q_2q_3\dots q_N$, so that $W_{s+1}\equiv W_s\cdot\sigma$ corresponds to the input configuration $\textbf{M}_2^\pazocal{L}$ with input $w$.  But then $\pazocal{C}_2$ can be identified with a reduced computation of $\textbf{M}_2^\pazocal{L}$ accepting this input, so that \Cref{M_2 language} implies $w\in\pazocal{L}$.

Conversely, suppose $w\in\pazocal{L}$.  

As $\pazocal{L}\subseteq\pazocal{A}^*$, there exists a reduced computation $\pazocal{D}_1:V_0\to\dots\to V_r$ given by \Cref{M_3 start to sigma 2} such that $V_0$ is the input configuration with input $\widetilde{\varphi}_1(w)$ and $V_r\cdot\sigma$ corresponds to the input configuration of $\textbf{M}_2^\pazocal{L}$ with input $w$.

Since \Cref{M_2 language} implies that $w$ is an accepted input of $\textbf{M}_2^\pazocal{L}$, identifying such an accepting computation with a computation of $\textbf{M}_3^\pazocal{L}(2)$ yields a reduced computation $\pazocal{D}_2$ accepting $V_r\cdot\sigma$.

Concatenating these computations thus implies the statement.

\end{proof}

\begin{lemma} \label{M_3 time-space}

For any accepted configuration $W$ of $\textbf{M}_3^\pazocal{L}$ with $|W|_a=n$, there exists an accepting computation $\pazocal{C}:W\equiv W_0\to\dots\to W_t$ satisfying
%
$t\leq c_0\TM_\pazocal{L}(c_0n)^3+nc_0^{n}+2c_0n+2c_0$.

%

\end{lemma}

\begin{proof}

Let $\pazocal{D}:W\equiv V_0\to\dots\to V_s$ be a reduced computation of $\textbf{M}_3^\pazocal{L}$ which accepts $W$.

If $\pazocal{D}$ is a computation of $\textbf{M}_3^\pazocal{L}(2)$, then it can be identified with a computation of $\textbf{M}_2^\pazocal{L}$.  But then \Cref{M_2 language} produces a computation $\pazocal{C}$ accepting $W$ with length $\leq c_0\TM_\pazocal{L}(c_0n)^3+c_0n+c_0$.

Hence, by \Cref{M_3(2) to M_3(2)}, it suffices to assume that there exists a factorization $H\equiv H_1\sigma H_2$ of the history $H$ of $\pazocal{D}$ such that:

\begin{itemize}

\item The subcomputation $\pazocal{D}_1:V_0\to\dots\to V_r$ with history $H_1$ is a computation of $\textbf{M}_3^\pazocal{L}(1)$

\item The subcomputation $\pazocal{D}_2:V_{r+1}\to\dots\to V_s$ with history $H_2$ is a computation of $\textbf{M}_3^\pazocal{L}(2)$

\end{itemize}

As above, $\pazocal{D}_2$ can be identified with a computation of $\textbf{M}_2^\pazocal{L}$ accepting $V_{r+1}$, so that \Cref{M_2 language} provides a computation $\pazocal{C}_2$ of $\textbf{M}_3^\pazocal{L}(2)$ accepting $V_{r+1}$ such that the history $H_2'$ of $\pazocal{C}_2$ satisfies: 
$$\|H_2'\|\leq c_0\TM_\pazocal{L}(c_0|V_{r+1}|_a)^3+c_0|V_{r+1}|_a+c_0$$
As $V_{r+1}$ is $\sigma^{-1}$-admissible, it corresponds to an input configuration of $\textbf{M}_2^\pazocal{L}$.  \Cref{M_2 language} then implies there exists $w\in\pazocal{L}$ such that $V_r\equiv q_0q_1\widetilde{\varphi}_2(w)q_2q_3\dots q_N$.

Let $\pazocal{D}_1':V_0'\to\dots\to V_r'$ be the restriction of $\pazocal{D}_1$ to the base $Q_0^\pazocal{L}Q_1^\pazocal{L}Q_2^\pazocal{L}$ and fix $w_1\in F(\pazocal{A}_1\sqcup\pazocal{B})$ and $w_2\in F(\pazocal{A}_2)$ such that $V_0'\equiv q_0w_1q_1w_2q_2$.  So, $W\equiv q_0w_1q_1w_2q_2\dots q_N$.

Then, $\pazocal{D}_1'$ can be identified with a reduced computation of $\textbf{M}_1^\pazocal{A}$ in the standard base, so that \Cref{projection argument} implies $|V_{r+1}|_a=\|w\|=\|\eps(V_t')\|=\|\eps(V_0')\|\leq\|w_1\|+\|w_2\|=|W|_a=n$.

Moreover, the restriction of $\pazocal{D}_1'$ to the base $Q_0^\pazocal{L}Q_1^\pazocal{L}$ can be identified with a computation of $\textbf{M}_1^\pazocal{A}$ with base $Q_0^\pazocal{A}Q_1^\pazocal{A}$ which is a shift of $w_1$.  \Cref{M_1 shift} then implies that $r\leq \|w_1\|+\|w_1\|c_0^{\|w_1\|}\leq n+nc_0^n$.

Thus, $H'\equiv H_1\sigma H_2'$ is the history of a reduced computation of $\textbf{M}_3^\pazocal{L}$ accepting $W$ such that:
$$\|H'\|\leq c_0\TM_\pazocal{L}(c_0n)^3+c_0n+c_0+1+n+nc_0^n$$
The statement then follows by taking $c_0\geq1$.

\end{proof}

As the only rules of $\textbf{M}_3^\pazocal{L}$ that do not lock the $Q_0^\pazocal{L}Q_1^\pazocal{L}$-sector are those of $\textbf{M}_3^\pazocal{L}(1)$, \Cref{semi locked sectors} implies that any non-trivial semi-computation of $\textbf{M}_3^\pazocal{L}$ in the $Q_0^\pazocal{L}Q_1^\pazocal{L}$-sector can be identified with a semi-computation of $\textbf{M}_1^\pazocal{A}$ in the $Q_0^{\pazocal{A}}Q_1^{\pazocal{A}}$-sector.  

\subsection{The machine $\textbf{M}_4^\pazocal{L}$} \label{sec-M_4} \

The noisy $S$-machine $\textbf{M}_4^\pazocal{L}$ is the composition of the machine $\textbf{M}_3^\pazocal{L}$ with a `reflected copy' of itself, introducing a level of symmetry to the model.  This composition is done in a manner similar to the methods employed in \cite{O18} and \cite{OS19}, and will be used explicitly in \Cref{sec-distortion-diagrams}.

Let $\pazocal{H}_3'=(\sqcup_{i=1}^N \pazocal{Y}_i^\pazocal{L},\sqcup_{i=0}^N R_i^\pazocal{L})$ be a copy of the hardware of $\textbf{M}_3^\pazocal{L}$.  The standard base of $\textbf{M}_4^\pazocal{L}$ is then: 
$$Q_0^\pazocal{L}Q_1^\pazocal{L}\dots Q_N^\pazocal{L}(R_N^\pazocal{L})^{-1}\dots(R_1^\pazocal{L})^{-1}(R_0^\pazocal{L})^{-1}$$
For each $i\in\{1,\dots,N\}$, the tape alphabet of the $Q_{i-1}^\pazocal{L}Q_i^\pazocal{L}$-sector is $Y_i^\pazocal{L}$, while that of the $(R_{i}^\pazocal{L})^{-1}(R_{i-1}^\pazocal{L})^{-1}$-sector is $\pazocal{Y}_i^\pazocal{L}$.  Finally, the tape alphabet of the $Q_N^\pazocal{L}(R_N^\pazocal{L})^{-1}$-sector is empty.

By construction, any configuration $W$ of $\textbf{M}_4^\pazocal{L}$ has an \textit{associated pair} of configurations of $\textbf{M}_3^\pazocal{L}$ $(W_1,W_2)$ such that $W\equiv W_1(W_2')^{-1}$ where $W_2'$ is the copy of $W_2$ over the hardware $\pazocal{H}_3'$.

The generalized rules of $\textbf{M}_4^\pazocal{L}$ correspond to those of $\textbf{M}_3^\pazocal{L}$, operating on admissible words whose base is a subword of either $Q_0^\pazocal{L}Q_1^\pazocal{L}\dots Q_N^\pazocal{L}$ or of $R_0^\pazocal{L}R_1^\pazocal{L}\dots R_N^\pazocal{L}$ as the corresponding rule operates on an analogous admissible word of $\textbf{M}_3^\pazocal{L}$.  

In particular, suppose the generalized rule $\theta$ of $\textbf{M}_3^\pazocal{L}$ has the part $q_i\to u_{i-1}q_i'v_i$.  Then, letting $r_i$ and $r_i'$ be the copies of $q_i$ and $q_i'$ in $R_i^\pazocal{L}$, respectively, then the corresponding rule $\bar{\theta}$ of $\textbf{M}_4^\pazocal{L}$ has the parts $q_i\to u_{i-1}q_i'v_i$ and $r_i^{-1}\to \bar{v}_i^{-1}(r_i')^{-1}\bar{u}_{i-1}^{-1}$, where $\bar{u}_{i-1}$ and $\bar{v}_i$ are the copies of $u_{i-1}$ and $v_i$ in $\pazocal{Y}_{i-1}^\pazocal{L}$ and $\pazocal{Y}_i^\pazocal{L}$, respectively.  

Further, for $\theta$ and $\bar{\theta}$ as above, $X_i(\bar{\theta})=X_i(\theta)$, $Z_i(\bar{\theta})=Z_i(\theta)$, and $f_{\bar{\theta},i}=f_{\theta,i}$ for all $i\in\{1,\dots,N\}$.  Similarly, $X_{2N+2-i}(\bar{\theta})$ and $Z_{2N+2-i}(\bar{\theta})$ are the natural copies of $X_i(\theta)$ and $Z_i(\theta)$ in $\pazocal{Y}_i^\pazocal{L}$, respectively, while the bijection $f_{\bar{\theta},2N+2-i}$ is the natural analogue of $f_{\theta,i}$.

Note that this construction may be applied to any generalized $S$-machine, and preserves the property of being a noisy $S$-machine (or of being an $S$-machine in the traditional sense).

As such, for any configuration $W$ with associated pair $(W_1,W_2)$, $W$ is $\bar{\theta}$-admissible if and only if both $W_1$ and $W_2$ are $\theta$-admissible, in which case $W\cdot\bar{\theta}$ is the configuration with associated pair $(W_1\cdot\theta,W_2\cdot\theta)$.  Hence, if $(W_1,W_2)$ is the associated pair of an accepted configuration of $\textbf{M}_4^\pazocal{L}$, then the parallel nature of the rules implies $W_1\equiv W_2$.  Consequently, any accepted configuration $W$ is essentially palindromic: $W^{-1}$ and $W$ are equivalent if $\pazocal{H}_3'$ is identified with the hardware of $\textbf{M}_3^\pazocal{L}$.  

This symmetry can be seen on another level: If $W$ is an admissible word whose base is a subword of $Q_0^\pazocal{L}Q_1^\pazocal{L}\dots Q_N^\pazocal{L}$, then define the \textit{reflection} of $W$ to be the admissible word which is the natural copy of $W^{-1}$ obtained over $\pazocal{H}_3'$.  For any rule $\bar{\theta}$, $W$ is $\bar{\theta}$-admissible if and only if its reflection is.

As the rules of $\textbf{M}_4^\pazocal{L}$ are in correspondence with the rules of $\textbf{M}_3^\pazocal{L}$ and operate similarly, the submachines $\textbf{M}_4^\pazocal{L}(1)$ and $\textbf{M}_4^\pazocal{L}(2)$ are defined as for $\textbf{M}_3^\pazocal{L}$.

Both the $Q_0^\pazocal{L}Q_1^\pazocal{L}$- and $(R_1^\pazocal{L})^{-1}(R_0^\pazocal{L})^{-1}$-sectors are the input sectors of $\textbf{M}_4^\pazocal{L}$, while the start and end letters are correspond to those of the machine $\textbf{M}_3^\pazocal{L}$.  In particular, letting $A_3$ be the accept configuration of $\textbf{M}_3^\pazocal{L}$, the accept configuration $A_4$ of $\textbf{M}_4^\pazocal{L}$ has associated pair $(A_3,A_3)$.

For any word $w\in F(\pazocal{A})$, let $I_3(w)$ be the input configuration of $\textbf{M}_3^\pazocal{L}$ with input $\widetilde{\varphi}_1(w)$.  Then, $I_4(w)$ is the input configuration of $\textbf{M}_4^\pazocal{L}$ whose associated pair is $(I_3(w),I_3(w))$.  

The following is thus a direct consequence of \Cref{M_3 language} and the mirror symmetry of the rules:

\begin{lemma} \label{M_4 language}


Suppose $W$ is an input configuration of $\textbf{M}_4^\pazocal{L}$ whose tape words are copies of words over $\pazocal{A}_1$.  Then, $W$ is accepted if and only if $W\equiv I_4(w)$ for some $w\in\pazocal{L}$.

\end{lemma}

%
%

Similarly, the next statement follows immediately from \Cref{M_3 time-space}:

\begin{lemma} \label{M_4 time-space}

For any accepted configuration $W$ of $\textbf{M}_4^\pazocal{L}$ with $|W|_a=n$, there exists an accepting computation $\pazocal{C}:W\equiv W\equiv W_0\to\dots\to W_t\equiv A_4$ satisfying $t\leq c_0\TM_\pazocal{L}(c_0n)^3+nc_0^{n}+c_0n+c_0$.

\end{lemma}

As all rules of $\textbf{M}_4^\pazocal{L}$ operate in the $Q_0^\pazocal{L}Q_1^\pazocal{L}$-sector in the same way as those of $\textbf{M}_3^\pazocal{L}$, semi-computations of $\textbf{M}_4^\pazocal{L}$ in this sector are the same as those in $\textbf{M}_3^\pazocal{L}$.  Hence, non-trivial semi-computations of $\textbf{M}_4^\pazocal{L}$ in the $Q_0^\pazocal{L}Q_1^\pazocal{L}$-sector can be identified with semi-computations of $\textbf{M}_1^\pazocal{A}$ in the $Q_0^\pazocal{A}Q_1^\pazocal{A}$-sector.

\medskip


\subsection{The machine $\textbf{M}_5^\pazocal{L}$.} \

The noisy $S$-machine $\textbf{M}_5^\pazocal{L}$ is the `circular' analogue of $\textbf{M}_4^\pazocal{L}$. It is defined in much the same way as the analogous machine in \cite{W}.

Letting $B_4^\pazocal{L}$ be the standard base of $\textbf{M}_4^\pazocal{L}$, the standard base of $\textbf{M}_5^\pazocal{L}$ is $\{t\}B_4^\pazocal{L}$, where $\{t\}$ is a singleton. The tape alphabet of the $\{t\}Q_0^\pazocal{L}$-sector is empty, while the tape alphabet of the other sectors are identified with the corresponding tape alphabets of $\textbf{M}_4^\pazocal{L}$.

However, there is a fundamental difference between $\textbf{M}_5^\pazocal{L}$ and the machines constructed in the previous sections: A tape alphabet is assigned to the space after $(R_0^\pazocal{L})^{-1}$, corresponding to the $(R_0^\pazocal{L})^{-1}\{t\}$-sector. As such, it is possible for an admissible word of $\textbf{M}_4^\pazocal{L}$ to have base $$(Q_1^\pazocal{L})^{-1}(Q_0^\pazocal{L})^{-1}\{t\}^{-1}R_0^\pazocal{L}(R_0^\pazocal{L})^{-1}\{t\}Q_0^\pazocal{L}Q_1^\pazocal{L}$$
{\frenchspacing i.e. so that it essentially} `wraps around' the standard base. A generalized $S$-machine with this property is called \textit{cyclic}, as the standard base can be visualized as being written on a circle.

In this machine, the tape alphabet of the $(R_0^\pazocal{L})^{-1}\{t\}$-sector is taken to be empty. The generalized rules of $\textbf{M}_5^\pazocal{L}$ correspond to those of $\textbf{M}_4^\pazocal{L}$, operating on the copy of the hardware of $\textbf{M}_4^\pazocal{L}$ in the same way and, as is compulsory by the definition of the tape alphabets, locking the new sectors.  Hence, $\textbf{M}_5^\pazocal{L}$ can indeed be viewed as a noisy $S$-machine, though with this new `circular' property.

As with the previous machine, the submachines $\textbf{M}_5^\pazocal{L}(1)$ and $\textbf{M}_5^\pazocal{L}(2)$ are adopted from the submachines of $\textbf{M}_3^\pazocal{L}$.  Similarly, any admissible word whose base is a subword of $Q_0^\pazocal{L}Q_1^\pazocal{L}\dots Q_N^\pazocal{L}$ has a \textit{reflection}, capturing the symmetry inherent to the machine.

The input sectors, start letters, and end letters of $\textbf{M}_5^\pazocal{L}$ are analogous to those of $\textbf{M}_4^\pazocal{L}$ (with the start and end letter of the part $\{t\}$ taken to be the only letter).  For any $w\in F(\pazocal{A})$, the configuration $tI_4(w)$ is thus an input configuration of $\textbf{M}_5^\pazocal{L}$, hereby denoted $I_5(w)$.

So, since the newly introduced sectors have empty tape alphabet, the following statements are direct consequences of Lemmas \ref{M_4 language} and \ref{M_4 time-space}:

%
%




%
%

\begin{lemma} \label{M_5 language}


Suppose $W$ is an input configuration of $\textbf{M}_5^\pazocal{L}$ whose tape words are copies of words over $\pazocal{A}_1$.  Then, $W$ is accepted if and only if $W\equiv I_5(w)$ for some $w\in\pazocal{L}$.

\end{lemma}

\begin{lemma} \label{M_5 time-space}

For any accepted configuration $W$ of $\textbf{M}_5^\pazocal{L}$ with $|W|_a=n$, there exists an accepting computation $\pazocal{C}:W\equiv W_0\to\dots\to W_t$ satisfying $t\leq c_0\TM_\pazocal{L}(c_0n)^3+nc_0^{n}+c_0n+c_0$.

\end{lemma}

Again, the rules of $\textbf{M}_5^\pazocal{L}$ are in correspondence with those of $\textbf{M}_3^\pazocal{L}$ and operate in the $Q_0^\pazocal{L}Q_1^\pazocal{L}$-sector analogously.  Hence, non-trivial semi-computations of $\textbf{M}_5^\pazocal{L}$ in the $Q_0^\pazocal{L}Q_1^\pazocal{L}$-sector can be identified with semi-computations of $\textbf{M}_1^\pazocal{A}$ in the $Q_0^\pazocal{A}Q_1^\pazocal{A}$-sector.

\medskip


\subsection{The machines $\textbf{M}_{6,1}^\pazocal{L}$ and $\textbf{M}_{6,2}^\pazocal{L}$.} \

The cyclic generalized $S$-machine $\textbf{M}_{6,1}^\pazocal{L}$ functions as the `parallel' composition of $\textbf{M}_5^\pazocal{L}$ with itself a (very large) number of times.

Letting $L$ be the parameter listed in \Cref{sec-parameters}, let $B_4^{\pazocal{L},1}(i)$ be a copy of the standard base $B_4^\pazocal{L}$ of $\textbf{M}_4^\pazocal{L}$ for all $1\leq i\leq L$, i.e with:
$$B_4^{\pazocal{L},1}(i)=Q_0^{\pazocal{L},1}(i)Q_1^{\pazocal{L},1}(i)\dots Q_N^{\pazocal{L},1}(i)(R_N^{\pazocal{L},1}(i))^{-1}\dots(R_1^{\pazocal{L},1}(i))^{-1}(R_0^{\pazocal{L},1}(i))^{-1}$$ 
Then the standard base of $\textbf{M}_{6,1}^\pazocal{L}$ is:
$$\{t(1)\}B_4^{\pazocal{L},1}(1)\{t(2)\}B_4^{\pazocal{L},1}(2)\dots\{t(L)\}B_4^{\pazocal{L},1}(L)$$
For any letter of $\{t(i)\}B_4^{\pazocal{L},1}(i)$ (or its inverse), the index $i$ is called its \textit{coordinate}.


The tape alphabet of any sector containing a singleton $\{t(i)\}$ (including the $(R_0^{\pazocal{L},1}(L))^{-1}\{t(1)\}$-sector) is taken to be empty, while the tape alphabet of any other sector is a copy of the tape alphabet of the corresponding sector of $\textbf{M}_5^\pazocal{L}$.

The generalized rules of $\textbf{M}_{6,1}^\pazocal{L}$ are in correspondence with those of $\textbf{M}_5^\pazocal{L}$, with each rule operating on every subword $\{t(i)\}B_4^{\pazocal{L},1}(i)$ of the standard base as the corresponding rule.  As such, $\textbf{M}_{6,1}^\pazocal{L}$ is a noisy $S$-machine which can be viewed as the composition of the submachines $\textbf{M}_{6,1}^\pazocal{L}(1)$ and $\textbf{M}_{6,1}^\pazocal{L}(2)$.

The input sectors of $\textbf{M}_{6,1}^\pazocal{L}$ are taken to be all the $Q_0^{\pazocal{L},1}(i)Q_1^{\pazocal{L},1}(i)$- and $(R_1^{\pazocal{L},1}(i))^{-1}(R_0^{\pazocal{L},1}(i))^{-1}$-sectors, while the start and end letters are taken to be the copies of those of $\textbf{M}_5^\pazocal{L}$.

Clearly, the statements of the previous section pertaining to the machine $\textbf{M}_5^\pazocal{L}$ have natural analogues to the machine $\textbf{M}_{6,1}^\pazocal{L}$. For example, for $w\in F(\pazocal{A})$, let $I_6(w)$ be the input configuration such that every admissible subword with base $\{t(i)\}B_4^{\pazocal{L},1}(i)$ is the natural copy of $I_5(w)$. 

The following statement is then the analogue of Lemma \ref{M_5 language}:

\begin{lemma} \label{M_{6,1} language}


Suppose $W$ is an input configuration of $\textbf{M}_{6,1}^\pazocal{L}$ whose tape words are copies of words over $\pazocal{A}_1$.  Then, $W$ is accepted if and only if $W\equiv I_6(w)$ for some $w\in\pazocal{L}$.

\end{lemma}

\begin{proof}

Given a reduced computation $\pazocal{C}$ accepting the input configuration $W$, the restriction of $\pazocal{C}$ to the base $\{t(2)\}B_4^{\pazocal{L},1}(2)$ can be identified with a reduced computation of $\textbf{M}_5^\pazocal{L}$.  \Cref{M_5 language} then implies the corresponding admissible subword of $W$ is the copy of $I_5(w)$ for some $w\in\pazocal{L}$.  But then the parallel nature of the rules yields $W\equiv I_6(w)$.

Conversely, given $w\in\pazocal{L}$, then \Cref{M_5 language} provides a reduced computation of $\textbf{M}_5^\pazocal{L}$ accepting $I_5(w)$, which then corresponds to a reduced computation of $\textbf{M}_{6,1}^\pazocal{L}$ accepting $I_6(w)$.

%

\end{proof}

The cyclic generalized $S$-machine $\textbf{M}_{6,2}^\pazocal{L}$ is constructed in much the same way as $\textbf{M}_{6,1}^\pazocal{L}$, but with one major alteration.

Letting $B_4^{\pazocal{L},2}(i)$ be a distinct copy of $B_4^\pazocal{L}$ for all $i\in\{1,\dots,L\}$, the standard base of $\textbf{M}_{6,2}^\pazocal{L}$ is
$$\{t(1)\}B_4^{\pazocal{L},2}(1)\{t(2)\}B_4^{\pazocal{L},2}(2)\dots\{t(L)\}B_4^{\pazocal{L},2}(L)$$
Similarly, the tape alphabets of $\textbf{M}_{6,2}^\pazocal{L}$ are defined in just the same way as those of $\textbf{M}_{6,1}^\pazocal{L}$.

However, there is one fundamental difference between $\textbf{M}_{6,2}^\pazocal{L}$ and its predecessor: While the positive rules of $\textbf{M}_{6,2}^\pazocal{L}$ are copies of those of $\textbf{M}_5^\pazocal{L}$, each locks the $Q_0^{\pazocal{L},2}(1)Q_1^{\pazocal{L},2}(1)$-sector. This sector is still called an input sector, though any configuration must have this sector empty for it to be $\theta$-admissible for any rule $\theta$ of $\textbf{M}_{6,2}^\pazocal{L}$.

Again, the statements from the previous section have analogues to the machine $\textbf{M}_{6,2}^\pazocal{L}$. For example, for any $w\in F(\pazocal{A})$, let $J_6(w)$ be the input configuration analogous to $I_6(w)$ except with empty $Q_0^{\pazocal{L},2}(1)Q_1^{\pazocal{L},2}(1)$-sector.  Then, the following statement is the analogue of Lemma \ref{M_5 language}, proved in much the same way as Lemma \ref{M_{6,1} language}:

\begin{lemma} \label{M_{6,2} language}


Suppose $W$ is an input configuration of $\textbf{M}_{6,2}^\pazocal{L}$ whose tape words are copies of words over $\pazocal{A}_1$.  Then, $W$ is accepted if and only if $W\equiv J_6(w)$ for some $w\in\pazocal{L}$.

\end{lemma}

\medskip


\section{The Main Machine}

\subsection{The machine $\textbf{M}^\pazocal{L}$} \

The main machine of this construction, the noisy $S$-machine $\textbf{M}^\pazocal{L}$, is the concatenation of the machines $\textbf{M}_{6,1}^\pazocal{L}$ and $\textbf{M}_{6,2}^\pazocal{L}$. However, unlike the compositions described in previous sections (but similar to the construction of the main machine of \cite{W}), the concatenation of these machines is done in a way so that they run `one or the other' instead of `one after another' or `in parallel'.

For every $j\in\{0,\dots,N\}$ and $i\in\{1,\dots,L\}$, define the sets:

\begin{itemize}

\item $Q_j^\pazocal{L}(i)=Q_j^{\pazocal{L},1}(i)\sqcup Q_j^{\pazocal{L},2}(i)\sqcup\{q_{j,s}(i),q_{j,a}(i)\}$ 

\item $R_j^\pazocal{L}(i)=R_j^{\pazocal{L},1}(i)\sqcup R_j^{\pazocal{L},2}(i)\sqcup\{r_{j,s}(i),r_{j,a}(i)\}$

\end{itemize}
Further, for all $i\in\{1,\dots,L\}$, denote $B_4^\pazocal{L}(i)=Q_0^\pazocal{L}(i)\dots Q_N^\pazocal{L}(i)(R_N^\pazocal{L}(i))^{-1}\dots(R_0^\pazocal{L}(i))^{-1}$.

Then, the standard base of $\textbf{M}^\pazocal{L}$ is:
$$\left(\{t(1)\}B_4^\pazocal{L}(1)\right)\left(\{t(2)\}B_4^\pazocal{L}(2)\right)\dots\left(\{t(L)\}B_4^\pazocal{L}(L)\right)$$


Similar to the setup of the machines $\textbf{M}_{6,1}^\pazocal{L}$ and $\textbf{M}_{6,2}^\pazocal{L}$, the input sectors of $\textbf{M}^\pazocal{L}$ are taken to be the $Q_0^\pazocal{L}(i)Q_1^\pazocal{L}(i)$- and $(R_1^\pazocal{L}(i))^{-1}(R_0^\pazocal{L}(i))^{-1}$-sectors.  For any $i$ and $j$, the letters $q_{j,s}(i)$ and $r_{j,s}(i)^{-1}$ are taken to be the start letters of $Q_j^\pazocal{L}(i)$ and $(R_j^\pazocal{L}(i))^{-1}$, respectively.  Similarly, $q_{j,a}(i)$ and $r_{j,a}(i)^{-1}$ are the end letters of $Q_j^\pazocal{L}(i)$ and $(R_j^\pazocal{L}(i))^{-1}$, respectively.

For any non-input sector, the associated tape alphabet is a copy of the corresponding tape alphabet of $\textbf{M}_{6,1}^\pazocal{L}$ (which is identified with the corresponding tape alphabet of $\textbf{M}_{6,2}^\pazocal{L}$). In particular, the tape alphabet of the $Q_0^\pazocal{L}(1)Q_1^\pazocal{L}(1)$-sector is identified with the alphabet $\pazocal{A}\sqcup\pazocal{A}_1\sqcup\pazocal{B}$.  As promised in \Cref{sec-M_1}, the reason for the inclusion of $\pazocal{A}$ in the input tape alphabets will finally become clear in this construction.

The set of generalized $S$-rules of $\textbf{M}^\pazocal{L}$, $\Theta$, is the disjoint union of two symmetric sets, denoted $\Theta_1$ and $\Theta_2$. Naturally, the positive (and negative) generalized rules are partitioned accordingly, i.e with $\Theta^+=\Theta_1^+\sqcup\Theta_2^+$ with $\Theta_i^+=\Theta^+\cap\Theta_i$ for $i=1,2$.

The rules of $\Theta_1^+$ are defined as follows:

\begin{itemize}

\item The transition rule $\theta(s)_1$ locks all sectors other than the input sectors and switches the state letters from the start letters of the machine to the copies of the start letters of $\textbf{M}_{6,1}^\pazocal{L}$. For each $i$ corresponding to an input sector, $X_i(\theta(s)_1)$ is the copy of $\pazocal{A}$, $Z_i(\theta(s)_1)$ is the copy of $\pazocal{A}_1$, and $f_{\theta(s)_1,i}$ operates as $\varphi_1$.  

\item The positive `working' rules of $\Theta_1$ correspond to the positive generalized $S$-rules of $\textbf{M}_{6,1}^\pazocal{L}$, with each rule operating on the copy of the hardware of $\textbf{M}_{6,1}^\pazocal{L}$ as its corresponding rule.

\item The transition rule $\theta(a)_1$ locks all sectors and switches the state letters from the copy of the end letters of $\textbf{M}_{6,1}^\pazocal{L}$ to the end letters of the machine.

\end{itemize}

The rules of $\Theta_2^+$ are defined as follows:

\begin{itemize}

\item The transition rule $\theta(s)_2$ operates in a similar manner to the rule $\theta(s)_1$, but with two exceptions: (i) The input $Q_0^{\pazocal{L}}(1)Q_1^{\pazocal{L}}(1)$-sector is locked, and (ii) The state letters are switched from the start letters of the machine to the copies of the start letters of $\textbf{M}_{6,2}^\pazocal{L}$.

\item The positive `working' rules of $\Theta_2$ correspond to the positive generalized $S$-rules of $\textbf{M}_{6,2}^\pazocal{L}$, with each rule operating on the copy of the hardware of $\textbf{M}_{6,2}^\pazocal{L}$ as its corresponding rule.

\item The transition rule $\theta(a)_2$ locks all sectors and switches the state letters from the copy of the end letters of $\textbf{M}_{6,2}^\pazocal{L}$ to the end letters of the machine.

\end{itemize}

The definition of the rules of $\textbf{M}^\pazocal{L}$ make it evident that the $Q_0^\pazocal{L}(1)Q_1^{\pazocal{L}}(1)$-sector stands out amongst the input sectors. Thus, it is henceforth fittingly referred to as the `special' input sector.

Observe that the operation of the positive rules $\theta(s)_j$ in the input sectors (other than the `special' input sector when $j=2$) satisfies condition (2) in the definition of noisy $S$-machines in \Cref{sec-noisy}, justifying the claiml that $\textbf{M}^\pazocal{L}$ is a noisy $S$-machine.



%

Note that for $w\in F(\pazocal{A})$, the natural copies of $I_6(w)$ and $J_6(w)$ in the hardware of $\textbf{M}^\pazocal{L}$ are configurations which are $\theta(s)_1^{-1}$-admissible and $\theta(s)_2^{-1}$-admissible, respectively. The configurations $I(w)$ and $J(w)$ are then defined to be the configurations resulting from applying these respective rules. Hence, $I(w)$ is the input configuration with the corresponding copy of $w$ written in each $Q_0^\pazocal{L}(i)Q_1^\pazocal{L}(i)$-sector and the copy of $w^{-1}$ written in each $(R_1^\pazocal{L}(i))^{-1}(R_0^\pazocal{L}(i))^{-1}$-sector, while $J(w)$ is the input configuration obtained from $I(w)$ by erasing the copy of $w$ in the `special' input sector.

\medskip

\subsection{Standard computations of $\textbf{M}^\pazocal{L}$} \label{sec-M-standard} \

As in \cite{W}, a reduced computation $\pazocal{C}$ is called a \textit{one-machine computation of the $i$-th machine} if every letter of the history of $\pazocal{C}$ corresponds to a rule of $\Theta_i$, i.e $H\in F(\Theta_i^+)$ for $H$ the history of $\pazocal{C}$. If $\pazocal{C}$ is not a one-machine computation, then it is called a \textit{multi-machine computation}.

The next two statements follow immediately by from the definition of the rules of $\textbf{M}^\pazocal{L}$.

\begin{lemma} \label{(s) (a) first or last}

Suppose $\pazocal{C}:W_0\to\dots\to W_t$ is a one-machine computation of the $i$-th machine in the standard base.  Then:

\begin{enumerate} [label=(\alph*)]

\item Any occurrence of $\theta(s)_i$ or of $\theta(a)_i^{-1}$ in the history of $\pazocal{C}$ is as the first letter.

\item Any occurrence of $\theta(s)_i^{-1}$ or of $\theta(a)_i$ in the history of $\pazocal{C}$ is as the last letter.

\end{enumerate}

\end{lemma}

%
%
%
%
%
%
%
%
%
%
%

\begin{lemma} \label{multi-machine step history}

Let $\pazocal{C}$ be a multi-machine computation of $\textbf{M}^\pazocal{L}$ in the standard base.  Suppose there exists a factorization $H\equiv H_1H_2$ of the history $H$ of $\pazocal{C}$ such that for $i=1,2$, the subcomputation $\pazocal{C}_i$ with history $H_i$ is a one-machine computation of the $i$-th machine.  Then either:

\begin{enumerate} [label=(\alph*)]

\item The last letter of $H_1$ is either $\theta(s)_1^{-1}$ and the first letter of $H_2$ is $\theta(s)_2$; or

\item The last letter of $H_1$ is $\theta(a)_1$ and the first letter of $H_2$ is $\theta(a)_2^{-1}$.

\end{enumerate}

\end{lemma}

%
%
%
%
%
%

\begin{lemma} \label{one-machine language}

For a start configuration $W$, there exists a one-machine computation of the first (respectively second) machine accepting $W$ if and only if there exists $w\in\pazocal{L}$ such that $W\equiv I(w)$ (respectively $W\equiv J(w)$). 

\end{lemma}

\begin{proof}

Given $w\in\pazocal{L}$, $I(w)\cdot\theta(s)_1$ is the copy of $I_6(w)$ over the copy of the hardware of $\textbf{M}_{6,1}^\pazocal{L}$.  But then \Cref{M_{6,1} language} implies there exists a one-machine computation of the first history between $I(w)\cdot\theta(s)_1$ and a configuration which is $\theta(a)_1$-admissible.  Concatenating these rules thus yields a one-machine computation of the first machine accepting $I(w)$.

An analogous argument applying \Cref{M_{6,2} language} to $J(w)\cdot\theta(s)_2$ produces a one-machine computation of the second machine accepting $J(w)$.

%
%

Now, suppose on the other hand that $\pazocal{C}:W\equiv W_0\to\dots\to W_t\equiv W_{ac}$ is a one-machine computation of the $i$-th machine such that $W$ is a start configuration.  Then the makeup of $\Theta$ dictates $W_1\equiv W\cdot\theta(s)_i$ and $W_{t-1}\equiv W_{ac}\cdot\theta(a)_i^{-1}$.  Hence, the subcomputation $W_1\to\dots\to W_{t-1}$ can be identified with a reduced computation of $\textbf{M}_{6,i}^\pazocal{L}$ accepting an input configuration whose tape words are all copies of words over $\pazocal{A}_1$.  But then the statement follows from \Cref{M_{6,1} language} if $i=1$ and \Cref{M_{6,2} language} if $i=2$.
%


\end{proof}

\begin{lemma} \label{one-machine I to I}

Let $\pazocal{C}:W_0\to\dots\to W_t$ be a one-machine computation of the first machine in the standard base. Suppose $W_t$ is a start configuration and $W_0\equiv I(u)$ for some $u\in\pazocal{L}$.  Then there exists $v\in\pazocal{L}$ such that $W_t\equiv I(v)$.

\end{lemma}

\begin{proof}

By \Cref{one-machine language}, there exists a one-machine computation of the first machine $\pazocal{D}_1$ accepting $I(u)$.  Letting $H$ and $H_1$ be the histories of $\pazocal{C}$ and $\pazocal{D}_1$, respectively, $H^{-1}H_1$ is freely equal to the history of a one-machine computation of the first machine accepting $W_t$.  The statement thus follows by \Cref{one-machine language}.

%
%
%
%

\end{proof}

The next statement is similarly implied by \Cref{one-machine language}:

\begin{lemma} \label{one-machine J to J}

Let $\pazocal{C}:W_0\to\dots\to W_t$ be a one-machine computation of the second machine in the standard base. Suppose $W_t$ is a start configuration and $W_0\equiv J(u)$ for some $u\in\pazocal{L}$.  Then there exists $v\in\pazocal{L}$ such that $W_t\equiv J(v)$.

\end{lemma}


%
%
%


\begin{lemma} \label{first machine J to J}

Let $\pazocal{C}:W_0\to\dots\to W_t$ be a one-machine computation of the first machine in the standard base. Suppose $W_t$ is a start configuration and $W_0\equiv J(u)$ for some $u\in\pazocal{L}$. Then $t=0$.

\end{lemma}

\begin{proof}

Suppose to the contrary that $t>0$.  

\Cref{(s) (a) first or last} then implies that there exists a factorization $H\equiv\theta(s)_1H'\theta(s)_1^{-1}$ of the history of $\pazocal{C}$ such that $H'$ is a non-empty word consisting entirely of working rules of the first machine.  The subcomputation $\pazocal{C}':W_1\to\dots\to W_{t-1}$ with history $H'$ can then be identified with a reduced computation of $\textbf{M}_{6,1}^\pazocal{L}$.

Suppose this is a computation of $\textbf{M}_{6,1}^\pazocal{L}(1)$.  Then, the restriction of $\pazocal{C}'$ to the base $Q_0^\pazocal{L}(1)Q_1^\pazocal{L}(1)Q_2^\pazocal{L}(1)$ can be identified with a reduced computation $\pazocal{D}_1:V_1\to\dots\to V_{t-1}$ of $\textbf{M}_1^\pazocal{A}$ in the standard base.  Since $W_{t-1}$ is $\theta(s)_1^{-1}$-admissible, $V_{t-1}$ must be of the form $q_0\widetilde{\varphi}_1(v)q_1q_2$ for some $v\in F(\pazocal{A})$.  As a result, $\pazocal{D}_1$ satisfies the hypotheses of \Cref{M_1 return to start}, so that $v=1$.  But then the restriction of $\pazocal{D}_1$ to the base $Q_0^\pazocal{A}Q_1^\pazocal{A}$ satisfies the hypotheses of \Cref{M_1 shift of 1}, yielding the contradiction $H'=1$.

Hence, $H'$ has a maximal proper prefix $H_1'$ such that the subcomputation $\pazocal{C}_1':W_1\to\dots\to W_r$ with history $H_1'$ can be identified with a computation of $\textbf{M}_{6,1}^\pazocal{L}(1)$.

For any $i\in\{1,\dots,L\}$, the restriction of $\pazocal{C}_1'$ to the base $Q_0^\pazocal{L}(i)\dots Q_N^\pazocal{L}(i)$ can be identified with a reduced computation $\pazocal{D}_i:U_1^{(i)}\to\dots\to U_r^{(i)}$ of $\textbf{M}_3^\pazocal{L}$ in the standard base.  Then, as $H_1'$ is a proper prefix of $H'$, $U_r^{(i)}$ must be $\sigma$-admissible for each $i$.  By construction, $U_1^{(1)}\equiv q_0q_1q_2\dots q_N$ and $U_1^{(j)}\equiv q_0\widetilde{\varphi}_1(u)q_1q_2\dots q_N$ for each $j\geq2$.  So, \Cref{M_3 start to sigma 1} implies that $U_r^{(1)}\equiv U_1^{(1)}$ and $U_r^{(j)}\equiv q_0q_1\widetilde{\varphi}_2(u)q_2\dots q_N$ for $j\geq2$.

Hence, the configuration $W_r$ has empty $Q_1^\pazocal{L}(1)Q_2^\pazocal{L}(1)$-sector and the corresponding copy of $\widetilde{\varphi}_2(u)$ written in the $Q_1^\pazocal{L}(j)Q_2^\pazocal{L}(j)$-sector for each $2\leq j\leq L$.   But all rules operate in parallel on the $Q_1^\pazocal{L}(i)Q_2^\pazocal{L}(i)$-sectors, so that the condition $u\neq1$ necessitated by $1\notin\pazocal{L}$ produces a contradiction.

%
%

\end{proof}

For any non-empty reduced computation $\pazocal{C}$ of $\textbf{M}^\pazocal{L}$, define $\ell(\pazocal{C})$ to be the number of maximal non-empty one-machine subcomputations of $\pazocal{C}$.  

Further, for any accepted configuration $W$ of $\textbf{M}^\pazocal{L}$, let $A(W)$ be the set of accepting computations of $W$.  If $W\neq W_{ac}$, then define $\ell(W)=\min\{\ell(\pazocal{C})\mid\pazocal{C}\in A(W)\}$.  For completeness, also set $\ell(W_{ac})=0$.

\begin{lemma} \label{ell(W)}

For any accepted configuration $W$ of $\textbf{M}^\pazocal{L}$, $\ell(W)\leq2$.  
\newline 
Moreover, if $\ell(W)=2$, then $W$ is not a start configuration and for any $\pazocal{C}\in A(W)$ with $\ell(\pazocal{C})=2$, there exists a factorization $H\equiv H_1H_2$ of the history of $\pazocal{C}$ such that: 

\begin{enumerate} [label=(\alph*)]

\item $H_i$ is the history of a one-machine computation of the $i$-th machine.

\item $W\cdot H_1\equiv J(w)$ for some $w\in\pazocal{L}$.

\end{enumerate}

\end{lemma}

\begin{proof}

Suppose $\ell(W)\geq2$ and let $\pazocal{C}\in A(W)$ be an accepting computation such that $\ell(\pazocal{C})=\ell(W)$.  Factor the history $H\equiv H_1\dots H_\ell$ of $\pazocal{C}$ such that $\ell=\ell(W)$ and each $H_j$ is the history of a non-empty maximal one-machine subcomputation of $\pazocal{C}$.  

For each $1\leq j\leq\ell$ fix $i(j)\in\{1,2\}$ such that $H_j$ is the history of a one-machine computation of the $i(j)$-th machine.  Note that $i(j)\neq i(j+1)$ for any $1\leq j\leq\ell-1$.

Suppose the last letter of $H_j$ is $\theta(a)_{i(j)}$.  Then, the configuration $W\cdot(H_1\dots H_j)$ must be the accept configuration $W_{ac}$, so that $H_1\dots H_j$ is the history of a reduced computation $\pazocal{D}\in A(W)$ with $\ell(\pazocal{D})=j\leq\ell$.  By the definition of $\ell$, we must then have $j=\ell$.

Hence, \Cref{multi-machine step history} implies that for any $1\leq j<\ell$, the last letter of $H_j$ must be $\theta(s)_{i(j)}^{-1}$.

Let $V_j\equiv W\cdot(H_1\dots H_j)$ for all $1\leq j<\ell$.  Then, $V_j$ is $\theta(s)_{i(j)}$-admissible, and so must be a start configuration.  \Cref{multi-machine step history} then also implies that the first letter of $H_{j+1}$ is $\theta(s)_{i(j+1)}$, {\frenchspacing i.e. $V_j$ must also be} $\theta(s)_{i(j+1)}$-admissible.  As a result, $V_j$ is both $\theta(s)_1$- and $\theta(s)_2$-admissible, and so must have empty `special' input sector.

In particular, $H_\ell$ is the history of a one-machine computation accepting the start configuration $V_{\ell-1}$, and thus by \Cref{one-machine language} there exists $u\in\pazocal{L}$ such that either $V_{\ell-1}\equiv I(u)$ (if $i(\ell)=1$) or $V_{\ell-1}\equiv J(u)$ (if $i(\ell)=2$).  But $V_{\ell-1}$ has empty `special' input sector, and so the assumption $1\notin\pazocal{L}$ implies $V_{\ell-1}\equiv J(u)$ and $i(\ell)=2$.

Assuming $\ell>2$, let $\pazocal{C}_{\ell-1}:V_{\ell-1}\to\dots\to V_{\ell-2}$ be the reduced computation with history $H_{\ell-1}^{-1}$.  Then, $\pazocal{C}_{\ell-1}$ is a one-machine computation of the first machine and $V_{\ell-1}\equiv J(u)$. But then $\pazocal{C}_{\ell-1}$ satisfies the hypotheses of \Cref{first machine J to J}, contradicting the assumption that each $H_j$ is non-empty.

Hence, $\ell(W)=2$.  By analogy, if $W$ is a start configuration, then the reduced computation $V_1\to\dots\to W$ with history $H_1^{-1}$ satisfies the hypotheses of \Cref{first machine J to J}, again contradicting the assumption that $H_1$ is non-empty.  Thus, the statement follows by construction.

\end{proof}

The following is thus a corollary of Lemmas \ref{one-machine language} and \ref{ell(W)}:

\begin{lemma} \label{M language}

A start configuration $W$ is accepted by $\textbf{M}^\pazocal{L}$ if and only if there exists $w\in\pazocal{L}$ such that either $W\equiv I(w)$ or $W\equiv J(w)$.

\end{lemma}

For any configuration $W$ of $\textbf{M}^\pazocal{L}$ and any $i\in\{1,\dots,L\}$, the \textit{$i$-th component of $W$}, denoted $W(i)$, is the admissible subword of $W$ with base $\{t(i)\}B_4^\pazocal{L}(i)$.

Since the tape alphabet of the $(R_0^\pazocal{L}(i))^{-1}\{t(i+1)\}$-sector is empty for each $i$ (taking $t(L+1)=t(1)$), any configuration is the concatenation of its components, i.e $W\equiv W(1)\dots W(L)$.

\begin{lemma} \label{M component size}

For any accepted configuration $W$ of $\textbf{M}^\pazocal{L}$ satisfying $\ell(W)=1$, $|W(1)|_a\leq|W(j)|_a$ for all $2\leq j\leq L$.

\end{lemma}

\begin{proof}

By hypothesis, there exists a non-empty computation $\pazocal{C}:W\equiv W_0\to\dots\to W_t\equiv W_{ac}$ with $\pazocal{C}\in A(W)$ and $\ell(\pazocal{C})=1$.
If $\pazocal{C}$ is a one-machine computation of the first machine, then the parallel nature of the rules implies $W(1)$ is a copy of $W(j)$ for each $2\leq j\leq L$.  Otherwise, $W(1)$ is is the copy of any $W(j)$ obtained by emptying the `special' input sector.

%
%

\end{proof}

\medskip


\subsection{Extending computations} \label{sec-almost-extendable} \

For simplicity, for any $w\in F(\pazocal{A})$ and $1\leq i\leq L$, we adopt the notation $I(w,i)\equiv(I(w))(i)$ and $J(w,i)\equiv(J(w))(i)$.  

Given an admissible word $V$ whose base consists entirely of letters with coordinate $i$, a \textit{coordinate shift} $V'$ of $V$ is an admissible word obtained from $V$ by changing each of the state letters' coordinates to some index $j\in\{1,\dots,L\}$ and taking the corresponding copies of the tape words. 

For example, for any $w\in F(\pazocal{A})\setminus\{1\}$, $J(w,i)$ and $J(w,j)$ are coordinate shifts of one another for $i,j\geq2$, but not of $J(w,1)$.

The next statement follows in just the same way as its analogue in \cite{W}:


\begin{lemma}[Lemma 5.9 in \cite{W}] \label{extend one-machine}

Let $V_0\to\dots\to V_t$ be a one-machine computation of the $i$-th machine with history $H$ and base $\{t(j)\}B_4^\pazocal{L}(j)$ for some $j\in\{2,\dots,L\}$. Then there exists a one-machine computation of the $i$-th machine $W_0\to\dots\to W_t$ in the standard base with history $H$ such that $W_\ell(j)\equiv V_\ell$ for all $\ell\in\{0,\dots,t\}$.  Moreover:

\begin{enumerate} [label=(\alph*)]

\item If $V_\ell\equiv W_{ac}(j)$, then $W_\ell\equiv W_{ac}$

\item If $V_\ell\equiv I(w,j)$ for some $w\in F(\pazocal{A})$, then

\begin{itemize}

\item $W_\ell\equiv I(w)$ if $i=1$, or

\item $W_\ell\equiv J(w)$ if $i=2$.

\end{itemize}

\end{enumerate}

\end{lemma}

%
%
%
%
%
%
%

The next three statements follow as corollaries of \Cref{extend one-machine}, using Lemmas \ref{one-machine language}, \ref{one-machine I to I}, and \ref{one-machine J to J}:

\begin{lemma} \label{one-machine projected end to end}

Let $\pazocal{C}:V_0\to\dots\to V_t$ be a one-machine computation of the $i$-th machine with history $H$ and base $\{t(j)\}B_4^\pazocal{L}(j)$ for some $j\in\{2,\dots,L\}$. Suppose $V_0\equiv V_t\equiv W_{ac}(j)$. Then $W_{ac}$ is $H$-admissible and $W_{ac}\cdot H\equiv W_{ac}$.

\end{lemma}

%
%

\begin{lemma} \label{one-machine projected start to end}

Let $\pazocal{C}:V_0\to\dots\to V_t$ be a one-machine computation of the $i$-th machine with history $H$ and base $\{t(j)\}B_4^\pazocal{L}(j)$ for some $j\in\{2,\dots,L\}$. Suppose $V_0$ is a start configuration and $V_t\equiv W_{ac}(j)$. Then there exists $u\in\pazocal{L}$ such that $V_0\equiv I(u,j)$.
\newline
Moreover, if $i=1$, then $I(u)$ is $H$-admissible with $I(u)\cdot H\equiv W_{ac}$; and if $i=2$, then $J(u)$ is $H$-admissible with $J(u)\cdot H\equiv W_{ac}$.

\end{lemma}

%
%

\begin{lemma} \label{one-machine projected start to start}

Let $\pazocal{C}:V_0\to\dots\to V_t$ be a one-machine computation of the $i$-th machine with history $H$ and base $\{t(j)\}B_4^\pazocal{L}(j)$ for some $j\in\{2,\dots,L\}$. Suppose $V_0\equiv I(u,j)$ for some $u\in\pazocal{L}$ and $V_t$ is an admissible subword of a start configuration. Then there exists $v\in\pazocal{L}$ such that $V_t\equiv I(v,j)$.  
\newline
Moreover, if $i=1$, then $I(u)$ is $H$-admissible with $I(u)\cdot H\equiv I(v)$; and if $i=2$, then $J(u)$ is $H$-admissible with $J(u)\cdot H\equiv J(v)$.

\end{lemma}

%
%
%

The following statement is thus a direct consequence of Lemmas \ref{one-machine projected end to end}, \ref{one-machine projected start to end}, and \ref{one-machine projected start to start}:

\begin{lemma} \label{projected end to end}

Let $j\in\{2,\dots,L\}$ and suppose $\pazocal{C}:W_{ac}(j)\to\dots\to W_{ac}(j)$ is a reduced computation. Let $H\equiv H_1\dots H_k$ be the factorization of the history of $\pazocal{C}$ such that for all $i\in\{1,\dots,k\}$, $H_i$ is the history of a maximal one-machine subcomputation of the $z_i$-th machine $\pazocal{C}_i:U_i\to\dots\to V_i$ of $\pazocal{C}$. Then for all $i$, either:

\begin{enumerate} [label=({\alph*})]

\item $V_i\equiv W_{ac}(j)$ or

\item $V_i\equiv I(w_i,j)$ for some $w_i\in\pazocal{L}$.

\end{enumerate}

In case (a), set $W_i^{(1)}\equiv W_i^{(2)}\equiv W_{ac}$; in case (b), set $W_i^{(1)}\equiv I(w_i)$ and $W_i^{(2)}\equiv J(w_i)$. Further, set $W_0^{(1)}\equiv W_0^{(2)}\equiv W_{ac}$.
\newline
Then for each $i\in\{1,\dots,k\}$, there exists a reduced computation $\pazocal{C}_i':W_{i-1}^{(z_i)}\to\dots\to W_i^{(z_i)}$ in the standard base with history $H_i$.

\end{lemma}

In other words, Lemma \ref{projected end to end} says that $\pazocal{C}$ can be `\textit{almost-extended}' to a reduced computation $W_{ac}\to\dots\to W_{ac}$, in that such a computation exists if one were to allow the insertion/deletion of elements of $\pazocal{L}$ in the `special' input sector between maximal one-machine subcomputations.

\medskip

\subsection{Accepted configurations with $\theta$-admissible components} \

The goal of this section is to study the situation where $W$ is an accepted configuration such that $W(j)$ is $\theta$-admissible for some $\theta\in\Theta$ and $j\in\{2,\dots,L\}$.  This situation becomes particularly relevant when $\ell(W)\leq1$, we it will help with the process of `transposing' a disk about a $\theta$-band in the diagrams over the groups associated to $\textbf{M}^\pazocal{L}$ (see \Cref{sec-transposition}).

\begin{lemma} \label{one-machine equal configurations}

Let $W$ and $W'$ be configurations of $\textbf{M}^\pazocal{L}$ which are both accepted by one-machine computations of the $i$-th machine.  If $W(j)\equiv W'(j)$ for some $2\leq j \leq L$, then $W\equiv W'$.


\end{lemma}

\begin{proof}

Suppose $i=1$.  For each $1\leq k\leq L$, note that $W_{ac}(k)$ is a coordinate shift of $W_{ac}(j)$.  So, by the parallel nature of the rules of $\Theta_1$, $W(k)$ and $W'(k)$ are coordinate shifts of $W(j)$ and $W'(j)$, respectively.  But $W(j)\equiv W'(j)$ then implies $W(k)\equiv W'(k)$, so that $W\equiv W'$.

If $i=2$, then the same argument applies to all $2\leq k\leq L$, and even can be adapted for all sectors of $W(1)$ and $W'(1)$ besides the `special' input sector.  But every rule of $\Theta_2$ locks the `special' input sector, so that the tape word in $W(1)$ and $W'(1)$ in this sector must be empty.  Hence, $W(1)\equiv W'(1)$, meaning $W\equiv W'$.

\end{proof}

The next two statements show that for an accepted configuration $W$ with $\ell(W)\leq1$, if $W(j)$ is $\theta$-admissible for some $j\geq2$, then $W$ is $\theta$-admissible with $\ell(W\cdot\theta)\leq1$ except in two particular circumstances.  In the setting of \Cref{sec-transposition}, these conditions allow us to transpose $\theta$-bands about disks, as even in the two particular circumstances the impediment is a difference of an $a$-relation.

\begin{lemma} \label{transposition computation not applicable}

Let $W$ be an accepted configuration of $\textbf{M}^\pazocal{L}$ with $\ell(W)\leq 1$ and $\theta\in\Theta$.  Suppose $W(j)$ is $\theta$-admissible for some $j\in\{2,\dots,L\}$, but $W$ is not $\theta$-admissible.  Then $\theta=\theta(s)_2$ and $W\equiv I(w)$ for some $w\in\pazocal{L}$.

\end{lemma}

\begin{proof}

Note that $W_{ac}(j)$ is $\theta$-admissible only for $\theta=\theta(a)_i^{-1}$, in which case $W_{ac}$ is $\theta$-admissible.  Hence, it suffices to assume $\ell(W)=1$.

Let $\pazocal{C}\in A(W)$ such that $\ell(\pazocal{C})=1$ and fix $i\in\{1,2\}$ such that $\pazocal{C}$ is a one-machine computation of the $i$-th machine.  Let $H\equiv\theta_1\dots\theta_t\in F(\Theta_i^+)$ be the history of $\pazocal{C}$.  Then, $\theta^{-1}H$ is freely equal to the history $H_1$ of a reduced computation $\pazocal{D}$ between $W(j)\cdot\theta$ and $W_{ac}(j)$.  

If $\theta\in\Theta_i$, then $\pazocal{D}$ is a one-machine computation of the $i$-th machine, and so \Cref{extend one-machine} may be applied to $\pazocal{D}$, producing a one-machine computation $\pazocal{D}'$ of the $i$-th machine with history $H_1$ accepting a configuration $V$ which satisfies $V(j)\equiv W(j)\cdot\theta$.  In particular, $H$ is the history of a one-machine computation $\pazocal{C}'$ of the $i$-th machine which accepts the configuration $W'\equiv V\cdot\theta^{-1}$.  

But then $W'(j)\equiv W(j)$ and each of $W$ and $W'$ are accepted by one-machine computations of the $i$-th machine, so that \Cref{one-machine equal configurations} implies $W\equiv W'$.
%
%
%
Hence, it suffices to assume $\theta\notin\Theta_i$.

In this case, $W(j)$ must be both $\theta_1$-admissible and $\theta$-admissible, {\frenchspacing i.e. it is admissible} for rules of both machines.  Hence, $W$ must either be a start or an end configuration.  As the only accepted end configuration is $W_{ac}$, though, $W$ must be a start configuration.  By \Cref{M language}, there then exists $w\in\pazocal{L}$ such that $W\equiv I(w)$ or $W\equiv J(w)$.

In either case, $W(j)\equiv I(w,j)$, and so $\theta=\theta(s)_k$ for $k\neq i$.  But $J(w)$ is both $\theta(s)_1$- and $\theta(s)_2$-admissible, and thus $W\equiv I(w)$.  \Cref{one-machine language} then implies $i=1$, so that $\theta=\theta(s)_2$.

\end{proof}

\begin{lemma} \label{transposition computation not one-machine}

Let $W$ be an accepted configuration of $\textbf{M}^\pazocal{L}$ with $\ell(W)\leq1$ and $\theta\in\Theta$.  Suppose $W$ is $\theta$-admissible with $\ell(W\cdot\theta)>1$.  Then $\theta=\theta(s)_1$ and $W\equiv J(w)$ for some $w\in\pazocal{L}$.

\end{lemma}

\begin{proof}

By definition, it suffices to assume $\ell(W)=1$, {\frenchspacing i.e. $W\neq W_{ac}$}.  As in the previous proof, let $\pazocal{C}\in A(W)$ such that $\ell(\pazocal{C})=1$, fix $i\in\{1,2\}$ such that $\pazocal{C}$ is a one-machine computation of the $i$-th machine, and let $H\equiv\theta_1\dots\theta_t\in F(\Theta_i^+)$ be the history of $\pazocal{C}$.  

If $\theta\in\Theta_i$, then the word $\theta^{-1}H$ is freely equal to the history of a one-machine computation of the $i$-th machine accepting $W\cdot\theta$, contradicting the hypotheses.  Hence, it suffices to assume $\theta\notin\Theta_i$.

Again, this means $W$ is both $\theta_1$- and $\theta$-admissible, and so must be a start configuration.  By Lemmas \ref{M language} and \ref{one-machine language}, there then exists $w\in\pazocal{L}$ such that $W\equiv I(w)$ if $i=1$ or $W\equiv J(w)$ if $i=2$.  But $I(w)$ is not admissible for any rule of $\Theta_2$, so that $W\equiv J(w)$ and $\theta=\theta(s)_1$.

\end{proof}


\medskip

\subsection{Complexity} \

The goal of this section is to study the accepting computations of configurations of $\textbf{M}^\pazocal{L}$ satisfying $\ell(W)=1$.  Specifically, for each such configuration, a particular accepting configuration is constructed which satisfies established bounds on its `length' and `width' (or `time' and `space', respectively) in terms of its $a$-length.

\begin{lemma} \label{M difference}

Let $W$ be a configuration of $\textbf{M}^\pazocal{L}$ that is $\theta$-admissible for some $\theta\in\Theta$.  Then $|W\cdot\theta|_a\leq c_0(|W|_a+2LN)$.

\end{lemma}

\begin{proof}

Let $W_{i,j}$ be the admissible subword of $W$ with base $Q_{j-1}^\pazocal{L}(i)Q_j^\pazocal{L}(i)$ for $j\in\{1,\dots,N\}$ and $i\in\{1,\dots,L\}$.  If $j=1$, then \Cref{M_1 difference} implies that $|W_{i,j}\cdot\theta|_a\leq c_0(|W_{i,j}|_a+1)$.  Otherwise, \Cref{simplify rules} implies $|W_{i,j}\cdot\theta|_a\leq|W_{i,j}|_a+1\leq c_0(|W_{i,j}|_a+1)$ for $c_0\geq1$.

Similarly, let $V_{i,j}$ be the admissible subword of $W$ with base $(R_j^\pazocal{L}(i))^{-1}(R_{j-1}^\pazocal{L}(i))^{-1}$ for $j\in\{1,\dots,N\}$ and $i\in\{1,\dots,L\}$.  Again, \Cref{M_1 difference} implies $|V_{i,1}\cdot\theta|_a\leq c_0(|V_{i,1}|_a+1)$, while \Cref{simplify rules} implies $|V_{i,j}\cdot\theta|_a\leq|V_{i,j}|_a+1\leq c_0(|V_{i,j}|_a+1)$ for $j\geq2$.

As any other sector is locked by every rule, $|W|_a=\sum_{i,j}|W_{i,j}|_a$ and $|W\cdot\theta|_a=\sum_{i,j}|W_{i,j}\cdot\theta|_a$.  Hence, $|W\cdot\theta|_a=\sum_{i,j}|W_{i,j}\cdot\theta|_a\leq c_0\sum_{i,j}(|W_{i,j}|_a+1)=c_0(|W|_a+2LN)$.

\end{proof}

\begin{lemma} \label{M main difference}

Let $\pazocal{C}:W_0\to\dots\to W_t\equiv W_{ac}$ be a computation of $\textbf{M}^\pazocal{L}$ accepting the configuration $W_0$.  Then $|W_i|_a\leq 4c_0^tLN$ for all $0\leq i\leq t$.

\end{lemma}

\begin{proof}

\Cref{M difference} immediately yields $|W_{i-1}|_a\leq c_0(|W_i|_a+2LN)$ for all $1\leq i\leq t$.  So, since $|W_t|_a=|W_{ac}|_a=0$, $|W_{t-1}|_a\leq 2c_0LN$. 

Assuming $|W_{t-i}|_a\leq\sum\limits_{j=1}^i 2c_0^jLN$, then: 
$$|W_{t-i-1}|_a\leq c_0\left(\sum\limits_{j=1}^i 2c_0^jLN+2LN\right)=2c_0\sum\limits_{j=0}^ic_0^jLN=\sum\limits_{j=1}^{i+1} 2c_0^jLN$$
Hence, by induction $|W_{t-i}|_a\leq\sum\limits_{j=1}^i 2c_0^jLN$ for all $i$.  Taking $c_0\geq2$, then $\sum\limits_{j=1}^{i-1} c_0^j\leq c_0^i$, and thus $|W_{t-i}|_a\leq4c_0^iLN\leq4c_0^tLN$.

%
%

\end{proof}

\begin{lemma} \label{M time-space}

Let $W$ be an accepted configuration of $\textbf{M}^\pazocal{L}$ with $\ell(W)=1$ and $|W(2)|_a=n$.  Then there exists an accepting computation $W\equiv W_0\to\dots\to W_t\equiv W_{ac}$ such that 
$$t\leq c_0\TM_\pazocal{L}(c_0n)^3+nc_0^{n}+c_0n+2c_0$$

\end{lemma}

\begin{proof}

Let $\pazocal{C}\in A(W)$ such that $\ell(\pazocal{C})=1$.  Fix $i\in\{1,2\}$ such that $\pazocal{C}$ is a one-machine computation of the $i$-th machine.  Then, \Cref{(s) (a) first or last} implies that there exists a factorization $H\equiv H_sH_i\theta(a)_i$ of the history of $\pazocal{C}$ such that:
\begin{itemize} 

\item $H_s$ is either empty or $H_s\equiv\theta(s)_i$, and

\item $H_i$ consists only of working rules in $\Theta_i$.

\end{itemize}

Let $W'\equiv W\cdot H_s$.  Then, $|W'(j)|_a=|W(j)|_a$ for all $1\leq j\leq L$, and hence $|W'(2)|_a=|W(2)|_a=n$.

Now, let $\pazocal{C}_i$ be the subcomputation of $\pazocal{C}$ with history $H_i$.  Then, $\pazocal{C}_i$ can be identified with a reduced computation of $\textbf{M}_{6,i}^\pazocal{L}$ which accepts the configuration corresponding to $W'$.  In turn, the restriction of $\pazocal{C}_i$ to the base $\{t(2)\}B_4^\pazocal{L}(2)$ can be identified with a reduced computation of $\textbf{M}_5^\pazocal{L}$ accepting the configuration corresponding to $W'(2)$.  By \Cref{M_5 time-space}, there then exists a one-machine computation of the $i$-th machine $\pazocal{D}':V_0\to\dots\to V_z$ with base $\{t(2)\}B_4^\pazocal{L}(2)$ satisfying:

\begin{itemize}

\item $V_0\equiv W'(2)$

\item $V_z\equiv W_{ac}(2)\cdot\theta(a)_i^{-1}$ 

\item $z\leq c_0\TM_\pazocal{L}(c_0n)^3+nc_0^{n}+c_0n+c_0$

\end{itemize}

Let $H'$ be the history of $\pazocal{D}'$.  Then, there exists a one-machine computation of the $i$-th machine $\pazocal{D}:V_0\to\dots\to V_z\to W_{ac}(2)$ with base $\{t(2)\}B_4^\pazocal{L}(2)$ and history $H'\theta(a)_i$.

By applying \Cref{extend one-machine} to $\pazocal{D}$, there then exists a one-machine computation of the $i$-th machine $\pazocal{E}:W_0\to\dots\to W_z\to W_{z+1}$ in the standard base with history $H'\theta(a)_i$ such that $W_\ell(2)\equiv V_\ell$ for all $0\leq\ell\leq z$ and $W_{z+1}\equiv W_{ac}$.

Hence, $W_0$ and $W'$ are both configurations accepted by one-machine computations of the $i$-th machine with $W_0(2)\equiv V_0\equiv W'(2)$, so that \Cref{one-machine equal configurations} implies $W_0\equiv W'$.

Thus, $H_sH'\theta(a)_i$ is the history of an accepting computation of $W$ with
$$\|H_sH'\theta(a)_i\|\leq z+2\leq c_0\TM_\pazocal{L}(c_0n)^3+nc_0^{n}+c_0n+c_0+2$$
so that the statement follows by taking $c_0\geq2$.

\end{proof}


\medskip

\subsection{Semi-computations in the `special' input sector} \label{sec-M-semi} \

As the rules of $\Theta_2$ lock the `special' input sector, \Cref{semi locked sectors} implies that any non-trivial semi-computation of $\textbf{M}^\pazocal{L}$ in the `special' input sector must consist entirely of rules from the first machine.  

In particular, any rule of such a semi-computation is either $\theta(s)_1^{\pm1}$ or can be identified with the application (in the sense of semi-computations) of a rule of $\textbf{M}_1^\pazocal{A}$ to a tape word of the $Q_0^\pazocal{A}Q_1^\pazocal{A}$-sector.

The following statement is an immediate consequence of the definition of the rules of $\textbf{M}^\pazocal{L}$:

\begin{lemma} \label{semi applicable}

Let $w$ be a non-trivial word over the tape alphabet of the `special' input sector and $\theta\in\Theta$.  Then $w$ is $\theta$-applicable if and only if:

\begin{itemize}

\item $w\in F(\pazocal{A}_1\sqcup\pazocal{B})$ if $\theta\neq\theta(s)_1^{\pm1}$

\item $w\in F(\pazocal{A})$ if $\theta=\theta(s)_1$

\item $w\in F(\pazocal{A}_1)$ if $\theta=\theta(s)_1^{-1}$

\end{itemize}

\end{lemma}

Hence, the next statement is an immediate corollary of \Cref{semi applicable}:

\begin{lemma} \label{semi subword}

Let $w$ be a non-trivial word over the tape alphabet of the `special' input sector and $v$ be a subword of a cyclic permutation of $w^{\pm1}$.  If $w$ is $\theta$-applicable for some $\theta\in\Theta$, then $v$ is also $\theta$-applicable.

\end{lemma}

%
%
%
%
%
%
%
%
%
%
%
%
%
%

Recall from \Cref{sec-M_1} that a reduced word over $(\pazocal{A}_1\sqcup\pazocal{B})^{\pm1}$ is defined to be compressed if it both begins and ends with a letter of $\pazocal{A}_1^{\pm1}$.  This is now extended to reduced words over $\pazocal{A}^{\pm1}$, which are all taken to be compressed.

Note that, by definition, a non-trivial word $w$ in the tape alphabet of the `special' input sector is $\theta(s)_1$-admissible if and only if $w\in F(\pazocal{A})$, in which case $w\cdot\theta(s)_1\in F(\pazocal{A}_1)$.  So, a non-trivial word $w$ which is $\theta(s)_1^{\pm1}$-admissible is necessarily compressed.  As such, the definition of the compressed application of a rule is extended to include applications of $\theta(s)_1^{\pm1}$.

The following statement is thus a consequence of \Cref{M_1 compressed semi-computation three A}:

\begin{lemma} \label{M compressed semi-computation three A}

Let $\pazocal{S}_\mathscr{C}:w_0\to\dots\to w_t$ be a non-empty reduced compressed semi-computation of $\textbf{M}^\pazocal{L}$ in the `special' input sector.  Suppose $w_0\equiv y_1^{\delta_1}y_2^{\delta_2}y_3^{\delta_3}\in F(\pazocal{A})$.  Then, setting $x_i=\varphi_1(y_i)$, there exist $u_1,u_2\in F(\pazocal{B})$ such that:

\begin{enumerate}

\item $w_t\equiv x_1^{\delta_1}u_1x_2^{\delta_2}u_2x_3^{\delta_3}$

\item $\frac{1}{2}D_\pazocal{A}(t-1)\leq\|u_1\|+\|u_2\|\leq 3D_\pazocal{A}(t-1)$  

\item The pair $(u_1,u_2)$ uniquely determine the history of $\pazocal{S}_\mathscr{C}$

\end{enumerate}
    
\end{lemma}

\begin{proof}

As $w_0\in F(\pazocal{A})$, there exists a factorization $H\equiv\theta(s)_1H'$ of the history $H$ of $\pazocal{S}$.  In particular, $w_1\equiv w_0\cdot\theta(s)_1\equiv x_1^{\delta_1}x_2^{\delta_2}x_3^{\delta_3}$.

Suppose $H'$ is non-empty.  Then $H'$ must have a non-empty maximal prefix $H''$ consisting entirely of working rules of the first machine.  In this case, it follows that the reduced compressed semi-computation $\pazocal{S}'':w_1\to\dots\to w_s$ with history $H''$ can be identified with a reduced compressed semi-computation of $\textbf{M}_1^\pazocal{A}$ in the $Q_0^\pazocal{A}Q_1^\pazocal{A}$-sector satisfying the hypotheses of \Cref{M_1 compressed semi-computation three A}.  But then $w_s$ is not $\theta(s)_1^{\pm1}$-admissible, so that $H'=H''$.

The statement then follows from \Cref{M_1 compressed semi-computation three A}.

\end{proof}

By an identical argument, the following statement is a consequence of \Cref{M_1 compressed semi-computation two A}:

\begin{lemma} \label{M compressed semi-computation two A}

Let $\pazocal{S}_\mathscr{C}:w_0\to\dots\to w_t$ be a non-empty reduced compressed semi-computation of $\textbf{M}^\pazocal{L}$ in the `special' input sector.  Suppose $w_0\equiv y_1^{\delta_1}y_2^{\delta_2}\in F(\pazocal{A})$ such that $\delta_1\neq1$ or $\delta_2\neq-1$.  Then, setting $x_i=\varphi_1(y_i)$, there exists $u_1\in F(\pazocal{B})$ such that:

\begin{enumerate}

\item $w_t\equiv x_1^{\delta_1}u_1x_2^{\delta_2}$

\item $\frac{1}{2}D_\pazocal{A}(t-1)\leq\|u_1\|\leq 2D_\pazocal{A}(t-1)$ 

\item $u_1$ uniquely determines the history of $\pazocal{S}_\mathscr{C}$

\end{enumerate}
    
\end{lemma}

For any subset $\Lambda^\pazocal{A}$ of $(\pazocal{A}\cup\pazocal{A}^{-1})^*$ consisting of cyclically reduced words of length at least $C$, let $\pazocal{E}(\Lambda^\pazocal{A})$ be the set of reduced words $w$ over $(\pazocal{A}\sqcup\pazocal{A}_1\sqcup\pazocal{B})^{\pm1}$ for which there exists a semi-computation of $\textbf{M}^\pazocal{L}$ in the `special' input sector of the form $\pazocal{S}:w\equiv w_0\to\dots\to w_t$ such that $w_t\in\Lambda^\pazocal{A}$.  In this case, the semi-computation $\pazocal{S}$ is then said to \textit{$\Lambda^\pazocal{A}$-accept} $w$.

Let $\Lambda_1^\pazocal{A}=\{\widetilde{\varphi}_1(w): w\in\Lambda^\pazocal{A}\}=\{w\cdot\theta_1(s): w\in\Lambda^\pazocal{A}\}$.  Note that $\Lambda_1^\pazocal{A}$ is then subset of $(\pazocal{A}_1\cup\pazocal{A}_1^{-1})^*$ consisting of cyclically reduced words of length at least $C$.  Hence, the set $\Lambda_1^\pazocal{A}$ may be studied as in \Cref{sec-M_1}, allowing for the application \Cref{M_1 Lambda semi-computations}.


\begin{lemma} \label{M Lambda semi-computations} 

Let $\Lambda^\pazocal{A}$ be a subset of $(\pazocal{A}\cup\pazocal{A}^{-1})^*$ consisting of cyclically reduced words of length at least $C$.  Then:

\begin{enumerate}

\item $\pazocal{E}(\Lambda^\pazocal{A})=\Lambda^\pazocal{A}\sqcup\pazocal{E}_1(\Lambda_1^\pazocal{A})$.  

\item For any $w\in\pazocal{E}(\Lambda^\pazocal{A})$, there is a unique semi-computation  $\pazocal{S}(w):w_0\to\dots\to w_t$ of $\textbf{M}^\pazocal{L}$ in the `special' input sector which $\Lambda^\pazocal{A}$-accepts $w$.

\item Let $w\equiv u_0x_1^{\delta_1}u_1x_2^{\delta_2}\dots x_k^{\delta_k}u_k\in\pazocal{E}_1(\Lambda_1^\pazocal{A})$ for some $x_i\in\pazocal{A}_1$, $\delta_i\in\{\pm1\}$, and $u_i\in F(\pazocal{B})$.  Then the history of $\pazocal{S}(w)$ has the form $H\theta(s)_1^{-1}$ where:

\begin{itemize}

\item $\frac{1}{2}D_\pazocal{A}\|H\|\leq \|u_{i-1}\|+\|u_i\|\leq 3D_\pazocal{A}\|H\|$ for any $i\in\{2,\dots,k-1\}$

\item $\frac{1}{2}D_\pazocal{A}\|H\|\leq \|u_ku_0\|+\|u_j\|\leq 3D_\pazocal{A}\|H\|$ for any $j\in\{1,k-1\}$

\item $\|u_0\|,\|u_k\|\leq D_\pazocal{A}\|H\|$

\end{itemize}


\end{enumerate}


\end{lemma}

\begin{proof}

Suppose $\pazocal{S}:w_0\to\dots\to w_t$ is a non-empty reduced semi-computation of $\textbf{M}^\pazocal{L}$ in the `special' input sector such that $w_t\in \Lambda^\pazocal{A}$.  Then, as $w_t\in F(\pazocal{A})$, there exists a factorization $H_w\equiv H\theta(s)_1^{-1}$ of the history of $\pazocal{S}$.  In particular, $w_{t-1}\equiv w_t\cdot\theta(s)_1\equiv\widetilde{\varphi}_1(w_t)\in\Lambda_1^\pazocal{A}$.

If $H$ is nonempty, then it must contain a non-empty maximal suffix $H'$ consisting entirely of working rules of the first machine.  But then the semi-computation $w_r\to\dots\to w_{t-1}$ with history $H'$ can be identified with a reduced semi-computation of $\textbf{M}_1^\pazocal{A}$ in the $Q_0^\pazocal{A}Q_1^\pazocal{A}$-sector.  The makeup of $w_t$ implies this semi-computation $\Lambda_1^\pazocal{A}$-accepts $w_r$, so that \Cref{M_1 Lambda semi-computations} implies $|w_r|_b>0$.  But then $w_r$ is not $\theta(s)_1^{\pm1}$-applicable, {\frenchspacing i.e. $H'=H$ and $r=0$}.

%
%

The rest of the statement then follows from \Cref{M_1 Lambda semi-computations}.

%
%
%
%
%
%
%

\end{proof}

\begin{lemma} \label{Lambda not cyclically reduced}

Let $\Lambda^\pazocal{A}$ be a subset of $(\pazocal{A}\cup\pazocal{A}^{-1})^*$ consisting of cyclically reduced words of length at least $C$.  Let $w'\in\pazocal{E}(\Lambda^\pazocal{A})$ and let $w\in F(\pazocal{A}\cup\pazocal{A}_1\cup\pazocal{B})$ be a cyclically reduced word which is freely conjugate to $w'$.  Then for any rule $\theta\in\Theta$, $w$ is $\theta$-applicable if and only if $w'$ is $\theta$-applicable.

\end{lemma}

\begin{proof}

If $w'\in\Lambda^\pazocal{A}$, then by hypothesis $w$ and $w'$ are non-trivial cyclic permutations of one another.  Hence, the statement follows from \Cref{semi subword}.

So, by \Cref{M Lambda semi-computations}(1), it suffices to assume that $w'\in\pazocal{E}_1(\Lambda_1^\pazocal{A})$.

By \Cref{M Lambda semi-computations}(2), there then exists a unique semi-computation $\pazocal{S}(w'):w'\equiv w_0\to\dots\to w_t$ of $\textbf{M}^\pazocal{L}$ in the `special' input sector which $\Lambda^\pazocal{A}$-accepts $w'$.  Let $w_t\equiv y_1^{\delta_1}\dots y_k^{\delta_k}\in\Lambda^\pazocal{A}$.  

Then, $w'\equiv u_0x_1^{\delta_1}u_1x_2^{\delta_2}\dots x_k^{\delta_k}u_k\in\pazocal{E}_1(\Lambda_1^\pazocal{A})$ where $u_j\in F(\pazocal{B})$ and $x_j=\varphi_1(y_j)\in\pazocal{A}_1$ for all $j$.  As a result, \Cref{M Lambda semi-computations}(3) implies the history $H'$ of $\pazocal{S}(w')$ is of the form $H'\equiv H\theta(s)_1^{-1}$ with $\|u_1\|+\|u_2\|\geq\frac{1}{2}D_\pazocal{A}\|H\|$.

Suppose $\|H\|=0$.  Then, $w'\equiv w_t\cdot\theta(s)_1\equiv x_1^{\delta_1}\dots x_k^{\delta_k}$.  But then $w'$ is cyclically reduced, so that again \Cref{semi subword} implies the statement.  Hence, $|w'|_b\geq\|u_1\|+\|u_2\|>0$.

Now, let $p$ be the maximal suffix of $u_k$ such that $p^{-1}$ is a prefix of $u_0$.  Further, let $u_0'$ and $u_k'$ be the (perhaps trivial) words over $\pazocal{B}\cup\pazocal{B}^{-1}$ such that $u_0\equiv p^{-1}u_0'$ and $u_k\equiv u_k'p$.  

Then, the maximality of $p$ and the assumption that $w_t\in\Lambda^\pazocal{A}$ is cyclically reduced imply that the word $pw'p^{-1}=u_0'x_1^{\delta_1}u_1x_2^{\delta_2}\dots x_k^{\delta_k}u_k'$ is cyclically reduced.

By hypothesis, $w$ is then a cyclic permutation of this word.  As a result, $w\in F(\pazocal{A}_1\cup\pazocal{B})$ with $|w|_b\geq\|u_1\|+\|u_2\|>0$ by a parameter choice $k\geq C>3$.  Hence, the statement follows from \Cref{semi applicable}.

\end{proof}

\begin{lemma} \label{semi extension}

Let $\Lambda^\pazocal{A}$ be a subset of $(\pazocal{A}\cup\pazocal{A}^{-1})^*$ consisting of cyclically reduced words of length at least $C$.  Further, let $w\in\pazocal{E}(\Lambda^\pazocal{A})$ and $\theta\in\Theta$.  Suppose there exists a $\theta$-applicable subword $v$ of $w$ such that $|v|_\pazocal{A}\geq3$.  Then $w$ is also $\theta$-applicable.

\end{lemma}

\begin{proof}

If $\theta\neq\theta(s)_1^{\pm1}$, then \Cref{semi applicable} implies that $v\in F(\pazocal{A}_1\sqcup\pazocal{B})$.  But then \Cref{M Lambda semi-computations}(1) then implies that $w\in\pazocal{E}_1(\Lambda_1^\pazocal{A})$, so that the statement follows from \Cref{semi applicable}.

Similarly, if $\theta=\theta(s)_1$, then \Cref{semi applicable} yields $v\in F(\pazocal{A})$, so that \Cref{M Lambda semi-computations}(1) implies $w\in\Lambda^\pazocal{A}$ so that the statement follows again by \Cref{semi applicable}.

Finally, suppose $\theta=\theta(s)_1^{-1}$.  As in the first case, \Cref{semi applicable} implies $v\in F(\pazocal{A}_1)$, so that \Cref{M Lambda semi-computations}(1) implies $w\in\pazocal{E}_1(\Lambda_1^\pazocal{A})$.  However, since $|v|_\pazocal{A}\geq3$ and $|v|_b=0$, \Cref{M Lambda semi-computations}(3) implies $w\in\Lambda_1^\pazocal{A}$, and thus the statement again follows from \Cref{semi applicable}.

\end{proof}

\bigskip


\section{Groups Associated to Generalized $S$-machines}

\subsection{The groups} \label{sec-the-groups} \

As in previous literature (for example \cite{O18}, \cite{OSnon-amen}, \cite{OS19}, \cite{W}), we now associate finitely presented groups to a cyclic generalized $S$-machine $\textbf{S}$. In the case $\textbf{S}=\textbf{M}^\pazocal{L}$ (or indeed any noisy $S$-machine), the groups `simulate' the work of $\textbf{M}^\pazocal{L}$ in the precise sense described in \Cref{sec-trapezia}.

Let $\textbf{S}$ be a cyclic recognizing generalized $S$-machine with hardware $(Y,Q)$, where $Q=\sqcup_{i=0}^s Q_i$ and $Y=\sqcup_{i=1}^{s+1} Y_i$, and software the set of rules $\Theta(\textbf{S})=\Theta^+(\textbf{S})\sqcup\Theta^-(\textbf{S})$. For notational purposes, set $Q_{s+1}=Q_0$, set $Y_0=Y_{s+1}$, and denote the accept word of $\textbf{S}$ by $W_{ac}$.

For $\theta\in\Theta^+(\textbf{S})$, let $\theta=[q_0\to u_{s+1}q_0'v_1, \ q_1\to u_1q_1'v_2, \ \dots, \ q_{s-1}\to u_{s-1}q_{s-1}'v_s, \ q_s\to u_sq_s'v_{s+1}]$ where some of the arrows may take the form $\xrightarrow{\ell}$. Further, for all $i$, let $X_i(\theta)$ and $Z_i(\theta)$ be the finite subsets of $F(Y_i)$ prescribed by $\theta$ and let $f_{\theta,i}:X_i(\theta)\to Z_i(\theta)$ be the associated bijection inducing the isomorphism $\widetilde{f}_{\theta,i}:\gen{X_i(\theta)}\to\gen{Z_i(\theta)}$.

Define $T=\{\theta_i: \theta\in\Theta^+(\textbf{S}),0\leq i\leq s\}$. For notational convenience, set $\theta_{s+1}=\theta_0$ for all $\theta\in\Theta^+(\textbf{S})$.

The group $M(\textbf{S})$ is then defined by taking the (finite) generating set $\pazocal{X}=Q\cup Y\cup T$ and imposing the (finite number of) relations:

\begin{itemize}

\item $q_i\theta_{i+1}=\theta_i v_iq_i'u_{i+1}$ for all $\theta\in\Theta^+(\textbf{S})$ and $0\leq i\leq s$,

\item $x\theta_i=\theta_i \cdot f_{\theta,i}(x)$ for all $0\leq i\leq s$ and $x\in X_i(\theta)$.

\end{itemize}

As in the language of computations of the machines, letters from $Q\cup Q^{-1}$ are called \textit{$q$-letters} and those from $Y\cup Y^{-1}$ are called \textit{$a$-letters}. Additionally, those from $T\cup T^{-1}$ are called \textit{$\theta$-letters}. 

The relations of the form $q_i\theta_{i+1}=\theta_iv_iq_i'u_{i+1}$ are called \textit{$(\theta,q)$-relations}, while those of the form $x\theta_i=\theta_i\cdot f_{\theta,i}(x)$ are called \textit{$(\theta,a)$-relations}; when specificity is required, this $(\theta,a)$-relation said to be a \textit{$(\theta,a)$-relation of the $Q_{i-1}Q_i$-sector}.

Note that if $\theta$ locks the $i$-th sector, then there is no relation between $\theta$ and the elements of $F(Y_i)$.

In the particular setting of $\textbf{S}=\textbf{M}^\pazocal{L}$, let $a$ be an $a$-letter from the tape alphabet of an input sector.  

\begin{enumerate} [label=(\alph*)]

\item If $a$ is the natural copy of a letter from $\pazocal{A}\sqcup\pazocal{A}_1$, then $a^{\pm1}$ is called an \textit{$\pazocal{A}$-letter}.

\item If $a$ is a copy of a letter from $\pazocal{B}$ then $a^{\pm1}$ is called a \textit{$b$-letter}.  

\end{enumerate}

Any other $a$-letter is called \textit{ordinary $a$-letter}.  Observe that this terminology can be generalized for the groups associated to any noisy $S$-machine.

Note that for any $\theta\in\Theta^+$, every domain $X_i(\theta)$ consists of letters from the corresponding tape alphabet.  Naturally, based on the type of $a$-letter of $x\in X_i(\theta)$, the $(\theta,a)$-relation $x\theta_i=\theta_i\cdot f_{\theta,i}(x)$ is called a \textit{$(\theta,\pazocal{A}$)-relation}, a \textit{$(\theta,b)$-relation}, or an \textit{ordinary $(\theta,a)$-relation}.

The \textit{coordinate} of a $(\theta,q)$-relation of $M(\textbf{M}^\pazocal{L})$ is the coordinate of either of its $q$-letters.  Accordingly, the \textit{coordinate} of a $(\theta,a)$-relation of $M(\textbf{M}^\pazocal{L})$ is taken to be $i$ if the tape letters are from $(Y_j^\pazocal{L}(i)\cup\pazocal{Y}_j^\pazocal{L}(i))^{\pm1}$ for some $j$.


However, the group $M(\textbf{S})$ evidently lacks any reference to the accept configuration. To amend this, the group $G(\textbf{S})$ is constructed by adding one more relation to those defining $M(\textbf{S})$, namely the \textit{hub-relation} $W_{ac}=1$. In other words, $G(\textbf{S})\cong M(\textbf{S})/\gen{\gen{W_{ac}}}$.

Moreover, it is useful for the purposes of this article to consider extra relations, called \textit{$a$-relations}, within the language of tape letters. If $\Omega$ is the set of relators defining these $a$-relations, then the groups arising from the addition of $a$-relations are denoted by $M_\Omega(\textbf{S})$ and $G_\Omega(\textbf{S})$. Hence, $M_\Omega(\textbf{S})\cong M(\textbf{S})/\gen{\gen{\Omega}}$ and $G_\Omega(\textbf{S})\cong G(\textbf{S})/\gen{\gen{\Omega}}$.

It is henceforth taken as an assumption that any $a$-relation adjoined to the groups associated to the machine $\textbf{M}^\pazocal{L}$ corresponds to a word over the alphabet $\pazocal{A}\cup\pazocal{A}_1\cup\pazocal{B}$ of the `special' input sector.  

In particular, it is assumed that $\Omega$ is the set of all cyclically reduced words over $(\pazocal{A}\cup\pazocal{A}_1\cup\pazocal{B})^{\pm1}$ which are freely conjugate to an element of $\pazocal{E}(\Lambda^\pazocal{A})$, where $\Lambda^\pazocal{A}$ is a subset of $(\pazocal{A}\cup\pazocal{A}^{-1})^*$ satisfying:

\begin{enumerate}[label=(L\arabic*)]

\item $\Lambda^\pazocal{A}$ consists entirely of cyclically reduced words of length at least $C$

\item $\Lambda^\pazocal{A}$ is closed under taking inverses

\item $\Lambda^\pazocal{A}$ is closed under taking cyclic permutations

\item For any $w_1,w_2\in\Lambda^\pazocal{A}$, either $w_1\equiv w_2^{-1}$ or $w_1w_2$ is freely conjugate to an element of $\Lambda^\pazocal{A}$

\item $\pazocal{L}\subseteq\Lambda^\pazocal{A}$

\end{enumerate}

The following is then a consequence of these conditions:

\begin{lemma} \label{Omega inverses}

The set of $a$-relators $\Omega$ is closed under taking inverses.

\end{lemma}

\begin{proof}

Let $w\in\Omega$.  Then, there exists a word $v\in\pazocal{E}(\Lambda^\pazocal{A})$ which is freely conjugate to $w$.

By definition, there then exists a (unique) semi-computation $\pazocal{S}(v):v\equiv v_0\to\dots\to v_t$ of $\textbf{M}^\pazocal{L}$ in the `special' input sector which $\Lambda^\pazocal{A}$-accepts $v$.  Let $H\equiv\theta_1\dots\theta_t$ be the history of $\pazocal{S}(v)$.

As the application of each rule of a semi-computation is the application of an isomorphism, it follows that for all $i\in\{1,\dots,t\}$, $v_{i-1}^{-1}$ is $\theta_i$-applicable with $v_{i-1}^{-1}\cdot\theta_i=(v_{i-1}\cdot\theta)^{-1}=v_i^{-1}$.  Hence, there exists a semi-computation $\overline{\pazocal{S}(v)}:v^{-1}\equiv v_0^{-1}\to\dots\to v_t^{-1}$ of $\textbf{M}^\pazocal{L}$ in the `special' input sector with history $H$.

But condition (L2) implies that $v_t^{-1}\in\Lambda^\pazocal{A}$, so that $\overline{\pazocal{S}(w)}$ $\Lambda^\pazocal{A}$-accepts $v^{-1}$.  Hence, $w^{-1}$ is a cyclically reduced word which is freely conjugate to $v^{-1}$, so that $w^{-1}\in\Omega$.

\end{proof}


Note that though the canonical presentations of $M_\Omega(\textbf{S})$ and $G_\Omega(\textbf{S})$ have finite generating sets, they may not be finite. In fact, in all relevant situations encountered in the sequel, these presentations necessarily have infinitely many relators.

\smallskip


\subsection{Bands and annuli} \label{sec-bands and annuli} \

The majority of the arguments presented in the forthcoming sections rely on van Kampen and Schupp diagrams (see \Cref{sec-Diagrams}) over the presentations of the groups introduced in \Cref{sec-the-groups}. To present these arguments efficiently, it is convenient to first differentiate between the types of edges and cells that abound in such diagrams, doing so in a way similar to that employed in \cite{O18}, \cite{OS19}, and \cite{W}. 

For simplicity, when possible the presence of $0$-edges and $0$-cells will be disregarded in these diagrams.  Hence, adjacent edges are generally identified in these settings.  However, even when ignored, the existence of 0-cells should be kept in mind, as $0$-refinement ensures that many of the diagrammatic operations performed in the sequel do not alter the desired topological properties of the diagram (for example, so that the process of removing a pair of cancellable cells in a circular diagram results in a circular diagram).

Additionally, it is henceforth taken as an assumption that the contour of any circular diagram, the contour of any subdiagram, the contour of any cell, and the outer contour of any annular diagram is traced in the counterclockwise direction.  Conversely, it is assumed that the inner contour of an annular diagram is traced in the clockwise direction.

For any diagram $\Delta$ over $G_\Omega(\textbf{S})$ (or any group associated to a generalized $S$-machine $\textbf{S}$), an edge labelled by a $q$-letter is called a \textit{$q$-edge}. Similarly, an edge labelled by a $\theta$-letter is called a \textit{$\theta$-edge} and one labelled by an $a$-letter is a \textit{$a$-edge}. 

For a path \textbf{p} in $\Delta$, the (combinatorial) length of $\textbf{p}$ is denoted $\|\textbf{p}\|$. Further, the path's \textit{$a$-length} $|\textbf{p}|_a$ is the number of $a$-edges in the path. The path's \textit{$\theta$-length} and \textit{$q$-length}, denoted $|\textbf{p}|_{\theta}$ and $|\textbf{p}|_q$, respectively, are defined similarly. 

A cell whose contour label corresponds to a $(\theta,q)$-relation is called a \textit{$(\theta,q)$-cell}. Similarly, there are \textit{$(\theta,a)$-cells}, \textit{$a$-cells}, and \textit{hubs}.  More specifically, a $(\theta,a)$-cell is called a \textit{$(\theta,a)$-cell of the $Q_{i-1}Q_i$-sector} if its contour label corresponds to such a $(\theta,a)$-relation, while the coordinate of a $(\theta,q)$-cell or $(\theta,a)$-cell is defined similarly.

In the particular setting where $\textbf{S}=\textbf{M}^\pazocal{L}$, an $a$-edge is called an \textit{$\pazocal{A}$-edge}, a \textit{$b$-edge}, or an \textit{ordinary $a$-edge} based on the type of $a$-letter labelling it.   

The \textit{$\pazocal{A}$-length}, \textit{$b$-length}, and \textit{ordinary $a$-length} of the path $\textbf{p}$, denoted $|\textbf{p}|_\pazocal{A}$, $|\textbf{p}|_b$, and $|\textbf{p}|_o$, respectively, are then defined in much the same way as above.  Note that $|\textbf{p}|_a=|\textbf{p}|_\pazocal{A}+|\textbf{p}|_b+|\textbf{p}|_o$ for any path $\textbf{p}$.  Moreover, if $\lab(\textbf{p})$ is a reduced word over the tape alphabet of an input sector, then $|\textbf{p}|_\pazocal{A}$ and $|\textbf{p}|_b$ agree with $|\lab(\textbf{p})|_\pazocal{A}$ and $|\lab(\textbf{p})|_b$, respectively.  Conversely, if $\lab(\textbf{p})$ is a reduced word over the tape alphabet of any other sector, then $|\textbf{p}|_o=|\textbf{p}|_a$.

A $(\theta,a)$-cell is called a \textit{$(\theta,\pazocal{A})$-cell}, a \textit{$(\theta,b)$-cell}, or an \textit{ordinary $(\theta,a)$-cell} based on the type of $(\theta,a)$-relation defining its boundary label. Note that it is a consequence of these definitions that $(\theta,b)$-cells and ordinary $(\theta,a)$-cells correspond to relators of the form $[\theta_i,y]$ for some index $i$ and some $a$-letter $y$.

Observe again that the definitions relating to $\textbf{M}^\pazocal{L}$ can be generalized to applying to the groups associated to all noisy $S$-machines.

In the general setting of a reduced diagram $\Delta$ over any presentation with generating set $X$, fix a subset $\pazocal{Z}\subseteq X$. For $m\geq1$, suppose $\pazocal{S}=(\pi_1,\dots,\pi_m)$ is a sequence of distinct cells in $\Delta$, $(\textbf{e}_0,\textbf{e}_1,\dots,\textbf{e}_m)$ is a sequence of edges of $\Delta$, and $\eps\in\{\pm1\}$ is a number such that the following conditions hold:

\begin{itemize}

\item $\textbf{e}_{i-1}^{-1}$ and $\textbf{e}_i$ are edges of $\partial\pi_i$

\item $\lab(\textbf{e}_i)\in\pazocal{Z}^\eps$

\item $\textbf{e}_{i-1}^{-1}$ and $\textbf{e}_i$ are the only edges of $\partial\pi_i$ labelled by a letter of $\pazocal{Z}\cup\pazocal{Z}^{-1}$

\end{itemize}

Then $\pazocal{S}$ is called a \textit{$\pazocal{Z}$-band of length $m$} with \textit{defining edge sequence} $(\textbf{e}_0,\textbf{e}_1,\dots,\textbf{e}_m)$ comprised of the \textit{defining edges} $\pazocal{I}_\pazocal{S}=\{\textbf{e}_0,\textbf{e}_1,\dots,\textbf{e}_m\}$.  In this case, $\pazocal{S}$ is called a \textit{positive} or \textit{negative} $\pazocal{Z}$-band depending on the value of $\eps$.

Using only edges from the boundaries of $\pi_1,\dots,\pi_m$, there exists a simple closed path $\textbf{e}_0^{-1}\textbf{q}_1\textbf{e}_m(\textbf{q}_2)^{-1}$ such that $\textbf{q}_1$ and $\textbf{q}_2$ are simple (perhaps closed) paths.  What's more, using $0$-refinement (or gluing), it may be assumed that $\textbf{q}_1$ and $\textbf{q}_2$ both have reduced label. In this case, $\textbf{q}_1$ is called the \text{bottom} of $\pazocal{S}$, denoted $\textbf{bot}(\pazocal{S})$, while $\textbf{q}_2$ is called the \textit{top} of $\pazocal{S}$ and denoted $\textbf{top}(\pazocal{S})$. When the top and bottom of the band need not be distinguished, they are called the \textit{sides} of the band.

If $\textbf{e}_0=\textbf{e}_m$ in a $\pazocal{Z}$-band $\pazocal{S}$ of length $m\geq1$, then $\pazocal{S}$ is called a \textit{$\pazocal{Z}$-annulus}. 

If $\pazocal{S}$ is a non-annular $\pazocal{Z}$-band, then identifying $\pazocal{S}$ with the subdiagram of $\Delta$ consisting of its cells, $\textbf{e}_0^{-1}\textbf{q}_1\textbf{e}_m\textbf{q}_2^{-1}$ is called the \textit{standard factorization} of $\partial\pazocal{S}$.

\begin{figure}[H]
\centering
\begin{subfigure}[b]{0.48\textwidth}
\centering
\raisebox{0.5in}{\includegraphics[scale=1.35]{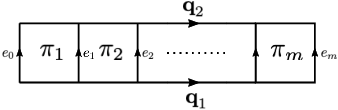}}
\caption{Non-annular $\pazocal{Z}$-band of length $m$}
\end{subfigure}\hfill
\begin{subfigure}[b]{0.48\textwidth}
\centering
\includegraphics[scale=1.35]{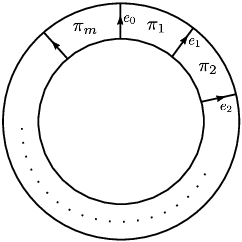}
\caption{Annular $\pazocal{Z}$-band of length $m$}
\end{subfigure}
\caption{ \ }
\end{figure}

Note that $\bar{\pazocal{S}}=(\pi_m,\dots,\pi_1)$ is a $\pazocal{Z}$-band of length $m$ with defining edge sequence $(\textbf{e}_m^{-1},\dots,\textbf{e}_1^{-1},\textbf{e}_0^{-1})$ (and so $\pazocal{I}_{\bar{\pazocal{S}}}=\pazocal{I}_\pazocal{S}^{-1}$), so that $\pazocal{S}$ is a positive $\pazocal{Z}$-band if and only if $\bar{\pazocal{S}}$ is a negative $\pazocal{Z}$-band.  Hence, a $\pazocal{Z}$-band of length $m\geq1$ can be identified with the collection of cells that comprise it along with a \textit{direction} determined by whether the band is positive or negative.

For completeness, the definition of $\pazocal{Z}$-band is extended by saying that any edge $\textbf{e}$ labelled by a letter of $\pazocal{Z}^{\pm1}$ is a \textit{$\pazocal{Z}$-band of length zero} with defining edge sequence $(\textbf{e})$.  Naturally, this band is positive or negative depending on whether $\lab(\textbf{e})$ is an element of $\pazocal{Z}$ or $\pazocal{Z}^{-1}$, respectively.

A $\pazocal{Z}$-band $\pazocal{S}_1$ is a \textit{(proper) subband} of a $\pazocal{Z}$-band $\pazocal{S}_2$ if the defining edge sequence of $\pazocal{S}_1$ is a (proper) subsequence of that of $\pazocal{S}_2$.  A $\pazocal{Z}$-band is said to be \textit{maximal} if it is not a proper subband of any other $\pazocal{Z}$-band.  Note that every edge labelled by a letter of $\pazocal{Z}$ (resp. $\pazocal{Z}^{-1}$) is a defining edge of a maximal positive (resp. negative) $\pazocal{Z}$-band; moreover, if it is non-annular, then this maximal $\pazocal{Z}$-band is unique.

If $\pazocal{S}$ is a non-annular $\pazocal{Z}$-band, then $\textbf{e}_0$ and $\textbf{e}_m$ are called the \textit{ends} of $\pazocal{S}$.  If $\textbf{e}_0$ (or $\textbf{e}_m^{-1}$) is an edge of $\partial\pi$ for some cell $\pi$ which is not a cell comprising $\pazocal{S}$, then $\pazocal{S}$ is said to \textit{have an end on $\pi$}.   Naturally, $\pazocal{S}$ can \textit{have two ends on $\pi$} if both $\textbf{e}_0$ and $\textbf{e}_m^{-1}$ are edges of $\partial\pi$.  Similarly, if $\textbf{e}_0^{-1}$ (or $\textbf{e}_m$) is an edge of a subpath $\textbf{t}$ of a boundary component of $\Delta$, then $\pazocal{S}$ is said to \textit{have an end on} $\textbf{t}$.

A $\pazocal{Z}_1$-band and a $\pazocal{Z}_2$-band \textit{cross} if they have a common cell and $\pazocal{Z}_1^{\pm1}\cap\pazocal{Z}_2^{\pm1}=\emptyset$.

In the particular setting of a reduced diagram $\Delta$ over a group associated to a generalized $S$-machine, there exist \textit{$q$-bands} corresponding to bands arising from taking $\pazocal{Z}$ to be some part of the state letters. Note that the makeup of the relations precludes the inclusion of a hub in a $q$-band, so that every cell of the band is a $(\theta,q)$-cell.  

The natural projection of the label of the top (or bottom) of a $q$-band onto $\Theta^+\sqcup\Theta^-$ is called the \textit{history} of the band.  Note that the structure of the relations implies that any reduction of adjacent $\theta$-edges in a side would necessitate a pair of cancellable $(\theta,q)$-cells in the band.  Hence, if $H$ is the history of a $q$-band $\pazocal{Q}$, then $H\in F(\Theta^+)$ and $\pazocal{Q}$ has length $\|H\|$.

Similarly, for a positive (generalized) rule $\theta$ of the machine, there exist \textit{$\theta$-bands} given by taking $\pazocal{Z}$ to be the set of all letters $\theta_i$.  The \textit{history} of a $\theta$-band $\pazocal{S}$ is taken to be $\theta$ if $\pazocal{S}$ is a positive $\theta$-band and $\theta^{-1}$ if it is negative.  The natural projection (without reduction) of the top (or bottom) of a $\theta$-band onto the alphabet given by the letters of the standard base is called the \textit{base} of the band.  As above, the length of the base of the band is equal to the number of $(\theta,q)$-cells in the band.

As opposed to the groups associated to typical $S$-machines (see \cite{W}), though, letters from the tape alphabet of an arbitrary generalized $S$-machine do not obviously define bands in the associated diagrams. However, in the particular setting of diagrams over the groups associated to the machine $\textbf{M}^\pazocal{L}$ (or to any noisy $S$-machine), these bands can be defined by restricting the types of cells which can be present.  Such bands are called \textit{$a$-bands} and are classified as follows:

\begin{enumerate}[label=({\arabic*})]

\item For any $a\in\pazocal{A}$ and any input tape alphabet, there exist $a$-bands given by $\pazocal{Z}=\{\tilde{a},\tilde{a}_1\}$, where $\tilde{a}$ and $\tilde{a}_1$ are the corresponding copies of $a$ and $\varphi_1(a)$, respectively, in this input tape alphabet.


\item For any $b\in\pazocal{B}$ and any input tape alphabet, there exist $a$-bands given by $\pazocal{Z}=\{\tilde{b}\}$, where $\tilde{b}$ is the corresponding copy of $b$ in this input tape alphabet.


\item For any tape letter $a$ of a non-input tape alphabet, there exist $a$-bands given by $\pazocal{Z}=\{a\}$

\end{enumerate}

The $a$-bands of type (1) are called \textit{$\pazocal{A}$-bands}.  Similarly, those of type (2) are called \textit{$b$-bands} and those of type (3) are called \textit{ordinary $a$-bands}.

In all cases, the inclusion of $(\theta,q)$- or $a$-cells in an $a$-band is forbidden, so that any such band must consist only of $(\theta,a)$-cells. Moreover, the inclusion of $(\theta,\pazocal{A})$-cells is forbidden in $b$-bands. Hence, each cell of any $\pazocal{A}$-band is a $(\theta,\pazocal{A})$-cell, each cell of any $b$-band is a $(\theta,b)$-cell, and each cell of any ordinary $a$-band is an ordinary $(\theta,a)$-cell.

Given a $b$-band or an ordinary $a$-band $\pazocal{S}$, the makeup of the groups' relations dictates that the defining edges are labelled identically. Similarly, the defining edges of a $\theta$-band correspond to the same rule, though the index of these edges may differ.

The \textit{history} of an $a$-band is defined in much the same way as it is for $q$-bands.  As in that setting, if $H$ is the history of an $a$-band $\pazocal{U}$, then $H\in F(\Theta^+)$ and $\pazocal{U}$ has length $\|H\|$.

Note that distinct maximal $q$-bands either consist of the same cells with opposite direction or do not intersect at all.  In particular, distinct maximal positive $q$-bands cannot intersect.  Analogous observations apply to distinct maximal $\theta$-bands and distinct maximal $a$-bands.

Given the makeup of the relations of the groups defined in \Cref{sec-the-groups}, a maximal band in a reduced diagram over the canonical presentation of $G_\Omega(\textbf{M}^\pazocal{L})$ can have ends in the following ways:

\begin{itemize}

\item a maximal $\pazocal{A}$-band can have an end on a $(\theta,q)$-cell, on an $a$-cell, or on the diagram's boundary;

\item a maximal $b$-band can have an end on a $(\theta,\pazocal{A})$-cell, on a $(\theta,q)$-cell, on an $a$-cell, or on the diagram's boundary;

\item a maximal ordinary $a$-band can have an end on a $(\theta,q)$-cell or on the diagram's boundary;

\item a maximal $q$-band can have an end on a hub or on the diagram's boundary; and

\item a maximal $\theta$-band can have an end only on the diagram's boundary.

\end{itemize}

Note that if one of these bands has an end as above in one part of the diagram, then it must also have another end in another part of the diagram as it cannot be an annulus.


Suppose the sequence of cells $(\pi_0,\pi_1,\dots,\pi_m)$ comprises a $\theta$-band and $(\gamma_0,\gamma_1,\dots,\gamma_\ell)$ a $q$-band such that $\pi_0=\gamma_0$, $\pi_m=\gamma_\ell$, and no other cells are shared. Suppose further that $\partial\pi_0$ and $\partial\pi_m$ both contain edges on the outer countour of the annulus bounded by the two bands. Then the union of these two bands is called a \textit{$(\theta,q)$-annulus} and $\pi_0$ and $\pi_m$ are called its \textit{corner} cells. 

A \textit{$(\theta,a)$-annulus} is defined similarly, with a $\theta$-band and an $a$-band intersecting twice. If the $a$-band defining this annulus is an $\pazocal{A}$-band, then the $(\theta,a)$-annulus is called a \textit{$(\theta,\pazocal{A})$-annulus}. A \textit{$(\theta,b)$-annulus} and an \textit{ordinary $(\theta,a)$-annulus} are defined similarly.

\begin{figure}[H]
\centering
\includegraphics[scale=0.8]{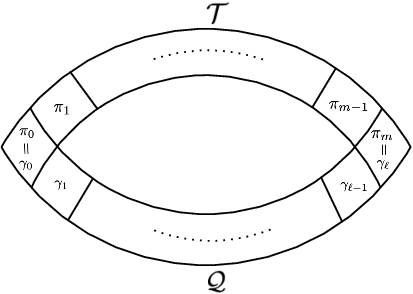}
\caption{$(\theta,q)$-annulus with defining $\theta$-band $\pazocal{T}$ and $q$-band $\pazocal{Q}$}
\end{figure}

Despite the adjustment to the relations, the next statement follows in just the same way as its analogue for diagrams over the groups associated to $S$-machines.

\begin{lemma}[Compare to Lemma 6.1 of \cite{O97}] \label{basic annuli 1}

For any generalized $S$-machine $\textbf{S}$, a reduced circular diagram $\Delta$ over $G_\Omega(\textbf{S})$ contains no:

\begin{enumerate}[label=({\arabic*})]

\item $(\theta,q)$-annuli

\item $q$-annuli

\end{enumerate}

\end{lemma}

For the diagrams over the groups associated to the generalized $S$-machine $\textbf{M}^\pazocal{L}$ (or more generally those associated to any noisy $S$-machine), the existence and makeup of $a$-bands allow for a similar study.  While the makeup of the group relations differ from those in previous settings, some particular subdiagrams can be ruled out in just the same way as in those settings.  For example, the next statement follows by an argument nearly identical to its analogue in \cite{W}.

\begin{lemma}[Lemma 8.1 of \cite{W}] \label{basic annuli 2}

A reduced circular diagram $\Delta$ over $G_\Omega(\textbf{M}^\pazocal{L})$ contains no:

\begin{enumerate}[label=({\arabic*})]

\item $(\theta,a)$-annuli

\item $a$-annuli

\end{enumerate}

\end{lemma}

As a result, in a reduced circular diagram $\Delta$ over $G_\Omega(\textbf{M}^\pazocal{L})$, if a maximal $\theta$-band and a maximal $q$-band (resp. $a$-band) cross, then their intersection is exactly one $(\theta,q)$-cell (resp. one $(\theta,a)$-cell). 

Similarly, the following statement is proved in much the same way as its analogue in \cite{W}:

\begin{lemma}[Lemma 8.2 of \cite{W}] \label{a-bands on a-cell}

If $\Delta$ is a reduced circular diagram over $G_\Omega(\textbf{M}^\pazocal{L})$ and $\pi$ is an $a$-cell in $\Delta$, then no $a$-band of positive length has two ends on $\pi$.

\end{lemma}

%
%

\begin{figure}[H]
\centering
\includegraphics[scale=0.9]{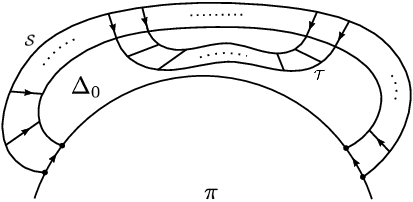}
\caption{$a$-band with two ends on an $a$-cell}
\label{aband1}
\end{figure}

Note that \Cref{a-bands on a-cell} does not rule out the possibility that an $a$-band of length $0$ has two ends on the $a$-cell $\pi$.  This is possible if there exists an edge $\textbf{e}$ of $\partial\pi$ such that $\textbf{e}^{-1}$ is also an edge of $\partial\pi$ (again, this is ignoring the existence of $0$-cells; for topological purposes, we may employ a $0$-refinement so that there exists an edge $\textbf{f}$ adjacent to $\textbf{e}$ such that $\textbf{f}^{-1}$, not $\textbf{e}^{-1}$, is an edge of $\partial\pi$).  In this case, $\pi$ is called a \textit{pinched $a$-cell} and $\textbf{e}^{\pm1}$ are called \textit{pinched edges} of $\pi$.

Given a pinched $a$-cell $\pi$, let $\textbf{s}$ be a maximal subpath of $\partial\pi$ consisting of pinched edges.  Then, there exists a decomposition $\partial\pi=\textbf{s}^{\pm1}\textbf{q}\textbf{s}^{\mp1}\textbf{p}$ such that $\textbf{p}^{-1}$ is the contour of a subdiagram $\Psi_{\pi,\textbf{s}}$ of $\Delta$ not containing $\pi$ (see \Cref{fig-pinched-a-cell}).  In this case, $\textbf{s}^{\pm1}\textbf{q}\textbf{s}^{\mp1}\textbf{p}$ is called the \textit{pinched factorization} of $\partial\pi$ with respect to the \textit{pinched subpath} $\textbf{s}$.

Observe that $\textbf{q}$ is the contour of a subdiagram $\Phi_{\pi,\textbf{s}}$ of $\Delta$ consisting of $\pi$ and $\Psi_{\pi,\textbf{s}}$.  What's more, since $\lab(\partial\pi)\in\Omega$ is cyclically reduced, $\textbf{p}$ and $\textbf{q}$ must be non-trivial subpaths of $\partial\pi$.

Note that $\textbf{s}$ need not be the only pinched subpath of $\partial\pi$.  Indeed, in the terminology above, some edges of $\textbf{p}$ may be pinched.  However, in this case we may use $0$-refinement to assume that $\Psi_{\pi,\textbf{s}}$ is indeed a (circular) subdiagram.

A reduced diagram $\Delta$ over $G_\Omega(\textbf{M}^\pazocal{L})$ is called \textit{smooth} if it contains no pinched $a$-cells.

\begin{figure}[H]
\centering
\includegraphics[scale=1.1]{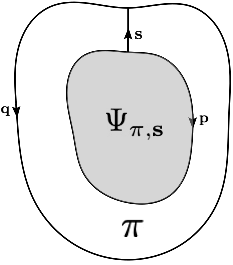}
\caption{The subdiagram $\Phi_{\pi,\textbf{s}}$ corresponding to a pinched $a$-cell $\pi$}
\label{fig-pinched-a-cell}
\end{figure}

\begin{lemma} \label{M(M) annuli}

A reduced circular diagram $\Delta$ over $M(\textbf{M}^\pazocal{L})$ contains no $\theta$-annuli.

\end{lemma}

\begin{proof}

Suppose $\pazocal{S}$ is a $\theta$-annulus in $\Delta$ and let $\Delta_\pazocal{S}$ be the subdiagram bounded by the outer component of $\pazocal{S}$.

If $\Delta_\pazocal{S}$ contains a $(\theta,q)$-cell, then it contains a maximal $q$-band $\pazocal{Q}$.  By \Cref{basic annuli 1}(2), $\pazocal{Q}$ must end on $\partial\Delta_\pazocal{S}$ in two places.  But then $\pazocal{Q}$ and a subband of $\pazocal{S}$ form a $(\theta,q)$-annulus in $\Delta$, contradicting \Cref{basic annuli 1}(1).

Hence, we may assume $\Delta_\pazocal{S}$ consists entirely of $(\theta,a)$-cells.

A similar argument applies if $\Delta_\pazocal{S}$ contains a $(\theta,\pazocal{A})$-cell: In this case, $\Delta_\pazocal{S}$ contains a maximal $\pazocal{A}$-band $\pazocal{U}$.  As $\pazocal{A}$-bands can only end on $(\theta,q)$-cells or on the boundary, \Cref{basic annuli 2}(2) implies $\pazocal{U}$ ends twice on $\partial\Delta_\pazocal{S}$.  But then $\pazocal{U}$ and a subband of $\pazocal{S}$ form a $(\theta,\pazocal{A})$-annulus that contradicts \Cref{basic annuli 2}(1).

Hence, $\Delta_\pazocal{S}$ must consist entirely of $(\theta,b)$-cells and ordinary $(\theta,a)$-cells.  But then any maximal $a$-band in $\Delta_\pazocal{S}$ most form a $(\theta,a)$-annulus with a subband of $\pazocal{S}$, again contradicting \Cref{basic annuli 2}(1).

\end{proof}

As a result, in a reduced circular diagram $\Delta$ over $M(\textbf{M}^\pazocal{L})$, each maximal $\theta$-band and each maximal $q$-band has two ends on $\partial\Delta$.

\smallskip

\subsection{Semi-trapezia} \label{sec-semi-trapezia} \

We now introduce a new classification of reduced diagram over $M(\textbf{S})$ that is unique to this setting.  As trapezia correspond to reduced computations of an $S$-machine (see \Cref{sec-trapezia}), so these diagrams correspond to reduced semi-computations of a generalized $S$-machine.  

Denote the hardware of the generalized $S$-machine $\textbf{S}$ by $(Y,Q)$ with $Y=\sqcup_{i=1}^{s+1} Y_i$ and $Q=\sqcup_{i=0}^s Q_i$.  The next two statements combine to tell us that particular $\theta$-bands correspond to applications of the corresponding rule in the sense of semi-computations.

\begin{lemma} \label{theta-band is one-rule semi-computation}

Let $\pazocal{T}$ be a $\theta$-band of positive length in a reduced diagram $\Delta$ over $M(\textbf{S})$ consisting entirely of $(\theta,a)$-cells of the $Q_{i-1}Q_i$-sector. If the history of $\pazocal{T}$ is $\theta$, then $\lab(\textbf{bot}(\pazocal{T}))$ is $\theta$-applicable and $\lab(\textbf{bot}(\pazocal{T}))\cdot\theta\equiv\lab(\textbf{top}(\pazocal{T}))$.

\end{lemma}

\begin{proof}

Let $\pazocal{T}=(\pi_1,\dots,\pi_k)$.

First, suppose $\pazocal{T}$ is a positive $\theta$-band, i.e $\theta\in\Theta^+(\textbf{S})$.  By the makeup of the relations of $M(\textbf{S})$, for all $j\in\{1,\dots,k\}$, there exist $x_j\in X_i(\theta)$ and $\eps_j\in\{\pm1\}$ such that $\lab(\partial\pi_j)\equiv\theta_i^{-1}x_j^{\eps_j}\theta_i f_{\theta,i}(x_j)^{-\eps_j}$.

As a result, $\lab(\textbf{bot}(\pazocal{T}))\equiv x_1^{\eps_1}\dots x_k^{\eps_k}\in\gen{X_i(\theta)}$ and 
$$\lab(\textbf{top}(\pazocal{T}))=f_{\theta,i}(x_1)^{\eps_1}\dots f_{\theta,i}(x_k)^{\eps_k}=\widetilde{f}_{\theta,i}(x_1^{\eps_1}\dots x_k^{\eps_k})$$ 
Hence, $\lab(\textbf{bot}(\pazocal{T}))$ is $\theta$-applicable $\lab(\textbf{top}(\pazocal{T}))\equiv\lab(\textbf{bot}(\pazocal{T}))\cdot\theta$.

Conversely, suppose $\pazocal{T}$ is a negative $\theta$-band, i.e $\theta\in\Theta^-(\textbf{S})$.  Then, since $\theta^{-1}\in\Theta^+(\textbf{S})$, for all $j\in\{1,\dots,k\}$, there exist $z_j\in X_i(\theta^{-1})$ and $\delta_j\in\{\pm1\}$ such that $\lab(\partial\pi_j)\equiv\theta_i^{-1}f_{\theta^{-1},i}(z_j)^{\delta_j}\theta_i z_j^{-\delta_j}$.

As a result, $\lab(\textbf{top}(\pazocal{T}))\equiv z_1^{\delta_1}\dots z_k^{\delta_k}\in\gen{X_i(\theta^{-1})}$ and $$\lab(\textbf{bot}(\pazocal{T}))=f_{\theta^{-1},i}(z_1)^{\delta_1}\dots f_{\theta^{-1},i}(z_k)^{\delta_k}=\widetilde{f}_{\theta^{-1},i}(z_1^{\delta_1}\dots z_k^{\delta_k})$$
Since $Z_i(\theta^{-1})=X_i(\theta)$, it follows that $\lab(\textbf{bot}(\pazocal{T}))\in\gen{X_i(\theta)}$, i.e $\lab(\textbf{bot}(\pazocal{T}))$ is $\theta$-applicable.  But since $\widetilde{f}_{\theta^{-1},i}=\widetilde{f}_{\theta,i}^{-1}$ by definition, it immediately follows that $$\lab(\textbf{bot}(\pazocal{T}))\cdot\theta=\widetilde{f}_{\theta,i}(\lab(\textbf{bot}(\pazocal{T})))=\lab(\textbf{top}(\pazocal{T}))$$

\end{proof}

\begin{lemma} \label{one-rule semi-computations are theta-bands}

Let $u\to v$ be a semi-computation of $\textbf{S}$ in the $Q_{i-1}Q_i$-sector with history $H$ of length 1, so that $H=\theta\in\Theta(\textbf{S})$. Then there exists a $\theta$-band $\pazocal{T}$ of length $l_\theta(u)$ history $\theta$ consisting entirely of $(\theta,a)$-cells of the $Q_{i-1}Q_i$-sector such that $\lab(\textbf{bot}(\pazocal{T}))\equiv u$ and $\lab(\textbf{top}(\pazocal{T}))\equiv v$.

\end{lemma}

\begin{proof}

First, suppose $\theta\in\Theta^+(\textbf{S})$.  Note that $u\in\gen{X_i(\theta)}$, so that there exist $x_1,\dots,x_k\in X_i(\theta)$ and $\eps_1,\dots,\eps_k\in\{\pm1\}$ such that $u=x_1^{\eps_1}\dots x_k^{\eps_k}$.  By the makeup of the relations, for each $j=1,\dots,k$ one can construct a $(\theta,a)$-cell $\pi_j$ satisfying $\lab(\partial\pi_j)\equiv\theta_i^{-1}x_j^{\eps_j}\theta_i f_{\theta,i}(x_j)^{-\eps_j}$.  Pasting along the $\theta$-edges (and making any necessary cancellations through 0-refinement or gluing) then gives a $\theta$-band $\pazocal{T}^+=(\pi_1,\dots,\pi_k)$ with $\lab(\textbf{bot}(\pazocal{T}^+))\equiv u$ and $\lab(\textbf{top}(\pazocal{T}^+))\equiv\widetilde{f}_{\theta,i}(u)\equiv v$. Hence, since the length of $\pazocal{T}^+$ is $k=|u|_{X_i(\theta)}=l_\theta(u)$, the band $\pazocal{T}^+$ satisfies the statement.

Conversely, suppose $\theta\in\Theta^-(\textbf{S})$.  Then, $v\cdot\theta^{-1}\equiv u$ with $\theta^{-1}\in\Theta^+(\textbf{S})$.  Let $z_1,\dots,z_\ell\in X_i(\theta^{-1})$ and $\delta_1,\dots,\delta_\ell\in\{\pm1\}$ such that $v=z_1^{\delta_1}\dots z_\ell^{\delta_\ell}$.  As above, the makeup of the relations then allows one to construct $(\theta,a)$-cells $\pi_1',\dots,\pi_\ell'$ such that $\lab(\partial\pi_j')\equiv\theta_i^{-1}f_{\theta^{-1},i}(z_j)^{\delta_j}\theta_iz_j^{-\delta_j}$.  Pasting $\pi_1',\dots,\pi_\ell'$ along their $\theta$-edges and making any necessary cancellations then gives a $\theta$-band $\pazocal{T}^-=(\pi_1',\dots,\pi_\ell')$ with $\lab(\textbf{top}(\pazocal{T}^-))\equiv v$ and $\lab(\textbf{bot}(\pazocal{T}^-))\equiv\widetilde{f}_{\theta^{-1},i}(v)\equiv u$.  Thus, the statement follows as above by noting that $\pazocal{T}^-$ has length $\ell=|v|_{X_i(\theta^{-1})}=l_{\theta^{-1}}(v)=l_\theta(u)$.

\end{proof}

Fix $i\in\{1,\dots,s\}$ and suppose $\Delta$ is a reduced circular diagram over $M(\textbf{S})$ which can be decomposed into maximal $\theta$-bands $\pazocal{T}_1,\dots,\pazocal{T}_h$ such that: 

\begin{itemize}

\item $\textbf{top}(\pazocal{T}_j)=\textbf{bot}(\pazocal{T}_{j+1})$ for each $j\in\{1,\dots,h-1\}$

\item $\pazocal{T}_j$ consists entirely of $(\theta,a)$-cells in the $Q_{i-1}Q_i$-sector

\end{itemize}

Then $\Delta$ is called a \textit{semi-trapezium} with height $h$ over $M(\textbf{S})$ in the $Q_{i-1}Q_i$-sector.

In this case, the maximal $\theta$-bands $\pazocal{T}_1,\dots,\pazocal{T}_h$ are said to be enumerated \textit{from bottom to top}.  Further, the \textit{bottom} and \textit{top} of $\Delta$ are defined to be $\textbf{bot}(\Delta)=\textbf{bot}(\pazocal{T}_1)$ and $\textbf{top}(\Delta)=\textbf{top}(\pazocal{T}_h)$, respectively.  Finally, if $\theta_j$ is the history of $\pazocal{T}_j$, then the \textit{history} of $\Delta$ is $\theta_1\dots\theta_h$.

As a semi-trapezium consists entirely of $(\theta,a)$-cells, for each maximal $\theta$-band $\pazocal{T}_j$ the defining edges are labelled identically.  In particular, there exists a factorization $\partial\Delta=\textbf{p}_1^{-1}\textbf{q}_1\textbf{p}_2\textbf{q}_2^{-1}$ such that:

\begin{itemize}

\item $\textbf{q}_1=\textbf{bot}(\Delta)$ and $\textbf{q}_2=\textbf{top}(\Delta)$

\item $\lab(\textbf{p}_1)\equiv\lab(\textbf{p}_2)$, with each a copy of the history of $\Delta$

\end{itemize}  

In particular, $\textbf{bot}(\Delta)$ and $\textbf{top}(\Delta)$ are conjugate in $M(\textbf{S})$.

An iteration of applications of Lemmas \ref{theta-band is one-rule semi-computation} and \ref{one-rule semi-computations are theta-bands} then imply the following two statements, producing the desired correspondence between semi-trapezia and semi-computations:

\begin{lemma} \label{semi-trapezia are semi-computations}

Let $\Delta$ be a semi-trapezium over $M(\textbf{S})$ in the $Q_{i-1}Q_i$-sector with maximal $\theta$-bands $\pazocal{T}_1,\dots,\pazocal{T}_h$ enumerated from bottom to top.  Let $H\equiv\theta_1\dots\theta_h$ be the history of $\Delta$.  Then, letting $w_{j-1}=\lab(\textbf{bot}(\pazocal{T}_j))$ for $j=1,\dots,h$ and $w_h=\lab(\textbf{top}(\pazocal{T}_h))$, there exists a semi-computation $w_0\to\dots\to w_h$ of $\textbf{S}$ in the $Q_{i-1}Q_i$-sector with history $H$. 

\end{lemma}

\begin{lemma} \label{semi-computations are semi-trapezia}

For any reduced semi-computation $w_0\to\dots\to w_t$ of $\textbf{S}$ in the $Q_{i-1}Q_i$-sector with history $H\equiv\theta_1\dots\theta_t$, there exists a semi-trapezium $\Delta$ over $M(\textbf{S})$ in the $Q_{i-1}Q_i$-sector satisfying:

\begin{enumerate} [label=(\alph*)]

\item $\lab(\textbf{bot}(\Delta))\equiv w_0$

\item $\lab(\textbf{top}(\Delta))\equiv w_t$

\item The history of $\Delta$ is $H$

\item $\text{Area}(\Delta)=\sum\limits_{j=1}^t l_{\theta_j}(w_{j-1})$

\end{enumerate}

\end{lemma}

%
%
%
%
%
%
%
%

\medskip


\subsection{Trapezia} \label{sec-trapezia} \

The goal of this section is to define the reduced diagrams over $M(\textbf{S})$ that `simulate' computations of the generalized machine $\textbf{S}$.  This is achieved much in the same way as the semi-trapezia of the last section `simulate' semi-computations (and in much the same way as has been studied in related previous literature).

Let $\pazocal{T}$ be a $\theta$-band over $M(\textbf{S})$ whose first and last cells are $(\theta,q)$-cells.  The maximal subpath of $\textbf{bot}(\pazocal{T})$ whose first and last edges are $q$-edges is called the \textit{trimmed bottom} of the band, denoted $\textbf{tbot}(\pazocal{T})$.  The \textit{trimmed top} $\textbf{ttop}(\pazocal{T})$ is defined similarly.

\begin{figure}[H]
\centering
\includegraphics[scale=1.75]{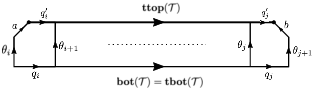}
\caption{$\theta$-band $\pazocal{T}$ with trimmed top}
\end{figure}

\begin{lemma}[Compare with Lemma 6.2 in \cite{W}] \label{theta-bands are one-rule computations}

Let $\textbf{S}$ be a generalized $S$-machine and $\pazocal{T}$ be a $\theta$-band in a reduced diagram $\Delta$ over $M(\textbf{S})$ whose first and last cells are $(\theta,q)$-cells. Suppose the history of $\pazocal{T}$ is $\theta\in\Theta(\textbf{S})$.  Then: 

\begin{enumerate} [label=(\alph*)]

\item $\lab(\textbf{tbot}(\pazocal{T}))$ and $\lab(\textbf{ttop}(\pazocal{T}))$ are admissible words 

\item $\lab(\textbf{tbot}(\pazocal{T}))$ is $\theta$-admissible

\item $\lab(\textbf{tbot}(\pazocal{T}))\cdot\theta\equiv\lab(\textbf{ttop}(\pazocal{T}))$

\end{enumerate}


\end{lemma}

\begin{proof}

If $\theta\in\Theta^+(\textbf{S})$, then the statement follows from an argument similar to that presented as the proof of Lemma 6.2 in \cite{W} for the analogous statement, employing \Cref{theta-band is one-rule semi-computation} for the portions of $\pazocal{T}$ consisting entirely of $(\theta,a)$-cells.  

If instead $\theta\in\Theta^-(\textbf{S})$, then note that:

\begin{itemize}

\item Any $(\theta,q)$-cell of $\pazocal{T}$ corresponds to a part $q_i\to u_iq_i'v_{i+1}$ of $\theta^{-1}$

\item Any $(\theta,a)$-cell of $\pazocal{T}$ corresponds to a relation $\theta_i x_i \theta_i^{-1}=f_{\theta^{-1},i}(x_i)$ for $x_i\in X_i(\theta^{-1})$

\end{itemize}

We may then construct the `mirror' $\theta$-band $\overline{\pazocal{T}}$ by reflecting the cells (see \Cref{mirror}).  The history of $\overline{\pazocal{T}}$ is $\theta^{-1}\in\Theta^+(\textbf{S})$, and so the statement follows by applying the argument of \cite{W} to $\overline{\pazocal{T}}$ instead.

\end{proof}

\begin{figure}[H]
\centering
\begin{subfigure}[b]{0.48\textwidth}
\centering
\includegraphics[scale=1.3]{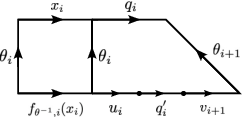}
\caption{A $\theta$-band $\pazocal{T}$ of length 2 with history $\theta$}
\end{subfigure}\hfill
\begin{subfigure}[b]{0.48\textwidth}
\centering
\includegraphics[scale=1.3]{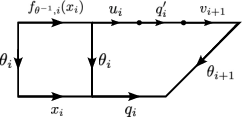}
\caption{The `mirror' $\theta$-band $\overline{\pazocal{T}}$ with history $\theta^{-1}$}
\end{subfigure}
\caption{ \ }
\label{mirror}
\end{figure}

\begin{lemma}[Compare with Lemma 6.3 of \cite{W}] \label{one-rule computations are theta-bands}

Let $U\to V$ be a computation of a generalized $S$-machine $\textbf{S}$ with history $\theta\in\Theta(\textbf{S})$. Then there exists a $\theta$-band $\pazocal{T}$ over $M(\textbf{S})$ with history $\theta$ whose first and last cells are $(\theta,q)$-cells and such that $\lab(\textbf{tbot}(\pazocal{T}))\equiv U$ and $\lab(\textbf{ttop}(\pazocal{T}))\equiv V$.  Moreover, the length of $\pazocal{T}$ is:

\begin{itemize}

\item $l_\theta(U)$ if $\theta\in\Theta^+(\textbf{S})$

\item $l_{\theta^{-1}}(V)$ if $\theta\in\Theta^-(\textbf{S})$

\end{itemize}

\end{lemma}

\begin{proof}

If $\theta\in\Theta^+(\textbf{S})$, then again the statement is proved in much the same way as in \cite{W}, using \Cref{one-rule semi-computations are theta-bands} to construct the subbands consisting of $(\theta,a)$-cells to transform the tape words.

If instead $\theta\in\Theta^-(\textbf{S})$, then there exists a computation $V\to U$ with history $\theta^{-1}\in\Theta^+(\textbf{S})$, and hence we may construct a $\theta$-band $\pazocal{T}'$ as above.  But then the statement follows by taking $\pazocal{T}$ to be the `mirror' $\theta$-band constructed from $\pazocal{T}'$ (see \Cref{mirror}).

\end{proof}

Now, let $\Delta$ be a reduced circular diagram over $M(\textbf{S})$ such that $\partial\Delta=\textbf{p}_1^{-1}\textbf{q}_1\textbf{p}_2\textbf{q}_2^{-1}$, where:
\begin{itemize}

\item $\textbf{p}_1$ and $\textbf{p}_2$ are sides of maximal $q$-bands

\item $\textbf{q}_1$ and $\textbf{q}_2$ are the trimmed sides of maximal $\theta$-bands 

\end{itemize}
Then $\Delta$ is called a \textit{trapezium} over $M(\textbf{S})$.

In this case, $\textbf{p}_1^{-1}\textbf{q}_1\textbf{p}_2\textbf{q}_2^{-1}$ is called the \textit{standard factorization} of the contour.  The paths $\textbf{q}_1$ and $\textbf{q}_2$ are called the \textit{trimmed bottom} and \textit{trimmed top} of the trapezium, respectively, denoted $\textbf{tbot}(\Delta)$ and $\textbf{ttop}(\Delta)$.  Further, $\textbf{p}_1$ and $\textbf{p}_2$ are the \textit{left} and \textit{right} sides of $\Delta$.

\begin{figure}[H]
\centering
\includegraphics[scale=1.25]{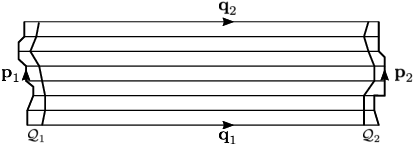}
\caption{Trapezium with side $q$-bands $\pazocal{Q}_1$ and $\pazocal{Q}_2$}
\label{trapezium}
\end{figure}

Let $\textbf{e}_1$ be the first and $\textbf{e}_2$ the last edge of $\textbf{q}_1$.  Then, by the definition of trapezium, there exist maximal $q$-bands $\pazocal{Q}_1$ and $\pazocal{Q}_2$ of $\Delta$ such that $\textbf{e}_j^{-1}$ is a defining edge of $\pazocal{Q}_j$.  As such, $\textbf{top}(\pazocal{Q}_1)=\textbf{p}_1$ and $\textbf{bot}(\pazocal{Q}_2)=\textbf{p}_2$.  The \textit{history} of the trapezium is the history of $\pazocal{Q}_2$ and the length of this history is the trapezium's \textit{height}. The base of $\text{Lab}(\textbf{q}_1)$ is called the \textit{base} of the trapezium.

It is evident from this definition that a non-annular $\theta$-band $\pazocal{T}$ whose first and last cells are $(\theta,q)$-cells can be viewed as a trapezium of height 1, with the standard factorization of $\partial\pazocal{T}$ giving the standard factorization of the trapezium.

However, note that for an arbitrary generalized $S$-machine $\textbf{S}$, a trapezium over $M(\textbf{S})$ need not have the desired structure resembling \Cref{trapezium}, as we have not ruled out the presence of $\theta$-annuli.  For $\textbf{M}^\pazocal{L}$ (or indeed any noisy $S$-machine), though, \Cref{M(M) annuli} takes care of this loose end, implying the following statement:

\begin{lemma} \label{trapezia decomposition}

Let $\Delta$ be a trapezium over $M(\textbf{M}^\pazocal{L})$ with height $h$ and standard factorization $\textbf{p}_1^{-1}\textbf{q}_1\textbf{p}_2\textbf{q}_2^{-1}$.  Then $\Delta$ can be decomposed into maximal $\theta$-bands $\pazocal{T}_1,\dots,\pazocal{T}_h$ such that:

\begin{enumerate}

\item For each $i\in\{1,\dots,h\}$ and $j\in\{1,2\}$, an edge of $\textbf{p}_j$ is a defining edge of $\pazocal{T}_i$

\item $\textbf{ttop}(\pazocal{T}_i)=\textbf{tbot}(\pazocal{T}_{i+1})$ for each $i\in\{1,\dots,h-1\}$

\item $\textbf{tbot}(\Delta)=\textbf{tbot}(\pazocal{T}_1)$ and $\textbf{ttop}(\Delta)=\textbf{ttop}(\pazocal{T}_h)$
 
\end{enumerate} 

\end{lemma}

In the setting of \Cref{trapezia decomposition}, the $\theta$-bands $\pazocal{T}_1,\dots,\pazocal{T}_h$ comprising the trapezium $\Delta$ are said to be \textit{enumerated from bottom to top}.  Hence, the next two statements follow from Lemmas \ref{theta-bands are one-rule computations} and \ref{one-rule computations are theta-bands} and exemplify how the group $M(\textbf{M}^\pazocal{L})$ `simulates' the machine's computational structure:

\begin{lemma} \label{trapezia are computations}

Let $\Delta$ be a trapezium over $M(\textbf{M}^\pazocal{L})$ with history $H\equiv\theta_1\dots\theta_h$ for $h\geq1$ and maximal $\theta$-bands $\pazocal{T}_1,\dots,\pazocal{T}_h$ enumerated from bottom to top. If $W_{j-1}\equiv\lab(\textbf{tbot}(\pazocal{T}_j))$ for $j=1,\dots,h$ and $W_h\equiv\lab(\textbf{ttop}(\pazocal{T}_h))$, then there exists a reduced computation $W_0\to\dots\to W_h$ of $\textbf{M}^\pazocal{L}$ with history $H$.

\end{lemma}

\begin{lemma} \label{computations are trapezia}

For any non-empty reduced computation $W_0\to\dots\to W_t$ of $\textbf{M}^\pazocal{L}$ with history $H$, there exists a trapezium $\Delta$ such that: 

\begin{enumerate} [label=(\alph*)]

\item $\lab(\textbf{tbot}(\Delta))\equiv W_0$

\item $\lab(\textbf{ttop}(\Delta))\equiv W_t$

\item The history of $\Delta$ is  $H$

\item $\text{Area}(\Delta)\leq t\max(\|W_0\|,\dots,\|W_t\|)$

\end{enumerate}

\end{lemma}

\begin{proof}

Note that for any $\theta\in\Theta^+$, $X_i(\theta)\subseteq Y_i$ for all $i$.  Hence, $l_\theta(W)=\|W\|$ for any $\theta$-admissible word $W$.  Thus, the statement follows from \Cref{one-rule computations are theta-bands}.

\end{proof}

\medskip


\section{Diagarams over the Groups Associated to $\textbf{M}^\pazocal{L}$}

\subsection{Compressed semi-trapezia} \

Recall that for noisy $S$-machines, one can introduce the notion of `compressed' semi-computations in a sector of the standard base (see \Cref{sec-M_1} and \Cref{sec-M-semi}).  In the particular setting of $\textbf{M}^\pazocal{L}$, this notion is relevant in the `special' input sector, and will be especially important for the proofs that follow.  

To aid in the study of these compressed semi-computations, we now introduce another class of reduced diagrams (novel to this article) which correspond to reduced compressed semi-computations of $\textbf{M}^\pazocal{L}$ in the `special' input sector in exactly the same way that (semi-)trapezia correspond to reduced (semi-)computations.  As in previous sections, this notion can be made precise for diagrams over $M(\textbf{S})$ for any noisy $S$-machine $\textbf{S}$, but for simplicity we restrict attention here to the setting of interest, $\textbf{M}^\pazocal{L}$.

Let $\pazocal{T}$ be a $\theta$-band over $M(\textbf{M}^\pazocal{L})$ consisting only of $(\theta,a)$-cells over the `special' input sector.  Suppose the first and last cells of $\pazocal{T}$ are $(\theta,\pazocal{A})$-cells.  

The maximal subpath of $\textbf{bot}(\pazocal{T})$ whose first and last edges are $\pazocal{A}$-edges is called the \textit{compressed bottom} of the band, denoted $\mathscr{C}\textbf{bot}(\pazocal{T})$.  The \textit{compressed top} $\mathscr{C}\textbf{top}(\pazocal{T})$ is defined analogously.  As with previous definitions, the compressed bottom and compressed top of $\pazocal{T}$ are collectively called the \textit{compressed sides} of the band.

Note that, as a consequence of its definition, $\mathscr{C}\textbf{bot}(\pazocal{T})$ is the subpath of $\textbf{bot}(\pazocal{T})$ satisfying $\lab(\mathscr{C}\textbf{bot}(\pazocal{T}))\equiv\mathscr{C}(\lab(\textbf{bot}(\pazocal{T})))$.  An analogous observation may be made about $\mathscr{C}\textbf{top}(\pazocal{T})$.

Since the $\theta$-band $\pazocal{T}$ consists only of $(\theta,a)$-cells of a particular sector, the following statement is an immediate consequence of \Cref{theta-band is one-rule semi-computation}:

\begin{lemma} \label{theta-band is one-rule compressed semi-computation}

Let $\pazocal{T}$ be a $\theta$-band with history $\theta$ in a reduced diagram $\Delta$ over $M(\textbf{M}^\pazocal{L})$ consisting entirely of $(\theta,a)$-cells of the `special' input sector.  Suppose the first and last cells of $\pazocal{T}$ are $(\theta,\pazocal{A})$-cells.  Then $\lab(\mathscr{C}\textbf{bot}(\pazocal{T}))*\theta\equiv\lab(\mathscr{C}\textbf{top}(\pazocal{T}))$.

\end{lemma}

Similarly, the following statement is a consequence of \Cref{one-rule semi-computations are theta-bands}:

\begin{lemma} \label{one-rule compressed semi-computations are theta-bands}

Let $u\to v$ be a reduced compressed semi-computation of $\textbf{M}^\pazocal{L}$ in the `special' input sector with history $\theta\in\Theta$.  Then there exists a $\theta$-band $\pazocal{T}$ with history $\theta$ consisting entirely of $(\theta,a)$-cells of the `special' input sector whose first and last cells are $(\theta,\pazocal{A})$-cells and such that $\lab(\mathscr{C}\textbf{bot}(\pazocal{T}))\equiv u$ and $\lab(\mathscr{C}\textbf{top}(\pazocal{T}))\equiv v$.  Moreover, the length of $\pazocal{T}$ is $\|u\|$ if $\theta\in\Theta^+$ and $\|v\|$ if $\theta\in\Theta^-$.

\end{lemma}

%
%
%

Now, let $\Delta$ be a reduced circular diagram over $M(\textbf{M}^\pazocal{L})$ consisting entirely of $(\theta,a)$-cells of the `special' input sector such that $\partial\Delta=\textbf{p}_1^{-1}\textbf{q}_1\textbf{p}_2\textbf{q}_2^{-1}$ where:

\begin{itemize} 

\item $\textbf{p}_1$ and $\textbf{p}_2$ are sides of maximal $\pazocal{A}$-bands

\item $\textbf{q}_1$ and $\textbf{q}_2$ are compressed sides of maximal $\theta$-bands

\end{itemize}

Then $\Delta$ is called a \textit{compressed semi-trapezium} over $M(\textbf{M}^\pazocal{L})$ in the `special' input sector.

As in the setting trapezia, $\textbf{p}_1^{-1}\textbf{q}_1\textbf{p}_2\textbf{q}_2^{-1}$ is called the \textit{standard factorization} of $\Delta$.  Similarly, $\textbf{q}_1$ and $\textbf{q}_2$ are called the \textit{compressed bottom} and \textit{compressed top} of $\Delta$, respectively, and denoted $\mathscr{C}\textbf{bot}(\Delta)$ and $\mathscr{C}\textbf{top}(\Delta)$.  The paths $\textbf{p}_1$ and $\textbf{p}_2$ are called the \textit{left} and \textit{right sides} of $\Delta$.

Let $\textbf{e}_1$ and $\textbf{e}_2$ be the first edges of $\textbf{q}_1$.  Then, noting that $\textbf{e}_i$ is an $\pazocal{A}$-edge, let $\pazocal{U}_i$ be the maximal $\pazocal{A}$-band of $\Delta$ with $\textbf{e}_i^{-1}\in\pazocal{I}_{\pazocal{U}_i}$.  So, $\textbf{p}_1=\textbf{top}(\pazocal{U}_1)$ and $\textbf{p}_2=\textbf{bot}(\pazocal{U}_2)$.  The \textit{history} of $\Delta$ is the history of $\pazocal{U}_2$, while the length of this history is the compressed semi-trapezium's \textit{height}.

Noting the similarity between the definitions of this section and those of \Cref{sec-trapezia}, we have the following analogue of \Cref{trapezia decomposition}, which is proved in exactly the same way:

\begin{lemma} \label{compressed semi-trapezia decomposition}

Let $\Delta$ be a compressed semi-trapezium over $M(\textbf{M}^\pazocal{L})$ in the `special' input sector with height $h$ and standard factorization $\textbf{p}_1^{-1}\textbf{q}_1\textbf{p}_2\textbf{q}_2^{-1}$.  Then $\Delta$ can be decomposed into maximal $\theta$-bands $\pazocal{T}_1,\dots,\pazocal{T}_h$ such that:

\begin{enumerate}

\item For each $i\in\{1,\dots,h\}$ and $j\in\{1,2\}$, an edge of $\textbf{p}_j$ is a defining edge of $\pazocal{T}_i$

\item $\mathscr{C}\textbf{top}(\pazocal{T}_i)=\mathscr{C}\textbf{bot}(\pazocal{T}_{i+1})$ for each $i\in\{1,\dots,h-1\}$

\item $\mathscr{C}\textbf{bot}(\Delta)=\mathscr{C}\textbf{bot}(\pazocal{T}_1)$ and $\mathscr{C}\textbf{top}(\Delta)=\mathscr{C}\textbf{top}(\pazocal{T}_h)$

\end{enumerate} 

\end{lemma}

In this setting, the $\theta$-bands $\pazocal{T}_1,\dots,\pazocal{T}_h$ are again said to be \textit{enumerated from bottom to top}.

Hence, an iteration of applications of Lemmas \ref{theta-band is one-rule compressed semi-computation} and \ref{one-rule compressed semi-computations are theta-bands} imply the following statements:

%
%

\begin{lemma} \label{compressed semi-trapezia are compressed semi-computations}

Let $\Delta$ be a compressed semi-trapezium over $M(\textbf{M}^\pazocal{L})$ in the `special' input sector with history $H\equiv\theta_1\dots\theta_h$ for $h\geq1$ and maximal $\theta$-bands $\pazocal{T}_1,\dots,\pazocal{T}_h$ enumerated from bottom to top.  If $w_{j-1}=\lab(\mathscr{C}\textbf{bot}(\pazocal{T}_j))$ for $j=1,\dots,h$ and $w_h=\lab(\mathscr{C}\textbf{top}(\pazocal{T}_h))$, then there exists a reduced compressed semi-computation $w_0\to\dots\to w_h$ of $\textbf{M}^\pazocal{L}$ in the `special' input sector with history $H$.

\end{lemma}

%
%

\begin{lemma} \label{compressed semi-computations are compressed semi-trapezia}

For any non-empty reduced compressed semi-computation $w_0\to\dots\to w_t$ of $\textbf{M}^\pazocal{L}$ in the `special' input sector with history $H$, there exists a compressed semi-trapezium $\Delta$ over $M(\textbf{M}^\pazocal{L})$ in the `special' input sector such that:

\begin{enumerate} [label=(\alph*)]

\item $\lab(\mathscr{C}\textbf{bot}(\Delta))\equiv w_0$

\item $\lab(\mathscr{C}\textbf{top}(\Delta))\equiv w_t$

\item The history of $\Delta$ is $H$

\item $\text{Area}(\Delta)\leq t\max(\|w_0\|,\dots,\|w_t\|)$

\end{enumerate}

\end{lemma}

\medskip

\subsection{Disks} \

%
%
%
%
%
%
%
%

Next, a new set of relations are added to the canonical presentations of the groups $G(\textbf{M}^\pazocal{L})$ and $G_\Omega(\textbf{M}^\pazocal{L})$ in much the same way as done in \cite{W}. The added relations, called \textit{disk relations}, are given by all relations of the form $W=1$ such that $W$ is a configuration accepted by $\textbf{M}^\pazocal{L}$ with $\ell(W)=1$, {\frenchspacing i.e. so that there exists a one-machine computation} of $\textbf{M}^\pazocal{L}$ accepting $W$ (see \Cref{sec-M-standard}).  For simplicity, the hub relation $W_{ac}=1$ is also called a disk relation in this setting, so that the disk relations are of the form $W=1$ for all accepted configurations $W$ such that $\ell(W)\leq1$.

\begin{lemma} \label{disk relations}

For any configuration $W$ accepted by $\textbf{M}^\pazocal{L}$, there exists a reduced circular diagram $\Gamma_W$ over $G(\textbf{M}^\pazocal{L})$ containing a single hub such that $\lab(\partial\Gamma_W)\equiv W$.

\end{lemma}

\begin{proof}

Let $\pazocal{C}$ be an accepting computation of $W$ and $H$ be its history. By Lemma \ref{computations are trapezia}, there exists a trapezium $\Delta$ corresponding to $\pazocal{C}$ with $\lab(\textbf{tbot}(\Delta))\equiv W$ and $\lab(\textbf{ttop}(\Delta))\equiv W_{ac}$.

As this is a computation of the standard base and the $(R_0^\pazocal{L}(L))^{-1}\{t(1)\}$-sector has empty tape alphabet, no trimming is necessary in $\Delta$. So, the left and right sides of $\Delta$ are labelled by the identical copies of $H$.  Hence, we may paste the sides of $\Delta$ together to produce a reduced annular diagram $\Delta'$ over $M(\textbf{M}^\pazocal{L})$ with outer contour label $W$ and inner contour label $W_{ac}^{-1}$.

But a single hub can now be pasted into the center of $\Delta'$ to produce a diagram $\Gamma_W$ satisfying the statement.

\end{proof}

As a result of Lemma \ref{disk relations}, the presentation given by adding the disk relations to the canonical presentation of $G(\textbf{M}^\pazocal{L})$ defines a group isomorphic to $G(\textbf{M}^\pazocal{L})$. Moreover, since $G_\Omega(\textbf{M}^\pazocal{L})$ is a quotient of $G(\textbf{M}^\pazocal{L})$, the same is true for the presentation given by adding disk relations to the canonical presentation of $G_\Omega(\textbf{M}^\pazocal{L})$. 

These new presentations are called the \textit{disk presentations} of the groups $G(\textbf{M}^\pazocal{L})$ and $G_\Omega(\textbf{M}^\pazocal{L})$. For a diagram over the disk presentation of one of these groups, a cell corresponding to a disk relation (or its inverse) is referred to simply as a \textit{disk}.  Note that in addition to the possibilities outlined in Section 5.2, a maximal $q$-band or maximal $a$-band (of any type) in a diagram over the disk presentation of $G_\Omega(\textbf{M}^\pazocal{L})$ may have an end on a disk.

Finally, note that since maximal $\theta$-bands cannot end on disks, Lemmas \ref{basic annuli 1}-\ref{a-bands on a-cell} have direct analogues for reduced circular diagrams over the disk presentation of $G_\Omega(\textbf{M}^\pazocal{L})$.

The following analogue of \Cref{a-bands on a-cell} for disks follows by a similar proof, using the fact that a counterexample would produce a $(\theta,a)$-annulus that contradicts \Cref{basic annuli 2}.

\begin{lemma} \label{a-bands on disk}

If $\Delta$ is a reduced circular diagram over the disk presentation of $G_\Omega(\textbf{M}^\pazocal{L})$ and $\Pi$ is a disk in $\Delta$, then no $a$-band of positive length has two ends on $\Pi$.

\end{lemma}

%
%

As with \Cref{a-bands on a-cell}, \Cref{a-bands on disk} does not rule out the possibility that an $a$-band of length $0$ has two ends on the disk $\Pi$.  In this case, $\Pi$ is called a \textit{pinched disk} and the corresponding $a$-edges are called \textit{pinched edges} of $\Pi$.  As in the setting of pinched $a$-cells, any maximal subpath $\textbf{s}$ of $\partial\Pi$ consisting of pinched edges induces a \textit{pinched decomposition} $\textbf{s}^{\pm1}\textbf{q}\textbf{s}^{\mp1}\textbf{p}$ of $\partial\Pi$ with respect to the \textit{pinched subpath} $\textbf{s}$, so that $\textbf{p}^{-1}$ bounds a subdiagram $\Psi_{\Pi,\textbf{s}}$ of $\Delta$ not containing $\Pi$ (see \Cref{fig-pinched-a-cell}).

Note that since disk relations are cyclically reduced by construction, $\textbf{p}$ and $\textbf{q}$ must be non-trivial subpaths of $\partial\Pi$.  Moreover, by the structure of configurations of $\textbf{M}^\pazocal{L}$, exactly one of $\textbf{p}$ or $\textbf{q}$ contains $q$-edges, while the other consists entirely of $a$-edges labelled by letters from the same tape alphabet as those of $\textbf{s}$.

\medskip


\subsection{Weights} \label{sec-weights} \

Next, in a way similar to that outlined in \cite{W}, the method with which one counts the area of a diagram over the disk presentation of $G_\Omega(\textbf{M}^\pazocal{L})$ is altered. This is done by introducing a \textit{weight function}, wt, on the cells of such diagrams.  Before doing so, we first define several auxiliary unary functions on the natural numbers:

\begin{itemize}

\item $\chi(n)=nc_0^n$

\item $h_\pazocal{L}(n)=c_0\TM_\pazocal{L}(c_0n)^3+nc_0^{n}+c_0n+L$

\item $f_\pazocal{L}(n)=c_1 \chi(h_\pazocal{L}(n))$

\item $g_\pazocal{L}(n)=c_0n^3+nf_\pazocal{L}(c_0n)$

\end{itemize}

It is easy to see that $\chi$ is non-decreasing, and so each function above is also non-decreasing.  Further, as the class of computable functions is closed under sums, products, and composition, each of the functions above is computable.  Finally, it is important to note that since $f_\pazocal{L}$ is non-decreasing, $g_\pazocal{L}$ is \textit{super-additive}; that is, for any $m,n\in\N$,
\begin{align*}
g_\pazocal{L}(m+n)&\geq
g_\pazocal{L}(m)+g_\pazocal{L}(n)
\end{align*}

Now, define the weight of a cell $\Pi$ of a diagram $\Delta$ over the disk presentation of $G_\Omega(\textbf{M}^\pazocal{L})$ as follows:

\begin{itemize}

\item If $\Pi$ is a $(\theta,q)$-cell or a $(\theta,a)$-cell (of any type), then $\text{wt}(\Pi)=1$.

\item If $\Pi$ is a disk, then letting $W$ be the configuration of $\textbf{M}^\pazocal{L}$ such that $\lab(\partial\Pi)\equiv W^{\pm1}$, $\text{wt}(\Pi)=f_\pazocal{L}(\|W(2)\|)$.

\item If $\Pi$ is an $a$-cell, then $\text{wt}(\Pi)=g_\pazocal{L}(\|\partial\Pi\|)$.

\end{itemize}



Naturally, this definition is extended to give the \textit{weight} $\text{wt}(\Delta)$ of a reduced diagram $\Delta$ over the disk presentation of $G_\Omega(\textbf{M}^\pazocal{L})$, so that it is given by the sum of the weights of the cells of $\Delta$.

\medskip


\section{Diagrams without disks} \label{sec-no-disks}

\subsection{$M$-minimal diagrams} \label{sec-M-minimal} \

The goal of this section is to study diagrams over $M_\Omega(\textbf{M}^\pazocal{L})$, yielding an upper bound on the weight of a reduced circular diagram in terms of its perimeter. However, this goal is not achieved for any possible reduced circular diagram over $M_\Omega(\textbf{M}^\pazocal{L})$, but rather for a specific class of such diagrams that will be shown to be `generic' in a particular sense.

For any $\pazocal{A}$-edge $\textbf{e}$ of a reduced circular diagram over $M_\Omega(\textbf{M}^\pazocal{L})$, the (unique) maximal $\pazocal{A}$-band for which $\textbf{e}$ is a defining edge is denoted $\pazocal{U}(\textbf{e})$.  Then, given an $a$-cell $\pi$ and a maximal $\theta$-band $\pazocal{T}$, $E(\pi,\pazocal{T})$ is defined to be the set of $\pazocal{A}$-edges $\textbf{e}$ of $\partial\pi$ such that $\pazocal{U}(\textbf{e})$ crosses $\pazocal{T}$.

Now, a reduced circular diagram $\Delta$ over $M_\Omega(\textbf{M}^\pazocal{L})$ is then called \textit{$M$-minimal} if the following conditions are satisfied:

\begin{addmargin}[1em]{0em}

\begin{enumerate}[label=(MM{\arabic*})]

\item For any $a$-cell $\pi$ and maximal $\theta$-band $\pazocal{T}$ in $\Delta$, $|E(\pi,\pazocal{T})|\leq\frac{1}{2}|\partial\pi|_\pazocal{A}$.


\item Let $\pi_1$ and $\pi_2$ be two $a$-cells in $\Delta$.  Suppose there exist three consecutive $\pazocal{A}$-edges $\textbf{e}_1,\textbf{e}_2,\textbf{e}_3$ of $\partial\pi_1$ such that $\pazocal{U}(\textbf{e}_j)$ has an end on $\pi_2$.  Let $\Psi$ be the subdiagram of $\Delta$ bounded by the $\pazocal{A}$-bands $\pazocal{U}(\textbf{e}_j)$ and the corresponding subpaths of $\partial\pi_1$ and $\partial\pi_2$ such that $\Psi$ does not contain $\pi_1$ or $\pi_2$ (see \Cref{MM2}). Then $\Psi$ contains an $a$-cell.

\end{enumerate}

\end{addmargin}

\begin{figure}[H]
\centering
\includegraphics[scale=1.25]{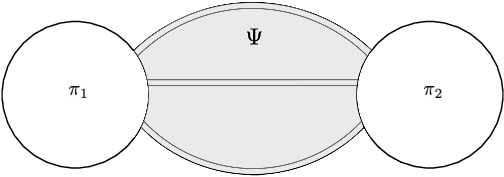}
\caption{Condition (MM2)}
\label{MM2}
\end{figure}


Note that it is a consequence of this definition that a subdiagram of a (smooth) $M$-minimal diagram is necessarily a (smooth) $M$-minimal diagram.

\medskip


\subsection{$\pazocal{A}$-bands and $\theta$-annuli in Smooth Diskless Diagrams} \label{sec-M-minimal-bands} \

The next goal is to study the makeup of smooth circular diagrams over $M_\Omega(\textbf{M}^\pazocal{L})$ to understand their makeup. 

Given a smooth circular diagram $\Delta$ over $M_\Omega(\textbf{M}^\pazocal{L}$), let $\pazocal{Q}=(\Pi_1,\dots,\Pi_m)$ be a maximal positive $q$-band of length $m\geq1$.  Suppose there exists an $a$-cell $\pi$ and an $\pazocal{A}$-edge $\textbf{e}$ of $\partial\pi$ such that $\pazocal{U}(\textbf{e})$ has an end on a $(\theta,q)$-cell of $\pazocal{Q}$.  Let $\pazocal{U}(\textbf{e})=(\pi_1,\dots,\pi_k)$ and $\Pi_\ell$ be the $(\theta,q)$-cell on which $\pazocal{U}(\textbf{e})$ has this end.  Then, define $\pazocal{V}(\textbf{e})$ to be the sequence of cells $$\pazocal{V}(\textbf{e})=(\pi_1,\dots,\pi_k,\Pi_\ell,\dots,\Pi_m)$$
By construction, $\pazocal{V}(\textbf{e})$ can be identified with a union of the $\pazocal{A}$-band $\pazocal{U}(\textbf{e})$ with a subband of the $q$-band $\pazocal{Q}$ (see \Cref{V(e)}).  So, since $\pi_k$ and $\Pi_\ell$ share a boundary edge, $\pazocal{V}(\textbf{e})$ is a subdiagram of $\Delta$.  While it is not itself a band, $\pazocal{V}(\textbf{e})$ does have a band-like structure, connecting $\textbf{e}$ to $\partial\Delta$ by a sequence of cells in which each consecutive pair of cells shares a boundary edge.
\begin{figure}[H]
\centering
\includegraphics[scale=1]{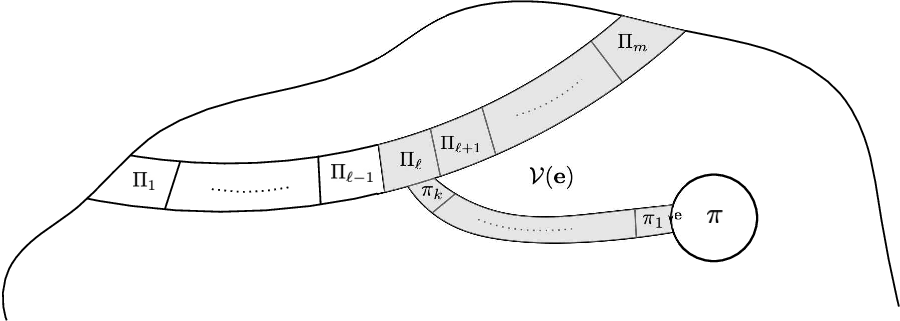}
\caption{Construction of $\pazocal{V}(\textbf{e})$ in the $M$-minimal diagram $\Delta$}
\label{V(e)}
\end{figure}





Let $\textbf{f}$ be the end of $\pazocal{U}(\textbf{e})$ which is on the boundary of $\Pi_\ell$.  Then, by the construction of the relations, $\pazocal{Q}$ is a positive $q$-band corresponding to the part $Q_1^\pazocal{L}(1)$ of the state letters of $\textbf{M}^\pazocal{L}$ and $\textbf{f}^{-1}$ is an edge of $\textbf{bot}(\pazocal{Q})$.

Now, to any smooth circular diagram $\Delta$ over $M_\Omega(\textbf{M}^\pazocal{L})$, construct the (unoriented) graph $\Gamma_a(\Delta)$ as follows:

\begin{enumerate}

\item The set of vertices is $\{v_0,v_1,\dots,v_\ell\}$, where each $v_i$ for $i\geq1$ corresponds to one of the $\ell$ $a$-cells of $\Delta$ and $v_0$ is a single exterior vertex.  

\item For $i,j\geq1$ and for any positive $\pazocal{A}$-band which has ends on the $a$-cells corresponding to $v_i$ and $v_j$, there is a corresponding edge $(v_i,v_j)$. Such an edge is called \textit{internal}.

\item For $i\geq1$ and any positive $\pazocal{A}$-band which has one end on the $a$-cell corresponding to $v_i$ and the other end on either a $(\theta,q)$-cell or on $\partial\Delta$, there is a corresponding edge $(v_0,v_i)$. Such an edge is called \textit{external}.

\end{enumerate}

\begin{lemma} \label{Gamma_a planar}

For any smooth circular diagram $\Delta$ over $M_\Omega(\textbf{M}^\pazocal{L})$, the graph $\Gamma_a(\Delta)$ can be constructed to be planar.

\end{lemma}

\begin{proof}

The graph $\Gamma_a(\Delta)$ is constructed as an `estimating graph' that is `auxiliary' to the planar graph underlying the diagram $\Delta$ (see Section 9.5 of \cite{O}).  Note the resemblance between this construction and that of the dual graph to $\Delta$.

Each interior vertex of $\Gamma_a(\Delta)$ is placed at the center of the corresponding $a$-cell in $\Delta$, while the exterior vertex is placed at some point in the unbounded component $X$ of the complement of $\partial\Delta$ in the plane.

To define the edges, we construct several arcs in the plane and implicitly appeal to the Jordan curve and Jordan-Sch{\"o}nflies theorems (see Section 9.1 of \cite{O}).  Viewing all arcs as images of the unit interval $[0,1]$, two arcs $\gamma_1$ and $\gamma_2$ are \textit{disjoint} if $\gamma_1(0,1)\cap\gamma_2(0,1)=\emptyset$.  Similarly, given a connected region $U$ of the plane, the arc $\gamma$ is \textit{contained in $U$} if $\gamma(0,1)\subseteq U$.

Note that for any finite set $F$ of points of $\partial\Delta$, one can construct a set of $|F|$ (pairwise) disjoint arcs contained in $X$ connecting $v_0$ to the points of $F$.  Hence, in place of an external edge of $\Gamma_a(\Delta)$, it suffices to construct the subpath which connects the corresponding interior vertex to a distinct point of $\partial\Delta$.

First, let $\pi_i$ be the $a$-cell corresponding to the vertex $v_i$.  Then, as above we construct $|\partial\pi_i|_\pazocal{A}$ disjoint arcs contained in the interior of $\pi_i$ connecting the vertex $v_i$ to the midpoints of the $\pazocal{A}$-edges of $\partial\pi_i$.  For an $\pazocal{A}$-edge $\textbf{e}$ of $\partial\pi_i$, denote the corresponding arc by $t_i(\textbf{e})$.

Next, let $\pazocal{U}$ be a positive $\pazocal{A}$-band which has an end on the $a$-cell $\pi_i$.  For every cell $\Pi$ comprising $\pazocal{U}$, construct an arc $t_\pazocal{U}(\Pi)$ contained in the interior of $\Pi$ connecting the midpoints of the corresponding defining edges of $\pazocal{U}$.  

Let $\textbf{e}$ be the $\pazocal{A}$-edge of $\partial\pi_i$ such that $\textbf{e}^{\pm1}$ is an end of $\pazocal{U}$.  So, $\pazocal{U}$ and $\pazocal{U}(\textbf{e})$ consist of the same cells, but perhaps have different directions.  Let $\textbf{f}$ be the $\pazocal{A}$-edge distinct from $\textbf{e}$ that is an end of $\pazocal{U}(\textbf{e})$.

Suppose $\pazocal{U}$ has an end on the $a$-cell $\pi_j$ for $j\neq i$.  Consequently, $\textbf{f}$ is an edge of $(\partial\pi_j)^{-1}$.  Then, the arcs $t_\pazocal{U}(\Pi)$, $t_i(\textbf{e})$, and $t_j(\textbf{f}^{-1})$ together form an arc $\gamma(\pazocal{U})$ connecting $v_i$ to $v_j$.  This arc is taken as the internal edge corresponding to $\pazocal{U}$ (see \Cref{a-graph-internal}).

\begin{figure}[H]
\centering
\includegraphics[scale=3.25]{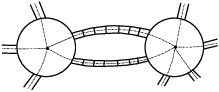}
\caption{The construction of internal edges of $\Gamma_a(\Delta)$}
\label{a-graph-internal}
\end{figure}

Hence, by \Cref{a-bands on a-cell} and the assumption that $\Delta$ is smooth, it suffices to assume that $\pazocal{U}$ has an end on either a $(\theta,q)$-cell or on $\partial\Delta$.  Then, as above, the arcs $t_\pazocal{U}(\Pi)$ and $t_i(\textbf{e})$ together form an arc $\gamma(\pazocal{U})$ connecting $v_i$ to the midpoint of $\textbf{f}$.


If $\textbf{f}$ is an edge of $\partial\Delta$, then $\gamma(\pazocal{U})$ is taken as the subpath of the external edge corresponding to $\pazocal{U}$.

Otherwise, $\textbf{f}$ is an edge of $\textbf{bot}(\pazocal{Q})^{-1}$ for some maximal positive $q$-band $\pazocal{Q}=(\Pi_1,\dots,\Pi_m)$.  In this case, fix $\ell\in\{1,\dots,m\}$ such that $\textbf{f}$ is an edge of $(\partial\Pi_\ell)^{-1}$.  Note that, by the definition of the rules of $\textbf{M}^\pazocal{L}$, $\textbf{f}^{-1}$ is the only $\pazocal{A}$-edge of $\partial\Pi_\ell$.

Letting $(\textbf{e}_0,\textbf{e}_1,\dots,\textbf{e}_m)$ be the defining edge sequence of $\pazocal{Q}$, add $m$ auxiliary vertices to the interior of each $\textbf{e}_j$, enumerated by their proximity to $\textbf{top}(\pazocal{Q})$.  Then, we construct an arc $t_\ell(\pazocal{U})$ contained in the interior of $\Pi_\ell$ connecting the midpoint of $\textbf{f}$ and the $\ell$-th auxiliary vertex of $\textbf{e}_\ell$.

Similarly, for each $j\in\{\ell+1,\dots,m\}$, construct the arc $t_j(\pazocal{U})$ contained in the interior of $\Pi_j$ connecting the  $\ell$-th auxiliary vertices of $\textbf{e}_{j-1}$ and $\textbf{e}_j$.

Then, the arcs $\gamma(\pazocal{U})$ and $t_j(\pazocal{U})$ for $\ell\leq j\leq m$ together form an arc $\rho(\pazocal{U})$ connecting $v_i$ with the $\ell$-th auxiliary vertex of $\textbf{e}_m$.

Note that, by construction, if two positive $\pazocal{A}$-bands $\pazocal{U}$ and $\pazocal{U}'$ both have ends on $(\theta,q)$-cells of $\pazocal{Q}$, then these ends are on distinct $(\theta,q)$-cells.  Hence, in this case the arcs $\rho(\pazocal{U})$ and $\rho(\pazocal{U}')$ can be constructed to be disjoint (see \Cref{a-graph-q-band}).

Hence, $\rho(\pazocal{U})$ can be taken as the subpath of the external edge corresponding to $\pazocal{U}$.

Thus, as distinct maximal $\pazocal{A}$-bands cannot intersect and $\pazocal{A}$-bands and $q$-bands cannot cross, these arcs together define $\Gamma_a(\Delta)$ as a planar graph.

\end{proof}

\begin{figure}[H]
\centering
\includegraphics[scale=3.25]{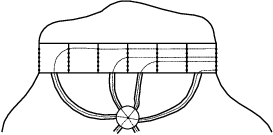}
\caption{The construction of external edges of $\Gamma_a(\Delta)$ for $\pazocal{A}$-bands with one end on an $a$-cell and the other on a $(\theta,q)$-cell}
\label{a-graph-q-band}
\end{figure}

Given a smooth circular diagram, Lemma \ref{a-bands on a-cell} implies that $\Gamma_a(\Delta)$ contains no loops. Further, letting $d(v)$ be the degree of the interior vertex $v$ in $\Gamma_a(\Delta)$, condition (L1) and \Cref{semi-computation deltas} imply $d(v)\geq C$.

For two interior vertices $v$ and $w$ of $\Gamma_a(\Delta)$, suppose there exist consecutive edges $e_1,\dots,e_\ell$ joining $v$ and $w$ such that $e_i$ and $e_{i+1}$ bound a 2-gon for all $i=1,\dots,\ell-1$. If $\Delta$ satisfies (MM2), then $\ell\leq2$. If in this case $\ell=2$, then the edges $e_1$ and $e_2$ are called a \textit{doubled pair}.

The planar graph $\Gamma_a'(\Delta)$ is then formed from $\Gamma_a(\Delta)$ by simply replacing any doubled pair of edges with a single edge. Note that the set of vertices of $\Gamma_a'(\Delta)$ can be identified with that of $\Gamma_a(\Delta)$.

By construction, $\Gamma_a'(\Delta)$ contains no loop and also contains no 2-gon on a pair of interior vertices. Further, letting $d'(v)$ be the degree of the interior vertex $v$ in $\Gamma_a'(\Delta)$, then $d'(v)\geq C/2$.

These properties and the parameter choice $C\geq12$ imply the following statement:

\begin{lemma}[Lemma 3.2 of \cite{O97}] \label{Gamma_a' special cell}

Suppose $\Delta$ is a smooth circular diagram over the canonical presentation of $M_\Omega(\textbf{M}^\pazocal{L})$ which satisfies condition (MM2).  If $\Delta$ contains at least one $a$-cell, then there exists an interior vertex $v$ of $\Gamma_a'(\Delta)$ such that at least $d'(v)-3$ consecutive edges join $v$ with the exterior vertex and there are no other vertices between these edges.

\end{lemma}

The following is an immediate consequence of the construction of $\Gamma_a'(\Delta)$ from $\Gamma_a(\Delta)$:

\begin{lemma} \label{Gamma_a special cell}

Suppose $\Delta$ is a smooth circular diagram over the canonical presentation of $M_\Omega(\textbf{M}^\pazocal{L})$ which satisfies condition (MM2).  If $\Delta$ contains at least one $a$-cell, then there exists an interior vertex $v$ of $\Gamma_a(\Delta)$ such that at least $d(v)-6$ consecutive edges join $v$ with the exterior vertex and there are no other vertices between these edges.

\end{lemma}

\begin{lemma} \label{M-minimal theta-annuli}

A smooth $M$-minimal diagram $\Delta$ contains no $\theta$-annuli.

\end{lemma}

\begin{proof}

Suppose to the contrary that $\Delta$ contains a $\theta$-annulus $\pazocal{S}$ and let $\Delta_\pazocal{S}$ be the subdiagram bounded by a side of $\pazocal{S}$ which contains $\pazocal{S}$.

As in the proof of \Cref{M(M) annuli}, $\Delta_\pazocal{S}$ cannot contain any $(\theta,q)$-cell, as such a cell would imply the existence of a $(\theta,q)$-annulus contradicting \Cref{basic annuli 1}(1).

Further, Lemma \ref{M(M) annuli} implies that $\Delta_\pazocal{S}$ must contain an $a$-cell. So, applying Lemma \ref{Gamma_a special cell}, there exists an interior vertex $v$ of $\Gamma_a(\Delta_\pazocal{S})$ such that at least $d(v)-6$ edges join $v$ to the exterior vertex.

Let $\pi$ be the $a$-cell of $\Delta_\pazocal{S}$ corresponding to the vertex $v$.  Then, an edge of $\Gamma_a(\Delta_\pazocal{S})$ corresponds to a maximal positive $\pazocal{A}$-band $\pazocal{U}$ which has ends on both $\pi$ and on $\partial\Delta$.  Letting $\textbf{e}$ be the edge of $\partial\pi$ such that $\textbf{e}^{\pm1}$ is an end of $\pazocal{U}$, this implies $\pazocal{U}(\textbf{e})$ must cross $\pazocal{S}$.  So, $\textbf{e}\in E(\pi,\pazocal{S})$.

Hence, $|E(\pi,\pazocal{S})|\geq d(v)-6$.  But then the parameter choice $C>12$ implies $d(v)-6>d(v)/2$, so that $|E(\pi,\pazocal{S})|>\frac{1}{2}|\partial\pi|_\pazocal{A}$, contradicting (MM1).

\end{proof}

\begin{lemma} \label{M-minimal is smooth}

Every $M$-minimal diagram is smooth.

\end{lemma}

\begin{proof}

Suppose to the contrary that the $M$-minimal diagram $\Delta$ contains a pinched $a$-cell.  Choose an $a$-cell $\pi$ and a pinched subpath $\textbf{s}$ of $\partial\pi$ such that the subdiagram $\Psi_{\pi,\textbf{s}}$ has minimal weight.

If $\Psi_{\pi,\textbf{s}}$ contains a pinched $a$-cell, then we can pass to the subdiagram corresponding to this pinched $a$-cell, contradicting the minimal weight hypothesis.
%
%
Hence, $\Psi_{\pi,\textbf{s}}$ is a smooth $M$-minimal diagram.

Let $\textbf{s}^{\pm1}\textbf{q}\textbf{s}^{\mp1}\textbf{p}$ be the pinched factorization of $\partial\pi$ with respect to $\textbf{s}$.  Since $\textbf{p}$ consists entirely of $a$-edges, Lemmas \ref{basic annuli 1}(2) and \ref{M-minimal theta-annuli} imply that any (positive) cell of $\Psi_{\pi,\textbf{s}}$ is an $a$-cell.  Moreover, since $\lab(\partial\pi)\in\Omega$ is cyclically reduced, $\lab(\textbf{p})$ must be a non-trivial reduced word, so that $\Psi_{\pi,\textbf{s}}$ contains at least one $a$-cell.

As a result, \Cref{Gamma_a special cell} implies $\Psi_{\pi,\textbf{s}}$ contains an $a$-cell $\pi_0$ and $\ell\geq|\partial\pi_0|_\pazocal{A}-6\geq C-6$ consecutive $\pazocal{A}$-edges $\textbf{e}_1,\dots,\textbf{e}_\ell$ of $\partial\pi_0$ such that $\pazocal{U}(\textbf{e}_j)$ has an end on $\textbf{p}^{-1}$ and such that no $a$-cell is between these $a$-bands.

But since $\textbf{p}$ is a subpath of $\partial\pi$, the parameter choice $C\geq9$ then implies $\pi$ and $\pi_0$ form a counterexample to condition (MM2).

\end{proof}

\medskip


\subsection{$a$-scopes} \label{sec-a-scopes} \

Before establishing the upper bound on the weight of $M$-minimal diagrams, we first study a consequence of \Cref{Gamma_a special cell} that will prove useful for future arguments.

Let $\pi$ be an $a$-cell and $\textbf{t}$ be a subpath of a boundary component of a reduced diagram $\Delta$ over the canonical presentation of $M_\Omega(\textbf{M}^\pazocal{L})$.  Let $\textbf{e}_1$ and $\textbf{e}_2$ be $\pazocal{A}$-edges of $\partial\pi$ such that $\pazocal{U}(\textbf{e}_i)$ has an end on $\textbf{t}$.  Suppose there exists a subpath $\textbf{s}$ of $\partial\pi$ such that $\textbf{s}$, a subpath of $\textbf{t}$, and the bands $\pazocal{U}(\textbf{e}_1),\pazocal{U}(\textbf{e}_2)$ bound a (circular) subdiagram $\Psi$ of $\Delta$ which contains neither $\pi$ nor any $(\theta,q)$-cell.

%
%
%
%
%

Then $\Psi$ is called an \textit{$a$-scope} on $\textbf{t}$ with \textit{associated $a$-cell} $\pi$, \textit{associated subpath} $\textbf{s}$, and \textit{size} $|\textbf{s}|_\pazocal{A}$. 

If in this case $|\textbf{s}|_\pazocal{A}>\frac{1}{2}|\partial\pi|_\pazocal{A}$, then $\Psi$ is called a \textit{big $a$-scope}.  If $\Psi$ contains no $a$-cell, then it is called a \textit{pure $a$-scope}.  Note that there exists a subdiagram $\tilde{\Psi}$ of $\Delta$ consisting of $\Psi$ and $\pi$.  In this case, $\tilde{\Psi}$ is called the \textit{completion} of $\Psi$.

\begin{lemma} \label{big a-scope}

Let $\textbf{t}$ be a subpath of a boundary component of a reduced diagram $\Delta$ over the canonical presentation of $M_\Omega(\textbf{M}^\pazocal{L})$.  Suppose $\Delta$ contains an $a$-scope $\Psi_0$ on $\textbf{t}$ such that the completion $\tilde{\Psi}_0$ is smooth and satisfies condition (MM2).  If $\Psi_0$ is not a pure $a$-scope, then there exists a big $a$-scope $\Psi_1$ on $\textbf{t}$ such that the completion $\tilde{\Psi}_1$ is a subdiagram of $\Psi_0$.

\end{lemma}

\begin{proof}

Let $\textbf{s}_0$ be the associated subpath of $\Psi_0$ and let $\textbf{t}_0$ be the subpath of $\textbf{t}$ which is shared with $\partial\Psi_0$.  As $\Psi_0$ is not pure, it contains at least one $a$-cell.  So, since $\Psi_0$ is itself smooth and satisfies condition (MM2), \Cref{Gamma_a special cell} implies the existence of an $a$-cell $\pi$ in $\Psi_0$ and $\ell\geq|\partial\pi|_\pazocal{A}-6$ consecutive $\pazocal{A}$-edges $\textbf{e}_1,\dots,\textbf{e}_\ell$ of $\partial\pi$ such that the maximal $\pazocal{A}$-bands $\pazocal{U}(\textbf{e}_1),\dots,\pazocal{U}(\textbf{e}_\ell)$ in $\Psi_0$ each correspond to external edges of the graph $\Gamma_a(\Psi_0)$.  In particular, since $\Psi_0$ contains no $(\theta,q)$-cell, each band $\pazocal{U}(\textbf{e}_i)$ ends on $\partial\Psi_0$.

As $\pazocal{A}$-bands cannot cross, each band $\pazocal{U}(\textbf{e}_i)$ must have an end on either $\textbf{s}_0^{-1}$ or on $\textbf{t}_0$.  Since $\tilde{\Psi}_0$ satisfies condition (MM2), though, no three consecutive such bands can end on $\textbf{s}_0^{-1}$.  So, because condition (L1) and \Cref{semi-computation deltas} imply that $|\partial\pi|_\pazocal{A}\geq C$, the parameter choice $C\geq12$ implies the existence of two indices $i_1,i_2\in\{1,\dots,\ell\}$ such that $\pazocal{U}(\textbf{e}_{i_j})$ has an end on $\textbf{t}_0$.

Now, let $\textbf{f}_1$ be the first edge of $\textbf{t}_0$ which is the end of an $\pazocal{A}$-band $\pazocal{U}(\textbf{e}_i)$.  Similarly, let $\textbf{f}_2$ be the last such edge of $\textbf{t}_0$.  Fix the indices $k_1,k_2\in\{1,\dots,\ell\}$ such that $\textbf{f}_i$ is an end of $\pazocal{U}(\textbf{e}_{k_i})$.

Let $\textbf{s}$ be the subpath of $\partial\pi$ with first edge $\textbf{e}_{k_1}$ and last edge $\textbf{e}_{k_2}$.  As distinct $\pazocal{A}$-bands cannot cross, if $\pazocal{U}(\textbf{e}_i)$ has an end on $\textbf{t}_0$, then $\textbf{e}_i$ is an edge of $\textbf{s}$.

Hence, $\textbf{s}$, $\textbf{t}_0$, and the bands $\pazocal{U}(\textbf{e}_{k_i})$ bound a subdiagram $\Psi_1$ that does not contain $\pi$.  Hence, $\Psi_1$ is an $a$-scope on $\textbf{t}$ with associated $a$-cell $\pi$, associated subpath $\textbf{s}$, and size $|\textbf{s}|_\pazocal{A}$.  Note that, by construction, the completion $\tilde{\Psi}_1$ is a subdiagram of $\Psi_0$.  Let $\textbf{s}'$ be the complement of $\textbf{s}$ in $\partial\pi$.  

Suppose there exist five indices $m_1,\dots,m_5\in\{1,\dots,\ell\}$ with $m_i<m_{i+1}$ such that $\textbf{e}_{m_i}$ is an edge of $\textbf{s}'$.  Since $\textbf{s}'$ is a subpath of $\partial\pi$ containing these edges, if it does not contain $\textbf{e}_i$ for all $m_1\leq i\leq m_3$, then it must contain $\textbf{e}_i$ for all $m_3\leq i\leq m_5$.  Either way, $\textbf{s}'$ must contain at least three consecutive $\pazocal{A}$-edges $\textbf{e}_i,\textbf{e}_{i+1},\textbf{e}_{i+2}$.  But then $\pazocal{U}(\textbf{e}_i)$, $\pazocal{U}(\textbf{e}_{i+1})$, and $\pazocal{U}(\textbf{e}_{i+2})$ each has an end on $\textbf{s}_0^{-1}$, producing a contradiction to condition (MM2).  Hence, $|\textbf{s}|_\pazocal{A}\geq\ell-4\geq|\partial\pi|_\pazocal{A}-10$.  Taking $C\geq21$ then implies that $\Psi_1$ is a big $a$-scope.

\end{proof}

\begin{lemma} \label{pure a-scope}

Let $\textbf{t}$ be a subpath of a boundary component of a reduced diagram $\Delta$ over the canonical presentation of $M_\Omega(\textbf{M}^\pazocal{L})$.  Suppose $\Delta$ contains an $a$-scope $\Psi_0$ on $\textbf{t}$ such that the completion $\tilde{\Psi}_0$ is smooth and satisfies condition (MM2).  If $\Psi_0$ is not a pure $a$-scope, then there exists a pure big $a$-scope $\Psi$ on $\textbf{t}$ such that the completion $\tilde{\Psi}$ is a subdiagram $\Psi_0$.

\end{lemma}

\begin{proof}

By \Cref{big a-scope}, there exists a big $a$-scope $\Psi_1$ on $\textbf{t}$ such that the completion $\tilde{\Psi}_1$ is a subdiagram of $\Psi_0$.  Note that this implies that $\text{Area}(\Psi_1)\leq\text{Area}(\Psi_0)-1$.

As a subdiagram of $\Psi_0$, $\tilde{\Psi}_1$ must also be smooth and satisfy condition (MM2).  So, if $\Psi_1$ is not a pure $a$-scope, we may again apply \Cref{big a-scope} to find a big $a$-scope $\Psi_2$ on $\textbf{t}$ such that the completion $\tilde{\Psi}_2$ is a subdiagram of $\Psi_1$.  Again, this implies $\text{Area}(\Psi_2)<\text{Area}(\Psi_1)-1$.

Iterating, this process must terminate with a big $a$-scope $\Psi$ on $\textbf{t}$ which is also pure.

\end{proof}

\medskip


\subsection{Upper bound on weights} \label{sec M-minimal upper bound} \

Let $\pi$ be an $a$-cell in an $M$-minimal diagram $\Delta$ and let $\pazocal{U}$ be a maximal positive $a$-band in $\Delta$ which has an end on $\pi$.  If the other end of $\pazocal{U}$ is on another $a$-cell, then $\pazocal{U}$ is called an \textit{internal} $a$-band in $\Delta$.  Otherwise, $\pazocal{U}$ is called an \textit{external} $a$-band.

Note that if $\pazocal{U}$ is a maximal positive $\pazocal{A}$-band, then Lemmas \ref{a-bands on a-cell} and \ref{M-minimal is smooth} imply that $\pazocal{U}$ is an internal $a$-band ({\frenchspacing i.e. an internal} $\pazocal{A}$-band) if and only if it corresponds to an internal edge of the auxiliary graph $\Gamma_a(\Delta)$.  However, this definition now extends this to include $b$-bands.

For any $M$-minimal diagram $\Delta$, define the values:

\begin{itemize}

\item $\a_i(\Delta)$ is the number of internal $\pazocal{A}$-bands in $\Delta$

\item $\a_e(\Delta)$ is the number of external $\pazocal{A}$-bands in $\Delta$

\item $\b_i(\Delta)$ is the number of internal $b$-bands in $\Delta$

\item $\b_e(\Delta)$ is the number of external $b$-bands in $\Delta$

\end{itemize}

\begin{lemma} \label{internal vs external bands}

For any $M$-minimal diagram $\Delta$:

\begin{enumerate}

\item $\a_i(\Delta)\leq\frac{7}{C}\a_e(\Delta)$

\item $\b_i(\Delta)\leq\frac{49}{C}\b_e(\Delta)$

\end{enumerate}

\end{lemma}

\begin{proof}

We prove both statements simultaneously by induction on the number $n$ of $a$-cells in $\Delta$, with the statement clear if $n=0,1$ as then Lemmas \ref{a-bands on a-cell} and \ref{M-minimal is smooth} imply $\a_i(\Delta)=\b_i(\Delta)=0$.

For the inductive step, as $n\geq2$, there exists an interior vertex $v$ of $\Gamma_a(\Delta)$ satisfying the statement of Lemma \ref{Gamma_a special cell}. 
Let $\pi$ be the $a$-cell of $\Delta$ corresponding to $v$ and let $e_1,\dots,e_\ell$ be the $\ell\geq d(v)-6$ consecutive external edges of $\Gamma_a(\Delta)$ connecting $v$ to $v_0$, enumerated counterclockwise about $\partial\pi$.  For $i\in\{1,\dots,\ell\}$, let $\pazocal{U}_i$ be the maximal positive $\pazocal{A}$-band corresponding to $e_i$.  Further, let $\textbf{e}_i$ be the edge of $\partial\pi_i$ such that $\textbf{e}_i^{\pm1}$ is an end of $\pazocal{U}_i$.

Then, for $i\in\{1,\dots,\ell\}$, let $\pazocal{S}_i$ be:

\begin{itemize}

\item the maximal $\pazocal{A}$-band $\pazocal{U}(\textbf{e}_i)$ if $\pazocal{U}_i$ has an end on $\partial\Delta$, or

\item the subdiagram $\pazocal{V}(\textbf{e}_i)$ if $\pazocal{U}_i$ has an end on a $(\theta,q)$-cell.

\end{itemize}

By construction, $\pazocal{S}_1,\dots,\pazocal{S}_\ell$ and $\pi$ together bound a subdiagram $\Delta_0$ of $\Delta$ (see \Cref{star-a}).  Further, as there are no vertices between $e_1,\dots,e_\ell$, $\pi$ must be the only $a$-cell of $\Delta_0$.

\begin{figure}[H]
\centering
\includegraphics[scale=1]{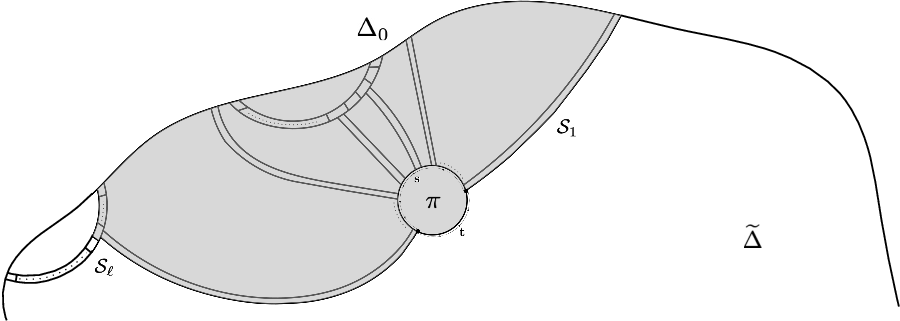}
\caption{The subdiagram $\Delta_0$}
\label{star-a}
\end{figure}

Let $\widetilde{\Delta}$ be the complement of $\Delta_0$ in $\Delta$.  Then $\widetilde{\Delta}$ is an $M$-minimal diagram containing $n-1$ $a$-cells, so that the inductive hypotheses imply $\a_i(\widetilde{\Delta})\leq\frac{7}{C}\a_e(\widetilde{\Delta})$ and $\b_i(\widetilde{\Delta})\leq\frac{49}{C}\b_e(\widetilde{\Delta})$.

Note that any external $a$-band of $\Delta$ that has an end on an $a$-cell other than $\pi$ corresponds to an external $a$-band of $\widetilde{\Delta}$.  Similarly, any internal $a$-band of $\Delta$ that does not have an end on $\pi$ corresponds to an internal $a$-band of $\widetilde{\Delta}$.

Consider the decomposition $\partial\pi=\textbf{s}\textbf{t}$ where $\textbf{s}$ is the minimal subpath containing the $\ell$ consecutive $\pazocal{A}$-edges $\textbf{e}_1,\dots,\textbf{e}_\ell$.  By construction, for any internal $a$-band of $\Delta$ having an end on $\pi$, this end must be an edge of $\textbf{t}^{\pm1}$.  On the other hand, each of these bands corresponds to an external $a$-band in $\widetilde{\Delta}$.  Hence, $\a_i(\Delta)\leq\a_i(\widetilde{\Delta})+|\textbf{t}|_\pazocal{A}$ and $\b_i(\Delta)\leq\b_i(\widetilde{\Delta})+|\textbf{t}|_b$.  

Since the number of $\pazocal{A}$-edges in $\textbf{t}$ is $d(v)-\ell\leq6$, this implies $\a_i(\Delta)\leq\a_i(\widetilde{\Delta})+6$.

Conversely, the $\ell$ consecutive external $\pazocal{A}$-bands with ends $\textbf{e}_1^{\pm1},\dots,\textbf{e}_\ell^{\pm1}$ of $\Delta$ are completely removed in passing to $\widetilde{\Delta}$.  So, $\a_e(\widetilde{\Delta})\leq\a_e(\Delta)-\ell+d(v)-\ell\leq\a_e(\Delta)-d(v)+12$.

Hence, $\a_i(\Delta)\leq\frac{7}{C}\a_e(\widetilde{\Delta})+6\leq\frac{7}{C}\a_e(\Delta)-\frac{7}{C}d(v)+\frac{84}{C}+6$.

So, (1) holds if $7d(v)\geq6C+84$. But by condition (L1) and Lemma \ref{semi-computation deltas}, $d(v)\geq C$ and so the statement follows by the parameter choice $C\geq84$.

Now, let $w\in\Omega$ such that $\lab(\partial\pi)\equiv w^{\pm1}$.  If $|w|_b=0$, then no internal or external $b$-band has an end on $\pi$, so that $\b_i(\Delta)=\b_i(\widetilde{\Delta})$ and $\b_e(\Delta)=\b_e(\widetilde{\Delta})$.  

Otherwise, there exists a word $w'$ freely conjugate to $w$ such that $w'\in\pazocal{E}(\Lambda^\pazocal{A})$.  Letting $t$ be the length of the semi-computation $\pazocal{S}(w')$ which $\Lambda^\pazocal{A}$-accepts $w'$, \Cref{M Lambda semi-computations} implies both that $|\textbf{t}|_b\leq12D_\pazocal{A}(t-1)$ and that $|\textbf{s}|_b\geq\lfloor\frac{\ell-1}{2}\rfloor\cdot \frac{1}{2}D_\pazocal{A}(t-1)\geq\frac{\ell-2}{4}D_\pazocal{A}(t-1)\geq\frac{C-8}{4}D_\pazocal{A}(t-1)$.  

So, $\b_i(\Delta)\leq\b_i(\widetilde{\Delta})+12D_\pazocal{A}(t-1)$.

Further, as with the $\ell$ consecutive external $\pazocal{A}$-bands with ends $\pi$, any maximal positive $b$-band of $\Delta_0$ with one end on $\textbf{s}^{-1}$ is an external $b$-band which is removed in passing to $\widetilde{\Delta}$, and thus 
$$\b_e(\widetilde{\Delta})\leq\b_e(\Delta)-\frac{C-8}{4}D_\pazocal{A}(t-1)+12D_\pazocal{A}(t-1)$$
Hence, $\b_i(\Delta)\leq\frac{49}{C}\b_e(\widetilde{\Delta})+12D_\pazocal{A}(t-1)\leq\frac{49}{C}\b_e(\Delta)+D_\pazocal{A}(t-1)\left(12+\frac{49}{C}(12-\frac{C-8}{4})\right)$.

As above, (2) then holds if $\frac{49}{C}(\frac{C-8}{4}-12)\geq12$.  

But this is equivalent to the parameter choice $C\geq2744$, so that the statement follows.



\end{proof}

\begin{lemma} \label{diskless weight}

If $\Delta$ is an $M$-minimal diagram, then $\text{wt}(\Delta)\leq c_0\|\partial\Delta\|^4+g_\pazocal{L}(c_0\|\partial\Delta\|^3)$.

\end{lemma}

\begin{proof}

First, note that if every cell of $\Delta$ is an $a$-cell, then Lemmas \ref{M-minimal is smooth} and \ref{Gamma_a special cell} imply $\partial\Delta$ contains at least $C-6$ $\pazocal{A}$-edges.  So, a parameter choice for $C$ also implies $\|\partial\Delta\|\geq2$.  Otherwise, $\Delta$ must contain a maximal $\theta$-band, in which case \Cref{M-minimal theta-annuli} implies $\|\partial\Delta\|\geq2$.

%

Now, Lemmas \ref{basic annuli 1}(2) and \Cref{M-minimal theta-annuli} implies there are at most $\frac{1}{2}\|\partial\Delta\|$ maximal positive $q$-bands and at most $\frac{1}{2}\|\partial\Delta\|$ maximal $q$-bands in $\Delta$.  Hence, it follows from \Cref{basic annuli 1}(1) that $\Delta$ contains at most $\frac{1}{4}\|\partial\Delta\|^2$ $(\theta,q)$-cells.

Let $\a$ be the number of maximal positive $\pazocal{A}$-bands in $\Delta$ and set $\a_i=\a_i(\Delta)$ and $\a_e=\a_e(\Delta)$.  Note that a maximal positive $\pazocal{A}$-band need not be internal or external, as such $\pazocal{A}$-bands must have at least one end on an $a$-cell.  Hence, $\a_i+\a_e\leq\a$.

Similarly, letting $\b$ be the number of maximal positive $b$-bands in $\Delta$ and setting $\b_i=\b_i(\Delta)$ and $\b_e=\b_e(\Delta)$, we have $\b_i+\b_e\leq\b$.

By the makeup of the relations, the boundary of any $(\theta,q)$-cell can have at most one $\pazocal{A}$-edge.  Hence, since \Cref{basic annuli 2}(2) implies that any maximal positive $\pazocal{A}$-band that is not internal must have at least one end on $\partial\Delta$ or on a $(\theta,q)$-cell, $\a-\a_i\leq\|\partial\Delta\|+\frac{1}{4}\|\partial\Delta\|^2\leq\frac{3}{4}\|\partial\Delta\|^2$.

Further, Lemma \ref{internal vs external bands}(1) implies that $\a_i\leq\frac{7}{C}\a_e\leq\frac{7}{C}(\a-\a_i)$, so that the parameter choice $C>21$ implies $\a_i\leq\frac{1}{3}(\a-\a_i)$.  So, $\a=\a_i+(\a-\a_i)\leq\frac{4}{3}(\a-\a_i)\leq\|\partial\Delta\|^2$.

Hence, \Cref{basic annuli 2}(1) implies that the number of $(\theta,\pazocal{A})$-cells in $\Delta$ is at most $\frac{1}{2}\|\partial\Delta\|^3$.

Next, note that the boundary of any $(\theta,q)$- or $(\theta,\pazocal{A})$-cell can have at most $D_\pazocal{A}$ $b$-edges.  So, as above, since any maximal positive $b$-band that is not internal must have at least one end on $\partial\Delta$, on a $(\theta,q)$-cell, or on a $(\theta,\pazocal{A})$-cell, $\b-\b_i\leq\|\partial\Delta\|+D_\pazocal{A}(\frac{1}{4}\|\partial\Delta\|^2+\frac{1}{2}\|\partial\Delta\|^3)\leq\|\partial\Delta\|+\frac{3}{4}D_\pazocal{A}\|\partial\Delta\|^3$.  

So, recalling that the value of $D_\pazocal{A}$ is dependent on $C$, a parameter choice for $C$ then yields $\b-\b_i\leq\frac{7}{8}D_\pazocal{A}\|\partial\Delta\|^3$.  Lemma \ref{internal vs external bands}(2) then implies $\b_i\leq\frac{49}{C}\b_e\leq\frac{49}{C}(\b-\b_i)$, so that the parameter choice $C\geq343$ yields $\b_i\leq\frac{1}{7}(\b-\b_i)$.

Hence, as above, $\b=\b_i+(\b-\b_i)\leq\frac{8}{7}(\b-\b_i)\leq D_\pazocal{A}\|\partial\Delta\|^3$, and so the number of $(\theta,b)$-cells in $\Delta$ is at most $\frac{1}{2}D_\pazocal{A}\|\partial\Delta\|^4$.

Finally, note that the boundary of any $(\theta,q)$-cell contains at most one ordinary $a$-edge.  So, since any maximal ordinary $a$-band has two ends which are on $\partial\Delta$ or on a $(\theta,q)$-cell, the number of maximal positive ordinary $a$-bands in $\Delta$ is at most $\frac{1}{2}(\|\partial\Delta\|+\frac{1}{2}\|\partial\Delta\|^2)\leq\frac{1}{2}\|\partial\Delta\|^2$.  Hence, the number of ordinary $(\theta,a)$-cells in $\Delta$ is at most $\frac{1}{4}\|\partial\Delta\|^3$.

Thus, letting $\pi_1,\dots,\pi_n$ be the $a$-cells of $\Delta$, a parameter choice for $C$ implies:
\begin{align*}
\text{wt}(\Delta)&\leq\sum_{i=1}^n\text{wt}(\pi_i)+\frac{1}{4}\|\partial\Delta\|^2+\frac{1}{2}\|\partial\Delta\|^3+\frac{1}{2}D_\pazocal{A}\|\partial\Delta\|^4+\frac{1}{4}\|\partial\Delta\|^3\leq\sum_{i=1}^n g_\pazocal{L}(\|\partial\pi_i\|)+D_\pazocal{A}\|\partial\Delta\|^4
\end{align*}
Since $g_\pazocal{L}$ is super-additive, $\sum\limits_{i=1}^n g_\pazocal{L}(\|\partial\pi_i\|)\leq g_\pazocal{L}\left(\sum\limits_{i=1}^n\|\partial\pi_i\|\right)$.  But since all of the edges on the boundary of an $a$-cell are $\pazocal{A}$- or $b$-edges, 
$$\sum\limits_{i=1}^n\|\partial\pi_i\|\leq2(\a+\b)\leq2\|\partial\Delta\|^2+2D_\pazocal{A}\|\partial\Delta\|^3\leq3D_\pazocal{A}\|\partial\Delta\|^3$$
Thus, the statement follows from the parameter choice $c_0>>C$.

\end{proof}

\medskip


\section{Diagrams with disks}

\subsection{Minimal diagrams} \

Analogous to the approach to diagrams over $M_\Omega(\textbf{M}^\pazocal{L})$ in \Cref{sec-no-disks}, the objective of this section is to study diagrams over the disk presentation of $G_\Omega(\textbf{M}^\pazocal{L})$ for the purpose of finding an upper bound of the weight of a reduced circular diagram in terms of its perimeter.  Again, this goal is not achieved for any possible reduced circular diagram, but rather for a `generic' class.

Recall that the standard base of $\textbf{M}^\pazocal{L}$ is $\left(\{t(1)\}B_4^\pazocal{L}(1)\right)\left(\{t(2)\}B_4^\pazocal{L}(2)\right)\dots\left(\{t(L)\}B_4^\pazocal{L}(L)\right)$ where:
$$B_4^\pazocal{L}(i)=Q_0^\pazocal{L}(i)Q_1^\pazocal{L}(i)\dots Q_N^\pazocal{L}(i)(R_N^\pazocal{L}(i))^{-1}\dots(R_1^\pazocal{L}(i))^{-1}(R_0^\pazocal{L}(i))^{-1}$$
for each $i=1,\dots,L$.  Letting $\pazocal{X}$ be the generators of the groups associated to $\textbf{M}^\pazocal{L}$ (see \Cref{sec-the-groups}), a $q$-letter of a word over $\pazocal{X}\cup\pazocal{X}^{-1}$ of the form $t(i)^{\pm1}$ for $2\leq i\leq L$ is called a \textit{$t$-letter}. Accordingly, a $q$-edge labelled by a $t$-letter is called a \textit{$t$-edge}, a $(\theta,q)$-relation corresponding to a $t$-letter is called a \textit{$(\theta,t)$-relation}, and a $q$-band corresponding to a part $\{t(i)\}$ for $i\geq2$ is called a \textit{$t$-band}. 

Note that for each positive rule $\theta$ and each $t$-letter, the corresponding $(\theta,t)$-relation is simply $\theta_jt(i)=t(i)\theta_{j+1}$.  Hence, each side of a $t$-band is labelled by a copy of the band's history.

Now, as in \cite{W}, we introduce a `grading' (see Section 13 of \cite{O} for the general definition of graded presentations) on the disk presentation of $G_\Omega(\textbf{M}^\pazocal{L})$ as follows:

For any diagram $\Delta$ over the disk presentation of $G_\Omega(\textbf{M}^\pazocal{L})$, define the values:
\begin{itemize}

\item $\sigma_1(\Delta)$ is the number of disks in $\Delta$

\item $\sigma_2(\Delta)$ is the number of $(\theta,t)$-cells in $\Delta$


\item $\sigma_3(\Delta)$ is the number of $a$-cells in $\Delta$

\item $\sigma_4(\Delta)$ is the number of $(\theta,\pazocal{A})$-cells in $\Delta$


\end{itemize}

The \textit{signature} of $\Delta$ is taken to be the quadruple $\tau(\Delta)=(\sigma_1(\Delta),\dots,\sigma_4(\Delta))$.  For $j\in\{1,2,3\}$, we also define the \textit{$j$-signature} of $\Delta$ to be the $j$-tuple $\tau_j(\Delta)=(\sigma_1(\Delta),\dots,\sigma_j(\Delta))$.  Signatures and $j$-signatures of diagrams over the disk presentation of $G_\Omega(\textbf{M}^\pazocal{L})$ are ordered lexicographically.  
%
%
%
%


A circular diagram $\Delta$ over the disk presentation of $G_\Omega(\textbf{M}^\pazocal{L})$ is \textit{minimal} if for any circular diagram $\Gamma$ over this presentation satisfying $\lab(\partial\Gamma)\equiv\lab(\partial\Delta)$, then $\tau(\Delta)\leq\tau(\Gamma)$.  Analogously, a circular diagram $\Delta$ over the disk presentation of $G_\Omega(\textbf{M}^\pazocal{L})$ is \textit{$j$-minimal} if it has the smallest possible $j$-signature amongst all circular diagrams with the same contour label.  Observe that minimal diagrams are necessarily $j$-minimal for any $j$, while $j$-minimal diagrams are necessarily $(j-1)$-minimal for appropriate $j$.

Note that for a minimal diagram $\Delta$ and a circular diagram $\Gamma$ over the disk presentation of $G_\Omega(\textbf{M}^\pazocal{L})$ satisfying $\lab(\partial\Delta)\equiv\lab(\partial\Gamma)$, it is not necessarily the case that $\text{wt}(\Delta)\leq\text{wt}(\Gamma)$.  In particular, in the sequel we define operations that add many cells of `low rank' in order to remove one or two cells of `high rank'; such an operation reduces the type of the diagram but can (and usually does) increase the weight.  However, despite this, the definition of minimal diagram provides a convenient setting for studying the structure of the group $G_\Omega(\textbf{M}^\pazocal{L})$.

Further, observe that the removal of cancellable cells (see \Cref{cancellable}) in a diagram over the disk presentation of $G_\Omega(\textbf{M}^\pazocal{L})$ can only decrease the ($j$)-signature of the diagram.  Hence, for any ($j$-)minimal diagram, there exists a reduced ($j$-)minimal diagram with the same contour label obtained by simply removing any pairs of cancellable cells.

Suppose $W$ is a word over $\pazocal{X}\cup\pazocal{X}^{-1}$ which represents the trivial element in $G_\Omega(\textbf{M}^\pazocal{L})$.  It follows from van Kampen's Lemma (see \Cref{sec-Diagrams}) that there exists a circular diagram $\Gamma$ over the disk presentation of $G_\Omega(\textbf{M}^\pazocal{L})$ such that $\lab(\partial\Gamma)\equiv W$.  As the lexicographic ordering on tuples of natural numbers is a well-ordering, without loss of generality $\tau(\Gamma)$ (or $\tau_j(\Gamma)$) is minimal amongst all such diagrams.  Hence, the next statement follows immediately, establishing the sense in which minimal diagrams are `generic':

\begin{lemma} \label{generic minimal diagram}

Let $W$ a word over $\pazocal{X}\cup\pazocal{X}^{-1}$ which represents the trivial element in $G_\Omega(\textbf{M}^\pazocal{L})$.  Then there exists a reduced ($j$-)minimal diagram $\Delta$ satisfying $\lab(\partial\Delta)\equiv W$.

\end{lemma}

\medskip


\subsection{Removal surgeries} \

In this section, we define two types of surgery on reduced diagrams over the disk presentation of $G_\Omega(\textbf{M}^\pazocal{L})$ which reduce the type of the diagram.  These operations help describe the makeup of a minimal diagram, allowing for the estimates that follow.

\subsubsection{Removing $a$-cells} \

Our first operation uses the definition of $\Lambda^\pazocal{A}$ to study the $\pazocal{A}$-bands of reduced diagrams over $M_\Omega(\textbf{M}^\pazocal{L})$ which have ends on $a$-cells, demonstrating the condition (MM2) in minimal diagrams.

\begin{lemma} \label{cancellable a-cells}

Suppose the circular diagram $\Delta$ over the disk presentation of $G_\Omega(\textbf{M}^\pazocal{L})$ contains $a$-cells $\pi_1$ and $\pi_2$ such that:

\begin{itemize}

\item $\lab(\partial\pi_1),\lab(\partial\pi_2)\in\Lambda^\pazocal{A}$

\item There exists a simple path $\textbf{t}$ in $\Delta$ between vertices of $\partial\pi_1$ and $\partial\pi_2$ such that $\lab(\textbf{t})$ is freely trivial

\end{itemize} 
Then $\Delta$ is not $3$-minimal.

\end{lemma}

\begin{proof}

Let $O_1$ and $O_2$ be the vertices of $\partial\pi_1$ and $\partial\pi_2$ such that the initial and terminal points of $\textbf{t}$ are $O_1$ and $O_2$ (see \Cref{cancellable}).  Then let $w_i\in F(\pazocal{A})$ be $\lab(\partial\pi_i)$ read starting at $O_i$.

The process of $0$-refinement then produces a diagram $\Delta_0$ satisfying $\lab(\partial\Delta_0)\equiv\lab(\partial\Delta)$ and $\tau(\Delta_0)=\tau(\Delta)$ such that there exists a subdiagram $\Gamma$ of $\Delta_0$ with $\tau_3(\Gamma)=(0,0,2)$ and $$\lab(\partial\Gamma)\equiv\lab(\partial\pi_1)\lab(\textbf{t})\lab(\partial\pi_2)\lab(\textbf{t})^{-1}=_{F(\pazocal{X})}w_1w_2$$
By condition (L3), $w_1,w_2\in\Lambda^\pazocal{A}$.  So, by condition (L4), $w_1w_2$ is either freely trivial or freely equal to an element of $\Lambda^\pazocal{A}$.  

Hence, there exists a (reduced) circular diagram $\Gamma'$ over the disk presentation of $G_\Omega(\textbf{M}^\pazocal{L})$ with $\lab(\partial\Gamma')\equiv\lab(\partial\Gamma)$ such that $\tau_3(\Gamma')\leq(0,0,1)<\tau_3(\Gamma)$.

But then excising $\Gamma$ from $\Delta_0$ and replacing it with $\Gamma'$ produces a circular diagram $\Delta'$ over the disk presentation of $G_\Omega(\textbf{M}^\pazocal{L})$ with $\lab(\partial\Delta')\equiv\lab(\partial\Delta)$ such that $\tau_3(\Delta')<\tau_3(\Delta)$.

\end{proof}

\begin{lemma} \label{pure a-cells}

For any $w\in\Omega$, there exists a reduced circular diagram $\Gamma_w$ over $M_\Omega(\textbf{M}^\pazocal{L})$ satisfying:

\begin{itemize}

\item $\lab(\partial\Gamma_w)\equiv w$

\item $\tau_3(\Gamma_w)=(0,0,1)$

\item Letting $\pi$ be the unique $a$-cell of $\Gamma_w$, $\lab(\partial\pi)\in\Lambda^\pazocal{A}$

\end{itemize}

\end{lemma}

\begin{proof}

By the definition of $\Omega$, there exists a word $w'\in\pazocal{E}(\Lambda^\pazocal{A})$ which is freely conjugate to $w$.  

Then, \Cref{M Lambda semi-computations} produces a (unique) semi-computation of $\textbf{M}^\pazocal{L}$ in the `special' input sector $\pazocal{S}(w'):w'\equiv w_0\to\dots\to w_t$ which $\Lambda^\pazocal{A}$-accepts $w'$.

By \Cref{semi-computations are semi-trapezia}, there then exists a semi-trapezium $\Delta_w$ over $M(\textbf{M}^\pazocal{L})$ in the `special' input sector such that $\lab(\textbf{bot}(\Delta_w))\equiv w'$ and $\lab(\textbf{top}(\Delta_w))\equiv w_t$.  

By definition, $\tau_3(\Delta_w)=(0,0,0)$ and the sides of $\Delta_w$ are labelled by identical copies of the history of $\pazocal{S}(w')$.  So, pasting the sides of $\Delta_w$ together produces an annular diagram $\Delta_w'$ over $M(\textbf{M}^\pazocal{L})$ with outer contour label $w'$, inner contour label $w_t^{-1}$, and $3$-signature $\tau_3(\Delta_w')=(0,0,0)$.

As $\pazocal{S}(w')$ is a $\Lambda^\pazocal{A}$-accepting computation, necessarily $w_t\in\Lambda^\pazocal{A}$, and so $w_t^{-1}\in\Lambda^\pazocal{A}$ by \Cref{Omega inverses}.  Hence, letting $\pi$ be an $a$-cell with $\lab(\partial\pi)\equiv w_t^{-1}$, $\pi$ can be pasted into the center of $\Delta_w'$ to produce a circular diagram $\Gamma_w'$ with $\lab(\partial\Gamma_w')\equiv w'$ and $\tau_3(\Gamma_w')=(0,0,1)$.

Thus, since $w'$ is freely conjugate to $w$, applying $0$-refinement (or gluing) and cancellation to $\Gamma_w'$ produces a diagram $\Gamma_w$ satisfying the statement.

%

\end{proof}

%
%
%
%
%
%
%
%
%
%
%
%

\begin{lemma} \label{minimal MM2}

Every $3$-minimal diagram $\Delta$ satisfies condition (MM2).

\end{lemma}

\begin{proof}


Suppose $\Delta$ does not satisfy (MM2).  So, there exist $a$-cells $\pi_1$ and $\pi_2$ and a subdiagram $\Psi$ contradicting the condition (see \Cref{MM2}).  Replacing each of $\pi_1$ and $\pi_2$ with the subdiagram arising in \Cref{pure a-cells} (and removing any cancellable cells) then produces a similar counterexample whose $a$-cells are labelled by elements of $\Lambda^\pazocal{A}$.  Hence, we may assume without loss of generality that $\lab(\partial\pi_i)\in\Lambda^\pazocal{A}$.

Let $\textbf{e}_1,\textbf{e}_2,\textbf{e}_3$ be the corresponding consecutive $\pazocal{A}$-edges of $\partial\pi_1$.  So, there exist edges $\textbf{f}_1,\textbf{f}_2,\textbf{f}_3$ of $\partial\pi_2$ such that $\textbf{f}_j^{-1}$ is an end of $\pazocal{U}(\textbf{e}_j)$.  Let $\textbf{s}_1$ be the subpath of $(\partial\pi_1)^{-1}$ with initial edge $\textbf{e}_3^{-1}$ and final edge $\textbf{e}_1^{-1}$.  Similarly, let $\textbf{s}_2$ be the subpath of $\partial\pi_2$ with initial edge $\textbf{f}_3$ and final edge $\textbf{f}_1$.  Further, let $\textbf{t}_1=\textbf{top}(\pazocal{U}(\textbf{e}_3))$ and $\textbf{t}_2=\textbf{bot}(\pazocal{U}(\textbf{e}_1))$.

Then, since by hypothesis $\Psi$ contains no $a$-cells, $\Psi$ is a compressed semi-trapezium over $M(\textbf{M}^\pazocal{L})$ in the `special' input sector with standard factorization $\textbf{t}_1^{-1}\textbf{s}_1\textbf{t}_2\textbf{s}_2^{-1}$.  Letting $H$ be the history of $\Psi$, note that $\lab(\textbf{t}_j)$ is a copy of $H$.

By \Cref{compressed semi-trapezia are compressed semi-computations}, there then exists a reduced compressed semi-computation $\pazocal{S}_\mathscr{C}:w_0\to\dots\to w_t$ of $\textbf{M}^\pazocal{L}$ in the `special' input sector with history $H$ such that $w_0\equiv\lab(\textbf{s}_1)$ and $w_t\equiv\lab(\textbf{s}_2)$.  

As $\lab(\partial\pi_1)\in\Lambda^\pazocal{A}$, there exist $y_i\in\pazocal{A}$ and $\delta_i\in\{\pm1\}$ such that $w_0\equiv y_1^{\delta_1}y_2^{\delta_2}y_3^{\delta_3}$.  Further, since $\lab(\partial\pi_2)\in\Lambda^\pazocal{A}$, \Cref{semi-computation deltas} implies there exist $z_i\in\pazocal{A}$ and $\eps_i\in\{\pm1\}$ such that $w_t\equiv z_1^{\eps_1}z_2^{\eps_2}z_3^{\eps_3}$.  In particular, $\pazocal{S}_\mathscr{C}$ satisfies the hypotheses of \Cref{M compressed semi-computation three A}, so that $H$ must be freely trivial.

But then $\lab(\textbf{t}_j)$ is also freely trivial, meaning \Cref{cancellable a-cells} yields a contradiction.

\end{proof}

\subsubsection{Removing disks} \label{sec-removing-disks} \

The next operation is used to study $t$-bands in minimal diagrams which have ends on two disks, yielding a condition for disks similar to (MM2).  This treatment is analogous to the that in \cite{W}.

First, we construct a diagram to simulate the `almost-extendability' of $\textbf{M}^\pazocal{L}$ (see \Cref{sec-almost-extendable}):

\begin{lemma} \label{almost-extendable a-trapezia hubs}

Let $j\in\{2,\dots,L\}$ and suppose $\pazocal{C}:W_{ac}(j)\to\dots\to W_{ac}(j)$ is a reduced computation of $\textbf{M}^\pazocal{L}$ with history $H$.  Then there exists a reduced circular diagram $\Delta$ over $M_\Omega(\textbf{M}^\pazocal{L})$ with $\partial\Delta=\textbf{t}_1^{-1}\textbf{s}_1\textbf{t}_2\textbf{s}_2^{-1}$ such that:

\begin{itemize}

\item $\lab(\textbf{s}_1)\equiv\lab(\textbf{s}_2)\equiv W_{ac}$

\item $\textbf{t}_1$ and $\textbf{t}_2$ are sides of maximal negative $q$-bands whose labels are identical copies of $H$

\item The history of every maximal negative $q$-band in $\Delta$ is $H$

\end{itemize}

\end{lemma}

\begin{proof}

Let $H\equiv H_1\dots H_k$ be the factorization of $H$ such that for each $i=1,\dots,k$, $H_i$ is the history of a maximal one-machine subcomputation $\pazocal{C}_i$ of $\pazocal{C}$.  By \Cref{extend one-machine}, there then exists a one-machine computation $\pazocal{D}_i:U_i\to\dots\to V_i$ of $\textbf{M}^\pazocal{L}$ in the standard base with history $H_i$ extending $\pazocal{C}_i$.  So, \Cref{computations are trapezia} provides a trapezium $\Delta_i$ with $\lab(\textbf{tbot}(\Delta_i))\equiv U_i$, $\lab(\textbf{ttop}(\Delta_i))\equiv V_i$, and history $H_i$.  Note that by the definition of trapezia and \Cref{basic annuli 1}, every maximal negative $q$-band of $\Delta_i$ has history $H_i$.  As no trimming is necessary for computations in the standard base, the sides of $\Delta_i$ are hence labelled by identical copies of $H_i$.

Now, \Cref{projected end to end} implies that $V_i$ and $U_{i+1}$ differ by the insertion/deletion of an element of $\pazocal{L}$ in the `special' input sector.  But conditions (L1) and (L5) imply $\pazocal{L}\subseteq\Lambda^\pazocal{A}\subseteq\Omega$, so that the top of $\Delta_i$ and the bottom of $\Delta_{i+1}$ can be glued along a single $a$-cell to produce a reduced circular diagram $\Delta$ over $M_\Omega(\textbf{M}^\pazocal{L})$.  Note that this procedure glues all maximal negative $q$-bands together, so that such a band is the concatenation of the corresponding bands in $\Delta_1,\dots,\Delta_k$.  Thus, the statement is satisfied by letting $\textbf{s}_1=\textbf{tbot}(\Delta_1)$ and $\textbf{s}_2=\textbf{ttop}(\Delta_k)$.

\end{proof}

\begin{lemma} \label{almost-extendable a-trapezia}

Let $W_1$ and $W_2$ be accepted configurations of $\textbf{M}^\pazocal{L}$ with $\ell(W_1),\ell(W_2)\leq1$.  Suppose $\pazocal{C}:W_1(j)\to\dots\to W_2(j)$ is a reduced computation of $\textbf{M}^\pazocal{L}$ with history $H$ for some $j\in\{2,\dots,L\}$.  Then there exists a reduced circular diagram $\Delta$ over $M_\Omega(\textbf{M}^\pazocal{L})$ with $\partial\Delta=\textbf{t}_1^{-1}\textbf{s}_1\textbf{t}_2\textbf{s}_2^{-1}$ such that:

\begin{itemize}

\item $\lab(\textbf{s}_i)\equiv W_i$ for $i=1,2$

\item $\textbf{t}_1$ and $\textbf{t}_2$ are sides of maximal negative $q$-bands whose labels are identical copies of $H$

\item The history of every maximal negative $q$-band in $\Delta$ is $H$

\end{itemize}

\end{lemma}

\begin{proof}

Let $\pazocal{C}_i$ be a reduced computation of $\textbf{M}^\pazocal{L}$ accepting $W_i$ with $\ell(\pazocal{C}_i)=\ell(W_i)$.  Let $H_i$ be the history $\pazocal{C}_i$ and let $\bar{\pazocal{C}}_i$ be the inverse computation of $\pazocal{C}_i$.

The restriction of $\pazocal{C}_i$ to the base $\{t(j)\}B_4^\pazocal{L}(j)$ is then a reduced computation with history $H_i$ of the form $W_i(j)\to\dots\to W_{ac}(j)$, so that $H_1^{-1}HH_2$ is freely equal to the history of a reduced computation $\pazocal{C}':W_{ac}(j)\to\dots\to W_{ac}(j)$.  Let $\Delta'$ be the diagram corresponding to $\pazocal{C}'$ given by \Cref{almost-extendable a-trapezia hubs} and set $\partial\Delta'=(\textbf{t}_1')^{-1}(\textbf{s}_1')(\textbf{t}_2')(\textbf{s}_2')^{-1}$ as in that setting.

Let $\Delta_1$ and  $\bar{\Delta}_2$ be the trapezia corresponding to $\pazocal{C}_1$ and $\bar{\pazocal{C}}_2$, respectively, as given by \Cref{computations are trapezia}.  Noting that $\lab(\textbf{ttop}(\Delta_1))\equiv W_{ac}\equiv\lab(\textbf{s}_1')$ and $\lab(\textbf{tbot}(\bar{\Delta}_2))\equiv W_{ac}\equiv\lab(\textbf{s}_2')$, we construct a reduced diagram $\Delta$ by pasting together $\Delta_1$, $\Delta'$, and $\bar{\Delta}_2$ and making any necessary cancellations.

As in the proof of \Cref{almost-extendable a-trapezia hubs}, all maximal negative $q$-bands of $\Delta$ arise as the concatenation of such a band in $\Delta_1$, $\Delta'$, and $\bar{\Delta}_2$ (and making any necessary cancellations).  By construction, the history of this band is then freely equal to $H_1(H_1^{-1}HH_2)H_2^{-1}$, and so is $H$.  Thus, the statement follows by letting $\textbf{s}_1=\textbf{tbot}(\Delta_1)$ and $\textbf{s}_2=\textbf{ttop}(\bar{\Delta}_2)$.

\end{proof}

Let $\Pi$ be a disk in a reduced circular diagram $\Delta$ over the disk presentation of $G_\Omega(\textbf{M}^\pazocal{L})$.  A maximal $t$-band which has an end on $\Pi$ is called a \textit{$t$-spoke} of $\Pi$.  Given a $t$-edge $\textbf{e}$ of $\partial\Pi$, the $t$-spoke of $\Pi$ for which $\textbf{e}$ is a defining edge is denoted $\pazocal{Q}(\textbf{e})$.

With \Cref{almost-extendable a-trapezia}, we now arrive at the following analogue of \Cref{minimal MM2}, providing a version of condition (MM2) for $t$-bands connecting disks.  

\begin{lemma}[Compare with Lemma 9.5 in \cite{W}] \label{disk MM2}

Let $\Pi_1$ and $\Pi_2$ be two disks of a reduced $1$-minimal diagram $\Delta$.  Suppose there exist consecutive $t$-edges $\textbf{e}_1$ and $\textbf{e}_2$ of $\partial\Pi_1$ such that both $\pazocal{Q}(\textbf{e}_1)$ and $\pazocal{Q}(\textbf{e}_2)$ have ends on $\Pi_2$.  Let $\Psi$ be the subdiagram of $\Delta$ bounded by $\pazocal{Q}(\textbf{e}_i)$ and subpaths of $\partial\Pi_i$ such that neither $\Pi_1$ nor $\Pi_2$ is contained in $\Psi$ (see \Cref{fig-t-spoke}).  Then $\Psi$ contains a disk.

\end{lemma}

\begin{proof}

%
%
%
%

Let $\pazocal{Q}_i=\pazocal{Q}(\textbf{e}_i)$ and suppose to the contrary that $\Psi$ contains no disk.

First, suppose the adjacent $t$-letters corresponding to $\pazocal{Q}_1$ and $\pazocal{Q}_2$ are $\{t(j),t(j+1)\}$ for some $2\leq j\leq L-1$.  Then the subpath of $\partial\Psi$ shared with $(\partial\Pi_i)^{-1}$ is an admissible word whose base is either $\{t(j)\}B_4^\pazocal{L}(j)$ or its inverse.  So, since the sides of $t$-bands consist entirely of $\theta$-edges, no edge of $\partial\Psi$ is an $\pazocal{A}$-edge labelled by a letter from the tape alphabet of the `special' input sector.  Hence, \Cref{minimal MM2}, \Cref{Gamma_a special cell}, and the parameter choice $C\geq7$ imply that $\Psi$ contains no $a$-cells.

It then follows that $\Psi$ is a trapezium satisfying the hypotheses of \Cref{almost-extendable a-trapezia}.  But then the proof of the statement is completed by an identical argument to that presented in the analogous statement in \cite{W}.

Conversely, suppose the adjacent $t$-letters corresponding to $\pazocal{Q}_1$ and $\pazocal{Q}_2$ are $\{t(L),t(2)\}$.  Then there exists a maximal $t$-band $\pazocal{Q}_3$ in $\Psi$ corresponding to $\{t(1)\}$ which ends on both $\Pi_1$ and $\Pi_2$.  But then the above argument may be applied to the subdiagram $\Psi_0$ bounded by $\pazocal{Q}_3$ and the $\pazocal{Q}_i$ corresponding to $\{t(L)\}$ (see \Cref{fig-t-spoke}(b)), again implying the statement.

\renewcommand\thesubfigure{\alph{subfigure}}
\begin{figure}[H]
\centering
\begin{subfigure}[b]{\textwidth}
\centering
\includegraphics[scale=0.9]{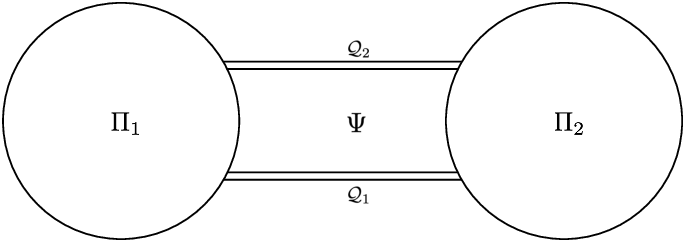}
\caption{Adjacent $t$-letters are $\{t(j),t(j+1)\}$ for $2\leq j\leq L-1$}
\end{subfigure} \\ \vspace{0.2in}
\begin{subfigure}[b]{\textwidth}
\centering
\includegraphics[scale=0.9]{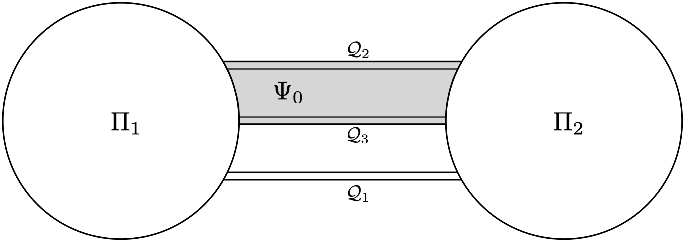}
\caption{Adjacent $t$-letters are $\{t(L),t(2)\}$           }
\end{subfigure}
\caption{Lemma \ref{disk MM2}}
\label{fig-t-spoke}
\end{figure}

\end{proof}

With \Cref{disk MM2}, we now adapt the methods of \Cref{sec-M-minimal-bands} to this context, defining an auxiliary graph to a reduced circular diagram $\Delta$ over the disk presentation of $G_\Omega(\textbf{M}^\pazocal{L})$ which is an estimating graph constructed from the disks of the diagram.  Note that this treatment is analogous to that of  \cite{O18}, \cite{OS19}, \cite{W}, and others.

To any reduced circular diagram $\Delta$ over the disk presentation $G_\Omega(\textbf{M}^\pazocal{L})$, construct the (unoriented) graph $\Gamma(\Delta)$ as follows:

\begin{enumerate}

\item The set of vertices is $\{v_0,v_1,\dots,v_\ell\}$, where each $v_i$ for $i\geq1$ corresponds to one of the $\ell$ disks of $\Delta$ and $v_0$ is a single exterior vertex.  

\item For $i,j\geq1$ and for any positive $t$-band which has ends on the disks corresponding to $v_i$ and $v_j$, there is a corresponding edge $(v_i,v_j)$. Such an edge is called \textit{internal}.

\item For $i\geq1$ and any positive $t$-band with one end on the disk corresponding to $v_i$ and the other end on $\partial\Delta$, there is a corresponding edge $(v_0,v_i)$. Such an edge is called \textit{external}.

\end{enumerate}

Analogous to the construction outlined in \Cref{Gamma_a planar}, $\Gamma(\Delta)$ can be constructed by placing interior vertices in the interior of the corresponding disk and constructing arcs running through the corresponding $t$-bands.  Hence, similar to that setting, $\Gamma(\Delta)$ can be assumed to be a planar graph (note this observation is easier to see in this setting given the simpler makeup of external edges).

By definition the label of the positive $q$-edges on the boundary of a disk is a representative of a different part of the state letters of $\textbf{M}^\pazocal{L}$.  Accordingly, a $q$-band can have at most one end on any particular disk.  In particular, any maximal positive $t$-band with an end on a disk corresponds to an edge of $\Gamma(\Delta)$.  Hence, $\Gamma(\Delta)$ contains no $1$-gons and the degree of any interior vertex is $L-1$.

Moreover, \Cref{disk MM2} implies that if $\Delta$ is a reduced $1$-minimal diagram, then no two internal edges of $\Gamma(\Delta)$ bound a 2-gon.

Thus, the next statement is a given by taking $L\geq7$, following in just the same way as \Cref{Gamma_a' special cell}, yielding a conclusion analogous to \Cref{Gamma_a special cell}:

\begin{lemma}[Lemma 3.2 of \cite{O97}] \label{Gamma special cell}

If $\Delta$ is a reduced $1$-minimal diagram containing at least one disk, then $\Delta$ contains a disk $\Pi$ such that $L-4$ consecutive $t$-spokes $\pazocal{Q}_1,\dots,\pazocal{Q}_{L-4}$ of $\Pi$ have ends on $\partial\Delta$ and such that every subdiagram $\Gamma_i$ bounded by $\pazocal{Q}_i$, $\pazocal{Q}_{i+1}$, $\partial\Pi$, and $\partial\Delta$ ($i=1,\dots,L-5$) contains no disks.

\end{lemma}

\begin{figure}[H]
\centering
\includegraphics[scale=0.95]{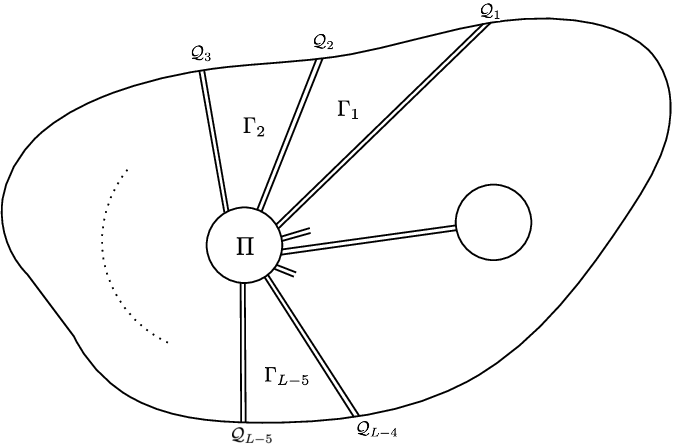}
\caption{Lemma \ref{Gamma special cell}}
\end{figure}


\medskip

\subsection{scopes} \label{sec-scopes} \

As in \Cref{sec-a-scopes}, we now take a brief interlude to investigate a consequence of \Cref{Gamma special cell} that will be useful for future arguments.

Let $\Pi$ be a disk and $\textbf{t}$ be a subpath of a boundary component of a reduced diagram $\Delta$ over the disk presentation of $G_\Omega(\textbf{M}^\pazocal{L})$.  Let $\textbf{e}_1$ and $\textbf{e}_2$ be $t$-edges of $\partial\Pi$ such that the $t$-bands $\pazocal{Q}(\textbf{e}_i)$ has an end on $\textbf{t}$.  Suppose there exists a subpath $\textbf{s}$ of $\partial\Pi$ such that $\textbf{s}$, a subpath of $\textbf{t}$, and the bands $\pazocal{Q}(\textbf{e}_1),\pazocal{Q}(\textbf{e}_2)$ bound a (circular) subdiagram $\Psi$ of $\Delta$ which does not contain $\Pi$.

Then $\Psi$ is called a \textit{scope} on $\textbf{t}$ with \textit{associated disk} $\Pi$, \textit{associated subpath} $\textbf{s}$, and \textit{size} $|\textbf{s}|_t$.

Analogous to the terminology of $a$-scopes, $\Psi$ is called a \textit{pure scope} if it contains no disk.  Further, the \textit{completion} of $\Psi$ is the subdiagram $\tilde{\Psi}$ consisting of both $\Psi$ and $\Pi$.

\begin{lemma} \label{big scope}

Let $\textbf{t}$ be a subpath of a boundary component of a reduced diagram $\Delta$ over the disk presentation of $G_\Omega(\textbf{M}^\pazocal{L})$.  Suppose $\Delta$ contains a scope $\Psi_0$ on $\textbf{t}$ such that the completion $\tilde{\Psi}_0$ is $1$-minimal.  If $\Psi_0$ is not a pure scope, then there exists a scope $\Psi_1$ on $\textbf{t}$ of size $\ell\geq L-6$ such that the completion $\tilde{\Psi}_1$ is a subdiagram of $\Psi_0$.

\end{lemma}

\begin{proof}

The proof follows much the same outline as that of \Cref{big a-scope}, using \Cref{disk MM2} in place of condition (MM2) and \Cref{Gamma special cell} in place of \Cref{Gamma_a special cell}.

\end{proof}

Similarly, the following statement is proved in much the same way as \Cref{pure a-scope}, using iterated applications of \Cref{big scope}:

\begin{lemma} \label{pure scope}

Let $\textbf{t}$ be a subpath of a boundary component of a reduced diagram $\Delta$ over the disk presentation of $G_\Omega(\textbf{M}^\pazocal{L})$.  Suppose $\Delta$ contains a scope $\Psi_0$ on $\textbf{t}$ such that the completion $\tilde{\Psi}_0$ is $1$-minimal.  If $\Psi_0$ is not a pure scope, then there exists a pure scope $\Psi$ on $\textbf{t}$ of size $\ell\geq L-6$ such that the completion $\tilde{\Psi}_1$ is a subdiagram of $\Psi_0$.

\end{lemma}


\medskip

\subsection{Transposition} \label{sec-transposition} \

Next, we define a process that allows us to move a $\theta$-band about an $a$-cell or a disk.  These operations appear similar as those in \cite{W}; however, the setting of the generalized $S$-machine $\textbf{M}^\pazocal{L}$ introduces some new obstructions for each.

\subsubsection{Transposition of a $\theta$-band and an $a$-cell} \

Let $\Delta$ be a circular diagram over the disk presentation of $G_\Omega(\textbf{M}^\pazocal{L})$ containing an $a$-cell $\pi$ and a reduced $\theta$-band $\pazocal{T}$ such that $|E(\pi,\pazocal{T})|\geq5$.  By \Cref{Omega inverses}, there exists $w\in\Omega$ such that $\lab(\partial\pi)\equiv w$.  Let $\theta$ be the history of $\pazocal{T}$.

Suppose $\partial\pi=\textbf{s}_1\textbf{s}_2$ where $\textbf{s}_1$ is a path satisfying:

\begin{itemize}

\item $\textbf{s}_1$ contains at least $5$ edges of $E(\pi,\pazocal{T})$

\item The first and last edges of $\textbf{s}_1$ are edges of $E(\pi,\pazocal{T})$

\item $\textbf{s}_1^{-1}$ is a subpath of $\textbf{bot}(\pazocal{T})$.

\end{itemize}

Then, let $\textbf{y}$ and $\textbf{z}$ be the minimal (perhaps trivial) subpaths of $\textbf{bot}(\pazocal{T})$ such that there exists a subband $\pazocal{T}'$ of $\pazocal{T}$ with $\textbf{bot}(\pazocal{T}')=\textbf{y}\textbf{s}_1^{-1}\textbf{z}_1$ (see \Cref{a-transposition}(a)).  Denote by $\Gamma$ the subdiagram of $\Delta$ consisting of $\pi$ and $\pazocal{T}'$.

Suppose $\textbf{y}$ is a non-trivial path.  Then, there exists a cell $\gamma$ of $\pazocal{T}'$ such that $\textbf{y}$ is a subpath of $\partial\gamma$.  In this case, $\gamma$ is a $(\theta,\pazocal{A})$-cell and the last edge of $\textbf{s}_1$ is an $\pazocal{A}$-edge of $(\partial\gamma)^{-1}$.  By the definition of the $(\theta,\pazocal{A})$-relations, $\lab(\textbf{y})\in F(\pazocal{B})$ and must be $\theta$-applicable with $\lab(\textbf{y})\cdot\theta\equiv\lab(\textbf{y})$.  

\begin{figure}[H]
\centering
\begin{subfigure}[b]{0.48\textwidth}
\centering
\raisebox{0.3825in}{\includegraphics[scale=0.7]{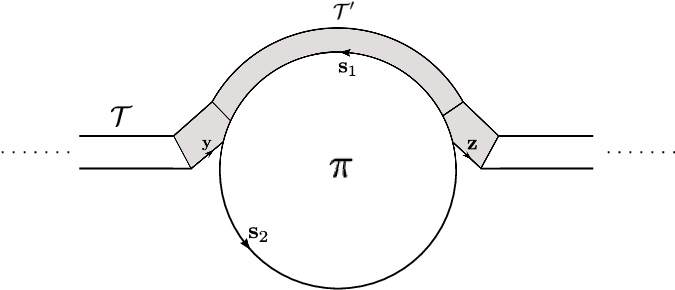}}
\caption{The subdiagram $\Gamma$}
\end{subfigure}\hfill
\begin{subfigure}[b]{0.48\textwidth}
\centering
\includegraphics[scale=0.7]{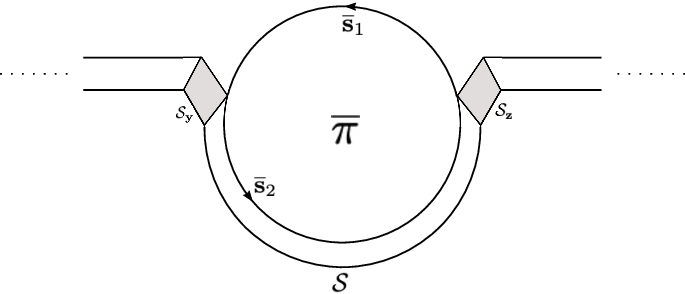}
\caption{The resulting subdiagram $\Gamma'$}
\end{subfigure}
\caption{The transposition of a $\theta$-band with an $a$-cell}
\label{a-transposition}
\end{figure}

Hence, \Cref{one-rule semi-computations are theta-bands} yields a $\theta$-band $\pazocal{S}_{\textbf{y}}$ with history $\theta$ consisting entirely of $(\theta,b)$-cells such that $\lab(\textbf{bot}(\pazocal{S}_{\textbf{y}}))\equiv\lab(\textbf{top}(\pazocal{S}_{\textbf{y}}))\equiv \lab(\textbf{y})$.

Similarly, if $\textbf{z}$ is a non-trivial path, then $\lab(\textbf{z})\in F(\pazocal{B})$ and \Cref{one-rule semi-computations are theta-bands} produces an analogous $\theta$-band $\pazocal{S}_{\textbf{z}}$.

In particular, $\textbf{bot}(\pazocal{T}')$ contains no $q$-edges, and hence \Cref{basic annuli 1} implies that $\pazocal{T}'$ consists entirely of $(\theta,a)$-cells.

Let $\textbf{s}_1''$ be the maximal subpath of $\textbf{s}_1$ such that there exists a subband $\pazocal{T}''$ of $\pazocal{T}'$ satisfying $\textbf{bot}(\pazocal{T}'')=(\textbf{s}_1'')^{-1}$.  Then, at most two cells of $\pazocal{T}'$ are not contained in $\pazocal{T}''$, and so the makeup of the relations implies $|\textbf{s}_1''|_\pazocal{A}\geq|E(\pi,\pazocal{T})|-2\geq3$.

Applying \Cref{theta-band is one-rule semi-computation} to the $\theta$-band $\pazocal{T}''$ then implies that $\lab((\textbf{s}_1'')^{-1})\equiv\lab(\textbf{s}_1'')^{-1}$ is $\theta$-applicable, and so \Cref{semi subword} implies $\lab(\textbf{s}_1'')$ is also $\theta$-applicable.  In particular, $\lab(\textbf{s}_1'')$ is a $\theta$-applicable subword of a cyclic permutation of $w$ with $|\textbf{s}_1''|_\pazocal{A}\geq3$, so that \Cref{semi extension} implies that $w$ is also $\theta$-applicable.

Let $v_i\equiv\lab(\textbf{s}_i)$ for $i=1,2$.  As $\lab(\textbf{s}_1\textbf{s}_2)$ is a cyclic permutation of $w$, \Cref{semi subword} implies $v_1$, $v_2$, and $v_1v_2$ are all $\theta$-applicable with $(v_1v_2)\cdot\theta=(v_1\cdot\theta)(v_2\cdot\theta)$.

Further, letting $u_{\textbf{y}}=\lab(\textbf{y})$ and $u_{\textbf{z}}=\lab(\textbf{z})$ (with these words taken to be trivial if the corresponding path is trivial), applying \Cref{theta-band is one-rule semi-computation} to $\pazocal{T}'$ implies $u_{\textbf{y}}v_1^{-1}u_{\textbf{z}}$ is $\theta$-applicable with $\lab(\textbf{top}(\pazocal{T}'))\equiv(u_{\textbf{y}}v_1^{-1}u_{\textbf{z}})\cdot\theta=u_{\textbf{y}}(v_1\cdot\theta)^{-1}u_{\textbf{z}}$.

Let $\pazocal{S}$ be the $\theta$-band given by \Cref{one-rule semi-computations are theta-bands} corresponding to the semi-computation $v_2\to(v_2\cdot\theta)$.  So, $\pazocal{S}$ has history $\theta$ with $\textbf{bot}(\pazocal{S})\equiv v_2$ and $\textbf{top}(\pazocal{S})\equiv v_2\cdot\theta$.

As $w\in\Omega$, there exists a word $w'\in\pazocal{E}(\Lambda^\pazocal{A})$ which is freely conjugate to $w$.  Let $p$ be a word such that $w'\equiv p^{-1}wp$.  Since $w$ is $\theta$-applicable, \Cref{Lambda not cyclically reduced} implies $w'$ is also $\theta$-applicable.  Since \Cref{semi subword} also implies $p$ is $\theta$-applicable, $w'\cdot\theta=(p\cdot\theta)^{-1}(w\cdot\theta)(p\cdot\theta)$.  But $w'\cdot\theta\in\pazocal{E}(\Lambda^\pazocal{A})$ by definition.  Hence, $w\cdot\theta$ is freely conjugate to $w'\cdot\theta\in\pazocal{E}(\Lambda^\pazocal{A})$, and so is freely conjugate to an element of $\Omega$.

In particular, $0$-refining a single $a$-cell, one can construct a circular diagram $\bar{\pi}$ with $\lab(\partial\bar{\pi})\equiv w\cdot\theta$ such that $\tau_3(\bar{\pi})=(0,0,1)$.  Then, $\lab(\partial\bar{\pi})$ is a cyclic permutation of $(v_1v_2)\cdot\theta$, and so using $0$-refinement we may assume $\partial\bar{\pi}=\bar{\textbf{s}}_1\bar{\textbf{s}}_2$ such that $\lab(\bar{\textbf{s}}_i)\equiv v_i\cdot\theta$.

So, we may glue $\pazocal{S}$ to $\bar{\pi}$ by identifying $\textbf{top}(\pazocal{S})$ and $\bar{\textbf{s}}_2$.  Then, perhaps pasting $\pazocal{S}_{\textbf{y}}$ and $\pazocal{S}_{\textbf{z}}$ to the ends of $\pazocal{S}$ (and making any necessary cancellations) then produces a reduced circular diagram $\Gamma'$ with $\lab(\partial\Gamma')\equiv\lab(\partial\Gamma)$ (see \Cref{a-transposition}(b)).

In this case, excising $\Gamma$ from $\Delta$ and replacing it with $\Gamma'$ is called the \textit{transposition} of the $\theta$-band $\pazocal{T}$ with the $a$-cell $\pi$ along $\textbf{s}_1$.

Note that the circular diagram $\Delta'$ resulting from the transposition has the same contour label as $\Delta$. Further, $\tau(\Gamma)=(0,0,1,|\textbf{s}_1|_\pazocal{A})$ and $\tau(\Gamma')=(0,0,1,|\partial\pi|_\pazocal{A}-|\textbf{s}_1|_\pazocal{A})$.  Hence, if in this setting $|\textbf{s}_1|_\pazocal{A}>\frac{1}{2}|\partial\pi|_\pazocal{A}$, then this transposition demonstrates that $\Delta$ is not minimal.

Indeed, the next statement shows that this observation applies in a more general setting:

\begin{lemma}[Compare with Lemma 9.9 of \cite{W}] \label{minimal diskless is M-minimal}

Every smooth minimal diagram over $M_\Omega(\textbf{M}^\pazocal{L})$ is $M$-minimal.

\end{lemma}

\begin{proof}

Supposing $\Delta$ is a counterexample, \Cref{minimal MM2} implies does not satisfy (MM1).  Hence, there exists a pair $(\pi,\pazocal{T})$ such that $\pi$ is an $a$-cell and $\pazocal{T}$ is a maximal $\theta$-band in $\Delta$ satisfying $|E(\pi,\pazocal{T})|>\frac{1}{2}|\partial\pi|_\pazocal{A}$.  Let $\pazocal{P}(\Delta)$ be the set of all such pairs in $\Delta$.

For any $(\pi,\pazocal{T})\in\pazocal{P}(\Delta)$, define $\pazocal{B}(\pi,\pazocal{T})$ to be the set of all tuples $(\pazocal{B}_1,\dots,\pazocal{B}_s)$ consisting of $s>\frac{1}{2}|\partial\pi|_\pazocal{A}$ maximal $\pazocal{A}$-bands corresponding to edges of $E(\pi,\pazocal{T})$ and enumerated based on where they cross $\pazocal{T}$.  For fixed $(\pi,\pazocal{T})\in\pazocal{P}(\Delta)$ and $(\pazocal{B}_1,\dots,\pazocal{B}_s)\in\pazocal{B}(\pi,\pazocal{T})$, let $\pazocal{T}_0$ be the minimal subband of $\pazocal{T}$ such that each $\pazocal{B}_i$ crosses $\pazocal{T}_0$.  Then, there exists a subdiagram $\Delta_0$ not containing $\pi$ which is bounded by the $\theta$-band $\pazocal{T}_0$, subbands of the $\pazocal{A}$-bands $\pazocal{B}_i$, and a subpath $\textbf{x}$ of $\partial\pi$ (see \Cref{fig-a-bands}).

Now, fix $(\pi,\pazocal{T})\in\pazocal{P}(\Delta)$ and $(\pazocal{B}_1,\dots,\pazocal{B}_s)\in\pazocal{B}(\pi,\pazocal{T})$ such that the corresponding subdiagram $\Delta_0$ is of minimal area.

By \Cref{basic annuli 1}, $\Delta_0$ cannot contain a maximal $q$-band, as then this band would bound a $(\theta,q)$-annulus with a subband of $\pazocal{T}_0$.  Further, if $\Delta_0$ contains a maximal $\theta$-band other than $\pazocal{T}_0$, then by \Cref{M(M) annuli} it must cross each $\pazocal{B}_i$, producing a pair that contradicts the minimality of $\text{Area}(\Delta_0)$.


\begin{figure}[H]
\centering
\includegraphics[scale=1]{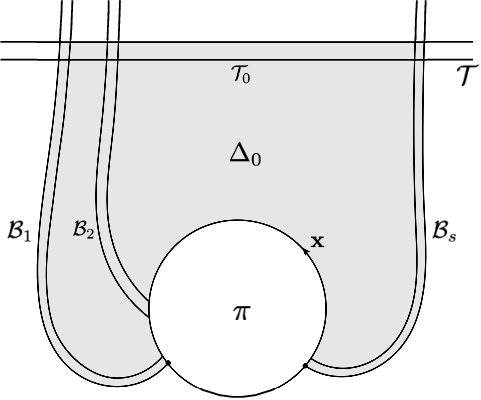}
\caption{Lemma \ref{minimal diskless is M-minimal}}
\label{fig-a-bands}
\end{figure}




Now, suppose $\Delta_0$ contains an $a$-cell.  Let $\tilde{\Delta}_0$ be the circular diagram consisting of both $\pi$ and $\Delta_0$ and let $\textbf{t}_0$ be the subpath of $\partial\tilde{\Delta}_0$ corresponding to the side of $\pazocal{T}_0$.  Then, as a subdiagram of $\tilde{\Delta}_0$, it then follows that $\Delta_0$ is a big $a$-scope on $\textbf{t}_0$ which is not pure.

Note that the completion of $\Delta_0$ is then $\tilde{\Delta}_0$, and so is smooth by hypothesis and satisfies (MM2) by \Cref{minimal MM2}.  \Cref{pure a-scope} then implies there exists a pure big $a$-scope $\Delta_1$ on $\textbf{t}_0$ such that the completion $\tilde{\Delta}_1$ is a subdiagram of $\Delta_0$.  In particular, $\text{Area}(\Delta_1)<\text{Area}(\Delta_0)$.  But then the associated $a$-cell and $\pazocal{T}$ form a pair which again contradicts the minimality of $\text{Area}(\Delta_0)$.

Thus, $\Delta_0$ cannot contain any cells apart from those of $\pazocal{T}_0$, {\frenchspacing i.e. $\textbf{x}$ or $\textbf{x}^{-1}$ is a subpath of a side of $\pazocal{T}$}. Without loss of generality assume that $\textbf{x}^{-1}$ is a subpath of $\textbf{bot}(\pazocal{T}_0)$.  A parameter choice for $C$ implies $s\geq5$, we may transpose $\pazocal{T}$ and $\pi$ along $\textbf{x}$ to produce a reduced circular diagram $\Delta'$.  

But then $\lab(\partial\Delta')\equiv\lab(\partial\Delta)$ and $\tau(\Delta')<\tau(\Delta)$, contradicting the hypothesis that $\Delta$ is minimal.

\end{proof}

Hence, \Cref{minimal diskless is M-minimal} implies the following analogue of \Cref{M-minimal is smooth}:

\begin{lemma} \label{minimal is smooth}

Any reduced minimal diagram is smooth.

\end{lemma}

\begin{proof}

Suppose to the contrary that the reduced minimal diagram $\Delta$ contains a pinched $a$-cell.  Choose an $a$-cell $\pi$ and a pinched subpath $\textbf{s}$ such that the subdiagram $\Psi_{\pi,\textbf{s}}$ has minimal weight.

As $\partial\Psi_{\pi,\textbf{s}}$ consists entirely of $a$-edges, by \Cref{Gamma special cell} it cannot contain a disk.  So, since the minimality of its weight implies $\Psi_{\pi,\textbf{s}}$ must be smooth, $\Psi_{\pi,\textbf{s}}$ is an $M$-minimal diagram by \Cref{minimal diskless is M-minimal}.

But then we arrive at a contradiction in exactly the same way as in the proof of \Cref{M-minimal is smooth}: 

Lemmas \ref{basic annuli 1} and \ref{M-minimal theta-annuli} imply that any cell of $\Psi_{\pi,\textbf{s}}$ must be an $a$-cell.  So, since $\lab(\partial\Psi_{\pi,\textbf{s}})$ is non-trivial, \Cref{Gamma_a special cell} yields an $a$-cell $\pi'$ in $\Psi_{\pi,\textbf{s}}$ which, together with $\pi$, produces a counterexample to condition (MM1).

\end{proof}




\subsubsection{Transposition of a $\theta$-band and a disk} \

We now adjust the above procedure in order to move a $\theta$-band about a disk.  Again, this is done in a manner similar to that of \cite{W} (and \cite{O18}, \cite{OS19}, etc), but with several more complications.

Let $\Delta$ be a circular diagram over the disk presentation of $G_\Omega(\textbf{M}^\pazocal{L})$ containing a disk $\Pi$ and a reduced $\theta$-band $\pazocal{T}$.  Let $\theta$ be the history of $\pazocal{T}$ and suppose $\lab(\partial\Pi)\equiv W^{-\eps}$ where $\eps\in\{\pm1\}$ and $W$ is an accepted configuration of $\textbf{M}^\pazocal{L}$ with $\ell(W)\leq1$.

Suppose the following conditions hold:

\begin{enumerate}

\item $W$ is $\theta$-admissible with $\ell(W\cdot\theta)\leq1$

\item There exists a decomposition $\partial\Pi=\textbf{s}_1\textbf{s}_2$ where $\textbf{s}_1$ is a path satisfying:

\begin{itemize}

\item $\textbf{s}_1$ contains $\ell\geq2$ $t$-edges

\item The first and last edges of $\textbf{s}_1$ are $t$-edges


\item $\textbf{s}_1^{-1}$ is a subpath of $\textbf{bot}(\pazocal{T})$.

\end{itemize}

\end{enumerate}

Note that each of the $\ell$ $t$-edges of $\textbf{s}_1^{-1}$ then correspond to positive $t$-spokes $\pazocal{Q}_1,\dots,\pazocal{Q}_\ell$ of $\Pi$ which cross $\pazocal{T}$.  Let $\pazocal{T}'$ be the minimal subband of $\pazocal{T}$ which crosses each of these $t$-spokes.  Then, since every $(R_0^\pazocal{L}(i))^{-1}\{t(i+1)\}$- and $\{t(i+1)\}Q_0^\pazocal{L}(i+1)$-sector is locked by each rule of $\textbf{M}^\pazocal{L}$, $\textbf{s}_1^{-1}=\textbf{bot}(\pazocal{T}')$.  In particular, $\Pi$ and $\pazocal{T}'$ form a subdiagram $\Gamma$ of $\Delta$ (see \Cref{fig-transposition}(a)).  

Let $V_1\equiv\lab(\textbf{s}_1)$ and $V_2\equiv\lab(\textbf{s}_2)$.  Then, since $V_1V_2$ is a cyclic permutation of the admissible word $W^{-\eps}$ and $V_1$ begins and ends with a $t$-letter, $V_1$ and $V_2$ are both admissible words.  Moreover, since $W$ is $\theta$-admissible, $V_1$ and $V_2$ are $\theta$-admissible with $\left(V_1\cdot\theta\right)\left(V_2\cdot\theta\right)$ a cyclic permutation of $(W\cdot\theta)^{-\eps}$.  Hence, since $\ell(W\cdot\theta)\leq1$, we may construct the disk $\bar{\Pi}$ with $\partial\bar{\Pi}=\bar{\textbf{s}}_1\bar{\textbf{s}}_2$ with $\lab(\bar{\textbf{s}}_i)\equiv V_i\cdot\theta$ for $i=1,2$.

Applying \Cref{theta-bands are one-rule computations} to $\pazocal{T}'$ implies $\lab(\textbf{ttop}(\pazocal{T}'))\equiv\lab(\textbf{s}_1^{-1})\cdot\theta\equiv\left(V_1\cdot\theta\right)^{-1}\equiv\lab(\bar{\textbf{s}}_1)^{-1}$.  As the first and last cells of $\pazocal{T}'$ are $(\theta,t)$-cells, no trimming is necessary in the band $\pazocal{T}'$, i.e $\lab(\textbf{top}(\pazocal{T}'))\equiv\lab(\textbf{ttop}(\pazocal{T}'))\equiv\lab(\bar{\textbf{s}}_1)^{-1}$.


Conversely, construct the $\theta$-band $\pazocal{S}$ given by \Cref{one-rule computations are theta-bands} corresponding to the computation $V_2\to V_2\cdot\theta$.  So, $\pazocal{S}$ has history $\theta$ with $\lab(\textbf{tbot}(\pazocal{S}))\equiv V_2$ and $\lab(\textbf{ttop}(\pazocal{S}))\equiv V_2\cdot\theta$.  As above, no trimming is necessary in the band $\pazocal{S}$, so that $\lab(\textbf{bot}(\pazocal{S}))\equiv V_2$ and $\lab(\textbf{top}(\pazocal{S}))\equiv V_2\cdot\theta$.

\renewcommand\thesubfigure{\alph{subfigure}}
\begin{figure}[H]
\centering
\begin{subfigure}[b]{0.48\textwidth}
\centering
\raisebox{0.2in}{\includegraphics[width=3in]{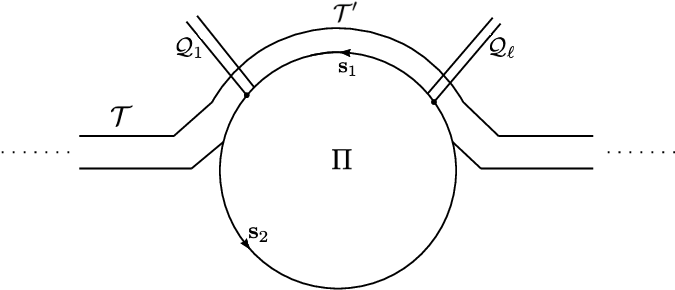}}
\caption{The subdiagram $\Gamma$}
\end{subfigure}\hfill
\begin{subfigure}[b]{0.48\textwidth}
\centering
\includegraphics[width=3in]{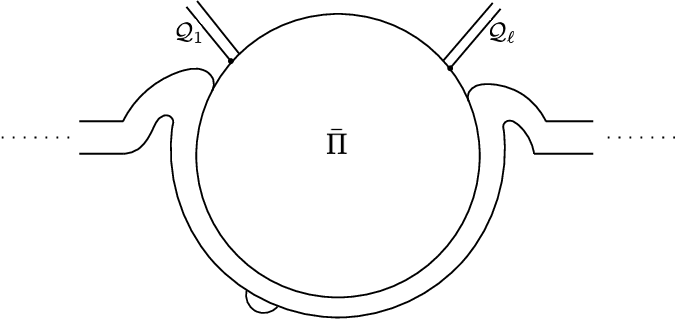}
\caption{The resulting subdiagram $\bar{\Gamma}$}
\end{subfigure}
\caption{The transposition of a $\theta$-band with a disk}
\label{fig-transposition}
\end{figure}

So, we may glue $\pazocal{S}$ to $\bar{\Pi}$ by identifying $\textbf{top}(\pazocal{S})$ and $\bar{\textbf{s}}_2$, producing a reduced circular diagram $\Gamma'$ which satisfies $\lab(\partial\Gamma')\equiv\lab(\partial\Gamma)$ (see \Cref{fig-transposition}(b)).

In this case, excising $\Gamma$ from $\Delta$ and replacing it with $\Gamma'$ is called the \textit{transposition} of the $\theta$-band $\pazocal{T}$ with the disk $\Pi$ along $\textbf{s}_1$.

Note that this procedure is identical to the one presented in previous literature (\cite{O18}, \cite{OS19}, etc.).  What's more, the circular diagram $\Delta'$ resulting from the transposition of $\pazocal{T}$ and $\Pi$ has the same contour label as $\Delta$.  Thus, since $\tau_2(\Gamma)=(1,\ell)$ and $\tau_2(\Gamma')=(1,L-1-\ell)$, if in this setting $\ell>(L-1)/2$, then the transposition demonstrates that $\Delta$ is not $2$-minimal, and so not minimal.

\medskip

Now, we adapt this procedure to a more general setting, assuming the disk is labelled by an arbitrary configuration and allowing some $a$-cells between the $\theta$-band and the disk.  Note that a similar procedure was discussed in \cite{W}, but as the machine in that setting was an $S$-machine (and not noisy), the adaptation was much simpler.

%
%
%
%

Let $\Phi$ be a reduced circular diagram over the disk presentation of $G_\Omega(\textbf{M}^\pazocal{L})$.  Suppose there exists a decomposition $\partial\Phi=\textbf{p}_1^{-1}\textbf{s}_2\textbf{p}_2\textbf{t}^{-1}$ such that:

\begin{itemize}

\item $\textbf{p}_1$ and $\textbf{p}_2$ are defining edges of a $\theta$-band $\pazocal{T}$ in $\Phi$

\item $\textbf{t}=\textbf{top}(\pazocal{T})$ and contains $\ell\geq2$ $t$-edges, two of which are its first and last edges

\item $\textbf{s}_2$ is a subpath of $\partial\Pi$ where $\Pi$ is the unique disk in $\Phi$

\end{itemize}

Let $\textbf{s}_1$ be the complement of $\textbf{s}_2$ in $\partial\Pi$.  By definition, $\textbf{p}_1^{-1}\textbf{s}_1^{-1}\textbf{p}_2\textbf{t}^{-1}$ is the contour of a subdiagram $\Psi$ of $\Phi$ which contains $\pazocal{T}$ but not $\Pi$.  If any cell of $\Psi$ other than those comprising $\pazocal{T}$ is an $a$-cell, then the diagram $\Phi$ is called a \textit{profile} and the subdiagram $\Psi$ is called its \textit{half-hat}.

In this case, $\ell$ is called the \textit{size} of the profile $\Phi$.  Note that necessarily $\tau_2(\Phi)=(1,\ell)$.

The \textit{history} of $\Phi$ is taken to be the history of the \textit{associated $\theta$-band} $\pazocal{T}$.  Further, the accepted configuration $W$ of $\textbf{M}^\pazocal{L}$ such that $\lab(\partial\Pi)^{-\eps}\equiv W$ for some $\eps\in\{\pm1\}$ is called the \textit{defining configuration} of $\Phi$.  Finally, the decomposition $\textbf{p}_1^{-1}\textbf{s}_2\textbf{p}_2\textbf{t}^{-1}$ is called the \textit{standard factorization} of $\partial\Phi$, while $\textbf{s}_1$ is called the \textit{hidden path} of $\Phi$.

If the half-hat $\Psi$ is a minimal diagram, then $\Phi$ is called \textit{flat profile}.  Conversely, if $\Psi$ contains no $a$-cells (i.e $\Psi$ is simply the associated $\theta$-band) then $\Phi$ is called a \textit{simple profile}.

%
%
%
%
%
%
%



\begin{lemma} \label{minimal half-hat}

Let $\Phi$ be a profile of size $\ell$ with history $\theta$ and defining configuration $W$.  Then, there exists a reduced circular diagram $\Phi'$ over the disk presentation of $G_\Omega(\textbf{M})$ such that:

\begin{enumerate}

\item $\lab(\partial\Phi')\equiv\lab(\partial\Phi)$

\item $\tau_2(\Phi')=\tau_2(\Phi)$

\item $\Phi'$ contains a subdiagram $\Phi_0'$ which is a flat profile of size $\ell$ with history $\theta$ and defining configuration $W$.

\end{enumerate}

\end{lemma}

\begin{proof}

Let $\textbf{p}_1^{-1}\textbf{s}_2\textbf{p}_2\textbf{t}^{-1}$ be the standard factorization of $\partial\Phi$ and $\textbf{s}_1$ be the hidden path of $\Phi$.

Letting $\Psi$ be the half-hat of $\Phi$, let $\Psi'$ be a reduced minimal diagram with $\lab(\partial\Psi')\equiv\lab(\partial\Psi)$.  So, there exists a decomposition $\partial\Psi'=(\textbf{p}_1')^{-1}(\textbf{s}_1')^{-1}\textbf{p}_2'(\textbf{t}')^{-1}$ with corresponding labels, i.e such that $\lab(\textbf{p}_i')\equiv\lab(\textbf{p}_i)$, $\lab(\textbf{s}_1')\equiv\lab(\textbf{s}_1)$, and $\lab(\textbf{t}')\equiv\lab(\textbf{t})$.

Let $\pazocal{T}'$ be the maximal $\theta$-band of $\Psi'$ for which $\textbf{p}_1'$ is a defining edge.  Then, since $|\textbf{t}'|_\theta=|\textbf{s}_1'|_\theta=0$, $\textbf{p}_1'$ and $\textbf{p}_2'$ must be the ends of $\pazocal{T}'$.  Since $\tau(\Psi')\leq\tau(\Psi)$ and $\Psi$ is a half-hat, $\Psi'$ also contains no disks.  Hence, Lemmas \ref{minimal is smooth}, \ref{minimal MM2}, and \ref{Gamma_a special cell} imply that $\Psi'$ contains no $\theta$-annuli.

Thus, any cell of $\Psi'$ other than those comprising $\pazocal{T}'$ is an $a$-cell.  In particular, any cell between $\textbf{top}(\pazocal{T}')$ and $\textbf{t}'$ is an $a$-cell, so that every $t$-edge of $\textbf{t}'$ is also an edge of $\textbf{top}(\pazocal{T}')$.

Now, let $\Phi'$ be the diagram obtained from $\Phi$ by replacing $\Psi$ with $\Psi'$.  Then, it follows immediately that $\lab(\partial\Phi')\equiv\lab(\partial\Phi)$ and $\tau_2(\Phi')\leq\tau_2(\Phi)$.

Let $\Phi_0'$ be the subdiagram of $\Phi'$ obtained by removing any $a$-cells of $\Psi'$ between $\textbf{top}(\pazocal{T}')$ and $\textbf{t}'$.  Then, by construction, $\Phi_0'$ is a flat profile satisfying (3), and so $\tau_2(\Phi')\geq\tau_2(\Phi_0')=(1,\ell)=\tau_2(\Phi)$.

\end{proof}

Let $\Phi$ be a profile with associated $\theta$-band $\pazocal{T}$.  Let $\textbf{e}_0,\textbf{e}_1,\dots,\textbf{e}_k$ be the enumeration of the $q$-edges of $\textbf{bot}(\pazocal{T})$.  By definition, these $q$-edges are in correspondence with the $q$-edges of the hidden path $\textbf{s}_1$ of $\Phi$, and so $\textbf{e}_0,\textbf{e}_1,\dots,\textbf{e}_k$ is also the enumeration of the $q$-edges of $\textbf{s}_1^{-1}$.  Letting $\Pi$ be the disk of $\Phi$, we can then continue this to obtain an enumeration $\textbf{e}_0,\textbf{e}_1,\dots,\textbf{e}_n$ of the $q$-edges of $(\partial\Pi)^{-1}$.

In this case, the sequence $(\textbf{e}_0,\dots,\textbf{e}_k ; \textbf{e}_{k+1},\dots,\textbf{e}_n)$ is called the \textit{$q$-enumeration} of $\Phi$.  Note that the value of $n$ is determined by simply the length of the standard base of $\textbf{M}^\pazocal{L}$, while the value of $k$ depends on the size (and makeup) of the profile.

Let $\textbf{y}_i$ be the (perhaps trivial) subpath of $(\partial\Pi)^{-1}$ between $\textbf{e}_{i-1}$ and $\textbf{e}_i$.  Then, by the definition of disk relations, $\lab(\textbf{e}_{i-1}\textbf{y}_i\textbf{e}_i)$ is an admissible word with reduced two-letter base $U_iV_i$, where $V_i=U_{i+1}$.  In particular, there exists a unique index $s=s(\Phi)$ such that $U_sV_s=(Q_0^\pazocal{L}(1)Q_1^\pazocal{L}(1))^{\pm1}$.

Now, for any $j\in\{0,\dots,k\}$, let $\gamma_j$ be the $(\theta,q)$-cell of $\pazocal{T}$ such that $\textbf{e}_j$ is an edge of $\partial\gamma_j$.  With this, let $\pazocal{T}_i$ be the minimal subband of $\pazocal{T}$ containing both $\gamma_{i-1}$ and $\gamma_i$.  Then, define the \textit{$i$-th cover} of the half-hat of $\Phi$ to be the subdiagram $\Psi_i$ bounded by $\pazocal{T}_i$ and $\textbf{y}_i$.  Note that, as indicated by the name, every cell of the half-hat is contained in a cover $\Psi_i$.  Moreover, $\gamma_i$ is the unique cell contained in both $\Psi_i$ and $\Psi_{i+1}$, while $\Psi_i$ and $\Psi_j$ share no cells if $|j-i|\geq2$.

Further, note that \Cref{theta-bands are one-rule computations} implies $\lab(\textbf{ttop}(\pazocal{T}_i))$ is, like $\lab(\textbf{e}_{i-1}\textbf{y}_i\textbf{e}_i)$, an admissible word with base $U_iV_i$.  Hence, any $a$-edge of $\textbf{ttop}(\pazocal{T}_i)$ or of $\textbf{y}_i$ is labelled by an $a$-letter from the tape alphabet of $\textbf{M}^\pazocal{L}$ corresponding to the $U_iV_i$-sector.  In particular, any $a$-edge of $\partial\Psi_i$ which is on the boundary of either a $(\theta,a)$- or an $a$-cell is labelled by an $a$-letter of this tape alphabet.




\begin{lemma} \label{profile makeup}

Let $\Phi$ be a flat profile with $q$-enumeration $(\textbf{e}_0,\dots\textbf{e}_k;\textbf{e}_{k+1},\dots,\textbf{e}_n)$.  If $\Phi$ is not a simple profile, then $s=s(\Phi)\leq k$ and any $a$-cell in $\Phi$ is contained in $\Psi_s$.






\end{lemma}

\begin{proof}

As $\Phi$ is not a simple profile, there exists $i\in\{1,\dots,k\}$ such that $\Psi_i$ contains an $a$-cell.  Since $\Phi$ is flat, this subdiagram $\Psi_i$ is a minimal diagram.  So, Lemmas \ref{minimal is smooth}, \ref{minimal diskless is M-minimal}, and \ref{Gamma_a special cell} imply that there exists an $a$-cell $\pi$ in $\Psi_i$ and $m\geq C-6$ consecutive $\pazocal{A}$-edges $\textbf{e}_1',\dots,\textbf{e}_m'$ of $\partial\pi$ such that each maximal $\pazocal{A}$-band $\pazocal{U}(\textbf{e}_j')$ of $\Psi_i$ has an end on either a $(\theta,q)$-cell or on $\partial\Psi_i$.

Note that the contour of any $(\theta,q)$-cell contains at most one $\pazocal{A}$-edge.  So, since $\Psi_i$ contains exactly two $(\theta,q)$-cells, the parameter choice $C\geq9$ implies that at least one $\pazocal{A}$-band $\pazocal{U}(\textbf{e}_j')$ has an end which is an edge of $\partial\Psi_i$.  This end is thus an $a$-edge of $\partial\Psi_i$ which is on the boundary of a $(\theta,a)$- or an $a$-cell and is labelled by an $\pazocal{A}$-letter of the `special' input sector.  Therefore, the index $i$ must correspond to the `special' input sector, i.e $i=s$.

\end{proof}

%

\begin{lemma} \label{profile projection}

Let $\Phi$ be a flat profile with history $\theta$ and defining configuration $W$.  Then, $W(2)$ is $\theta$-admissible.

\end{lemma}

\begin{proof}

Note that by the parallel nature of the rules of $\textbf{M}^\pazocal{L}$, it suffices to show that $W(j)$ is $\theta$-admissible for some $j\in\{2,\dots,L\}$.

Let $\Pi$ be the disk in $\Phi$ and fix $\eps\in\{\pm1\}$ such that $\lab(\partial\Pi)^{-\eps}\equiv W$.  Further, let $\pazocal{T}$ be the associated $\theta$-band and $\textbf{s}_1$ be the hidden path of $\Phi$.

If $\textbf{s}_1^{-1}$ has a subpath $\textbf{x}$ shared with $\textbf{bot}(\pazocal{T})$ such that $\lab(\textbf{x})$ is an admissible word with base $\left(\{t(j)\}B_4^\pazocal{L}(j)\{t(j+1)\}\right)^{\eps}$ for some $j\in\{2,\dots,L\}$, then the this admissible subword of $W^{\pm1}$ is $\theta$-admissible, and so $W(j)$ is as well.

%

Now, let $(\textbf{e}_0,\dots,\textbf{e}_k;\textbf{e}_{k+1},\dots,\textbf{e}_n)$ be the $q$-enumeration of $\Phi$ and set $m\in\{1,\dots,n\}$ as the minimal index such that $\textbf{e}_m$ is a $t$-edge.  As the size of a profile is at least $2$, it must hold that $m\leq k$.  Let $\textbf{z}$ be the initial subpath of $\textbf{s}_1^{-1}$ whose last edge is $\textbf{e}_m$.

First, suppose $\{\lab(\textbf{e}_0),\lab(\textbf{e}_m)\}=\{t(j)^{\eps},t(j+1)^{\eps}\}$ for some $j\in\{2,\dots,L-1\}$.  Then, by definition, $\lab(\textbf{z})$ is then an admissible word with base $\left(\{t(j)\}B_4^\pazocal{L}(j)\{t(j+1)\}\right)^{\eps}$.  But then $s(\Phi)\notin\{0,\dots,m\}$, so that \Cref{profile makeup} implies $\textbf{z}$ is a subpath of $\textbf{bot}(\pazocal{T})$.  Hence, setting $\textbf{x}=\textbf{z}$ as above, we conclude that $W(j)$ is $\theta$-admissible.

Otherwise, $\{\lab(\textbf{e}_0),\lab(\textbf{e}_m)\}=\{t(L)^\eps,t(2)^\eps\}$.  By the makeup of the standard base, there then exists $r\in\{1,\dots,m-1\}$ such that $\lab(\textbf{e}_r)= t(1)^\eps$.  Define the subpath $\textbf{z}'$ of $\textbf{z}$ by:

\begin{itemize}

\item If $\eps=1$, then $\textbf{z}'$ is the initial subpath of $\textbf{z}$ whose last edge is $\textbf{e}_r$

\item If $\eps=-1$, then $\textbf{z}'$ is the terminal subpath of $\textbf{z}$ whose first edge is $\textbf{e}_r$

\end{itemize}

In either case, $\lab(\textbf{z}')$ is an admissible word with base $\left(\{t(L)\}B_4^\pazocal{L}(L)\{t(1)\}\right)^\eps$.  But then \Cref{profile makeup} again implies $\textbf{z}'$ is a subpath of $\textbf{bot}(\pazocal{T})$, so that $W(L)$ is $\theta$-admissible.

\end{proof}

\begin{lemma} \label{flat profile to simple profile}

Let $\Phi$ be a flat profile of size $\ell$ with history $\theta$ and defining configuration $W$.  If $W$ is $\theta$-admissible, then there exists a circular diagram $\Phi'$ over the disk presentation of $G_\Omega(\textbf{M}^\pazocal{L})$ such that:

\begin{enumerate}

\item $\lab(\partial\Phi')\equiv\lab(\partial\Phi)$

\item $\tau_2(\Phi')=\tau_2(\Phi)$

\item There exists a subdiagram $\Phi_0'$ of $\Phi'$ which is a simple profile of size $\ell$ with history $\theta$ and defining configuration $W$.

\end{enumerate}

\end{lemma}

\begin{proof}

We may assume $\Phi$ is not simple, as otherwise the statement is satisfied for $\Phi'=\Phi$.

Let $(\textbf{e}_0,\dots,\textbf{e}_k;\textbf{e}_{k+1},\dots,\textbf{e}_n)$ be the $q$-enumeration of $\Phi$.  By \Cref{profile makeup}, it then follows that $s=s(\Phi)\leq k$ and every $a$-cell of $\Phi$ is contained in the subdiagram $\Psi_s$ of the half-hat.

Define the $(\theta,q)$-cells $\gamma_i$, the subbands $\pazocal{T}_i$ of the associated $\theta$-band $\pazocal{T}$, and the paths $\textbf{y}_i$ as above.  For each $i$, fix the decomposition $\partial\gamma_i=\textbf{p}_i^{-1}\textbf{e}_i\textbf{q}_i\textbf{f}_i^{-1}$ such that $\textbf{f}_i$ is a $q$-edge of $\textbf{top}(\pazocal{T})$.  Then, $\partial\Psi_s=\textbf{p}_{s-1}^{-1}(\textbf{e}_{s-1}\textbf{y}_s\textbf{e}_s)\textbf{q}_s\textbf{ttop}(\pazocal{T}_s)^{-1}$. 

As $\textbf{e}_{s-1}\textbf{y}_s\textbf{e}_s$ is a subpath of $(\partial\Pi)^{-1}$ for $\Pi$ the unique disk of $\Phi$, $W_s'\equiv\lab(\textbf{e}_{s-1}\textbf{y}_s\textbf{e}_s)$ is the admissible subword of $W^\eps$ with base $\left(Q_0^\pazocal{L}(1)Q_1^\pazocal{L}(1)\right)^\eps$.  Hence, $W_s'$ is $\theta$-admissible.

Applying \Cref{one-rule computations are theta-bands} to the computation $W_s'\to W_s'\cdot\theta$ then produces a $\theta$-band $\pazocal{S}$ with history $\theta$ such that $\lab(\textbf{tbot}(\pazocal{S}))\equiv W_s'$ and $\lab(\textbf{ttop}(\pazocal{S}))\equiv W_s'\cdot\theta$.  Note that the first and last cells of $\pazocal{S}$ are copies of $\gamma_{s-1}$ and $\gamma_s$, respectively.  So, $\partial\pazocal{S}=(\textbf{p}_{s-1}')^{-1}\textbf{tbot}(\pazocal{S})\textbf{q}_s'\textbf{ttop}(\pazocal{S})^{-1}$ such that $\lab(\textbf{p}_{s-1}')\equiv\lab(\textbf{p}_{s-1})$ and $\lab(\textbf{q}_s')\equiv\lab(\textbf{q}_s)$.  Note that no $q$-edge of $\partial\pazocal{S}$ is a $t$-edge, and so \Cref{basic annuli 1} implies $\tau_2(\pazocal{S})=(0,0)$.

Next, consider the `mirror' $\theta$-band $\overline{\pazocal{S}}$ of $\pazocal{S}$ (see \Cref{mirror}), i.e the $\theta$-band with history $\theta^{-1}$ such that $\lab(\textbf{bot}(\overline{\pazocal{S}}))\equiv\lab(\textbf{top}(\pazocal{S}))$ and $\lab(\textbf{top}(\overline{\pazocal{S}}))\equiv\lab(\textbf{bot}(\pazocal{S}))$.  Then, $$\partial\overline{\pazocal{S}}=\textbf{p}_{s-1}''\textbf{tbot}(\overline{\pazocal{S}})(\textbf{q}_s'')^{-1}\textbf{ttop}(\overline{\pazocal{S}})^{-1}$$
where $\lab(\textbf{p}_{s-1}'')\equiv\lab(\textbf{p}_{s-1})$ and $\lab(\textbf{q}_s'')\equiv\lab(\textbf{q}_s)$.  

Now, construct the (unreduced) circular diagram $\text{I}$ over the disk presentation of $G_\Omega(\textbf{M}^\pazocal{L})$ obtained by pasting $\textbf{ttop}(\pazocal{S})$ to $\textbf{tbot}(\overline{\pazocal{S}})$.  Then, $\partial\text{I}=\textbf{p}_{s-1}''(\textbf{p}_{s-1}')^{-1}\textbf{tbot}(\pazocal{S})\textbf{q}_s'(\textbf{q}_s'')^{-1}\textbf{ttop}(\overline{\pazocal{S}})^{-1}$.

Further, as $\lab(\textbf{ttop}(\overline{\pazocal{S}}))\equiv\lab(\textbf{tbot}(\pazocal{S}))\equiv W_s'\equiv\lab(\textbf{e}_{s-1}\textbf{y}_s\textbf{e}_s)$, we can paste $\text{I}$ to $\Psi_s$ by identifying the subpath $\textbf{p}_{s-1}^{-1}(\textbf{e}_{s-1}\textbf{y}_s\textbf{e}_s)\textbf{q}_s$ of $\partial\Psi_s$ with the subpath $(\textbf{p}_{s-1}'')^{-1}\textbf{ttop}(\overline{\pazocal{S}})\textbf{q}_s''$ of $(\partial\text{I})^{-1}$.  This produces an unreduced circular diagram $\Psi_s'$ over the disk presentation of $G_\Omega(\textbf{M}^\pazocal{L})$ with $\partial\Psi_s'=(\textbf{p}_{s-1}')^{-1}\textbf{tbot}(\pazocal{S})\textbf{q}_s'\textbf{ttop}(\pazocal{T}_s)^{-1}$.

Hence, $\lab(\partial\Psi_s')\equiv\lab(\partial\Psi_s)$, and so we can construct the circular diagram $\Phi'$ by excising $\Psi_s$ from $\Phi$ and pasting $\Psi_s'$ in its place.  By construction, there exists a maximal reduced $\theta$-band $\pazocal{S}'$ of $\Phi'$ obtained from $\pazocal{T}$ by replacing $\pazocal{T}_s$ with $\pazocal{S}$.  By \Cref{profile makeup}, $\textbf{bot}(\pazocal{S}')$ is a subpath of $(\partial\Pi)^{-1}$.  

Thus, the subdiagram $\Phi_0'$ of $\Phi'$ consisting of $\Pi$ and $\pazocal{S}'$ is a flat profile satisfying the statement.

\end{proof}

\begin{lemma} \label{basic transpose a}

Let $\Phi$ be a profile of size $\ell$ with history $\theta$ and defining configuration $W$.  Suppose $W$ is $\theta$-admissible with $\ell(W\cdot\theta)\leq1$.  Then there exists a circular diagram $\Phi'$ over the disk presentation of $G_\Omega(\textbf{M}^\pazocal{L})$ with $\lab(\partial\Phi')\equiv\lab(\partial\Phi)$ such that $\tau_2(\Phi')=(1,L-1-\ell)$.

\end{lemma}

\begin{proof}

By Lemmas \ref{minimal half-hat} and \ref{flat profile to simple profile}, there exists a circular diagram $\Gamma$ over the disk presentation of $G_\Omega(\textbf{M}^\pazocal{L})$ such that:

\begin{itemize}

\item $\lab(\partial\Gamma)\equiv\lab(\partial\Phi)$

\item $\tau_2(\Gamma)=\tau_2(\Phi)$

\item There exists a subdiagram $\Gamma_0$ of $\Gamma$ which is a simple profile of size $\ell$ with history $\theta$ and defining configuration $W$

\end{itemize}

Let $\pazocal{T}_0$ be the associated $\theta$-band, $\textbf{s}_1$ be the hidden path, and $\Pi$ be the (unique) disk of $\Gamma_0$.  Then, $\pazocal{T}_0$ and $\Pi$ may be transposed along $\textbf{s}_1$, producing a diagram $\Gamma_0'$ with $\lab(\partial\Gamma_0')\equiv\lab(\partial\Gamma_0)$ and $\tau_2(\Gamma_0')=(1,L-1-\ell)$.  Thus, letting $\Phi'$ be the diagram obtained from $\Gamma$ by excising $\Gamma_0$ and pasting $\Gamma_0'$ in its place satisfies the statement.

\end{proof}

Finally, the next statement demonstrates \Cref{basic transpose a} in the general case, removing any assumption on the defining configuration:

\begin{lemma} \label{transpose a}

Let $\Phi$ be a profile of size $\ell$.  Then there exists a circular diagram $\Phi'$ over the disk presentation of $G_\Omega(\textbf{M}^\pazocal{L})$ with $\lab(\partial\Phi')\equiv\lab(\partial\Phi)$ such that $\tau_2(\Phi')=(1,L-1-\ell)$.

\end{lemma}

\begin{proof}

Let $\theta$ be the history, $W$ the defining configuration, and $(\textbf{e}_0,\dots,\textbf{e}_k;\textbf{e}_{k+1},\dots,\textbf{e}_n)$ be the $q$-enumeration of $\Phi$.  Letting $\Pi$ be the disk of $\Phi$, let $\eps\in\{\pm1\}$ such that $\lab(\partial\Pi)^\eps\equiv W$.

By \Cref{basic transpose a}, the statement holds if $W$ is $\theta$-admissible and $\ell(W\cdot\theta)\leq1$.

First, suppose $W$ is not $\theta$-admissible.  Then \Cref{transposition computation not applicable} implies that $\theta=\theta(s)_2$ and $W\equiv I(w)$ for some $w\in\pazocal{L}$.  By condition (L5), we can construct a circular diagram $\Sigma$ over the disk presentation of $G_\Omega(\textbf{M}^\pazocal{L})$ with $\tau(\Sigma)=(1,0,1,0)$ such that:

\begin{itemize}

\item The single disk $\bar{\Pi}$ of $\Sigma$ satisfies $\lab(\partial\bar{\Pi})^\eps\equiv J(w)$

\item The single $a$-cell $\pi$ of $\Sigma$ satisfies $\lab(\partial\pi)^\eps\equiv w$

\item $\lab(\partial\Sigma)^\eps\equiv W$

\end{itemize}

Excising $\Pi$ from $\Phi$ and replacing it with $\Sigma$ then yields a circular diagram $\Delta$ over the disk presentation of $G_\Omega(\textbf{M}^\pazocal{L})$ with $\lab(\partial\Delta)\equiv\lab(\partial\Phi)$ and $\tau_2(\Delta)=\tau_2(\Phi)$.

If $s(\Phi)\leq k$, then $\Delta=\Gamma$ is a profile of size $\ell$ with history $\theta(s)_2$ and defining configuration $J(w)$.

Otherwise, if $s(\Phi)>k$, then let $\Gamma$ be the subdiagram of $\Delta$ obtained by removing $\pi$.  Then $\Gamma$ is a profile of size $\ell$ with history $\theta(s)_2$ and defining configuration $J(w)$ such that $\tau_2(\Gamma)=\tau_2(\Phi)$.

Note that $J(w)$ is $\theta(s)_2$-admissible and $\ell(J(w)\cdot\theta(s)_2)=1$.  

Hence, by \Cref{basic transpose a}, there exists a circular diagram $\Phi_0$ over the disk presentation of $G_\Omega(\textbf{M}^\pazocal{L})$ with $\lab(\partial\Phi_0)\equiv\lab(\partial\Gamma)$ and $\tau_2(\Phi_0)=(1,L-1-\ell)$.  

Thus, the circular diagram $\Phi'$ obtained from $\Delta$ by replacing $\Gamma$ with $\Phi_0$ satisfies the statement.

Now, suppose $W$ is $\theta$-admissible but $\ell(W\cdot\theta)>1$.  Then \Cref{transposition computation not one-machine} implies that $\theta=\theta(s)_1$ and $W\equiv J(w)$ for some $w\in\pazocal{L}$.  As above, we can then construct a reduced diagram $\Sigma'$ over the disk presentation of $G_\Omega(\textbf{M}^\pazocal{L})$ with $\tau(\Sigma')=(1,0,1,0)$ such that:

\begin{itemize}

\item The single disk $\bar{\Pi}'$ of $\Sigma'$ satisfies $\lab(\partial\bar{\Pi}')^\eps\equiv I(w)$

\item The single $a$-cell $\pi'$ of $\Sigma'$ satisfies $\lab(\partial\bar{\Pi}')^\eps\equiv w$

\item $\lab(\partial\Sigma')^\eps\equiv W$

\end{itemize}

Again, excising $\Pi$ from $\Phi$ and replacing it with $\Sigma'$ then yields a circular diagram $\Delta'$ over the disk presentation of $G_\Omega(\textbf{M}^\pazocal{L})$ with $\lab(\partial\Delta')\equiv\lab(\partial\Phi)$ and $\tau_2(\Delta')=\tau_2(\Phi)$.

As above, the value of $s(\Phi)$ then determines a subdiagram $\Gamma'$ of $\Delta'$ which is a profile of size $\ell$ with history $\theta(s)_1$ and defining configuration $I(w)$.  As $\ell(I(w)\cdot\theta(s)_1)=1$, again \Cref{basic transpose a} provides a circular diagram $\Phi_0'$ with $\lab(\partial\Phi_0')\equiv\lab(\partial\Gamma')$ and $\tau_2(\Phi_0')=(1,L-1-\ell)$.  

Thus, the circular diagram $\Phi'$ obtained from $\Delta'$ by replacing $\Gamma'$ with $\Phi_0'$ satisfies the statement.

\end{proof}

Note that it is an immediate consequence that for any ($2$-)minimal diagram $\Delta$ over the disk presentation of $G_\Omega(\textbf{M}^\pazocal{L})$, the size of any profile contained in $\Delta$ is at most $(L-1)/2$.

Now, as with the transposition of a $\theta$-band and an $a$-cell, the following analogue of \Cref{minimal diskless is M-minimal} shows that this observation applies in a more general setting:

\begin{lemma} \label{minimal t-spokes theta-band}

Let $\Pi$ be a disk and $\pazocal{T}$ be a maximal $\theta$-band in a reduced minimal diagram $\Delta$.  Then $\pazocal{T}$ crosses at most $(L-1)/2$ positive $t$-spokes of $\Pi$.

\end{lemma}

\begin{proof}

Suppose $\Delta$ is a reduced diagram over the disk presentation of $G_\Omega(\textbf{M}^\pazocal{L})$ containing a $\theta$-band $\pazocal{T}$ which crosses $\ell>(L-1)/2$ positive $t$-spokes of a disk $\Pi$.

As in the proof of \Cref{minimal diskless is M-minimal}, a subband $\pazocal{T}_0$ of $\pazocal{T}$, a subpath $\textbf{x}$ of $\partial\Pi$, and the positive $t$-spokes of $\Pi$ that cross $\pazocal{T}$ bound a subdiagram $\Delta_0$ (see \Cref{fig-a-bands}).  Assuming $\Delta_0$ has minimal area among all subdiagrams arising in this way, then using \Cref{pure scope} in place of \Cref{pure a-scope} (and a parameter choice for $L$), the same argument as that employed in \Cref{minimal diskless is M-minimal} implies $\Delta_0$ contains no disk.

Perhaps replacing $\Delta_0$ with a minimal diagram with the same contour label, Lemmas \ref{M-minimal theta-annuli} and \ref{minimal diskless is M-minimal} imply that $\Delta_0$ contains no maximal $\theta$-band distinct from $\pazocal{T}_0$.  Hence, any cell of $\Delta_0$ which is not a part of $\pazocal{T}_0$ is an $a$-cell.

Finally, as in the proof of \Cref{minimal diskless is M-minimal}, perhaps passing to the $\theta$-band $\overline{\pazocal{T}}$ with opposite direction, it may be assumed that $\textbf{x}$ and $\textbf{bot}(\pazocal{T}_0)$ have the same endpoints.

But then the subdiagram $\Phi$ of $\Delta$ consisting of $\Pi$ and $\Delta_0$ is a profile of size $\ell$ with associated $\theta$-band $\pazocal{T}$, so that \Cref{transpose a} implies $\Delta$ is not minimal.

\end{proof}

\begin{lemma} \label{disk theta-annuli}

A reduced minimal diagram contains no $\theta$-annuli.

\end{lemma}

\begin{proof}

Suppose the reduced minimal diagram $\Delta$ contains a $\theta$-annulus $\pazocal{S}$.  Let $\Delta_\pazocal{S}$ be the subdiagram of $\Delta$ bounded by the outer contour of $\pazocal{S}$.

By Lemmas \ref{minimal is smooth}, \ref{minimal diskless is M-minimal}, and \ref{M-minimal theta-annuli}, $\Delta_\pazocal{S}$ must contain a disk.  So, \Cref{Gamma special cell} yields a disk $\Pi$ of $\Delta_\pazocal{S}$ such that $L-4$ consecutive $t$-spokes of $\Pi$ (in $\Delta_\pazocal{S}$) have ends on $\partial\Delta_\pazocal{S}$.  But taking $L\geq8$, then $L-4>(L-1)/2$ and so the $\theta$-band $\pazocal{S}$ and the disk $\Pi$ provide a contradiction to \Cref{minimal t-spokes theta-band}.

\end{proof}

As a consequence, we arrive at the following statement, essential for the proof that the maps defined in Sections 13-14 for the proof of Theorem A are embeddings:

\begin{lemma} \label{all A's is Lambda}

Suppose $\Delta$ is a reduced minimal diagram such that $\lab(\partial\Delta)$ is a word over $\pazocal{A}^{\pm1}$.  Then, letting $k$ be the number of $a$-cells of $\Delta$, $\lab(\partial\Delta)$ is freely equal to a product $w_1\dots w_k$ such that each $w_i$ is a word over $\pazocal{A}^{\pm1}$ freely conjugate to an element of $\Lambda^\pazocal{A}$.

\end{lemma}

\begin{proof}

First, note that the conclusion is independent of the vertex from which $\lab(\partial\Delta)$ is read.


As $\partial\Delta$ consists entirely of $a$-edges, Lemmas \ref{disk theta-annuli} and \Cref{Gamma special cell} imply every (positive) cell in $\Delta$ is an $a$-cell.  So, if $k=0$, then $\Delta$ is a diagram over the free group, meaning $\lab(\partial\Delta)$ is freely trivial, and so the statement is trivially satisfied.


Otherwise, assuming $k\geq1$, Lemmas \ref{minimal is smooth}, \ref{minimal diskless is M-minimal}, and \ref{Gamma_a special cell} yield an $a$-cell $\pi$ and $\ell\geq|\partial\pi|_\pazocal{A}-6$ consecutive $\pazocal{A}$-edges $\textbf{e}_1,\dots,\textbf{e}_\ell$ of $\partial\pi$ such that each $\pazocal{A}$-band $\pazocal{U}(\textbf{e}_i)$ has an end on $\partial\Delta$ and there are no $a$-cells between these $\pazocal{A}$-bands.  In particular, each of these edges is shared with $\partial\Delta$, and so  $\lab(\partial\pi)\in\Lambda^\pazocal{A}$ by \Cref{M Lambda semi-computations}.


Fix $1\leq i\leq\ell$ and let $\textbf{s}_i$ be the complement of $\textbf{e}_i$ in $\partial\pi$, {\frenchspacing i.e. so that $\partial\pi=\textbf{e}_i\textbf{s}_i$}.  Note that condition (L3) implies $w_i=\lab(\textbf{e}_i\textbf{s}_i)\in\Lambda^\pazocal{A}$.  Further, let $\textbf{t}_i$ be the path such that $\partial\Delta=\textbf{e}_i\textbf{t}_i$.



Using $0$-refinement to push the other edges away from the boundary, we may then excise $\pi$ from $\Delta$ by cutting along $\textbf{s}_i$, yielding a reduced minimal diagram $\Delta'$ with $\partial\Delta'=\textbf{s}_i^{-1}\textbf{t}_i$.  As $\lab(\textbf{s}_i^{-1})$ is a subword of $\lab(\partial\pi)\in\Lambda^\pazocal{A}$, $\lab(\partial\Delta')$ is a word over $\pazocal{A}^{\pm1}$.  So, since $\Delta'$ consists of $k-1$ $a$-cells, an inductive argument implies $\lab(\textbf{s}_i^{-1}\textbf{t}_i)=_{F(\pazocal{A})}w_2\dots w_k$ where each word $w_2,\dots,w_k$ is a word over $\pazocal{A}^{\pm1}$ freely conjugate to an element of $\Lambda^\pazocal{A}$.

Thus, $\lab(\textbf{e}_i\textbf{t}_i)=_{F(\pazocal{A})}\lab(\textbf{e}_i\textbf{s}_i)\lab(\textbf{s}_i^{-1}\textbf{t}_i)=_{F(\pazocal{A})}w_1w_2\dots w_k$, implying the statement.

%
%
%

\end{proof}

Similarly, the next two statements are essential for establishing the malnormality of the embeddings (see \Cref{no counterexample annulus}):

\begin{lemma} \label{equal b's are freely equal}

Let $w_1$ and $w_2$ are reduced words over $\pazocal{B}^{\pm1}$.  Identifying $\pazocal{B}$ with the corresponding subset of the tape alphabet of the `special' input sector, suppose $w_1$ and $w_2$ represent the same element of $G_\Omega(\textbf{M}^\pazocal{L})$.  Then $w_1\equiv w_2$.

\end{lemma}

\begin{proof}

By \Cref{generic minimal diagram}, there exists a reduced minimal diagram $\Delta$ over the disk presentation of $G_\Omega(\textbf{M}^\pazocal{L})$ which satisfies $\lab(\partial\Delta)\equiv w_1w_2^{-1}$.  As in the previous proof, Lemmas \ref{disk theta-annuli} and \Cref{Gamma special cell} imply every (positive) cell in $\Delta$ is an $a$-cell.  

But Lemmas \ref{minimal is smooth} and \ref{minimal diskless is M-minimal} imply $\Delta$ is a smooth $M$-minimal diagram, so that \Cref{Gamma_a special cell} implies $\Delta$ contains no $a$-cells.  Hence, $\Delta$ is a circular diagram over the free group, so that the statement follows from the hypothesis that $w_1$ and $w_2$ are reduced.

\end{proof}

\begin{lemma} \label{compressed semi-trapezia sides not equal}

Let $\Delta$ be a compressed semi-trapezium over $M(\textbf{M}^\pazocal{L})$ in the `special' input sector with $\lab(\mathscr{C}\textbf{bot}(\Delta))\equiv y_1^{\delta_1}\dots y_k^{\delta_k}$ for some $y_i\in\pazocal{A}$ and $\delta_i\in\{\pm1\}$.  Suppose:

\begin{enumerate}

\item $y_1^{\delta_1}\dots y_k^{\delta_k}$ is cyclically reduced

\item $\delta_1\neq-1$ or $\delta_k\neq1$

\item The history of $\Delta$ can be factored as $\theta(s)_1H\theta(s)_1^{-1}$

\end{enumerate}
Then the label of the sides of $\Delta$ are not equal in $G_\Omega(\textbf{M}^\pazocal{L})$.

\end{lemma}

\begin{proof}

As the history of $\Delta$ is reduced, $H$ must be a non-trivial word consisting entirely of working rules.  In particular, letting $H\equiv\theta_1\dots\theta_\ell$, there exists $z_j\in\pazocal{A}\cup\pazocal{B}$ and $\eps_j\in\{\pm1\}$ such that $\theta_j$ is the copy of the rule $\theta_{z_j}^{\eps_j}$ of $\textbf{M}_1^\pazocal{A}$ in $\Theta_1$.

Now, for any rule $\theta\in\Theta_1$, let $\theta'$ be the copy of $\theta$ in $T$ which is used to define the $(\theta,a)$-relations corresponding to the `special' input sector.  

Then, letting $\partial\Delta=\textbf{p}_1^{-1}\textbf{q}_1\textbf{p}_2\textbf{q}_2^{-1}$ be the standard factorization of $\Delta$, we have:

\begin{itemize}

\item If $\delta_1=1$, then $\lab(\textbf{p}_1)=\theta(s)_1'\left(\prod\limits_{j=1}^\ell(\theta_j'v(z_j,y_1))^{\eps_j}\right)(\theta(s)_1')^{-1}$

\item If $\delta_1=-1$, then $\lab(\textbf{p}_1)=\theta(s)_1'\left(\prod\limits_{j=1}^\ell (\theta_j')^{\eps_j}\right)(\theta(s)_1')^{-1}$

\item If $\delta_k=1$, then $\lab(\textbf{p}_2)=\theta(s)_1'\left(\prod\limits_{j=1}^\ell(\theta_j'v(z_j,y_k))^{\eps_j}\right)(\theta(s)_1')^{-1}$

\item If $\delta_k=-1$, then $\lab(\textbf{p}_1)=\theta(s)_1'\left(\prod\limits_{j=1}^\ell (\theta_j')^{\eps_j}\right)(\theta(s)_1')^{-1}$

\end{itemize}

Note that the definition of the rules of $\textbf{M}_1^\pazocal{A}$ dictates that all letters $\theta_j'$ commute with any $b$-letters in these products.

First, suppose $\delta_1=1=\delta_k$.  Then, assuming that the statement is false, the word $\prod_{j=1}^\ell v(z_j,y_1)^{\eps_j}$ must represent the identity in $G_\Omega(\textbf{M}^\pazocal{L})$.  \Cref{equal b's are freely equal} then implies that this word is freely trivial.  But \Cref{free subgroup} then implies that $H$ must be freely trivial.

Similarly, if $\delta_1=-1=\delta_k$, then assuming the statement is false implies word $\prod_{j=1}^\ell v(z_j,y_k)^{\eps_j}$ is freely trivial, which yields a contradiction in the same way.

Finally, suppose $\delta_1=1$ and $\delta_k=-1$.  Then, assuming the statement is false, the words $\prod_{j=1}^\ell v(z_j,y_1)^{\eps_j}$ and $\prod_{j=1}^\ell v(z_j,y_k)^{\eps_j}$ must be equal in $G_\Omega(\textbf{M}^\pazocal{L})$, and so must be freely equal by \Cref{equal b's are freely equal}.  But then \Cref{free subgroup} implies that $y_1=y_k$, so that the word $y_1^{\delta_1}\dots y_k^{\delta_k}$ is not cyclically reduced.

\end{proof}

\medskip


\subsection{Upper bound on weights} \label{sec-upper-bound} \

To aid with the weight estimates established in the next section, we now study the arrangement of particular maximal bands in a reduced minimal diagram.  This is done in an analogous manner as that employed in \Cref{sec M-minimal upper bound} to study of the positive $\pazocal{A}$- and $b$-bands of an $M$-minimal diagram.

Let $\Pi$ be a disk in a reduced minimal diagram $\Delta$ and let $\pazocal{Q}$ be a maximal positive $q$-band in $\Delta$ which has an end on $\Pi$.  If $\pazocal{Q}$ has an end on another disk, then $\pazocal{Q}$ is called an \textit{internal $q$-band} in $\Delta$.  Otherwise, $\pazocal{Q}$ is called an \textit{external $q$-band} in $\Delta$.

Note that the makeup of the disk relations dictates that no $q$-band can have two ends on the same disk.  In particular, the internal $t$-bands of $\Delta$ correspond to the internal edges of $\Gamma(\Delta)$.

For a reduced minimal diagram $\Delta$, define the following values:

\begin{itemize}

\item $\rho_i(\Delta)$ is the number of internal $q$-bands in $\Delta$

\item $\rho_e(\Delta)$ is the number of external $q$-bands in $\Delta$

\item $\mu_q(\Delta)$ is the number of $(\theta,q)$-cells in $\Delta$

\end{itemize}

The next statement then provides an analogue of \Cref{internal vs external bands} in this setting:

\begin{lemma} \label{internal vs external q-bands}

If $\Delta$ is a reduced minimal diagram, then $\rho_i(\Delta)\leq\rho_e(\Delta)$.

\end{lemma}

\begin{proof}

The proof proceeds by induction on the number $n$ of disks in $\Delta$, with the statement clear if $n=0,1$ as then $\rho_i(\Delta)=0$.

Let $\Pi$ be the disk and $\pazocal{Q}_1,\dots,\pazocal{Q}_{L-4}$ the consecutive positive $t$-spokes of $\Pi$ given by \Cref{Gamma special cell}.  Let $\textbf{e}_1,\dots,\textbf{e}_{L-4}$ be the $t$-edges of $\partial\Pi$ such that $\pazocal{Q}(\textbf{e}_i)$ is the $t$-spoke corresponding to $\pazocal{Q}_i$.  

Let $\textbf{s}_1$ be the subpath of $\partial\Pi$ with first edge $\textbf{e}_1$ and last edge $\textbf{e}_{L-4}$.  Then, letting $\textbf{s}_2$ be the complement of $\textbf{s}_1$ in $\partial\Pi$, let $\textbf{p}=\textbf{top}(\pazocal{Q}(\textbf{e}_{L-4}))^{-1}\textbf{s}_2\textbf{bot}(\pazocal{Q}(\textbf{e}_1))$.

Cutting along $\textbf{p}$ separates $\Delta$ into two subdiagrams $\bar{\Delta}_1$ and $\Delta_2$, where $\bar{\Delta}_1$ is the subdiagram consisting of $\Pi$ and the subdiagrams $\Gamma_1,\dots,\Gamma_{L-5}$ defined in \Cref{Gamma special cell}.  Let $\Delta_1$ be the subdiagram of $\bar{\Delta}_1$ obtained by removing $\Pi$ (see \Cref{fig-counterexample}).

By construction, $\Delta_2$ is a minimal diagram containing $n-1$ disks, so that the inductive hypothesis implies $\rho_i(\Delta_2)\leq\rho_e(\Delta_2)$.

\begin{figure}[H]
\centering
\includegraphics[scale=0.85]{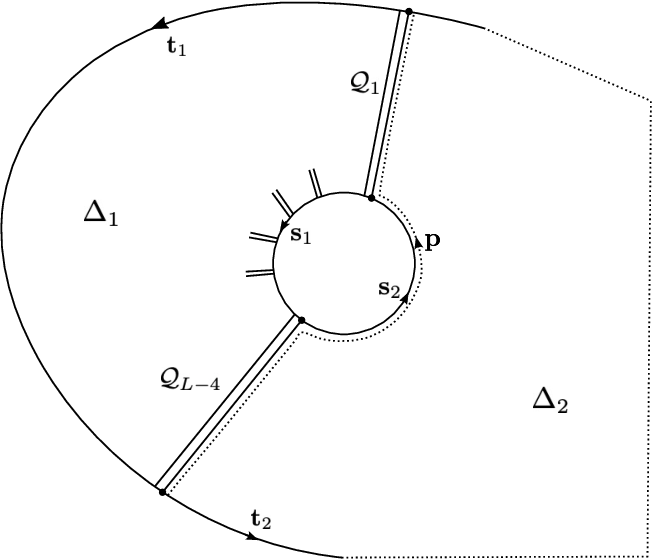}
\caption{Reduced minimal diagram $\Delta$}
\label{fig-counterexample}
\end{figure}

Note that any external $q$-band of $\Delta$ that has an end on a disk other than $\Pi$ corresponds to an external $q$-band of $\Delta_2$.  Similarly, any internal $q$-band of $\Delta$ that does not have an end on $\Pi$ corresponds to an internal $q$-band of $\Delta_2$.

Now, for any internal $q$-band of $\Delta$ which has an end on $\Pi$, this end must be an edge of $\textbf{s}_2^{\pm1}$.  On the other hand, each of these bands corresponds to an external $q$-band in $\Delta_2$.  Every $q$-edge of $\textbf{s}_1$, conversely, corresponds to an external $q$-band of $\Delta$ which is removed entirely when passing to $\Delta_2$.

Hence, $\rho_i(\Delta)\leq\rho_i(\Delta_2)+|\textbf{s}_2|_q$ and $\rho_e(\Delta_2)\leq\rho_e(\Delta)-|\textbf{s}_1|_q+|\textbf{s}_2|_q$, so that $$\rho_i(\Delta)\leq\rho_i(\Delta_2)+|\textbf{s}_2|_q\leq\rho_e(\Delta_2)+|\textbf{s}_2|_q\leq\rho_e(\Delta)-|\textbf{s}_1|_q+2|\textbf{s}_2|_q$$
As $|\textbf{s}_2|_t=3$, though, the makeup of the disk relations implies $|\textbf{s}_2|_q\leq 10(N+1)+4$.  But since $|\textbf{s}_1|_t=L-4$, the parameter choice $L>>N$ yields $|\textbf{s}_2|_q\leq\frac{1}{2}|\textbf{s}_1|_q$, thus implying the statement.

\end{proof}

\begin{lemma} \label{minimal rho}

If $\Delta$ is a reduced minimal diagram, then $\mu_q(\Delta)\leq\|\partial\Delta\|^2$.

\end{lemma}

\begin{proof}

By \Cref{basic annuli 1}(2), any maximal positive $q$-band of $\Delta$ which is not one must have an end on $\partial\Delta$.  So, letting $\rho(\Delta)$ be the number of maximal positive $q$-bands in $\Delta$, $\rho(\Delta)-\rho_i(\Delta)\leq\|\partial\Delta\|$.  But \Cref{internal vs external q-bands} implies $\rho_i(\Delta)\leq\rho_e(\Delta)$, and so $\rho(\Delta)\geq\rho_i(\Delta)+\rho_e(\Delta)\geq2\rho_i(\Delta)$.  Hence, $\rho(\Delta)\leq2\|\partial\Delta\|$.

\Cref{disk theta-annuli}, on the other hand, implies that any maximal $\theta$-band in $\Delta$ must have two ends on $\partial\Delta$.  So, the number of maximal positive $\theta$-bands in $\Delta$ is at most $\frac{1}{2}\|\partial\Delta\|$.  Thus, the statement follows by \Cref{basic annuli 1}(1).

\end{proof}

Now, let $a\in Y_j^\pazocal{L}(i)\cup\pazocal{Y}_j^\pazocal{L}(i)$ for $i\geq2$ and any $j$, {\frenchspacing i.e. $a$ is} an $a$-letter from the tape alphabet of either the $Q_{j-1}^\pazocal{L}(i)Q_j^\pazocal{L}(i)$- or $(R_j^\pazocal{L}(i))^{-1}(R_{j-1}^\pazocal{L}(i))^{-1}$-sector of the standard base.  Then $a^{\pm1}$ is called an \textit{unrestricted $a$-letter}.

As with other types of letters, an $a$-edge $\textbf{e}$ in a reduced minimal diagram is called an \textit{unrestricted $a$-edge} if $\lab(\textbf{e})$ is an unrestricted $a$-letter.  Accordingly, the unrestricted $a$-edges are partitioned into three types: \textit{unrestricted $\pazocal{A}$-edges}, \textit{unrestricted $b$-edges}, and \textit{unrestricted ordinary $a$-edges}.  

Further, $(\theta,a)$-cells and $a$-bands in reduced minimal diagrams are called \textit{unrestricted $(\theta,a)$-cells} and \textit{unrestricted $a$-bands} if they correspond to unrestricted $a$-letters.  

Naturally, \textit{unrestricted $(\theta,\pazocal{A})$-cells}, \textit{unrestricted $(\theta,b)$-cells}, \textit{unrestricted ordinary $(\theta,a)$-cells}, \textit{unrestricted $\pazocal{A}$-bands}, \textit{unrestricted $b$-bands}, and \textit{unrestricted ordinary $a$-bands} are defined in the obvious way.   Note that no unrestricted $a$-band can have an end on an $a$-cell.

\begin{lemma} \label{pinched disk unrestricted}

Let $\pazocal{U}$ be a maximal unrestricted $a$-band in a reduced minimal diagram $\Delta$.  If $\pazocal{U}$ has two ends on disks, then these ends are on distinct disks.

\end{lemma}

\begin{proof}

Assume toward contradiction that $\pazocal{U}$ has two ends on the disk $\Pi$.  By \Cref{a-bands on disk}, $\pazocal{U}$ must then be an $a$-band of length 0, so that $\Pi$ is a pinched disk.  

So, $\pazocal{U}$ has a unique defining edge $\textbf{e}$, and both $\textbf{e}$ and $\textbf{e}^{-1}$ are edges of $\partial\Pi$.  Let $\textbf{s}$ be the pinched subpath of $\partial\Pi$ containing the edge $\textbf{e}$ and let $\textbf{s}^{\pm1}\textbf{q}\textbf{s}^{\mp1}\textbf{p}$ be the pinched factorization of $\partial\Pi$ with respect to $\textbf{s}$.  Then, $\textbf{p}^{-1}$ bounds a subdiagram $\Psi_{\Pi,\textbf{s}}$.  

By the definition of disk relators, the label of each edge of $\textbf{p}^{-1}$ must be an $a$-letter corresponding to a letter of the same tape alphabet as $\lab(\textbf{e})$.  In particular, every edge of $\partial\Psi_{\Pi,\textbf{s}}$ is an unrestricted $a$-edge, while \Cref{disk theta-annuli} implies $\Psi_{\Pi,\textbf{s}}$ consists entirely of disks and $a$-cells.
But Lemmas \ref{Gamma special cell} and \ref{disk MM2} quickly rule out the existence of a disk, so that $\Psi_{\Pi,\textbf{s}}$ consists entirely of $a$-cells.  

Since disk relators are cyclically reduced by construction, $\Psi_{\Pi,\textbf{s}}$ must contain at least one $a$-cell.  Lemmas \ref{minimal is smooth}, \ref{minimal diskless is M-minimal}, and \ref{Gamma_a special cell} then produce an $a$-cell $\pi$ in $\Psi_{\Pi,\textbf{s}}$ and at least $C-6$ maximal positive $\pazocal{A}$-bands which have ends on both $\pi$ and $\partial\Psi_{\Pi,\textbf{s}}$.  Taking $C\geq7$, there exists an edge $\textbf{e}'$ of $\partial\Psi_{\Pi,\textbf{s}}$ corresponding to the end of such an $\pazocal{A}$-band $\pazocal{U}'$.  

But then $\textbf{e}'$ is not an unrestricted $a$-edge, yielding a contradiction.

\end{proof}

Note that the proof of \Cref{pinched disk unrestricted} can be expanded statement for $a$-bands which correspond to a subset of the tape alphabet of any sector other than the `special' input sector.  However, the consideration of unrestricted $a$-bands will suffice for our purposes.

A maximal positive unrestricted $\pazocal{A}$-band is called \textit{D-internal} if it has two ends on (distinct) disks.

For any reduced minimal diagram $\Delta$, define the values:

\begin{itemize}

\item $\rho_i'(\Delta)$ is the number of D-internal $\pazocal{A}$-bands in $\Delta$

\item $\rho'(\Delta)$ is the number of maximal positive unrestricted $\pazocal{A}$-bands in $\Delta$

\item $\mu_\pazocal{A}(\Delta)$ is the number of unrestricted $(\theta,\pazocal{A})$-cells in $\Delta$


\end{itemize}

\begin{lemma} \label{internal vs external A-bands}

If $\Delta$ is a reduced minimal diagram, then $\rho_i'(\Delta)\leq\frac{1}{2}\rho'(\Delta)$.

\end{lemma}

\begin{proof}

The proof follows by induction on the number $n$ of disks in $\Delta$, with base cases $n=0,1$ following immediately from \Cref{pinched disk unrestricted}.

Let $\Pi$ be the disk and $\pazocal{Q}_1,\dots,\pazocal{Q}_{L-4}$ be the consecutive positive $t$-spokes of $\Pi$ given by \Cref{Gamma special cell}.  Let $\partial\Pi=\textbf{s}_1\textbf{s}_2$ and define the subdiagrams $\Delta_1$ and $\Delta_2$ as in the proof of \Cref{internal vs external q-bands} (see \Cref{fig-counterexample}).  Finally, let $\partial\Delta=\textbf{t}_1\textbf{t}_2$ where $\textbf{t}_i$ is a subpath of $\partial\Delta_i$.  The inductive hypothesis implies $\rho_i'(\Delta_2)\leq\frac{1}{2}\rho'(\Delta_2)$.

Note that any internal $\pazocal{A}$-band of $\Delta$ which does not have an end on $\Pi$ corresponds to an internal $\pazocal{A}$-band of $\Delta_2$.  Conversely, for any internal $\pazocal{A}$-band of $\Delta$ which has an end on $\Pi$, this end must be an unrestricted $\pazocal{A}$-edge of $\textbf{s}_2^{\pm1}$.

Let $W$ be the configuration corresponding to $\lab(\partial\Pi)$.  Then, the parallel nature of the rules implies $|W(j)|_\pazocal{A}=|W(2)|_\pazocal{A}$ for each $j\in\{2,\dots,L\}$.  In particular, the number of unrestricted $\pazocal{A}$-edges of $\textbf{s}_2$ is equal to $4|W(2)|_\pazocal{A}$.  Hence, $\rho_i'(\Delta)\leq\rho_i'(\Delta_2)+4|W(2)|_\pazocal{A}\leq\frac{1}{2}\rho'(\Delta_2)+4|W(2)|_\pazocal{A}$.

However, each of the $(L-5)|W(2)|_\pazocal{A}$ unrestricted $\pazocal{A}$-edges of $\textbf{s}_1$ corresponds to a maximal positive unrestricted $\pazocal{A}$-band of $\Delta$ which cannot be internal.  This implies $\rho'(\Delta)\geq\rho'(\Delta_2)+(L-5)|W(2)|_\pazocal{A}$, {\frenchspacing i.e. $\rho_i'(\Delta)\leq\frac{1}{2}\rho'(\Delta)-\frac{L-5}{2}|W(2)|_\pazocal{A}+4|W(2)|_\pazocal{A}$}.  Thus, the statement follows by taking $L\geq13$.

\end{proof}

\begin{lemma} \label{minimal rho A}

If $\Delta$ is a reduced minimal diagram, then $\mu_\pazocal{A}(\Delta)\leq|\partial\Delta|_\theta(\mu_q(\Delta)+|\partial\Delta|_\pazocal{A})$.

\end{lemma}

\begin{proof}

Note that any maximal positive unrestricted $\pazocal{A}$-band which is not internal must have one end which is on a $(\theta,q)$-cell or on $\partial\Delta$.  So, since any $(\theta,q)$-cell has at most one boundary $\pazocal{A}$-edge, $\rho'(\Delta)-\rho_i'(\Delta)\leq \mu_q(\Delta)+|\partial\Delta|_\pazocal{A}$.  Hence, \Cref{internal vs external A-bands} implies $\rho'(\Delta)\leq2(\mu_q(\Delta)+|\partial\Delta|_\pazocal{A})$, so that the statement follows by Lemmas \ref{disk theta-annuli} and \ref{basic annuli 2}(1).


\end{proof}


\begin{lemma} \label{minimal number of disks}

If $\Delta$ is a reduced minimal diagram with $m_2\geq1$ disks, then $m_2+1\leq|\partial\Delta|_q$.

\end{lemma}

\begin{proof}


If $m_2=1$, then for $\Pi$ the unique disk of $\Delta$, every $q$-edge of $\partial\Pi$ is an end of a maximal $q$-band which also ends on $\partial\Delta$.  As a result, $|\partial\Delta|_q\geq|\partial\Pi|_q=L(2N+3)$.


Now suppose $m_2\geq2$.  Let $\Pi$ be the disk and $\pazocal{Q}_1,\dots,\pazocal{Q}_{L-4}$ be the consecutive positive $t$-spokes of $\Pi$ given by \Cref{Gamma special cell}.  Further, let $\Delta_1$ and $\Delta_2$ be the subdiagrams of $\Delta$ and $\textbf{s}_i$, $\textbf{t}_i$, and $\textbf{p}$ be the paths as in \Cref{internal vs external q-bands} (see \Cref{fig-counterexample}).  Every $q$-edge of $\textbf{s}_1$ is a defining edge (and an end) of a maximal $q$-band of $\Delta_1$ which must have an end on $\textbf{t}_1$.  So, $|\textbf{t}_1|_q\geq|\textbf{s}_1|_q\geq(L-5)(2N+3)$.  Conversely, since the sides of $q$-bands contain no $q$-edges, $|\textbf{p}|_q=|\textbf{s}_2|_q\leq5(2N+3)$.  In particular, since $\partial\Delta=\textbf{t}_1\textbf{t}_2$ while $\partial\Delta_2=\textbf{p}^{-1}\textbf{t}_2$, parameter choices yield $|\partial\Delta|_q-|\partial\Delta_2|_q\geq(L-10)(2N+3)\geq1$.

But $\Delta_2$ is a reduced minimal diagram containing $m_2-1\geq1$ disks, so that an inductive argument implies $m_2\leq|\partial\Delta_2|_q\leq|\partial\Delta|_q-1$.

\end{proof}

Thus, taking $\mu(\Delta)=\mu_q(\Delta)+\mu_\pazocal{A}(\Delta)$ for any reduced minimal diagram $\Delta$, we arrive at the following upper bound for the weight of such diagrams:

\begin{lemma} \label{minimal weight bound}

If $\Delta$ is a reduced minimal diagram with $m_2$ disks, then for $m_1=\|\partial\Delta\|+\mu(\Delta)$:
$$\text{wt}(\Delta)\leq (m_2+1)\left(Lm_1^4+g_\pazocal{L}(Lm_1^3)+f_\pazocal{L}(m_1)\right)$$

\end{lemma}

\begin{proof}

The proof follows by induction on the number $m_2$ of disks in the diagram.

If $m_2=0$, then Lemmas \ref{minimal diskless is M-minimal} and \ref{diskless weight} imply $\text{wt}(\Delta)\leq c_0\|\partial\Delta\|^4+g_\pazocal{L}(c_0\|\partial\Delta\|^3)$.  So, the statement follows from $|\partial\Delta|\leq m_1$ and the parameter choice $L>>c_0$.

Now, suppose $m_2\geq1$ and let $\Pi$ be the disk and $\pazocal{Q}_1,\dots,\pazocal{Q}_{L-4}$ be the consecutive positive $t$-spokes of $\Pi$ given by \Cref{Gamma special cell}.  Further, let $\Delta_1$ and $\Delta_2$ be the subdiagrams of $\Delta$ and $\textbf{s}_i$ and $\textbf{t}_i$ be the paths as in \Cref{internal vs external q-bands} (see \Cref{fig-counterexample}).

Let $\textbf{p}_1$ be the subpath of $\partial\Delta_1$ such that $\textbf{p}_1^{\pm1}$ is a side of $\pazocal{Q}_1$.  Then, as $\pazocal{Q}_1$ is a $t$-band, $h_1=\|\textbf{p}_1\|$ is the length of the band's history.  Similarly, define the subpath $\textbf{p}_{L-4}$ of $\partial\Delta_1$ and set $h_{L-4}=\|\textbf{p}_{L-4}\|$.

Let $\textbf{e}$ be an edge of $\pazocal{Q}_1$.  Then, $\textbf{e}$ is a $\theta$-edge, and so there exists a maximal $\theta$-band $\pazocal{T}_\textbf{e}'$ of $\Delta_1$ for which $\textbf{e}$ is a defining edge (and an end).  Since $\textbf{s}_1$ contains no $\theta$-edges, \Cref{basic annuli 1}(1) implies that $\pazocal{T}_\textbf{e}'$ must also have an end on either $\textbf{t}_1$ or on $\textbf{p}_{L-4}$.  By \Cref{disk theta-annuli}, there exists a unique maximal $\theta$-band $\pazocal{T}_\textbf{e}$ which contains $\pazocal{T}_\textbf{e}'$ as a subband.  If $\pazocal{T}_\textbf{e}'$ has an end $\textbf{p}_{L-4}$, then $\pazocal{T}_\textbf{e}$ must cross every $t$-spoke $\pazocal{Q}_i$ of $\Pi$.  But the parameter choice $L\geq8$ then implies that $\pazocal{T}_\textbf{e}$ and $\Pi$ form a counterexample to \Cref{minimal t-spokes theta-band}.
Hence, $\pazocal{T}_\textbf{e}'$ must have an end on $\textbf{t}_1$.  

Similarly, every edge of $\textbf{p}_{L-4}$ is a defining edge of a maximal $\theta$-band of $\Delta_1$ which has an end on $\textbf{t}_1$.  Thus, $h_1+h_{L-4}\leq|\textbf{t}_1|_\theta$.

Next, let $W$ be the accepted configuration of $\textbf{M}^\pazocal{L}$ with $\lab(\partial\Pi)\equiv W^{\pm1}$.  Then, the parallel nature of the machine and \Cref{M component size} imply that $|W(1)|_a\leq|W(2)|_a=|W(j)|_a$ for every $j\in\{2,\dots,L\}$.  So, by construction:
\begin{itemize}

\item $(L-5)|W(2)|_a\leq|\textbf{s}_1|_a\leq(L-4)|W(2)|_a$

\item $(L-5)(2N+3)+1\leq|\textbf{s}_1|_q\leq(L-4)(2N+3)+1$

\item $|\textbf{s}_2|_a\leq5|W(2)|_a$ 

\item $|\textbf{s}_2|_q\leq5(2N+3)$

\end{itemize}

Taking $L\geq10$, it then follows that $|\textbf{s}_1|_q\geq5(2N+3)+1\geq|\textbf{s}_2|_q+1$.  Note that each $q$-edge of $\textbf{s}_1$ corresponds to a maximal $q$-band of $\Delta_1$ which, by the makeup of disk relations, must have an end on $\textbf{t}_1$.  So, $|\textbf{s}_1|_q\leq|\textbf{t}_1|_q$.  Since $\partial\Delta_2=\textbf{t}_2(\textbf{p}_{L-4}\textbf{s}_2\textbf{p}_1)^{-1}$, this implies: 
\begin{align*}
\|\partial\Delta_2\|&=\|\textbf{t}_2\|+h_1+h_{L-4}+\|\textbf{s}_2\|\leq\|\textbf{t}_2\|+|\textbf{t}_1|_\theta+|\textbf{s}_2|_q+|\textbf{s}_2|_a \\
&\leq\|\textbf{t}_2\|+|\textbf{t}_1|_\theta+(|\textbf{s}_1|_q-1)+|\textbf{s}_2|_a\leq\|\textbf{t}_2\|+|\textbf{t}_1|_\theta+|\textbf{t}_1|_q+|\textbf{s}_2|_a \\
&\leq\|\partial\Delta\|-|\textbf{t}_1|_a+5|W(2)|_a
\end{align*}
Similarly, $\partial\Delta_1=\textbf{p}_{L-4}\textbf{s}_1^{-1}\textbf{p}_1\textbf{t}_1$, so that:
\begin{align*}
\|\partial\Delta_1\|&=h_1+h_{L-4}+\|\textbf{t}_1\|+\|\textbf{s}_1\|\leq\|\textbf{t}_1\|+|\textbf{t}_1|_\theta+|\textbf{s}_1|_q+|\textbf{s}_1|_a\leq\|\textbf{t}_1\|+|\textbf{t}_1|_\theta+|\textbf{t}_1|_q+|\textbf{s}_1|_a \\
&\leq2\|\textbf{t}_1\|-|\textbf{t}_1|_a+(L-4)|W(2)|_a
\end{align*}
Now, let $\kappa_o$ be the number of unrestricted ordinary $a$-edges of $\textbf{s}_1$.  Similarly, let $\kappa_\pazocal{A}$ and $\kappa_b$ be the number of unrestricted $\pazocal{A}$- and $b$-edges of $\textbf{s}_1$.  By construction, $\kappa_o+\kappa_\pazocal{A}+\kappa_b=(L-5)|W(2)|_a$.

Let $\textbf{e}$ be an unrestricted $a$-edge of $\textbf{s}_1$ and let $\pazocal{U}$ be the maximal $a$-band of $\Delta_1$ for which $\textbf{e}$ is a defining edge (and end).  Note that if $\pazocal{U}$ has two ends on $\partial\Delta_1$, then \Cref{pinched disk unrestricted} and the makeup of the $(\theta,t)$-relations imply that $\pazocal{U}$ must have an end on $\textbf{t}_1$.

If $\textbf{e}$ is an ordinary $a$-edge, then $\pazocal{U}$ must have an end on a $(\theta,q)$-cell of $\Delta_1$ or on $\textbf{t}_1$.  Since \Cref{simplify rules} implies that any $(\theta,q)$-cell has at most one boundary ordinary $a$-edge, it follows that at least $\max(\kappa_o-\mu_q(\Delta_1),0)$ unrestricted ordinary $a$-edges of $\textbf{s}_1$ are defining edges of $a$-bands which have an end on $\textbf{t}_1$.

Similarly, if $\textbf{e}$ is an $\pazocal{A}$-edge, then $\pazocal{U}$ must have an end on a $(\theta,q)$-cell of $\Delta_1$ or on $\textbf{t}_1$.  The makeup the rules then implies at least $\max(\kappa_\pazocal{A}-\mu_q(\Delta_1),0)$ unrestricted $\pazocal{A}$-edges of $\textbf{s}_1$ are defining edges of $a$-bands which have an end on $\textbf{t}_1$.

Finally, if $\textbf{e}$ is a $b$-edge, then $\pazocal{U}$ must have an end on a $(\theta,q)$-cell of $\Delta_1$, on a $(\theta,\pazocal{A})$-cell of $\Delta_1$, or on $\textbf{t}_1$.  But every $(\theta,q)$- and $(\theta,\pazocal{A})$-cell has at most $D_\pazocal{A}$ boundary $b$-edges, and so at least $\max(\kappa_b-D_\pazocal{A}\mu(\Delta_1),0)$ such edges are defining edges of $a$-bands which have an end on $\textbf{t}_1$.

Hence, accounting for the distinct unrestricted $a$-bands which have ends on both $\textbf{s}_1$ and $\textbf{t}_1$, we have $|\textbf{t}_1|_a\geq(L-5)|W(2)|_a-(D_\pazocal{A}+2)\mu(\Delta_1)$.  

Set $m_1'=\|\partial\Delta_2\|+\mu(\Delta_2)$.  Since $\Delta_2$ consists of $m_2-1$ disks, the inductive hypothesis implies:
$$\text{wt}(\Delta_2)\leq m_2(L(m_1')^4+g_\pazocal{L}(L(m_1')^3)+f_\pazocal{L}(m_1'))$$

%
%
%

\textbf{1.} Suppose $|W(2)|_a\leq\frac{2(D_\pazocal{A}+2)}{L-4}\mu(\Delta_1)$.  

The parameter choice $L>>C$ (recalling that $D_\pazocal{A}$ depends on $C$) then implies 
$$\|\partial\Delta_2\|\leq\|\partial\Delta\|+5|W(2)|_a\leq\|\partial\Delta\|+\mu(\Delta_1)$$  As a result, noting that $\mu(\Delta)=\mu(\Delta_1)+\mu(\Delta_2)$, we have $m_1'\leq m_1$.  Since $f_\pazocal{L}$ and $g_\pazocal{L}$ are non-decreasing functions, this means 
$$\text{wt}(\Delta_2)\leq m_2(Lm_1^4+g_\pazocal{L}(Lm_1^3)+f_\pazocal{L}(m_1))$$
Further, $\|\partial\Delta_1\|\leq2\|\textbf{t}_1\|+(L-4)|W(2)|_a\leq2\|\textbf{t}_1\|+2(D_\pazocal{A}+2)\mu(\Delta_1)\leq 2(D_\pazocal{A}+2)m_1$.  Hence, since $\Delta_1$ is an $M$-minimal diagram, \Cref{diskless weight} and the parameter choices $L>>c_0>>C$ yield $\text{wt}(\Delta_1)\leq Lm_1^4+g_\pazocal{L}(Lm_1^3)$.

Finally, $\|W(2)\|\leq|\textbf{s}_1|_q+|W(2)|_a\leq|\textbf{t}_1|_q+\mu(\Delta_1)\leq \|\partial\Delta\|+\mu(\Delta)=m_1$, and so as a consequence $\text{wt}(\Pi)=f_\pazocal{L}(\|W(2)\|)\leq f_\pazocal{L}(m_1)$.  Thus,
\begin{align*}
\text{wt}(\Delta)&=\text{wt}(\Delta_1)+\text{wt}(\Delta_2)+\text{wt}(\Pi)\leq (m_2+1)(Lm_1^4+g_\pazocal{L}(Lm_1^3)+f_\pazocal{L}(m_1))
\end{align*}

\textbf{2.} Suppose $|W(2)|_a\geq\frac{2(D_\pazocal{A}+2)}{L-4}\mu(\Delta_1)$.

Then, $|\textbf{t}_1|_a\geq(L-5)|W(2)|_a-(D_\pazocal{A}+2)\mu(\Delta_1)\geq(\frac{1}{2}L-3)|W(2)|_a$.  Taking $L\geq16$, this yields $|\textbf{t}_1|_a\geq5|W(2)|_a$, so that $\|\partial\Delta_2\|\leq\|\partial\Delta\|$.  This implies $m_1'\leq m_1$, so that as in the previous case $$\text{wt}(\Delta_2)\leq m_2(Lm_1^4+g_\pazocal{L}(Lm_1^3)+f_\pazocal{L}(m_1))$$
Further, $\|\partial\Delta_1\|\leq2\|\textbf{t}_1\|-|\textbf{t}_1|_a+(L-4)|W(2)|_a\leq2\|\textbf{t}_1\|+|\textbf{t}_1|_a\leq3\|\partial\Delta\|\leq 3m_1$, so that \Cref{diskless weight} and the parameter choice $L>>c_0$ yields $\text{wt}(\Delta_1)\leq Lm_1^4+g_\pazocal{L}(Lm_1^3)$.

Finally, $\|W(2)\|=|W(2)|_a+|\textbf{s}_1|_q\leq |\textbf{t}_1|_a+|\textbf{t}_1|_q\leq\|\partial\Delta\|\leq m_1$, so that $\text{wt}(\Pi)\leq f_\pazocal{L}(m_1)$.  

Thus, the desired bound again follows.

\end{proof}


\begin{lemma} \label{main upper bound}

If $\Delta$ is a reduced minimal diagram with $n=\|\partial\Delta\|$, then:
$$\text{wt}(\Delta)\leq n\left(Kn^{12}+g_\pazocal{L}(Kn^9)+f_\pazocal{L}(Kn^3)\right)$$

\end{lemma}

\begin{proof}

Clearly, we may assume $n\geq1$, as otherwise the diagram may be formed over a free group.
%
As in \Cref{minimal weight bound}, let $m_2$ be the number of disks in $\Delta$ and $m_1=\|\partial\Delta\|+\mu(\Delta)$.  \Cref{minimal number of disks} then implies $m_2+1\leq n$, while Lemmas \ref{minimal rho} and \ref{minimal rho A} yield $\mu(\Delta)\leq n^2+n(n^2+n)=n^3+2n^2$, so that $m_1\leq4n^3$.  Thus, the statement follows from \Cref{minimal weight bound} and the parameter choice $K>>L$.

\end{proof}

\medskip


\section{Annular Diagrams} \label{sec-annular-diagrams}

The goal of this section is to exhibit the malnormality of the subgroup $H_\pazocal{A}=\gen{\pazocal{A}}$ of $G_\Omega(\textbf{M}^\pazocal{L})$ (recall that $\pazocal{A}$ is identified with the subset of the tape alphabet of the `special' input sector).  To achieve this, we study the structure of annular diagrams over the disk presentation of $G_\Omega(\textbf{M}^\pazocal{L})$.  

Given the intricate nature of the necessary arguments, we must access the full power of $0$-refinement, and so generally cannot ignore the presence of $0$-edges and $0$-cells in this section.  This slightly alters some of the ways in which we refer to the structures described in previous sections.  For example, bands may now involve $0$-cells, so that the `defining edge sequence' of a band of length $0$ need not be a single edge but rather a sequence of edges such that any pair of consecutive edges are immediately adjacent. That said, this does not alter these conceptualizations in any meaningful way, as the presence of $0$-cells was simply implicit in previous settings.

Throughout this section, we assume $H_\pazocal{A}$ is not a malnormal subgroup of $G_\Omega(\textbf{M}^\pazocal{L})$ and fix group elements demonstrating this, {\frenchspacing i.e. $g\in G_\Omega(\textbf{M}^\pazocal{L})\setminus H_\pazocal{A}$ and $h_1,h_2\in H_\pazocal{A}\setminus\{1\}$} with $g^{-1}h_1g=h_2$.  


Let $\Delta$ be an annular diagram over the disk presentation of $G_\Omega(\textbf{M}^\pazocal{L})$ such that:

\begin{itemize}

\item There exists a vertex $O_1$ of the outer contour $\textbf{q}_1$ of $\Delta$ such that the word $w_1$ given by reading $\lab(\textbf{q}_1)$ from $O_1$ is a word over $\pazocal{A}\cup\pazocal{A}^{-1}$ that represents $h_1$ in $H_\pazocal{A}$

\item There exists a vertex $O_2$ of the inner contour $\textbf{q}_2$ of $\Delta$ such that the word $w_2$ given by reading $\lab(\textbf{q}_2^{-1})$ from $O_2$ is a word over $\pazocal{A}\cup\pazocal{A}^{-1}$ that represents $h_2$ in $H_\pazocal{A}$

\item There exists a path $\textbf{t}$ in $\Delta$ with $\textbf{t}_-=O_1$ and $\textbf{t}_+=O_2$ such that $\lab(\textbf{t})$ is a word $u$ over $\pazocal{X}\cup\pazocal{X}^{-1}$ that represents $g$ in $G_\Omega(\textbf{M}^\pazocal{L})$

\end{itemize}

Then, $\Delta$ is called a \textit{counterexample annulus}.  

In this case, the path $\textbf{t}$ is called a \textit{$g$-path} and the tuple of words $(u,w_1,w_2)$ is called the \textit{defining triple} of $\Delta$ with respect to $\textbf{t}$.  Note that for any word $u\in(\pazocal{X}\cup\pazocal{X}^{-1})^*$ that represents $g$ in $G_\Omega(\textbf{M}^\pazocal{L})$ and any pair of words $w_1,w_2\in(\pazocal{A}\cup\pazocal{A}^{-1})^*$ such that $w_i$ represents $h_i$ in $H_\pazocal{A}$, van Kampen's Lemma (see \Cref{sec-initial-embedding}) implies the existence of a counterexample annulus for which $(u,w_1,w_2)$ is a defining triple.  Hence, by hypothesis there must exist counterexample annuli.

The diagram $\Delta$ is called a \textit{minimal counterexample annulus} if $\tau(\Delta)\leq\tau(\Delta')$ for any counterexample annulus $\Delta'$.  A \textit{$j$-minimal counterexample annulus} is defined analogously.  Note that the existence of counterexample annuli implies the existence of ($j$-)minimal counterexample annuli.

A counterexample annulus $\Delta$ is called \textit{reduced} if it is a reduced annular diagram over the disk presentation of $G_\Omega(\textbf{M}^\pazocal{L})$.  A \textit{reduced ($j$-)minimal counterexample annulus} is defined analogously.  Note that the existence of a reduced ($j$-)minimal counterexample is slightly more subtle: While we may remove pairs of cancellable cells, we want to be able to do so without altering the $g$-path (or at least some $g$-path).  The next few statements address this subtlety.

\begin{lemma} \label{counterexample annuli trivial path}

If $\textbf{p}$ is a simple closed path in a counterexample annulus $\Delta$ which is not combinatorially null-homotopic, then $\lab(\textbf{p})$ represents a non-trivial element of $G_\Omega(\textbf{M}^\pazocal{L})$.

\end{lemma}

\begin{proof}

Cutting along $\textbf{p}$ separates $\Delta$ into two connected components, each of which is an annular diagram with one boundary component identified with $\textbf{p}^{\pm1}$ and the other identified with a boundary component of $\Delta$.  As a result, van Kampen's Lemma implies that $\lab(\textbf{p})$ (or $\lab(\textbf{p})^{-1}$) represents an element $h$ of $G_\Omega(\textbf{M}^\pazocal{L})$ which is conjugate to both $h_1$ and $h_2$.  But $h_1$ and $h_2$ are non-trivial elements of $H_\pazocal{A}$ by hypothesis, so that $h$ must be a non-trivial element of $G_\Omega(\textbf{M}^\pazocal{L})$.

\end{proof}

\begin{lemma} \label{counterexample annuli H path}

Let $\Delta$ be a counterexample annulus with outer contour $\textbf{q}_1$ and inner contour $\textbf{q}_2$.  For any path $\textbf{p}$ in $\Delta$ such that $\textbf{p}_-$ is a vertex of $\textbf{q}_1$ and $\textbf{p}_+$ is a vertex of $\textbf{q}_2$, $\lab(\textbf{p})$ represents an element of $G_\Omega(\textbf{M}^\pazocal{L})\setminus H_\pazocal{A}$.

\end{lemma}

\begin{proof}

Let $\textbf{t}$ be a $g$-path of $\Delta$ and let $(u,w_1,w_2)$ be the defining triple of $\Delta$ with respect to $\textbf{t}$.  

Let $\textbf{s}_1$ be the subpath of $\textbf{q}_1$ with $(\textbf{s}_1)_-=\textbf{t}_-$ and $(\textbf{s}_1)_+=\textbf{p}_-$.  Similarly, let $\textbf{s}_2$ be the subpath of $\textbf{q}_2$ with $(\textbf{s}_2)_-=\textbf{p}_+$ and $(\textbf{s}_2)_+=\textbf{t}_+$.  Then, $\textbf{y}=\textbf{s}_1\textbf{p}\textbf{s}_2$ is a path with $\textbf{y}_-=\textbf{t}_-$ and $\textbf{y}_+=\textbf{t}_+$.  

As a consequence of van Kampen's Lemma (see Lemma 11.4 of \cite{O}), there then exists an integer $k$ such that $w_1^k\lab(\textbf{y})$ represents $g$ in $G_\Omega(\textbf{M}^\pazocal{L})$.  In particular, $w_1^k\lab(\textbf{y})=w_1^k\lab(\textbf{s}_1)\lab(\textbf{p})\lab(\textbf{s}_2)$ represents an element of $G_\Omega(\textbf{M}^\pazocal{L})\setminus H_\pazocal{A}$.

But by construction the letters comprising $w_1$, $\lab(\textbf{s}_1)$, and $\lab(\textbf{s}_2)$ are from $\pazocal{A}\cup\pazocal{A}^{-1}$, so that $w_1^k\lab(\textbf{s}_1),\lab(\textbf{s}_2)\in H_\pazocal{A}$.  The statement thus follows.

%
%

\end{proof}

\begin{lemma} \label{counterexample annulus subdiagram}

Suppose $\Gamma$ is a subdiagram of a counterexample annulus $\Delta$.  Then there exists a counterexample annulus $\Delta_0$ containing a subdiagram $\Gamma_0$ such that:

\begin{itemize}

\item $\lab(\partial\Gamma_0)\equiv\lab(\partial\Gamma)$

\item $\tau(\Gamma_0)=\tau(\Gamma)$

\item $\tau(\Delta_0)=\tau(\Delta)$

\item There exists a $g$-path $\textbf{t}_0$ of $\Delta_0$ which is disjoint from $\Gamma_0$

\item $\Gamma_0$ is disjoint from the boundary of $\Delta_0$

\end{itemize}

\end{lemma}

\begin{proof}

Let $\textbf{t}$ be a $g$-path of $\Delta$.  If $\textbf{t}$ and $\Gamma$ are disjoint, then the statement is satisfied by $0$-refinement, adding a `band' of $0$-cells around the boundary of $\Gamma$ to assure its contour is disjoint from the diagram's boundary components.

Otherwise, let $\textbf{p}$ be a maximal subpath of $\textbf{t}$ which is contained in $\Gamma$.  By construction, $\textbf{p}_-$ and $\textbf{p}_+$ must be vertices of $\partial\Gamma$.  So, there exists a subpath $\textbf{p}'$ of $\partial\Gamma$ with $\textbf{p}'_-=\textbf{p}_-$ and $\textbf{p}'_+=\textbf{p}_+$.  By construction, $\textbf{p}'$ and $\textbf{p}$ are combinatorially homotopic, and so $\lab(\textbf{p}')$ and $\lab(\textbf{p})$ represent the same element of $G_\Omega(\textbf{M}^\pazocal{L})$.  Hence, replacing all maximal subpaths $\textbf{p}$ of $\textbf{t}$ with the corresponding path $\textbf{p}'$ produces a $g$-path $\textbf{t}'$ with $\textbf{t}'_-=\textbf{t}_-$ and $\textbf{t}'_+=\textbf{t}_+$.

As above, the $0$-refinement given by adding a band of $0$-cells around the boundary of $\Gamma$ then produces a counterexample annulus satisfying the statement.

%
%
%

\end{proof}

With \Cref{counterexample annulus subdiagram} in hand, we may now prove the existence of a reduced minimal counterexample annulus:

\begin{lemma} \label{reduced counterexample annuli}

For any counterexample annulus $\Delta$, there exists a reduced counterexample annulus $\tilde{\Delta}$ satisfying $\tau(\tilde{\Delta})\leq\tau(\Delta)$.

\end{lemma}

\begin{proof}

Suppose $\Delta$ contains a pair of cancellable cells $\Pi_1$ and $\Pi_2$.  Using $0$-refinement, we can assume that $\Delta$ contains a subdiagram $\Gamma$ consisting of this pair of cancellable cells (and $0$-cells) such that $\lab(\partial\Gamma)$ is freely trivial.
By \Cref{counterexample annulus subdiagram}, it may be assumed that $\Gamma$ is disjoint from both a $g$-path $\textbf{t}$ and the boundary of $\Delta$.  But then we may remove $\Gamma$ from $\Delta$ as in the description of cancellable cells in \Cref{sec-Diagrams} without affecting the $g$-path.  

Thus, we may produce the reduced counterexample $\tilde{\Delta}$ by simply iterating this process, altering the $g$-path as necessary to remove a pair of cancellable cells until there is no such pair remaining.

%
%
%
%

\end{proof}

\begin{lemma} \label{minimal counterexample annuli}

Any subdiagram of a reduced ($j$-)minimal counterexample annulus is a reduced ($j$-)minimal circular diagram.

\end{lemma}

\begin{proof}

Let $\Gamma$ be a subdiagram of a reduced minimal counterexample annulus $\Delta$ and let $\textbf{t}$ be a $g$-path of $\Delta$.  Using the $0$-refinement of \Cref{counterexample annulus subdiagram}, we may assume that $\Gamma$ is disjoint from both the $g$-path $\textbf{t}$ and the boundary of $\Delta$.

Let $\tilde{\Gamma}$ be a minimal diagram with $\lab(\partial\tilde{\Gamma})\equiv\lab(\partial\Gamma)$.  Then, let $\tilde{\Delta}$ be the annular diagram obtained from $\Delta$ by excising $\Gamma$ and pasting $\tilde{\Gamma}$ in its place.  As in the proof of \Cref{reduced counterexample annuli}, $\tilde{\Delta}$ is itself a counterexample diagram.

Since $\Delta$ is a minimal counterexample annulus, it then follows that $\tau(\Delta)\leq\tau(\tilde{\Delta})$.  On the other hand, as $\tilde{\Gamma}$ is minimal, $\tau(\tilde{\Gamma})\leq \tau(\Gamma)$, and hence $\tau(\tilde{\Delta})\leq\tau(\Delta)$.  But then $\tau(\Delta)=\tau(\tilde{\Delta})$ implies $\tau(\tilde{\Gamma})=\tau(\Gamma)$ by construction, so that $\Gamma$ must itself be minimal.

If $\Delta$ is a reduced $j$-minimal counterexample annulus, then an analogous argument applies.

\end{proof}

\begin{lemma} \label{minimal counterexample annuli disks}

Let $\Delta$ be a reduced $1$-minimal counterexample annulus.  Then $\Delta$ contains no disks.

\end{lemma}

\begin{proof}

Suppose $\Delta$ contains at least one disk.

Similar to the construction of \Cref{sec-removing-disks} (but omitting external edges), we construct the auxiliary graph $\Gamma(\Delta)$ as follows:

\begin{enumerate}

\item The set of vertices is $\{v_1,\dots,v_\ell\}$, where each $v_i$ corresponds to one of the $\ell$ disks of $\Delta$.  

\item For $i,j\geq1$ and for any positive $t$-band which has ends on the disks corresponding to $v_i$ and $v_j$, there is a corresponding edge $(v_i,v_j)$. 


\end{enumerate}

As in that setting, this graph can be constructed as an auxiliary graph to the graph underlying $\Delta$, {\frenchspacing i.e. it is constructed on an annulus}.  It follows immediately from the definition of the disk relations that $\Gamma(\Delta)$ has no $1$-gons.  Further, Lemmas \ref{minimal counterexample annuli} and \ref{disk MM2} imply that $\Gamma(\Delta)$ has no $2$-gons.

Hence, an appeal to the Euler characteristic of the annulus (see, for example, Lemma 10.1 of \cite{O}) implies there must exist a vertex $v$ of $\Gamma(\Delta)$ with degree at most $18$.

Note that, by definition, every boundary edge of $\Delta$ is an $a$-edge.  In particular, any maximal non-annular $t$-band of $\Delta$ must have two ends on (distinct) disks.  But then the degree of every vertex of $\Gamma(\Delta)$ must be $L-1$, so that a parameter choice for $L$ provides a contradiction.

\end{proof}

\begin{lemma} \label{minimal counterexample annuli q}

A reduced $1$-minimal counterexample annulus $\Delta$ contains no $q$-annuli or $a$-annuli.

\end{lemma}

\begin{proof}

Assuming the statement is false, let $\pazocal{S}$ be a maximal $q$-annulus or $a$-annulus in $\Delta$.  Then, each side of $\pazocal{S}$ can be assumed to be (perhaps with $0$-refinement) a simple closed path in $\Delta$.


If a side of $\pazocal{S}$ is combinatorially null-homotopic, then $\pazocal{S}$ bounds a subdiagram $\Gamma_0$ of $\Delta$.  But then $\Gamma_0$ is a reduced circular diagram, so that the presence of $\pazocal{S}$ provides a contradiction to either \Cref{basic annuli 1}(2) or \Cref{basic annuli 2}(2).

Hence, each side of $\pazocal{S}$ is not combinatorially null-homotopic, so that \Cref{counterexample annuli trivial path} implies the labels of these sides represent non-trivial elements of $G_\Omega(\textbf{M}^\pazocal{L})$.  In particular, $\pazocal{S}$ must be a band of length $\ell>0$.

Note that each of these $\ell$ cells correspond to the crossing of a maximal $\theta$-band with $\pazocal{S}$.  But $\Delta$ contains no boundary $\theta$-edges, so that these $\theta$-bands must cross $\pazocal{S}$ more than once, creating $(\theta,q)$- or $(\theta,a)$-annuli that contradict \Cref{basic annuli 1}(1) or \Cref{basic annuli 2}(1).

\end{proof}

\begin{lemma} \label{minimal counterexample annuli a-boundary}

Let $\pi$ be an $a$-cell in a 3-minimal counterexample annulus $\Delta$.  Then no edge of $\partial\pi$ is a boundary edge of $\Delta$.

\end{lemma}

\begin{proof}

Suppose there is an edge $\textbf{e}$ of $\partial\pi$ which is also an edge of the outer contour $\textbf{q}_1$ of $\Delta$.  Then, letting $\textbf{t}$ be a $g$-path of $\Delta$, let $\textbf{y}\textbf{e}\textbf{z}$ be the decomposition of $\textbf{q}_1$ as a loop about the vertex $o=\textbf{t}_-$, {\frenchspacing i.e. such that $\textbf{y}_-=\textbf{z}_+=o$}.  Further, let $\textbf{s}$ be the complement of $\textbf{e}$ in $\partial\pi$.  Perhaps $0$-refining, we may assume that no edge of $\textbf{s}$ is a boundary edge of $\Delta$.

Note that since $\textbf{e}$ is a boundary edge of $\Delta$, $\lab(\textbf{e})\in\pazocal{A}^{\pm1}$, and so $\lab(\partial\pi)\in\Lambda^\pazocal{A}$ by \Cref{M Lambda semi-computations}.  In particular, $\lab(\textbf{s})\in(\pazocal{A}\cup\pazocal{A}^{-1})^*$ and $\lab(\textbf{e})$ represents the same word as $\lab(\textbf{s})^{-1}$ in $G_\Omega(\textbf{M}^\pazocal{L})$.

Now, consider the annular diagram $\tilde{\Delta}$ obtained by cutting along $\textbf{s}$ and removing $\pi$.  By construction, the outer contour $\tilde{\textbf{q}}_1$ of $\tilde{\Delta}$ has a copy of the vertex $o$ (indeed of all vertices of $\textbf{q}_1$), so that the decomposition of $\tilde{\textbf{q}}_1$ as a loop about this vertex can be identified with $\textbf{y}\textbf{s}^{-1}\textbf{z}$.  By construction, $\lab(\textbf{y}\textbf{s}^{-1}\textbf{z})$ is a word over $\pazocal{A}\cup\pazocal{A}^{-1}$ which is equal to $\lab(\textbf{y}\textbf{e}\textbf{z})$ in $G_\Omega(\textbf{M}^\pazocal{L})$.  In particular, $\lab(\tilde{\textbf{q}}_1)$ read starting at $o$ represents $h_1$.

Similarly, replacing in $\textbf{t}$ any occurrence of the edge $\textbf{e}^{\pm1}$ with the subpath $\textbf{s}^{\mp1}$ produces a path $\tilde{\textbf{t}}$ in $\tilde{\Delta}$ whose label represents $g$ in $G_\Omega(\textbf{M}^\pazocal{L})$.  Hence, since the inner contour of $\Delta$ is undisturbed in passing to $\tilde{\Delta}$, it follows that $\tilde{\Delta}$ is itself a counterexample annulus.  But by construction $\tau_3(\tilde{\Delta})<\tau_3(\Delta)$, contradicting the hypothesis that $\Delta$ is a $3$-minimal counterexample annulus.

If there is an edge $\textbf{e}$ of $\partial\pi$ such that $\textbf{e}^{-1}$ is an edge of the inner contour, then an analogous argument produces a contradiction in the same way.

\end{proof}


\begin{lemma} \label{minimal counterexample annuli transposition}

Let $\pi$ be an $a$-cell and $\textbf{q}$ be a boundary component in a reduced minimal counterexample annulus $\Delta$.  If $\Psi$ is an $a$-scope on $\textbf{q}$ with associated $a$-cell $\pi$ and size $\ell\geq5$, then $\Psi$ is not a pure $a$-scope.

\end{lemma}

\begin{proof}

Suppose to the contrary that $\Psi$ is pure.  By the definition of pure $a$-scope, every cell of $\Psi$ is then a $(\theta,a)$-cell.

Let $\textbf{e}_1,\dots,\textbf{e}_\ell$ be the consecutive $\pazocal{A}$-edges of $\partial\pi$ which comprise the associated subpath of $\Psi$.  By \Cref{minimal counterexample annuli}, both $\Psi$ and its completion $\tilde{\Psi}$ are  reduced minimal diagrams.  Hence, Lemmas \ref{a-bands on a-cell}, \ref{minimal is smooth}, and \ref{minimal counterexample annuli a-boundary} imply that each $\pazocal{A}$-band $\pazocal{U}(\textbf{e}_i)$ is of positive length and has an end on $\textbf{q}$.

Further, \Cref{disk theta-annuli} implies $\Psi$ contains only non-annular $\theta$-bands.  In particular, since $\Delta$ contains no boundary $\theta$-edges, the positive cells of $\Psi$ consist entirely of those forming $\theta$-bands that all cross each of the $\pazocal{A}$-bands $\pazocal{U}(\textbf{e}_1),\dots,\pazocal{U}(\textbf{e}_\ell)$.


Using the $0$-refinement procedure of \Cref{counterexample annulus subdiagram}, we may construct a reduced minimal counterexample annulus $\Delta_0$ containing a subdiagram identified with $\tilde{\Psi}$ which is disjoint from both the boundary of $\Delta_0$ and and a $g$-path $\textbf{t}_0$ of $\Delta_0$.

As $\ell\geq5$ and $\Psi$ contains no $a$-cells, we may then iteratively transpose $\pi$ with each of the $\theta$-bands of $\Psi$, producing an annular diagram $\Delta_0'$ with corresponding subdiagram $\tilde{\Psi}_0$.  As $\tilde{\Psi}$ is disjoint from both $\textbf{t}_0$ and the boundary of $\Delta_0$, $\Delta_0'$ is itself a counterexample annulus.  Moreover, since the transposition of a $\theta$-band and an $a$-cell changes only the number of $(\theta,a)$-cells in the diagram, $\tau_3(\Delta_0')=\tau_3(\Delta_0)$.  In particular, $\Delta_0'$ is a $3$-minimal counterexample annulus.

Now, let $\pi'$ be the $a$-cell of $\tilde{\Psi}_0$.  By construction, each of the $\pazocal{A}$-edges $\textbf{e}_1',\dots,\textbf{e}_\ell'$ of $\partial\pi'$ is adjacent to an edge of the boundary of $\Delta_0'$.  But then removing the corresponding $0$-cells produces a counterexample annulus with the same $3$-signature and containing an $a$-cell that shares a boundary $\pazocal{A}$-edge (indeed $\ell$ such edges) with the boundary of the diagram, contradicting \Cref{minimal counterexample annuli a-boundary}.

\end{proof}

\begin{lemma} \label{minimal counterexample annuli a bound}

Let $\pi$ be an $a$-cell and $\textbf{q}$ be a boundary component of a minimal counterexample annulus $\Delta$.  Then at most 4 positive $\pazocal{A}$-bands have ends on both $\pi$ and $\textbf{q}$.

\end{lemma}

\begin{proof}

Assume to the contrary that there exist $\pazocal{A}$-edges $\textbf{e}_1,\dots,\textbf{e}_5$ of $\partial\pi$ such that each maximal $\pazocal{A}$-band $\pazocal{U}(\textbf{e}_1),\dots,\pazocal{U}(\textbf{e}_5)$ has an end on $\textbf{q}$.  

Perhaps with $0$-refinement, there then exists a subpath $\textbf{s}$ of $\partial\pi$ containing each of the $\pazocal{A}$-edges $\textbf{e}_i$ such that $\textbf{s}$, a subpath of $\textbf{q}$, and the $\pazocal{A}$-bands $\pazocal{U}(\textbf{e}_i)$ bound a subdiagram $\Psi_0$ of $\Delta$ not containing $\pi$.  As $\Delta$ contains no $(\theta,q)$-cell by \Cref{minimal counterexample annuli q}, $\Psi_0$ is an $a$-scope on $\textbf{q}$ with associated $a$-cell $\pi$ and associated subpath $\textbf{s}$.  Hence, the size of $\Psi_0$ is $|\textbf{s}|_\pazocal{A}\geq5$, so that \Cref{minimal counterexample annuli transposition} implies $\Psi_0$ cannot be a pure $a$-scope.

Further, Lemmas \ref{minimal counterexample annuli}, \ref{minimal MM2}, and \ref{minimal is smooth} imply that the completion $\tilde{\Psi}_0$ is smooth and satisfies condition (MM2).  Hence, \Cref{pure a-scope} implies the existence of a pure big $a$-scope $\Psi$ on $\textbf{q}$.  But then by condition (L1) and \Cref{semi-computation deltas}, the size of $\Psi$ is greater than $C/2$, so that the parameter choice $C\geq10$ yields a contradiction to \Cref{minimal counterexample annuli transposition}.

%
%
%

\end{proof}

\begin{lemma} \label{minimal counterexample annuli a-cells}

A reduced minimal counterexample annulus contains no $a$-cells.

\end{lemma}

\begin{proof}

Suppose the reduced minimal counterexample annulus $\Delta$ contains at least one $a$-cell.  Similar to the proof of \Cref{minimal counterexample annuli disks}, we begin by adapting the construction of the auxiliary graphs of \Cref{sec-M-minimal}.  To this end, we construct the graph $\Gamma_a(\Delta)$ as follows:

\begin{enumerate}

\item The set of vertices is $\{v_1,\dots,v_\ell\}$, where each $v_i$ corresponds to one of the $\ell$ $a$-cells of $\Delta$.  

\item For $i,j\geq1$ and for any positive $\pazocal{A}$-band which has ends on the $a$-cells corresponding to $v_i$ and $v_j$, there is a corresponding edge $(v_i,v_j)$.


\end{enumerate}

As in the proof of \Cref{Gamma_a planar}, the graph $\Gamma_a(\Delta)$ can be constructed as an auxiliary graph to the graph underlying $\Delta$, and so constructed on an annulus (indeed, the lack of $(\theta,q)$-cells makes this a much simpler version of that presented in \Cref{Gamma_a planar}).

By Lemmas \ref{minimal counterexample annuli}, \ref{minimal is smooth}, and \ref{a-bands on a-cell}, $\Gamma_a(\Delta)$ contains no $1$-gon.


Suppose $\Gamma_a(\Delta)$ contains edges $e_1,\dots,e_m$ connecting the vertices $v_i$ and $v_j$ such that $e_k$ and $e_{k+1}$ bound a $2$-gon for each $k=1,\dots,m-1$.  Let $\pi_i$ and $\pi_j$ be the $a$-cells corresponding to $v_i$ and $v_j$, respectively, and let $\pazocal{U}_k$ be the maximal positive $\pazocal{A}$-band corresponding to $e_k$.  

Suppose $m\geq3$.  Then $\pi_i$, $\pi_j$, $\pazocal{U}_1$, and $\pazocal{U}_3$ bound a circular subdiagram $\Gamma'$ of $\Delta$.  As $e_k$ and $e_{k+1}$ bound a $2$-gon, the only $a$-cells in $\Gamma'$ are $\pi_i$ and $\pi_j$.
But then $\Gamma'$ does not satisfy condition (MM2), contradicting Lemmas \ref{minimal counterexample annuli} and \ref{minimal MM2}.

Hence, as in \Cref{sec-M-minimal}, any $2$-gon in $\Gamma_a(\Delta)$ arises in the form of a doubled pair of edges.  As in this previous setting, we construct the graph $\Gamma_a'(\Delta)$ by simply replacing any doubled pair of edges with a single edge.

By construction, $\Gamma_a'(\Delta)$ has no $1$-gons or $2$-gons, and thus as in the proof of \Cref{minimal counterexample annuli disks} the Euler characteristic of the annulus implies the graph must contain a vertex with degree at most 18 (see Lemma 10.1 of \cite{O}).

Now, by \Cref{minimal counterexample annuli a bound}, for any $a$-cell $\pi$ in $\Delta$, at most $8$ maximal positive $\pazocal{A}$-bands have an end on $\pi$ and on a boundary component.  Hence, by condition (L1) and \Cref{semi-computation deltas} the degree of each vertex of $\Gamma_a(\Delta)$ is at least $C-8$.  But then the degree of every vertex of $\Gamma_a'(\Delta)$ is at least $\frac{1}{2}C-4$, so that the parameter choice $C>44$ provides a counterexample to the bound given by the Euler characteristic.

%

\end{proof}

Combining Lemmas \ref{minimal counterexample annuli a-cells}, \ref{minimal counterexample annuli q}, and \ref{minimal counterexample annuli disks}, a reduced minimal counterexample annulus $\Delta$ is a reduced annular diagram over $M(\textbf{M}^\pazocal{L})$ in which every (positive) cell is a $(\theta,a)$-cell.

\begin{lemma} \label{minimal counterexample annuli theta-bands}

Let $\pazocal{T}$ be a maximal $\theta$-band in a reduced minimal counterexample annulus $\Delta$.  Then $\pazocal{T}$ is a $\theta$-annulus of positive length.

\end{lemma}

\begin{proof}

As $\Delta$ contains no boundary $\theta$-edge, $\pazocal{T}$ must be a $\theta$-annulus.  By Lemmas \ref{minimal counterexample annuli} and \ref{disk theta-annuli}, $\textbf{bot}(\pazocal{T})$ cannot be combinatorially null-homotopic.  But then \Cref{counterexample annuli trivial path} implies $\lab(\textbf{bot}(\pazocal{T}))$ represents a non-trivial element of $G_\Omega(\textbf{M}^\pazocal{L})$, so that $\pazocal{T}$ is a band of positive length.



\end{proof}

\begin{lemma} \label{minimal counterexample annuli a-bands}

Let $\textbf{q}_1$ be the outer contour and $\textbf{q}_2$ be the inner contour of a reduced minimal counterexample annulus $\Delta$.

\begin{enumerate}

\item Any maximal $\pazocal{A}$-band of length $0$ has two ends on $\textbf{q}_i$ for some $i=1,2$

\item Any maximal $\pazocal{A}$-band of positive length has an end on $\textbf{q}_1$ and an end on $\textbf{q}_2$

\item $\Delta$ contains at least one $\pazocal{A}$-band of positive length

\end{enumerate}

\end{lemma}

\begin{proof}

Lemmas \ref{minimal counterexample annuli q} and \ref{minimal counterexample annuli a-cells} imply that any maximal $\pazocal{A}$-band must have two ends on the boundary of $\Delta$.  So, since a band of length $0$ ending on both $\textbf{q}_1$ and $\textbf{q}_2$ would yield a contradiction to \Cref{counterexample annuli H path}, (1) follows immediately.


Suppose $\pazocal{U}$ is a maximal $\pazocal{A}$-band of positive length which has two ends on $\textbf{q}_i$ for some $i=1,2$.  As any (positive) cell of $\pazocal{U}$ represents a crossing with a maximal $\theta$-band and $\textbf{q}_1$ contains no $\theta$-edges, a subband of $\pazocal{U}$ bounds a $(\theta,a)$-annulus with some $\theta$-band.  But then this provides a counterexample to \Cref{basic annuli 1}(1).  Hence, (2) must hold.


Finally, suppose every $\pazocal{A}$-band of $\Delta$ has length $0$.  In particular, $\Delta$ contains no $(\theta,\pazocal{A})$-cell.  So, Lemmas \ref{minimal counterexample annuli q} and \ref{minimal counterexample annuli a-cells} imply that any maximal $a$-band must have two ends on the boundary of $\Delta$.  As the boundary of $\Delta$ consists entirely of $\pazocal{A}$-edges, though, this means that $\Delta$ has no positive cells at all.  In particular, \Cref{minimal counterexample annuli q} implies the only $a$-edges of $\Delta$ are boundary edges, and so are $\pazocal{A}$-edges labelled by letters of the `special' input sector.  \Cref{minimal counterexample annuli theta-bands} further implies that $\Delta$ contains no $\theta$-edge, while Lemmas \ref{minimal counterexample annuli disks} and \ref{minimal counterexample annuli q} imply $\Delta$ contains no $q$-edge.  Thus, for any path $\textbf{p}$ in $\Delta$ such that $\textbf{p}_-$ is a vertex of $\textbf{q}_1$ and $\textbf{p}_+$ is a vertex of $\textbf{q}_2$, the only positive edges of $\textbf{p}$ are boundary edges.  But then $\lab(\textbf{p})$ represents an element of $H_\pazocal{A}$, contradicting \Cref{counterexample annuli H path}.

\end{proof}

We now reach the desired contradiction:

\begin{lemma} \label{no counterexample annulus}

There is no counterexample annulus.

\end{lemma}

\begin{proof}

Let $\textbf{q}_1$ be the outer contour and $\textbf{q}_2$ be the inner contour of a reduced minimal counterexample annulus $\Delta$.  By \Cref{minimal counterexample annuli a-bands}, there exists an $\pazocal{A}$-edge $\textbf{e}$ of $\textbf{q}_1^{-1}$ which is a defining edge (indeed an end) of a maximal $\pazocal{A}$-band $\pazocal{U}$ which has an end on $\textbf{q}_2$.

By \Cref{counterexample annuli H path}, $\lab(\textbf{top}(\pazocal{U}))$ must be an element of $G_\Omega(\textbf{M}^\pazocal{L})\setminus H_\pazocal{A}$.  In particular, this label must be non-trivial, so that the history $H$ of $\pazocal{U}$ is a reduced word with $\|H\|>0$.

Cutting $\Delta$ along $\textbf{top}(\pazocal{U})$ then produces a reduced circular diagram $\Gamma$ containing a maximal $\pazocal{A}$-band identified with $\pazocal{U}$ such that $\partial\Gamma=\textbf{p}_1^{-1}\textbf{s}_1\textbf{p}_2\textbf{s}_2^{-1}$ where: 

\begin{itemize}

\item $\textbf{p}_1=\textbf{top}(\pazocal{U})$

\item $\lab(\textbf{p}_2)\equiv\lab(\textbf{p}_1)$

\item $\textbf{s}_1$ is identified with $\textbf{q}_1$ read starting at $\textbf{e}_+=(\textbf{top}(\pazocal{U}))_-$

\item $\textbf{s}_2$ is identified with $\textbf{q}_2^{-1}$ read starting at $(\textbf{top}(\pazocal{U}))_+$

\end{itemize}

As noted above, Lemmas \ref{minimal counterexample annuli disks}, \ref{minimal counterexample annuli q}, and \ref{minimal counterexample annuli a-cells} imply any positive cell of $\Gamma$ is a $(\theta,a)$-cell.

Enumerate the $\theta$-edges of $\textbf{p}_1$ by $\textbf{e}_1,\dots,\textbf{e}_\ell$.  For each $i\in\{1,\dots,\ell\}$, let $\pazocal{T}_i$ be the maximal $\theta$-band of $\Gamma$ for which $\textbf{e}_i$ is an end.  As $\textbf{e}_i$ is on the boundary of a $(\theta,\pazocal{A})$-cell of $\pazocal{U}$, $\pazocal{T}_i$ cannot have two ends on $\textbf{p}_1^{-1}$, and so must have an end on $\textbf{p}_2$.  In particular, since $|\textbf{p}_2|_\theta=\ell$, every positive cell must be contained in one and only one $\theta$-band $\pazocal{T}_i$.

Let $\textbf{t}_1=\textbf{bot}(\pazocal{T}_1)$ and $\textbf{t}_2=\textbf{top}(\pazocal{T}_\ell)$.  Then, as any cell between $\textbf{t}_i$ and $\textbf{s}_i$ must be a $0$-cell, $\lab(\textbf{t}_i)$ and $\lab(\textbf{s}_i)$ must be equal in $F(\pazocal{X})$.  In particular, $\lab(\textbf{t}_i)$ is a freely reduced word which is conjugate in $F(\pazocal{A})$ to a word that represents $h_i$.  Note that this necessarily implies $\lab(\textbf{t}_i)$ is a non-trivial word over $\pazocal{A}\cup\pazocal{A}^{-1}$.

Let $\Gamma'$ be the subdiagram of $\Gamma$ obtained by removing any $0$-cells between $\textbf{t}_i$ and $\textbf{s}_i$.  Then, letting $\pazocal{U}'$ be the subband of $\pazocal{U}$ obtained by removing any initial or terminal subsequence of $0$-cells, $\partial\Gamma'=(\textbf{p}_1')^{-1}\textbf{t}_1\textbf{p}_2'\textbf{t}_2^{-1}$, where $\textbf{p}_1'=\textbf{top}(\pazocal{U}')$ and $\lab(\textbf{p}_2')\equiv\lab(\textbf{p}_1')$.  Note that, by definition, the history of $\pazocal{U}'$ is $H$.

Let $\textbf{f}$ be the initial positive edge of $\textbf{t}_1^{-1}$.  Noting that $\textbf{f}$ is then an $\pazocal{A}$-edge, let $\pazocal{V}'$ be the maximal $\pazocal{A}$-band of $\Gamma'$ with end $\textbf{f}$ (note that it is possible that $\textbf{e}=\textbf{f}$, in which case $\pazocal{V}'=\pazocal{U}'$).  Then, let $\textbf{p}_2''=\textbf{bot}(\pazocal{V}')$.  By construction, $\textbf{p}_2''$ and $\textbf{p}_2'$ bound a subdiagram consisting entirely of $0$-cells, and so $\lab(\textbf{p}_2'')$ is freely equal to $\lab(\textbf{p}_1')$.

Now, let $\Psi$ be the subdiagram of $\Gamma'$ with $\partial\Psi=(\textbf{p}_1')^{-1}\textbf{t}_1\textbf{p}_2''\textbf{t}_2^{-1}$.  Then, by construction, $\Psi$ is a compressed semi-trapezium in the `special' input sector with standard factorization $(\textbf{p}_1')^{-1}\textbf{t}_1\textbf{p}_2''\textbf{t}_2^{-1}$.  Further, the maximal $\theta$-bands $\pazocal{T}_1,\dots,\pazocal{T}_\ell$ are enumerated from bottom to top.

Hence, letting $w_{j-1}=\lab(\mathscr{C}\textbf{bot}(\pazocal{T}_j))$ for $j=1,\dots,\ell$ and $w_\ell=\lab(\mathscr{C}\textbf{top}(\pazocal{T}_\ell))$, \Cref{compressed semi-trapezia are compressed semi-computations} yields an associated reduced compressed semi-computation $\pazocal{S}_\mathscr{C}:w_0\to\dots\to w_\ell$ of $\textbf{M}^\pazocal{L}$ in the `special' input sector with history $H$.

Suppose $\|w_0\|\geq3$.  Then, letting $v_i$ be the minimal prefix of $w_i$ with $|v_i|_\pazocal{A}=3$, there exists a reduced compressed semi-computation $\pazocal{S}_\mathscr{C}':v_0\to\dots\to v_\ell$ of $\textbf{M}^\pazocal{L}$ in the `special' input sector with history $H$.  As $v_0\in F(\pazocal{A})$ with $\|v_0\|=3$, $\pazocal{S}_\mathscr{C}'$ satisfies the hypotheses of \Cref{M compressed semi-computation three A}.  Hence, $v_\ell$ must be a non-trivial word over $(\pazocal{A}_1\sqcup\pazocal{B})^{\pm1}$.  But $w_\ell\equiv\lab(\textbf{t}_2)$, yielding a contradiction.

Similarly, if $w_0\equiv y_1^{\delta_1}y_2^{\delta_2}\in F(\pazocal{A})$ such that $\delta_1\neq1$ or $\delta_2\neq-1$, then $\pazocal{S}_\mathscr{C}$ satisfies the hypotheses of \Cref{M compressed semi-computation two A}.  But then this implies $w_\ell$ is a non-trivial word over $(\pazocal{A}_1\sqcup\pazocal{B})^{\pm1}$, again yielding a contradiction.

Hence, we may assume that $w_0\equiv y_1y_2^{-1}\in F(\pazocal{A})$ or $w_0\in\pazocal{A}^{\pm1}$.  Either way, $\Psi$ satisfies the hypotheses of \Cref{compressed semi-trapezia sides not equal}.  But then $\lab(\textbf{p}_1')$ and $\lab(\textbf{p}_2'')$ do not represent the same element of $G_\Omega(\textbf{M}^\pazocal{L})$, yielding a contradiction.

\end{proof}

Thus, \Cref{no counterexample annulus} immediately implies:

\begin{lemma} \label{H_A malnormal}

 $H_\pazocal{A}\leq_{mal}G_\Omega(\textbf{M}^\pazocal{L})$.
 
 \end{lemma}


\section{Distortion diagrams} \label{sec-distortion-diagrams}

The goal of this section is to demonstrate that the subgroup $H_\pazocal{A}$ is undistorted in $G_\Omega(\textbf{M}^\pazocal{L})$.  This is accomplished by studying minimal circular diagrams with a particular contour decomposition, resembling the treatment of `$g$-minimal diagrams' in \cite{W}.

Before this, though, it will prove convenient to first modify the length of words over the disk presentation of $G_\Omega(\textbf{M}^\pazocal{L})$ and, by extension, the paths in diagrams over these presentations.  This is done in a way resembling that used in \cite{O18}, \cite{OS19}, and \cite{W}, but with a significant difference.

\subsection{Modified length function} \


To begin, a word $u$ over $\pazocal{X}\cup\pazocal{X}^{-1}$ is called a \textit{$(\theta,a)$-syllable} if $|u|_\theta=1$, $|u|_q=0$, and $|u|_\pazocal{A}+|u|_o\leq1$.
%
%
%
%
Note that, by definition, a single $\theta$-letter is a $(\theta,a)$-syllable and that the inverse of a $(\theta,a)$-syllable is another $(\theta,a)$-syllable.  Further, observe that there is no bound on the number of $b$-letters present in a $(\theta,a)$-syllable.  

Given a word $w$ over $\pazocal{X}\cup\pazocal{X}^{-1}$, a \textit{decomposition} of $w$ is a factorization $w\equiv u_1\dots u_k$ such that each $u_i$ is either a single letter or a $(\theta,a)$-syllable.  The \textit{length} of such a decomposition is the sum of the lengths of its factors, where we take the length of $u_i$ to be:



\begin{itemize}

\item $1$ if $u_i$ is a $q$-letter or a $(\theta,a)$-syllable

\item $\delta$ if $u_i$ is an $\pazocal{A}$-letter or an ordinary $a$-letter

\item $0$ if $u_i$ is a $b$-letter

\end{itemize}

As indicated in \Cref{sec-parameters}, the parameter $\delta$ assigned to be the length of an $\pazocal{A}$-letter or ordinary $a$-letter may be thought of as a very small positive number.

Finally, the \textit{length} of the word $w$, denoted $|w|$, is the minimal length of any of its decompositions.

Observe that this definition fundamentally differs from the analogous modified length function defined in previous literature for similar settings: While those functions were still equivalent to the natural combinatorial length function, designating the length of $b$-letters to $0$ means this function is not.  With that said, disregarding the $b$-letters entirely (but not reducing the resulting word) allows for an identical analysis of the properties of this length function.

For example, the next statement follows immediately from the definition:

\begin{lemma}[Compare to Lemma 6.2 of \cite{OS19}] \label{lengths 1}

Let $w\equiv w_1w_2$ be a word over $\pazocal{X}\cup\pazocal{X}^{-1}$.

\begin{enumerate} [label=(\alph*)]

\item $|w^{-1}|=|w|$

\item $|w|\geq|w|_q+|w|_\theta+\delta\max(0,|w|_\pazocal{A}-|w|_\theta)+\delta\max(0,|w|_o-|w|_\theta)$

\item $|w_1|+|w_2|-\delta\leq|w|\leq|w_1|+|w_2|$

\item If the last letter of $w_1$ or the first letter of $w_2$ is a $q$-letter, then $|w|=|w_1|+|w_2|$

\end{enumerate}

\end{lemma}

Naturally, given a diagram $\Delta$ over any of the presentations of the groups associated to $\textbf{M}^\pazocal{L}$ ({\frenchspacing e.g. the disk presentation} of $G_\Omega(\textbf{M}^\pazocal{L})$), the \textit{length} of a path $\textbf{s}$ in $\Delta$ is defined to be the length of its label, {\frenchspacing i.e. $|\textbf{s}|=|\lab(\textbf{s})|$}.

The next statement is thus given by the following observation regarding the definition of the defining relations: The word labelling a subpath of the boundary of a $(\theta,q)$-relation between two $q$-edges is a $(\theta,a)$-syllable.

\begin{lemma}[Compare to Lemma 6.2 of \cite{OS19}] \label{lengths 2}

Let $\textbf{s}$ be a path in a diagram $\Delta$ over the disk presentation of $G_\Omega(\textbf{M}^\pazocal{L})$.

\begin{enumerate}[label=(\alph*)]

\item If $\textbf{s}$ is a side of a $q$-band, then $|\textbf{s}|=|\textbf{s}|_\theta$

\item If $\textbf{s}$ is a side of a $\theta$-band, then $|\textbf{s}|=|\textbf{s}|_q+\delta|\textbf{s}|_\pazocal{A}+\delta|\textbf{s}|_o$ 

\end{enumerate}

\end{lemma}

%
%
%
%
%
%
%
%
%
%

\begin{lemma} \label{theta-band lengths}

Let $\pazocal{T}$ be a $\theta$-band of positive length in a diagram $\Delta$ over the disk presentation of $G_\Omega(\textbf{M}^\pazocal{L})$.  Letting $l_b$ be the length of the base of $\pazocal{T}$, then $-2\delta l_b\leq|\textbf{top}(\pazocal{T})|-|\textbf{bot}(\pazocal{T})|\leq2\delta l_b$.

\end{lemma}

\begin{proof}

Let $\theta$ be the history of $\pazocal{T}$ and suppose without loss of generality that $\theta\in\Theta^+$.

First, suppose $l_b=0$, so that every cell of $\pazocal{T}$ is a $(\theta,a)$-cell.  Then the defining $\theta$-edges must be labelled identically, so that the $(\theta,a)$-cells are all of the same sector.  \Cref{theta-band is one-rule semi-computation} then implies that $\lab(\textbf{bot}(\pazocal{T}))$ is $\theta$-applicable with $\lab(\textbf{bot}(\pazocal{T}))\cdot\theta\equiv\lab(\textbf{top}(\pazocal{T}))$.

If this semi-computation is of a non-input sector, then the definition of the rules necessitates that $\lab(\textbf{bot}(\pazocal{T}))\equiv\lab(\textbf{top}(\pazocal{T}))$, and so $|\textbf{bot}(\pazocal{T})|=|\textbf{top}(\pazocal{T})|$.  Otherwise, \Cref{semi-computation deltas} implies $|\textbf{bot}(\pazocal{T})|_\pazocal{A}=|\textbf{top}(\pazocal{T})|_\pazocal{A}$.  Hence, as no letter of an input alphabet is an ordinary $a$-letter, \Cref{lengths 2}(b) implies $|\textbf{bot}(\pazocal{T})|=|\textbf{top}(\pazocal{T})|$.

Now suppose $l_b>0$.  By the definition of the rules of $\textbf{M}^\pazocal{L}$ and \Cref{simplify rules}, the boundary of a $(\theta,q)$-cell has at most one $\pazocal{A}$-edge and at most one ordinary $a$-edge.  Hence, applying the above argument to maximal subbands of $\pazocal{T}$ with base of length $0$, the statement follows from Lemmas \ref{lengths 1}(d) and \ref{lengths 2}(b).

\end{proof}

\medskip


\subsection{$h$-distortion diagrams} \

Throughout the arguments spanning the rest of this section, we fix an element $h\in H_\pazocal{A}$.  

Recall that $H_\pazocal{A}$ is the subgroup of $G_\Omega(\textbf{M}^\pazocal{L})$ generated by $\pazocal{A}$.  Hence, we may define $|h|_\pazocal{A}$ in the standard way, {\frenchspacing i.e. the minimal number} of letters of $\pazocal{A}\cup\pazocal{A}^{-1}$ necessary to produce a word which represents $h$.

Conversely, since $h$ is an element of $G_\Omega(\textbf{M}^\pazocal{L})$, we define $|h|$ to be the minimal length of a word over $\pazocal{X}\cup\pazocal{X}^{-1}$ (in the sense defined in the previous section) which represents $h$ in $G_\Omega(\textbf{M}^\pazocal{L})$.  

Let $w$ be a word realizing this length.  It should be noted that, by definition, $w$ need not be a reduced word; indeed, $w$ may have a freely trivial subword of arbitrarily large size consisting of $b$-letters.  However, as free reduction cannot increase the length of the word, the reduced word $w'$ obtained from $w$ by a sequence of cancellations is another word over $\pazocal{X}\cup\pazocal{X}^{-1}$ realizing $|h|$.


Note that by definition and \Cref{lengths 1}(a), $|h^{-1}|_\pazocal{A}=|h|_\pazocal{A}$ and $|h^{-1}|=|h|$.

Then, a circular diagram $\Delta$ over the disk presentation of $G_\Omega(\textbf{M}^\pazocal{L})$ is called an \textit{$h$-distortion diagram} if there exists a factorization $\partial\Delta=\textbf{q}\textbf{p}$ such that for some $\eps\in\{\pm1\}$:

\begin{itemize}

\item $\lab(\textbf{q})$ is a reduced word over $\pazocal{X}\cup\pazocal{X}^{-1}$ representing $h^\eps$ such that $|\textbf{q}|=|h|$

\item $\lab(\textbf{p})$ is a (reduced) word over $\pazocal{A}\cup\pazocal{A}^{-1}$ representing $h^{-\eps}$ satisfying $\|\textbf{p}\|=|h|_\pazocal{A}$

\end{itemize}
 
In this case, $\partial\Delta=\textbf{qp}$ is called the \textit{standard factorization} of the contour of the $h$-distortion diagram.  Further, $\eps$ is called the \textit{sign} of the $h$-distortion diagram.

Note that per the definition, for any $h$-distortion diagram $\Delta$, there exists an $h$-distortion diagram $\bar{\Delta}$ with opposite sign formed by taking the `mirror copy' of each cell of $\Delta$, so that any cell $\pi$ of $\Delta$ corresponds to a cell $\bar{\pi}$ in $\bar{\Delta}$ with $\lab(\partial\bar{\pi})\equiv\lab(\partial\pi)^{-1}$.  As such, $\bar{\Delta}$ is called the \textit{mirror} of $\Delta$.

\begin{lemma} \label{distortion diagram q}

Any maximal $q$-band of a reduced $h$-distortion diagram $\Delta$ has an end on a disk.

\end{lemma}

\begin{proof}

Let $\partial\Delta=\textbf{qp}$ be the standard factorization of the contour of $\Delta$ and $\eps$ be the sign of $\Delta$.  As $|\textbf{p}|_q=0$, any $q$-edge of $\partial\Delta$ must be an edge of $\textbf{q}$.

Suppose there exists a maximal $q$-band which has no end on a disk.  Then, \Cref{basic annuli 1}(2) implies that this band has two ends on $\textbf{q}$.

Now, enumerate the $q$-edges $\textbf{e}_1,\dots,\textbf{e}_k$ of $\partial\Delta$ so that $\textbf{q}=\textbf{u}_0\textbf{e}_1\textbf{u}_1\dots\textbf{u}_{k-1}\textbf{e}_k\textbf{u}_k$ for some (perhaps trivial) subpaths $\textbf{u}_i$.  Then, for each $i\in\{1,\dots,k\}$, let $\pazocal{Q}_i$ be the maximal $q$-band of $\Delta$ for which $\textbf{e}_i$ is a defining edge.

By hypothesis, there then exists a pair of indices $i,j\in\{1,\dots,k\}$ with $i<j$ such that $\textbf{e}_i^{-1}$ is an end of $\pazocal{Q}_j$.  Let $\textbf{q}'$ be the subpath of $\textbf{q}$ with initial edge $\textbf{e}_i$ and terminal edge $\textbf{e}_j$.  Further, let $\textbf{s}_1$ and $\textbf{s}_2$ be the (perhaps trivial) subpaths of $\textbf{q}$ so that $\textbf{q}=\textbf{s}_1\textbf{q}'\textbf{s}_2$.

Then, $\textbf{q}'$ and $\textbf{top}(\pazocal{Q}_j)$ bound a subdiagram $\Delta_0$ of $\Delta$ containing $\pazocal{Q}$.  By \Cref{basic annuli 1}(1), any maximal $\theta$-band of $\Delta_0$ must have at least one end on $\textbf{q}'$.  Hence, $|\textbf{q}'|_\theta\geq|\textbf{top}(\pazocal{Q}_j)|_\theta$, so that \Cref{lengths 1}(b) implies $|\textbf{q}'|\geq|\textbf{top}(\pazocal{Q}_j)|_\theta+2$.  \Cref{lengths 2}(a) then implies $|\textbf{top}(\pazocal{Q}_j)|=|\textbf{top}(\pazocal{Q}_j)|_\theta<|\textbf{q}'|$.

By \Cref{lengths 1}(c), $|\textbf{s}_1\textbf{top}(\pazocal{Q}_j)\textbf{s}_2|\leq|\textbf{s}_1|+|\textbf{top}(\pazocal{Q}_j)|+|\textbf{s}_2|$.  Meanwhile, since $\textbf{q}'$ starts and ends with $q$-edges, \Cref{lengths 1}(d) implies $|\textbf{q}|=|\textbf{s}_1|+|\textbf{q}'|+|\textbf{s}_2|$.  Hence, $|\textbf{s}_1\textbf{top}(\pazocal{Q}_j)\textbf{s}_2|<|\textbf{q}|$.

But applying van Kampen's Lemma to $\Delta_0$, $\lab(\textbf{q}')$ and $\lab(\textbf{top}(\pazocal{Q}_j))$ represent the same element of $G_\Omega(\textbf{M}^\pazocal{L})$, so that $\lab(\textbf{s}_1\textbf{top}(\pazocal{Q}_j)\textbf{s}_2)$ is a word over $\pazocal{X}\cup\pazocal{X}^{-1}$ representing $h^\eps$.  Thus, $|\textbf{s}_1\textbf{top}(\pazocal{Q}_j)\textbf{s}_2|<|\textbf{q}|=|h|$ yields a contradiction.

\end{proof}

\begin{lemma} \label{distortion diagram A}

Let $\partial\Delta=\textbf{qp}$ be the standard factorization of the contour of a reduced $h$-distortion diagram $\Delta$.  Then no $\pazocal{A}$-band of $\Delta$ has two ends on $\textbf{p}$.

\end{lemma}

\begin{proof}

Analogous to the proof of \Cref{distortion diagram q}, enumerate the $\pazocal{A}$-edge of $\textbf{p}$ by $\textbf{e}_1,\dots,\textbf{e}_k$.  Note that by definition, $\textbf{p}=\textbf{e}_1\dots\textbf{e}_k$.

For each $i\in\{1,\dots,k\}$, let $\pazocal{U}_i$ be the maximal $\pazocal{A}$-band for which $\textbf{e}_i$ is a defining edge.  Then, assuming the statement is false, there must exist $i,j\in\{1,\dots,k\}$ with $i<j$ such that $\textbf{e}_j^{-1}$ is an end of $\pazocal{U}_i$.


Let $\textbf{p}'=\textbf{e}_i\dots\textbf{e}_j$ and let $\textbf{s}_1$ and $\textbf{s}_2$ be the (perhaps trivial) subpaths of $\textbf{p}$ such that $\textbf{p}=\textbf{s}_1\textbf{p}'\textbf{s}_2$.  Then, analogous to the proof of \Cref{distortion diagram q}, let $\textbf{p}'$ and $\textbf{bot}(\pazocal{U}_i)$ bound a subdiagram $\Delta_0$ of $\Delta$ containing $\pazocal{U}$.  By \Cref{basic annuli 2}(1), any maximal $\theta$-band of $\Delta_0$ must have at least one end on $\textbf{p}'$.  But $|\textbf{p}'|_\theta=0$ by definition, so that $\Delta_0$ must contain no $\theta$-bands.

In particular, this implies $\pazocal{U}_i$ must be an $\pazocal{A}$-band of length $0$, so that $\lab(\textbf{bot}(\pazocal{U}_i))\equiv1$.

But then letting $\eps$ be the sign of $\Delta$, $w\equiv\lab(\textbf{s}_1)\lab(\textbf{s}_2)$ is a word over $\pazocal{A}\cup\pazocal{A}^{-1}$ representing $h^{-\eps}$ with $\|w\|=\|\textbf{s}_1\|+\|\textbf{s}_2\|<\|\textbf{p}\|=|h|_\pazocal{A}$, contradicting the definition of $|h|_\pazocal{A}$.

\end{proof}

\begin{lemma} \label{distortion diagram a-cell}

Let $\partial\Delta=\textbf{qp}$ be the standard factorization of the contour of an $h$-distortion diagram $\Delta$.  Suppose there exists a subpath $\textbf{x}$ of $\textbf{p}$ or $\textbf{q}$ such that $\textbf{x}$ is a subpath of $\partial\pi$ for some $a$-cell $\pi$.  Then $|\textbf{x}|_\pazocal{A}\leq\frac{1}{2}|\partial\pi|_\pazocal{A}$.

\end{lemma}

\begin{proof}

Assume toward contradiction that $|\textbf{x}|_\pazocal{A}>\frac{1}{2}|\partial\pi|_\pazocal{A}$ and let $\eps$ be the sign of $\Delta$.

Set $\textbf{y}$ be the subpath of $(\partial\pi)^{-1}$ such that $\partial\pi=\textbf{xy}^{-1}$.  So, $\lab(\textbf{y})$ is a word consisting entirely of $\pazocal{A}$-letters and $b$-letters with $|\textbf{y}|_\pazocal{A}<|\textbf{x}|_\pazocal{A}$.  Moreover, $\lab(\textbf{y})$ and $\lab(\textbf{x})$ represent the same element as $G_\Omega(\textbf{M}^\pazocal{L})$.  

First, suppose $\textbf{x}$ is a subpath of $\textbf{p}$.  Then, $\lab(\textbf{x})$ is a non-trivial word over $\pazocal{A}\cup\pazocal{A}^{-1}$, so that \Cref{M Lambda semi-computations} implies $\lab(\partial\pi)\in\Lambda^\pazocal{A}$.  Let $\textbf{p}_1$ and $\textbf{p}_2$ be the (perhaps trivial) subpaths of $\textbf{p}$ satisfying $\textbf{p}=\textbf{p}_1\textbf{x}\textbf{p}_2$.  But then $\lab(\textbf{p}_1\textbf{y}\textbf{p}_2)$ is a word over $\pazocal{A}\cup\pazocal{A}^{-1}$ representing $h^{-\eps}$ and satisfying $\|\textbf{p}_1\textbf{y}\textbf{p}_2\|<\|\textbf{p}\|$, contradicting the definition of $h$-distortion diagram.

Now, suppose $\textbf{x}$ is a subpath of $\textbf{q}$.  Similar to the previous setting, let $\textbf{q}_1$ and $\textbf{q}_2$ be the (perhaps trivial) subpaths of $\textbf{q}$ satisfying $\textbf{q}=\textbf{q}_1\textbf{x}\textbf{q}_2$.  Then, $\lab(\textbf{q}_1\textbf{y}\textbf{q}_2)$ is a word representing $h^\eps$ in $G_\Omega(\textbf{M}^\pazocal{L})$.  Hence, the definition of $h$-distortion diagram necessitates $|\textbf{q}_1\textbf{y}\textbf{q}_2|\geq|\textbf{q}|$.

If $|\textbf{y}|_\pazocal{A}\leq1$, then $|\textbf{x}|_\pazocal{A}=|\partial\pi|_\pazocal{A}-|\textbf{y}|_\pazocal{A}\geq C-1$, so that a parameter choice for $C$ implies $|\textbf{x}|_\pazocal{A}>3$.  As $\lab(\textbf{x})$ and $\lab(\textbf{y})$ both consist entirely of $\pazocal{A}$-letters and $b$-letters, \Cref{lengths 1}(a) then implies $|\textbf{y}|\leq\delta$ and $|\textbf{x}|>3\delta$.  But then \Cref{lengths 1}(c) yields the contradiction:
$$|\textbf{q}_1\textbf{y}\textbf{q}_2|\leq|\textbf{q}_1|+|\textbf{y}|+|\textbf{q}_2|\leq|\textbf{q}_1|+\delta+|\textbf{q}_2|<|\textbf{q}_1|+|\textbf{x}|-2\delta+|\textbf{q}_2|\leq|\textbf{q}|$$
Hence, it may be assumed that $|\textbf{y}|_\pazocal{A}\geq2$.  

Now, by \Cref{lengths 1}(d), there exists $\ell\in\{0,1,2\}$ such that $|\textbf{q}_1|+|\textbf{x}|+|\textbf{q}_2|-|\textbf{q}|=\ell\delta$.  In this setting, the value of $\ell$ corresponds to $\pazocal{A}$-edges of $\textbf{x}$ such that the corresponding letter can be placed in a $(\theta,a)$-syllable in a decomposition of $\textbf{q}$.  Since $|\textbf{x}|_\theta=0$, such an edge must be the first or the last $\pazocal{A}$-edge of $\textbf{x}$.  

But since $|\textbf{y}|_\pazocal{A}\geq2$, the first or last $\pazocal{A}$-edge of $\textbf{y}$ then corresponds to a letter that can be placed in a $(\theta,a)$-syllable in a decomposition of $\lab(\textbf{q}_1\textbf{y}\textbf{q}_2)$, meaning $|\textbf{q}_1|+|\textbf{y}|+|\textbf{q}_2|=|\textbf{q}_1\textbf{y}\textbf{q}_2|+\ell\delta$.

Thus, $|\textbf{q}_1\textbf{y}\textbf{q}_2|<|\textbf{q}|$, again yielding a contradiction.

\end{proof}

\begin{lemma} \label{distortion diagram pure a-scope}

Let $\partial\Delta=\textbf{qp}$ be the standard factorization of the contour of a reduced minimal $h$-distortion diagram $\Delta$.  Then there exists no big $a$-scope on $\textbf{p}$.

\end{lemma}

\begin{proof}

Suppose $\Psi$ is such a big $a$-scope on $\textbf{p}$.  By \Cref{pure a-scope}, it may be assumed that $\Psi$ is a pure big $a$-scope.

Let $\pi$ be the associated $a$-cell and $\textbf{s}$ be the associated subpath of $\Psi$.  Then, there exist $\pazocal{A}$-edges $\textbf{e}_1$ and $\textbf{e}_2$ of $\partial\pi$ such that $\Psi$ is bounded by $\pazocal{U}(\textbf{e}_1)$, $\pazocal{U}(\textbf{e}_2)$, $\textbf{s}$, and a subpath $\textbf{t}$ of $\textbf{p}$.

By \Cref{minimal is smooth}, every $\pazocal{A}$-edge of $\textbf{s}$ is the end of a maximal $\pazocal{A}$-band which has an end on $\textbf{t}$.  So, if a $\theta$-band of $\pazocal{T}$ of $\Psi$ crosses both $\pazocal{U}(\textbf{e}_1)$ and $\pazocal{U}(\textbf{e}_2)$, then it crosses the $|\textbf{s}|_\pazocal{A}>\frac{1}{2}|\partial\pi|_\pazocal{A}$ maximal $\pazocal{A}$-bands of $\Psi$ that have ends on $\textbf{s}^{-1}$.  But Lemmas \ref{minimal is smooth} and \ref{minimal diskless is M-minimal} imply that $\Delta$ satisfies condition (MM1), yielding a contradiction.  Hence, no $\theta$-band can cross both $\pazocal{U}(\textbf{e}_1)$ and $\pazocal{U}(\textbf{e}_2)$.

By \Cref{basic annuli 2}(1), it then follows that every maximal $\theta$-band which has an end on the side of $\pazocal{U}(\textbf{e}_i)$ has an end on $\textbf{t}$.  But $|\textbf{t}|_\theta=0$, so that $\pazocal{U}(\textbf{e}_1)$ and $\pazocal{U}(\textbf{e}_2)$ must be $\pazocal{A}$-bands of length 0.  As a result, $\Psi$ contains no $\theta$-bands at all, so that it must consist entirely of $0$-cells as it is pure.

Since $\lab(\textbf{s})$ and $\lab(\textbf{t})$ are both reduced words over $\pazocal{A}\cup\pazocal{A}^{-1}$ which are freely equal, $\textbf{s}$ and $\textbf{t}$ can be identified as subpaths of $\partial\Delta$.  But then $\pi$ and $\textbf{s}$ form a contradiction to \Cref{distortion diagram a-cell}.

\end{proof}

Let $\partial\Delta=\textbf{qp}$ be the standard factorization of the contour of an $h$-distortion diagram $\Delta$.  Note that since $|\textbf{p}|_\theta=0$, any maximal $\theta$-band $\pazocal{T}$ which has an end on $\partial\Delta$ must have two ends on $\textbf{q}$.  Let $\textbf{e}_1\textbf{b}\textbf{e}_2$ be the subpath of $\textbf{q}$ such that $\textbf{e}_1$ and $\textbf{e}_2$ are the $\theta$-edges corresponding to the ends of $\pazocal{T}$.  Without loss of generality, let $\textbf{e}_1^{-1}$ and $\textbf{e}_2$ be defining edges (indeed, the ends) of $\pazocal{T}$.  Observe that $\textbf{b}$ must be a non-trivial path since $\lab(\textbf{q})$ is reduced.

In contrast to $q$-bands (see \Cref{distortion diagram q}), the length function does not allow us to immediately rule out the existence of $\theta$-bands that have two ends on the boundary of an $h$-distortion diagram.  However, the next statements do say something about such bands.

\begin{lemma} \label{distortion diagram theta 0}

Every maximal $\theta$-band which has an end on the boundary of a reduced $h$-distortion diagram has positive length.

\end{lemma}

\begin{proof}

Let $\partial\Delta=\textbf{qp}$ be the standard factorization of the contour of a reduced $h$-distortion diagram $\Delta$ and suppose to the contrary that there exists a maximal $\theta$-band $\pazocal{T}$ with an end on $\partial\Delta$ of length $0$.  Then let $\textbf{e}_1\textbf{b}\textbf{e}_2$ be the subpath of $\textbf{q}$ corresponding to $\pazocal{T}$ and define the (perhaps trivial) subpaths $\textbf{q}_1$ and $\textbf{q}_2$ such that $\textbf{q}=\textbf{q}_1\textbf{e}_1\textbf{b}\textbf{e}_2\textbf{q}_2$.

As $\pazocal{T}$ has length $0$, $\textbf{e}_1^{-1}$ and $\textbf{e}_2$ are adjacent edges, so that $\lab(\textbf{e}_1\textbf{b}\textbf{e}_2)$ is trivial in $G_\Omega(\textbf{M}^\pazocal{L})$.  Cutting along any $0$-edges defining the adjacency allows us to remove the subpath $\textbf{e}_1\textbf{b}\textbf{e}_2$, producing a path $\textbf{q}'$ in $\Delta$ homotopic to $\textbf{q}$ such that $\lab(\textbf{q}')\equiv\lab(\textbf{q}_1)\lab(\textbf{q}_2)$ represents $h^\eps$ in $G_\Omega(\textbf{M}^\pazocal{L})$.  

But then \Cref{lengths 1}(c) implies $|\textbf{q}'|\leq|\textbf{q}_1|+|\textbf{q}_2|\leq|\textbf{q}|-|\textbf{e}_1\textbf{b}\textbf{e}_2|+2\delta\leq|\textbf{q}|-2+2\delta$, so that taking $\delta^{-1}>1$ produces a contradiction to the minimality of $|\textbf{q}|$.

\end{proof}

In the context above, if any positive cell of $\Delta$ between $\textbf{b}$ and $\textbf{bot}(\pazocal{T})$ is an $a$-cell, then $\pazocal{T}$ is called a \textit{quasi-rim $\theta$-band}.  Further, if there are no such $a$-cells ({\frenchspacing i.e. if} $\textbf{b}=\textbf{bot}(\pazocal{T})$), then $\pazocal{T}$ is called a \textit{rim $\theta$-band}.

\begin{lemma} \label{distortion diagram rim theta}

The base of a rim $\theta$-band in a reduced $h$-distortion diagram $\Delta$ has length $l_b>K$.

\end{lemma}

\begin{proof}

Let $\pazocal{T}$ be a rim $\theta$-band and, letting $\partial\Delta=\textbf{qp}$ be the standard factorization of $\partial\Delta$, define $\textbf{q}=\textbf{q}_1\textbf{e}_1\textbf{b}\textbf{e}_2\textbf{q}_2$ as in the proof of \Cref{distortion diagram theta 0}.


As $\textbf{b}=\textbf{bot}(\pazocal{T})$, then \Cref{theta-band lengths} implies $|\textbf{top}(\pazocal{T})|\leq|\textbf{b}|+2\delta l_b$.  Further, since $\textbf{top}(\pazocal{T})$ and $\textbf{e}_1\textbf{b}\textbf{e}_2$ bound a subdiagram of $\Delta$, then for $\eps$ the sign of $\Delta$, $\lab(\textbf{q}_1\textbf{top}(\pazocal{T})\textbf{q}_2)$ represents $h^\eps$ in $G_\Omega(\textbf{M}^\pazocal{L})$.  So, $|\textbf{q}_1\textbf{top}(\pazocal{T})\textbf{q}_2|\geq|h|=|\textbf{q}|$.

But \Cref{lengths 1}(c) implies 
\begin{align*}
|\textbf{q}_1\textbf{top}(\pazocal{T})\textbf{q}_2|&\leq|\textbf{q}_1|+|\textbf{top}(\pazocal{T})|+|\textbf{q}_2|\leq|\textbf{q}_1|+|\textbf{b}|+|\textbf{q}_2|+2\delta l_b \\
&\leq|\textbf{q}_1|+|\textbf{e}_1|+|\textbf{b}|+|\textbf{e}_2|+|\textbf{q}_2|+2\delta l_b-2\leq|\textbf{q}|+2\delta(l_b+2)-2
\end{align*}
Hence, $l_b\geq\delta^{-1}-2$, so that the parameter choice $\delta^{-1}>>K$ implies $l_b>K$.

\end{proof}

\begin{lemma} \label{distortion diagram quasi-rim theta}

The base of a quasi-rim $\theta$-band in a reduced minimal $h$-distortion diagram $\Delta$ has length $l_b>K$.

\end{lemma}

\begin{proof}

As in the previous two proofs, let $\partial\Delta=\textbf{qp}$ be the standard factorization of $\partial\Delta$, let $\textbf{e}_1\textbf{b}\textbf{e}_2$ be the subpath of $\textbf{q}$ corresponding to the quasi-rim $\theta$-band $\pazocal{T}$, and let $\textbf{q}_1,\textbf{q}_2$ be the (perhaps trivial) subpaths of $\textbf{q}$ such that $\textbf{q}=\textbf{q}_1\textbf{e}_1\textbf{b}\textbf{e}_2\textbf{q}_2$.

By \Cref{distortion diagram rim theta}, it suffices to assume that (through $0$-refinement) $\textbf{b}$ and $\textbf{bot}(\pazocal{T})$ bound a subdiagram $\Delta_0$ of $\Delta$ consisting of $a$-cells $\pi_1,\dots,\pi_k$.

As no cell of $\Delta_0$ is a $(\theta,q)$- or $(\theta,a)$-cell, $|\textbf{b}|_\theta=0$ by \Cref{distortion diagram theta 0}.  An identical proof to that of \Cref{distortion diagram pure a-scope} then implies there is no big $a$-scope on $\textbf{b}$.  Hence, by \Cref{pure a-scope} any $a$-scope on $\textbf{b}$ with associated $a$-cell $\pi_i$ is a pure $a$-scope that is not big.  In particular, since $\pi_1,\dots,\pi_k$ comprise every positive cell of $\Delta_0$, for each $i$ there exists a maximal (perhaps trivial) subpath $\textbf{t}_i$ of $\partial\pi_i$ shared with $\textbf{b}$ such that every $\pazocal{A}$-edge shared by $\partial\pi_i$ and $\textbf{b}$ is an edge of $\textbf{t}_i$.  Observe that $|\textbf{t}_i|_\pazocal{A}\leq\frac{1}{2}|\partial\pi_i|_\pazocal{A}$ since the corresponding $a$-scope is not big.



Noting that $\Delta_0$ is an $M$-minimal diagram by \Cref{minimal diskless is M-minimal}, \Cref{Gamma_a special cell} implies there exists $j\in\{1,\dots,k\}$ such that $|\partial\pi_{j}|_\pazocal{A}-6$ $\pazocal{A}$-edges of $\partial\pi_{j}$ are shared with $\partial\Delta_0$.  Since $|\textbf{t}_{j}|_\pazocal{A}\leq\frac{1}{2}|\partial\pi_{j}|_\pazocal{A}$, at least $\frac{1}{2}|\partial\pi_{j}|_\pazocal{A}-6\geq \frac{1}{2}C-6$ $\pazocal{A}$-edges of $\partial\pi_{j}$ are shared with $\textbf{bot}(\pazocal{T})^{-1}$.  

So, \Cref{pure a-scope} and the parameter choice $C\geq26$ implies there exists a pure $a$-scope on $\textbf{bot}(\pazocal{T})^{-1}$ of size at least $7$.  Letting $\pi_{\ell_1}$ be the associated $a$-cell of this pure $a$-scope, the associated subpath $\textbf{s}_{\ell_1}$ is a subpath of both $\partial\pi_{\ell_1}$ and $\textbf{bot}(\pazocal{T})^{-1}$ and satisfies $|\textbf{s}_{\ell_1}|_\pazocal{A}\geq7$.  Hence, since $\pazocal{A}$-bands and $q$-bands cannot cross, $\textbf{s}_{\ell_1}$ contains at least $5$ edges of $E(\pi_{\ell_1},\pazocal{T})$, and so we may perform the transposition of $\pi_{\ell_1}$ and $\pazocal{T}$ along $\textbf{s}_{\ell_1}$.

Note that the result of this transposition is a reduced diagram $\Delta_1$ with the same contour label and $3$-signature, and so $\Delta_1$ is a reduced $3$-minimal $h$-distortion diagram.  
By construction, the $\theta$-band $\pazocal{T}_1$ arising from $\pazocal{T}$ has the same base.  Further, identifying $\textbf{b}$ with a subpath of $\partial\Delta_1$, $\textbf{bot}(\pazocal{T}_1)$ and $\textbf{b}$ bound a subdiagram $\Delta_{1,0}$ comprised of $k-1$ $a$-cells identified with the subdiagram of $\Delta_0$ obtained by removing $\pi_{\ell_1}$.  

Hence, while $\Delta_1$ need not be minimal, the subdiagram $\Delta_{1,0}$ is still $M$-minimal.  Thus, the process may be iterated to move all the $a$-cells past the $\theta$-band, producing a reduced $h$-distortion diagram containing a rim $\theta$-band whose base has length $l_b$.  The statement is therefore given by \Cref{distortion diagram rim theta}.

\end{proof}

\medskip


\subsection{Disks in $h$-distortion diagrams} \

Our next goal is to show that a reduced minimal $h$-distortion diagram cannot contain any disks.  This requires the most technical argument of this paper; to present it as efficiently as possible, we introduce auxiliary terminology specific to this setting.

Fix a reduced minimal $h$-distortion diagram $\Delta$ containing a disk.  Let $\partial\Delta=\textbf{qp}$ be the standard factorization of the contour and $\eps$ be the sign of $\Delta$.  Lemmas \ref{Gamma special cell} and \ref{pure scope} then imply the existence of a pure scope $\Psi$ on $\textbf{q}$ of size $L-6$.  

Let $\textbf{s}$ be the associated subpath and $\Pi$ be the associated disk of $\Psi$.  Perhaps passing to the mirror $\bar{\Delta}$, it may be assumed that $\lab(\partial\Pi)\equiv W$ for some accepted configuration $W$ of $\textbf{M}^\pazocal{L}$. 

Enumerate the $t$-edges of $\textbf{s}$ by $\textbf{e}_1,\dots,\textbf{e}_{L-6}$.  Then, for each $i\in\{1,\dots,L-6\}$, let $\pazocal{Q}_i$ be the $t$-spoke $\pazocal{Q}(\textbf{e}_i)$ of $\Pi$.  As $\Psi$ is a pure scope, each $\pazocal{Q}_i$ must have an end on $\textbf{q}$.  In particular, there exists a factorization $\partial\Psi=\textbf{s}^{-1}(\textbf{bot}(\pazocal{Q}_1))\textbf{t}(\textbf{top}(\pazocal{Q}_{L-6}))^{-1}$.

Let $\textbf{z}$ be the subpath of $(\partial\Pi)^{-1}$ such that $\textbf{z}^{-1}$ is the complement of $\textbf{s}$ in $\partial\Pi$, {\frenchspacing i.e. $\partial\Pi=\textbf{s}\textbf{z}^{-1}$}.  Then, define the path $\textbf{t}_0=(\textbf{bot}(\pazocal{Q}_1)^{-1})\textbf{z}(\textbf{top}(\pazocal{Q}_{L-6}))$ in $\Delta$.  Note that, by definition, $\textbf{t}_0$ is combinatorially homotopic to $\textbf{t}$.

\begin{lemma} \label{clove path}

For any word $w$ over $\pazocal{X}\cup\pazocal{X}^{-1}$ which represents the same element of $G_\Omega(\textbf{M}^\pazocal{L})$ as $\lab(\textbf{t}_0)$, $|w|\geq|\textbf{t}|$.

\end{lemma}

\begin{proof}

Let $\textbf{q}_1$ and $\textbf{q}_2$ be the (perhaps trivial) subpaths of $\textbf{q}$ such that $\textbf{q}=\textbf{q}_1\textbf{t}\textbf{q}_2$.  Then, $(\lab(\textbf{q}_1))w(\lab(\textbf{q}_2))$ represents $h^\eps$ in $G_\Omega(\textbf{M}^\pazocal{L})$.  So, \Cref{lengths 1}(c) implies:
$$|\textbf{q}_1|+|w|+|\textbf{q}_2|=|\lab(\textbf{q}_1)|+|w|+|\lab(\textbf{q}_2)|\geq|(\lab(\textbf{q}_1))w(\lab(\textbf{q}_2))|\geq|h|$$
But the first and last edges of $\textbf{t}$ are $q$-edges, so that \Cref{lengths 1}(d) implies 
$$|h|=|\textbf{q}|=|\textbf{q}_1|+|\textbf{t}|+|\textbf{q}_2|$$ 

\end{proof}

\begin{figure}[H]
\centering
\includegraphics[scale=2.75]{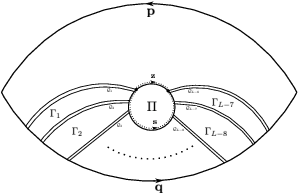}
\caption{Cloves formed by the disk $\Pi$ in the $h$-distortion diagram $\Delta$}
\label{h-distortion-disk}
\end{figure}


Now, for each $i\in\{1,\dots,L-7\}$, let $\Gamma_i$ be the subdiagram of $\Psi$ bounded by the $t$-bands $\pazocal{Q}_i$ and $\pazocal{Q}_{i+1}$ (see \Cref{h-distortion-disk}).  Each such subdiagram $\Gamma_i$ is called a \textit{clove}.  Note that for each $i$, there exist subpaths $\textbf{s}_i$ and $\textbf{t}_i$ of $\textbf{s}$ and $\textbf{t}$, respectively, such that $\partial\Gamma_i=\textbf{s}_i^{-1}(\textbf{bot}(\pazocal{Q}_i))\textbf{t}_i(\textbf{top}(\pazocal{Q}_{i+1}))^{-1}$.

For each $i$, $\Gamma_i$ and $\Gamma_{i+1}$ intersect along the $t$-band $\pazocal{Q}_{i+1}$.  So, the cloves $\Gamma_1,\dots,\Gamma_{L-7}$ form a `cover' of $\Psi$.  Moreover, for any $1\leq k<\ell\leq L-7$, there exists a subdiagram $\Psi_{k,\ell}$ `covered' by $\Gamma_k,\dots,\Gamma_{\ell-1}$; in other words, $\Psi_{k,\ell}$ is the subdiagram of $\Psi$ bounded by $\pazocal{Q}_k$ and $\pazocal{Q}_\ell$.  Note that it follows from this definition that $\Psi_{1,L-6}=\Psi$ and $\Psi_{i,i+1}=\Gamma_i$.

Let $\textbf{s}_{k,\ell}$ be the minimal subpath of $\textbf{s}$ containing each subpath $\textbf{s}_i$ for $k\leq i\leq\ell-1$.  Similarly, let $\textbf{t}_{k,\ell}$ be the minimal subpath of $\textbf{t}$ containing each $\textbf{t}_i$.  Then $\partial\Psi_{k,\ell}=\textbf{s}_{k,\ell}^{-1}(\textbf{bot}(\pazocal{Q}_k))\textbf{t}_{k,\ell}(\textbf{top}(\pazocal{Q}_\ell))^{-1}$.

\begin{lemma} \label{clove theta-bands}

Let $k,\ell\in\{1,\dots,L-6\}$ such that $\ell-k>(L-3)/2$.  Then any maximal $\theta$-band of $\Psi_{k,\ell}$ has one end on $\textbf{t}_{k,\ell}$ and crosses exactly one of either $\pazocal{Q}_k$ or $\pazocal{Q}_\ell$.

\end{lemma}

\begin{proof}

Let $\pazocal{T}$ be a maximal $\theta$-band in $\Psi_{k,\ell}$.  By \Cref{disk theta-annuli}, $\pazocal{T}$ must have two ends on $\partial\Psi_{k,\ell}$.  As $|\textbf{s}_{k,\ell}|_\theta=0$, these ends must be on $\textbf{t}_{k,\ell}$, on $\textbf{bot}(\pazocal{Q}_k)$, or on $\textbf{top}(\pazocal{Q}_\ell)^{-1}$.  By \Cref{basic annuli 1}(1), $\pazocal{T}$ has at most one end on $\textbf{bot}(\pazocal{Q}_k)$ and at most one end on $\textbf{top}(\pazocal{Q}_\ell)^{-1}$.

First, suppose $\pazocal{T}$ crosses both $\pazocal{Q}_k$ and $\pazocal{Q}_\ell$.  Then, $\pazocal{T}$ crosses each of the $t$-bands $\pazocal{Q}_i$ for $k\leq i\leq\ell$.  In particular, viewing it as a $\theta$-band in $\Delta$, $\pazocal{T}$ crosses $\ell-k+1>(L-1)/2$ $t$-spokes of $\Pi$.  But $\Delta$ is a reduced minimal diagram, so that $\Pi$ and $\pazocal{T}$ contradict \Cref{minimal t-spokes theta-band}.  Hence, $\pazocal{T}$ must have at least one end on $\textbf{t}_{k,\ell}$.  

Now, suppose $\pazocal{T}$ has two ends on $\textbf{t}_{k,\ell}$.  By \Cref{distortion diagram q} and the makeup of the disk relations, any maximal $q$-band of $\Psi_{k,\ell}$ has ends on both $\textbf{s}_{k,\ell}^{-1}$ and $\textbf{t}_{k,\ell}$.  So, since $\pazocal{T}$ crosses any of these bands at most once, the length of the base of $\pazocal{T}$ is at most $|\textbf{s}_{k,\ell}|_q\leq 3LN$.  In particular, the parameter assignments $K>>L>>N$ imply that the length of the base of $\pazocal{T}$ is at most $K$.  But then this implies the existence of a quasi-rim $\theta$-band in $\Delta$ with base of length at most $K$, contradicting \Cref{distortion diagram quasi-rim theta}.  Thus, $\pazocal{T}$ has exactly one end on $\textbf{t}_{k,\ell}$ and so the statement follows.

\end{proof}

For every $i\in\{1,\dots,L-6\}$, let $H_i$ be the history of $\pazocal{Q}_i$.  

For $2\leq k<(L-9)/2$, applying \Cref{clove theta-bands} to $\Psi_{k,L-6}$ and to $\Psi_{k-1,L-6}$ implies that any maximal $\theta$-band of $\Psi_{k-1,L-6}$ that crosses $\pazocal{Q}_k$ must have ends on $\textbf{bot}(\pazocal{Q}_{k-1})$ and on $\textbf{t}_{k,L-6}$.  In particular, this implies $H_k$ is a prefix of $H_{k-1}$.

Similarly, for $(L-1)/2<\ell\leq L-7$, any maximal $\theta$-band of $\Psi_{1,\ell+1}$ that crosses $\pazocal{Q}_\ell$ has ends on $\textbf{top}(\pazocal{Q}_{\ell+1})^{-1}$ and on $\textbf{t}_{1,\ell}$, so that $H_\ell$ is a prefix of $H_{\ell+1}$.

Hence, letting $h_i=\|H_i\|$ for each $i$, the parameter choice $L>23$ implies $h_1\geq\dots\geq h_7$ and $h_{L-12}\leq\dots\leq h_{L-6}$.

For each $i\in\{1,\dots,L-6\}$, fix the index $j_i\in\{2,\dots,L\}$ such that $\pazocal{Q}_i$ is a $t$-band corresponding to the part $\{t(j_i)\}$ of the standard base of $\textbf{M}^\pazocal{L}$.  If there exists an index $i\in\{1,\dots,L-7\}$ such that $j_i=L$, then $\Gamma_i$ is called the \textit{distinguished clove}.   The makeup of the disk relations immediately implies that there is at most one distinguished clove.  Note that $\lab(\textbf{s}_i)\equiv W(L)W(1)t(2)$ if $\Gamma_i$ is the distinguished clove, while $\lab(\textbf{s}_i)\equiv W(j_i)t(j_i+1)$ otherwise.

\begin{lemma} \label{non-distinguished a-cells}

If $\Gamma_i$ is not the distinguished clove, then it contains no $a$-cells.

\end{lemma}

\begin{proof}

By \Cref{distortion diagram q}, every maximal $q$-band of $\Gamma_i$ has an end on $\textbf{s}_i^{-1}$.  So, since $\Gamma_i$ is not the distinguished clove, every $q$-band corresponds to a part of the standard base of $\textbf{M}^\pazocal{L}$ with coordinate $j_i$ or $j_{i+1}$.  In particular, no $q$-band corresponds to a part with coordinate 1.

Further, since \Cref{clove theta-bands} implies that every $\theta$-band must cross at least one $q$-band, no $(\theta,\pazocal{A})$-cell of $\Gamma_i$ can correspond to a relation of the `special' input sector.  Hence, every $\pazocal{A}$-band with one end on an $a$-cell must be of length $0$.

Now, suppose $\Gamma_i$ contains an $a$-cell.  Then, Lemmas \ref{Gamma_a special cell}, \ref{minimal diskless is M-minimal}, and \ref{minimal is smooth} imply the existence of an $a$-cell $\pi_0$ and $\ell\geq|\partial\pi_0|_\pazocal{A}-6$ consecutive $\pazocal{A}$-edges $\textbf{f}_1,\dots,\textbf{f}_\ell$ of $\partial\pi_0$ such that each $\pazocal{A}$-band $\pazocal{U}(\textbf{f}_j)$ is external.  As such, each $\textbf{f}_j$ must be an edge of $\textbf{t}_i$.  As $|\partial\pi_0|_\pazocal{A}\geq C$, the parameter choice $C\geq13$ then implies the existence of a big $a$-scope on $\textbf{t}_i$ with associated $a$-cell $\pi_0$.  \Cref{pure a-scope} then implies the existence of a pure big $a$-scope on $\textbf{t}_i$.  

Letting $\pi$ be the associated $a$-cell and $\textbf{x}$ the associated subpath of this pure $a$-scope, $\textbf{x}$ is a subpath of $\partial\pi$ and each edge of $\textbf{x}$ is an edge of $\textbf{t}_i$.  But then since $\lab(\textbf{t}_i)$ is reduced, $\textbf{x}$ is a subpath of $\textbf{t}_i$ with $|\textbf{x}|_\pazocal{A}>\frac{1}{2}|\partial\pi|_\pazocal{A}$, contradicting \Cref{distortion diagram a-cell}.

\end{proof}

Fix $j\in\{2,\dots,L\}$ and suppose $\Sigma$ is a circular diagram over $M(\textbf{M}^\pazocal{L})$ such that every cell has coordinate $j$ and no cell is a $(\theta,t)$-cell.  So, for any cell $\pi$ in $\Sigma$, $\lab(\partial\pi)$ is given by either a $(\theta,q)$- or a $(\theta,a)$-relation with coordinate $j$.  Then, for any $r\in\{2,\dots,L\}$, the parallel nature of the rules of $\textbf{M}^\pazocal{L}$ implies the existence of another such relation obtained from this relation by:

\begin{itemize}

\item Switching the coordinate of any $q$-letter from $j$ to $r$

\item Taking the copy of any $a$-letter in the tape alphabet of the corresponding sector of $B_4^\pazocal{L}(r)$

\item Adjusting the index of the $\theta$-letters accordingly

\end{itemize}

The relation obtained can then be written on the boundary of a cell to produce a `copy' $\pi(r)$ of $\pi$, with the structure of the cell remaining much the same.  Replacing every cell of $\Sigma$ with its `copy' then produces a circular diagram $\Sigma(r)$ over $M(\textbf{M}^\pazocal{L})$ with much the same structure as $\Sigma$, but so that every cell has coordinate $r$ and no cell is a $(\theta,t)$-cell.

Note that by construction, if $\Sigma$ is a trapezium, then the label of the trimmed side of any maximal $\theta$-band of $\Sigma(r)$ is a coordinate shift (see \Cref{sec-almost-extendable}) of the label of the trimmed side of the corresponding $\theta$-band of $\Sigma$.  As such, $\Sigma(r)$ is called a \textit{coordinate shift} of $\Sigma$.

Further, suppose that for any $\theta$-band $\pazocal{T}$ of $\Sigma$ whose history is a rule of the second machine, no cell comprising $\pazocal{T}$ is either:

\begin{itemize}

\item a $(\theta,q)$-cell which is part of a $q$-band corresponding to the part $Q_1^\pazocal{L}(j)$, or

\item a $(\theta,a)$-cell of the input $Q_0^\pazocal{L}(j)Q_1^\pazocal{L}(j)$-sector.

\end{itemize}

Then in the same way as above, we may construct the coordinate shift $\Sigma(1)$.  In this case, $\Sigma$ is called \textit{exceptional}.

Next, recall the symmetry of the machine $\textbf{M}_4^\pazocal{L}$ arising from the `reflected copies' of the machine $\textbf{M}_3^\pazocal{L}$ in its construction (see \Cref{sec-M_4}).  Given this symmetry, for any cell $\pi$ in $\Sigma$, the $(\theta,q)$- or $(\theta,a)$-relation defining $\lab(\partial\pi)$ corresponds to a `reflected' such relation obtained by:

\begin{itemize}

\item Taking the inverse of maximal (cyclic) subwords not containing $\theta$-letters

\item Replacing any resulting $q$-letter of $Q_i^\pazocal{L}(j)$ with its copy in $R_i^\pazocal{L}(j)$, and vice versa

\item Replacing any resulting $a$-letter of $Y_i^\pazocal{L}(j)$ with its copy in $\pazocal{Y}_i^\pazocal{L}(j)$, and vice versa

\item Adjusting the index of the $\theta$-letters accordingly

\end{itemize}

The relation obtained can then be written on the boundary of a cell to produce a `reflected copy' $\bar{\pi}$ of $\pi$ whose structure is that of a `mirror image' of $\pi$ (see \Cref{Reflected-copy}(a)).  As such, for any maximal positive $\theta$-band $(\pi_1,\dots,\pi_k)$ of $\Sigma$, we may construct a maximal positive $\theta$-band $(\bar{\pi}_k,\dots,\bar{\pi}_1)$ with the same history.  Doing so for all maximal positive $\theta$-bands produces a circular diagram $\bar{\Sigma}$ over $M(\textbf{M}^\pazocal{L})$ such that every cell has coordinate $j$ and no cell is a $(\theta,t)$-cell (see \Cref{Reflected-copy}(b)).  Accordingly, $\bar{\Sigma}$ is called the \textit{reflected copy} of $\Sigma$.

\begin{figure}[H]
\centering
\begin{subfigure}[b]{0.48\textwidth}
\centering
\includegraphics[scale=1.5]{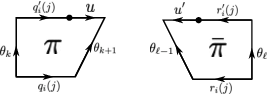}
\caption{The `reflected copy' $\bar{\pi}$ of a $(\theta,q)$-cell $\pi$ corresponding to a relation involving a part $Q_i^\pazocal{L}(j)$ of the standard base}
\end{subfigure}\hfill
\begin{subfigure}[b]{0.48\textwidth}
\centering
\includegraphics[scale=1]{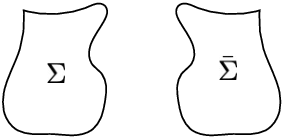}
\caption{The reflected copy $\bar{\Sigma}$ of a circular diagram $\Sigma$ \\ \ \\ \ }
\end{subfigure}
\caption{ \ }
\label{Reflected-copy}
\end{figure}


%
%

Now, fix an index $i\in\{1,\dots,L-7\}$ such that $\Gamma_i$ is not the distinguished clove.  Let $\Sigma_i$ be the subdiagram of $\Gamma_i$ obtained by removing the $t$-bands $\pazocal{Q}_i$ and $\pazocal{Q}_{i+1}$.  Combining Lemmas \ref{clove theta-bands} and \ref{non-distinguished a-cells}, $\Sigma_i$ is a circular diagram over $M(\textbf{M}^\pazocal{L})$ such that every cell has coordinate $j_i$ and no cell is a $(\theta,t)$-cell.  Hence, we may construct the reflected copy $\bar{\Sigma}_i$.

Consider the factorization $\partial\Sigma_i=\textbf{x}_i^{-1}\textbf{p}_{i}\textbf{y}_i\textbf{q}_{i}^{-1}$ such that $\textbf{p}_{i}=\textbf{top}(\pazocal{Q}_i)$, $\textbf{q}_{i}=\textbf{bot}(\pazocal{Q}_{i+1})$, and $\textbf{x}_i,\textbf{y}_i$ are subpaths of $\textbf{s}_i,\textbf{t}_i$, respectively.  Then, there exists a factorization  $\partial\bar{\Sigma}_i=(\bar{\textbf{x}}_i^{-1}\bar{\textbf{p}}_{i}\bar{\textbf{y}}_i\bar{\textbf{q}}_{i}^{-1})^{-1}$ where the naming of each subpath is indicative of its correspondence to a subpath of $\partial\Sigma_i$.  

As $\lab(\textbf{x}_i)$ is an admissible subword of the accepted configuration $W$ with base $B_4^\pazocal{L}(j_i)$, the symmetric nature of the rules of $\textbf{M}_4^\pazocal{L}$ then implies that $\lab(\textbf{x}_i)\equiv\lab(\bar{\textbf{x}}_i^{-1})$.  Moreover, $\lab(\bar{\textbf{p}}_{i})$ is the word over $T\cup T^{-1}$ obtained from $\lab(\textbf{p}_{i})$ by switching the index of each $\theta$-letter to that of the letters comprising $\lab(\textbf{q}_{i})$; $\lab(\bar{\textbf{q}}_{i})$ is obtained from $\lab(\textbf{q}_{i})$ analogously.

Since $\bar{\Sigma}_i$ is a circular diagram over $M(\textbf{M}^\pazocal{L})$ such that every cell has coordinate $j_i$ and no cell is a $(\theta,t)$-cell, we may then construct its coordinate shift $\bar{\Sigma}_i(r)$ for any $r\in\{2,\dots,L\}$.  Then, there exists a factorization $\partial\bar{\Sigma}_i(r)=((\bar{\textbf{x}}_i(r))^{-1}\bar{\textbf{p}}_{i}(r)\bar{\textbf{y}}_i(r)(\bar{\textbf{q}}_{i}(r))^{-1})^{-1}$ such that each subpath arises from the corresponding subpath of $\partial\bar{\Sigma}_i$ in the natural way.  As such, $\lab((\bar{\textbf{x}}_i(r))^{-1})$ is the coordinate shift of $\lab(\textbf{x}_i)$ with base $B_4^\pazocal{L}(r)$ while $\lab(\bar{\textbf{p}}_{i}(r))$ and $\lab(\bar{\textbf{q}}_i(r))$ are the words over $T\cup T^{-1}$ obtained by simply changing the indices of each $\theta$-letter accordingly.

Let $\pazocal{Q}_i(r,1)$ be the positive $q$-band corresponding to $\{t(r+1)\}$ with history $H_i$.  Note that $\pazocal{Q}_i(r,1)$ is a $t$-band unless $r=L$, in which case $r+1$ is taken to be $1$ and $\pazocal{Q}_i(r,1)$ is a $q$-band corresponding to $\{t(1)\}$.  Observe that $\lab(\textbf{bot}(\pazocal{Q}_i(r,1)))\equiv\lab(\bar{\textbf{p}}_{i}(r))$.  

On the other hand, letting $\pazocal{Q}_i(r,2)$ be the positive $t$-band corresponding to $\{t(r)\}$ with history $H_{i+1}$, then $\lab(\textbf{top}(\pazocal{Q}_i(r,2)))\equiv\lab(\bar{\textbf{q}}_{i}(r))$.

Hence, $\pazocal{Q}_i(r,1)$ and $\pazocal{Q}_i(r,2)$ may be pasted to $\bar{\Sigma}_i(r)$ by identifying the corresponding paths, yielding a circular diagram $\bar{\Gamma}_i(r)$.  By construction, there exists a factorization 
$$\partial\bar{\Gamma}_i(r)=(\textbf{s}_i(r))^{-1}\textbf{bot}(\pazocal{Q}_i(r,2))\textbf{t}_i(r)\textbf{top}(\pazocal{Q}_i(r,1))^{-1}$$ 
such that  $\lab(\textbf{s}_i(r))\equiv W(r)t(r+1)$ and $|\textbf{t}_i(r)|=|\textbf{t}_i|$.  As such, $\lab(\textbf{s}_i(r))$ is an admissible subword of $W$, and so $\textbf{s}_i(r)$ can be identified with a subpath of $\partial\Pi$.

Suppose neither $\Gamma_i$ nor $\Gamma_{i+1}$ is the distinguished clove and that $2< r\leq L$, {\frenchspacing i.e. so that we may construct both} $\bar{\Gamma}_i(r)$ and $\bar{\Gamma}_{i+1}(r-1)$.  By construction, the $t$-bands $\pazocal{Q}_i(r,2)$ and $\pazocal{Q}_{i+1}(r-1,1)$ are identical, and so we may paste $\bar{\Gamma}_i(r)$ and $\bar{\Gamma}_{i+1}(r-1)$ together by identifying these bands, producing a circular diagram $\bar{\Psi}_{i,i+2}(r)$ whose structure is that of a `mirror copy' of the diagram $\Psi_{i,i+2}$.  Iterating this procedure yields a circular diagram $\bar{\Psi}_{k,\ell}(r)$ whose structure is that of a `mirror copy' of $\Psi_{k,\ell}$ for appropriate choices of $k$, $\ell$, and $r$.  

Observe that the notation $\bar{\Psi}_{k,\ell}(r)$ indicates that the shift is made so that the copy of $\bar{\Gamma}_k$ has coordinate $r$, while the copy of $\bar{\Gamma}_{\ell-1}$ has coordinate $r-(\ell-1-k)$.  


\begin{lemma} \label{no distinguished clove}

There is no distinguished clove in $\Psi$.

\end{lemma}

\begin{proof}

Assume toward contradiction that $\Gamma_d$ is the distinguished clove for some $d\in\{1,\dots,L-7\}$.

Then, $j_{L-6}\in\{2,\dots,L-7\}$ with $\lab(\textbf{e}_{L-6}\textbf{z}^{-1}\textbf{e}_1)\equiv W(j_{L-6})\dots W(j_{L-6}+5)t(j_{L-6}+6)$ and $j_{L-6}+6=j_1$.

Suppose $d\geq7$.  Then, we may construct the circular diagram $\bar{\Psi}_{1,7}(j_1-1)$.  By construction, there exists a factorization $\partial\bar{\Psi}_{1,7}(j_1-1)=(\textbf{s}_{1,7}')^{-1}\textbf{bot}(\pazocal{Q}_6(j_1-6,2))\textbf{t}_{1,7}'\textbf{top}(\pazocal{Q}_1(j_1-1,1))^{-1}$ such that $|\textbf{t}_{1,7}'|=|\textbf{t}_{1,7}|$ and $\lab(\textbf{s}_{1,7}')\equiv\lab(\textbf{e}_{L-6}\textbf{z}^{-1}\textbf{e}_1)$.

Note that, by construction, $\pazocal{Q}_1(j_1-1,1)$ is identical to $\pazocal{Q}_1$.  Further, $\pazocal{Q}_6(j_1-6,2)$ is a $t$-band with history $H_7$ corresponding to the part $\{t(j_{L-6})\}$ of the standard base.

Let $\bar{\Phi}_{1,7}$ be the diagram obtained from $\bar{\Psi}_{1,7}(j_1-1)$ by removing the bands $\pazocal{Q}_1(j_1-1,1)$ and $\pazocal{Q}_6(j_1-6,2)$.  Then, applying \Cref{lengths 1}(d) and \Cref{lengths 2}(a), there exists a factorization $\partial\bar{\Phi}_{1,7}=(\textbf{s}_{1,7}'')^{-1}\textbf{p}_{1,7}''\textbf{t}_{1,7}''(\textbf{q}_{1,7}'')^{-1}$ such that:

\begin{itemize}

\item $\lab(\textbf{s}_{1,7}'')\equiv\lab(\textbf{z}^{-1})$

\item $|\textbf{t}_{1,7}''|=|\textbf{t}_{1,7}|-2$

\item $\lab(\textbf{q}_{1,7}'')\equiv\lab(\textbf{bot}(\pazocal{Q}_1))$

\item $|\textbf{p}_{1,7}''|=h_7$

\end{itemize}

In particular, $\lab(\textbf{p}_{1,7}''\textbf{t}_{1,7}'')^{-1}$ represents the same element of $G_\Omega(\textbf{M}^\pazocal{L})$ as $\lab(\textbf{bot}(\pazocal{Q}_1)^{-1}\textbf{z})$ and, by \Cref{lengths 1}(c), $|\textbf{p}_{1,7}''\textbf{t}_{1,7}''|\leq h_7+|\textbf{t}_{1,7}|-2$.  As a result, $w\equiv\lab(\textbf{p}_{1,7}''\textbf{t}_{1,7}'')^{-1}\lab(\textbf{top}(\pazocal{Q}_{L-6}))$ represents the same element of $G_\Omega(\textbf{M}^\pazocal{L})$ as $\lab(\textbf{t}_0)$ and satisfies $|w|\leq h_7+|\textbf{t}_{1,7}|-2+h_{L-6}$.

But \Cref{clove theta-bands} and a parameter choice for $L$ imply that $|\textbf{t}_{7,L-6}|_\theta= h_7+h_{L-6}$, so that \Cref{lengths 1} yields $|w|\leq|\textbf{t}_{1,7}|+|\textbf{t}_{7,L-6}|_\theta-2<|\textbf{t}_{1,7}|+|\textbf{t}_{7,L-6}|-1=|\textbf{t}|$, contradicting \Cref{clove path}.  

Hence, it may be assumed that $d\leq 6$.  

By a parameter choice for $L$, we may then assume that $L-12\geq d$.  This implies the ability to construct the diagram $\bar{\Psi}_{L-12,L-6}(j_{L-6}+5)$ and then to remove $\pazocal{Q}_{L-12}(j_{L-6}+5,1)$ and $\pazocal{Q}_{L-7}(j_{L-6},2)$ to produce the circular diagram $\bar{\Phi}_{L-12,L-6}$.  Observing that the $t$-band $\pazocal{Q}_{L-7}(j_{L-6},2)$ is identical to $\pazocal{Q}_{L-6}$, then analogous to the arguments in the previous case, there exists a factorization $\partial\bar{\Phi}_{L-12,L-6}=(\textbf{s}_{L-12,L-6}'')^{-1}\textbf{p}_{L-12,L-6}''\textbf{t}_{L-12,L-6}''(\textbf{q}_{L-12,L-6}'')^{-1}$ such that:

\begin{itemize}

\item $\lab(\textbf{s}_{L-12,L-6}'')\equiv\lab(\textbf{z}^{-1})$

\item $|\textbf{t}_{L-12,L-6}''|=|\textbf{t}_{L-12,L-6}|-2$

\item $\lab(\textbf{p}_{L-12,L-6}'')\equiv\lab(\textbf{top}(\pazocal{Q}_{L-6}))$

\item $|\textbf{q}_{L-12,L-6}''|=h_{L-12}$

\end{itemize}

So, $\lab(\textbf{t}_{L-12,L-6}''(\textbf{q}_{L-12,L-6}'')^{-1})^{-1}$ represents the same element of $G_\Omega(\textbf{M}^\pazocal{L})$ as $\lab(\textbf{z}~\textbf{top}(\pazocal{Q}_{L-6}))$, with \Cref{lengths 1}(c) implying the bound $|\textbf{t}_{L-12,L-6}''(\textbf{q}_{L-12,L-6}'')^{-1}|\leq|\textbf{t}_{L-12,L-6}|-2+h_{L-12}$.  In particular, $v\equiv\lab(\textbf{bot}(\pazocal{Q}_1))^{-1}\lab(\textbf{t}_{L-12,L-6}''(\textbf{q}_{L-12,L-6}'')^{-1})^{-1}$ represents the same element of $G_\Omega(\textbf{M}^\pazocal{L})$ as $\lab(\textbf{t}_0)$ and satisfies $|v|\leq h_1+|\textbf{t}_{L-12,L-6}|-2+h_{L-12}$.  

But then as above, \Cref{clove theta-bands}, \Cref{lengths 1}, and a parameter choice for $L$ imply
$$|v|<|\textbf{t}_{1,L-12}|+|\textbf{t}_{L-12,L-6}|-1=|\textbf{t}|$$ providing a contradiction to \Cref{clove path}.

\end{proof}

\Cref{no distinguished clove} implies there exists $d\in\{1,\dots,6\}$ such that $j_1=8-d$, $j_{L-6}=L-d+1$, and
$$\lab(\textbf{e}_{L-6}\textbf{z}^{-1}\textbf{e}_1)\equiv W(L-d+1)\dots W(L)W(1)\dots W(7-d)t(8-d)$$
Let $\textbf{z}_2$ be the subpath of $\textbf{z}$ such that $\lab(\textbf{z}_2^{-1})$ is an admissible word with base $B_4^\pazocal{L}(1)$.  Further, let $\textbf{z}_1$ and $\textbf{z}_3$ be the (perhaps trivial) subpaths such that $\textbf{z}=\textbf{z}_1\textbf{z}_2\textbf{z}_3$.

Let $k=7-d$.  If $k\geq2$, then we may construct the circular diagram $\bar{\Psi}_{1,k}(j_1-1)$.  In this case, similar to the proof of \Cref{no distinguished clove}, let $\bar{\Phi}_{1,k}$ be the diagram obtained from $\bar{\Psi}_{1,k}(j_1-1)$ by removing the $t$-band $\pazocal{Q}_1(j_1-1,1)$ (which is identical to $\pazocal{Q}_1$).  Then there exists a factorization of the contour $\partial\bar{\Phi}_{1,k}=(\textbf{s}_{1,k}'')^{-1}\textbf{p}_{1,k}''\textbf{t}_{1,k}''(\textbf{q}_{1,k}'')^{-1}$ where:

\begin{itemize}

\item $\lab(\textbf{s}_{1,k}'')\equiv \lab(\textbf{z}_1^{-1})$

\item $|\textbf{t}_{1,k}''|=|\textbf{t}_{1,k}|-1$

\item $\lab(\textbf{q}_{1,k}'')\equiv\lab(\textbf{bot}(\pazocal{Q}_1))$

\item $\lab(\textbf{p}_{1,k}'')\equiv\lab(\textbf{bot}(\pazocal{Q}_{k-1}(2,2)))$

\end{itemize}



%
%
%
%
%

For completeness, if $k=1$, then define $\textbf{t}_{1,1}$ to be the single $t$-edge of $\textbf{t}$ corresponding to the end of $\pazocal{Q}_1$, $\textbf{t}_{1,1}''$ as the trivial path at $(\textbf{t}_0)_-$, and $\pazocal{Q}_0(2,2)=\pazocal{Q}_1$.

Then, in any case we have $|\textbf{t}_{1,k}''|=|\textbf{t}_{1,k}|-1$ and
\begin{equation}
\lab(\textbf{bot}(\pazocal{Q}_1)^{-1}\textbf{z}_1)=_{G_\Omega(\textbf{M}^\pazocal{L})}\lab((\textbf{t}_{1,k}'')^{-1})\lab(\textbf{bot}(\pazocal{Q}_{k-1}(2,2)))^{-1}
\end{equation}

On the other side, letting $\ell=L-6-d$, we may construct the circular diagram $\bar{\Psi}_{\ell,L-6}(L)$.  As above, let $\bar{\Phi}_{\ell,L-6}$ be the diagram obtained from $\bar{\Psi}_{\ell,L-6}(L)$ by removing the $t$-band $\pazocal{Q}(j_{L-6},2)$, which is identical to $\pazocal{Q}_{L-6}$.  Then, again as in the proof of \Cref{no distinguished clove}, there exists a contour factorization $\partial\bar{\Phi}_{\ell,L-6}=(\textbf{s}_{\ell,L-6}'')^{-1}\textbf{p}_{\ell,L-6}''\textbf{t}_{\ell,L-6}''(\textbf{q}_{\ell,L-6}'')^{-1}$ such that:

\begin{itemize}

\item $\lab(\textbf{s}_{\ell,L-6}'')\equiv \lab(\textbf{z}_3^{-1})$

\item $|\textbf{t}_{\ell,L-6}''|=|\textbf{t}_{\ell,L-6}|-1$

\item $\lab(\textbf{p}_{\ell,L-6}'')\equiv\lab(\textbf{top}(\pazocal{Q}_{L-6}))$

\item $\lab(\textbf{q}_{\ell,L-6}'')\equiv\lab(\textbf{top}(\pazocal{Q}_{\ell}(L,1)))$

\end{itemize}

%
%
%
%
%
%

Then, we have:
\begin{equation}
\lab(\textbf{z}_3\textbf{top}(\pazocal{Q}_{L-6}))=_{G_\Omega(\textbf{M}^\pazocal{L})}\lab(\textbf{top}(\pazocal{Q}_{\ell}(L,1)))\lab(\textbf{t}_{\ell,L-6}'')^{-1}
\end{equation}

\begin{lemma} \label{z_2 length}

For any word $v$ over $\pazocal{X}\cup\pazocal{X}^{-1}$ which represents the same element of $G_\Omega(\textbf{M}^\pazocal{L})$ as $\lab(\textbf{bot}(\pazocal{Q}_{k-1}(2,2)))^{-1}\lab(\textbf{z}_2)\lab(\textbf{top}(\pazocal{Q}_{\ell}(L,1)))$, $|v|\geq|\textbf{t}_{k,\ell}|$.

\end{lemma}

\begin{proof}

By (12.1) and (12.2), $w\equiv\Big(\lab(\textbf{t}_{1,k}'')^{-1}\Big)v\Big(\lab(\textbf{t}_{\ell,L-6}'')^{-1}\Big)$ represents the same element of $G_\Omega(\textbf{M}^\pazocal{L})$ as $\lab(\textbf{t}_0)$.  So, \Cref{clove path} implies $|w|\geq|\textbf{t}|$.

Note that by construction, the last letter of $\lab(\textbf{t}_{\ell,L-6}'')$ is a $t$-letter.  Similarly, $\lab(\textbf{t}_{1,k}'')$ is either trivial (if $k=1$) or its first letter is a $t$-letter.  Hence, in any case \Cref{lengths 1}(d) yields $|w|=|\textbf{t}_{1,k}''|+|v|+|\textbf{t}_{\ell,L-6}''|=|\textbf{t}_{1,k}|+|v|+|\textbf{t}_{\ell,L-6}|-2$.

But \Cref{lengths 1} also implies $|\textbf{t}|=|\textbf{t}_{1,k}|+|\textbf{t}_{k,\ell}|+|\textbf{t}_{\ell,L-6}|-2$, implying the statement.

\end{proof}

Now, let $\textbf{z}_{2,2}$ be the subpath of $\textbf{z}_2$ such that $\lab(\textbf{z}_{2,2}^{-1})$ is the admissible subword of $W$ with base $Q_0^\pazocal{L}(1)Q_1^\pazocal{L}(1)$.  The structure of the standard base of $\textbf{M}^\pazocal{L}$ implies $\textbf{z}_{2,2}$ is a terminal subpath of $\textbf{z}_2$, and so there exists a factorization $\textbf{z}_2=\textbf{z}_{2,1}\textbf{z}_{2,2}$.

Then, letting $\textbf{f}$ be the first edge of $\textbf{z}_{2,2}$, $\lab(\textbf{f}^{-1}\textbf{z}_{2,1}^{-1})$ is the admissible subword of $W(1)$ with base $Q_1^\pazocal{L}(1)\dots Q_N^\pazocal{L}(1)(R_N^\pazocal{L}(1))^{-1}\dots (R_0^\pazocal{L}(1))^{-1}$.  In particular, as $W$ is an accepted configuration, the parallel nature of the rules of $\textbf{M}^\pazocal{L}$ implies that $\lab(\textbf{f}^{-1}\textbf{z}_{2,1}^{-1})$ is a coordinate shift of the corresponding admissible subword of $W(j)$ for any $j$.

\begin{lemma} \label{z_22 length}

For any word $u$ over $\pazocal{X}\cup\pazocal{X}^{-1}$ which represents the same element of $G_\Omega(\textbf{M}^\pazocal{L})$ as $\lab(\textbf{z}_{2,2})\lab(\textbf{top}(\pazocal{Q}_{\ell}(L,1)))$, $|u|\geq h_\ell+|\textbf{t}_{k+1,\ell}|_q+3$.

\end{lemma}

\begin{proof}

Let $\textbf{e}_k'$ be the edge of $\textbf{s}_k$ satisfying $\lab(\textbf{e}_k')\in(R_1^\pazocal{L}(j_k))^{-1}$.  Then, let $\pazocal{Q}_k'$ be the maximal positive $q$-band of $\Sigma_k$ for which $\textbf{e}_k'$ is a defining edge.  

Cutting along $\textbf{bot}(\pazocal{Q}_k')$ separates $\Sigma_k$ into two subdiagrams, one of which, denoted $\Sigma_k'$, does not contain $\pazocal{Q}_k'$.  By construction, we may factor the contour $\partial\Sigma_k'=(\textbf{x}_k')^{-1}\textbf{p}_{k}\textbf{y}_k'(\textbf{q}_{k}')^{-1}$ such that $\textbf{p}_{k}=\textbf{top}(\pazocal{Q}_k)$ (note this subpath is as denoted in $\partial\Sigma_k$), $\textbf{q}_{k}'=\textbf{bot}(\pazocal{Q}_k')$, $\textbf{x}_k'$ is a subpath of $\textbf{x}_k$, and $\textbf{y}_k'$ is a subpath of $\textbf{y}_k$.  

Let $\textbf{x}_k''$ be the subpath of $\textbf{x}_k$ such that $\textbf{x}_k=\textbf{x}_k'\textbf{x}_k''$.  Since $k\leq6$, $\ell-k=L-6-2k\geq L-18$.  So, the parameter choice $L>33$ implies $\Psi_{k,\ell}$ satisfies the hypotheses of \Cref{clove theta-bands}.  As such, $|\textbf{t}_{k,\ell}|_\theta=h_k+h_\ell$ and every maximal $\theta$-band of $\Sigma_k'$ has an end on $\textbf{p}_{k}$.  So, letting $H_k'$ be the history of $\pazocal{Q}_k'$, we have $$h_k=|\textbf{p}_{k}|_\theta=|\textbf{y}_k'|_\theta+|\textbf{q}_{k}'|_\theta=|\textbf{y}_k'|_\theta+\|H_k'\|$$
with $|\textbf{q}_{k}'|=\|H_k'\|$ by \Cref{lengths 2}(a).

Let $\textbf{f}_k$ be the edge of $\textbf{t}$ which is an end of $\pazocal{Q}_k$.  Then, let $\textbf{t}_{k,\ell}'$ be the subpath of $\textbf{t}_{k,\ell}$ such that $\textbf{t}_{k,\ell}=\textbf{f}_k\textbf{y}_k'\textbf{t}_{k,\ell}'$.  By construction, $|\textbf{t}_{k,\ell}'|_q=|\textbf{t}_{k+1,\ell}|_q+2$ and the first edge of $\textbf{t}_{k,\ell}'$ is a $q$-edge.  So, \Cref{lengths 1}(d) implies $$|\textbf{t}_{k,\ell}|=|\textbf{f}_k|+|\textbf{y}_k'|+|\textbf{t}_{k,\ell}'|=|\textbf{y}_k'|+|\textbf{t}_{k,\ell}'|+1$$
Further, $|\textbf{t}_{k,\ell}|_\theta=|\textbf{y}_k'|_\theta+|\textbf{t}_{k,\ell}'|_\theta$, so that $|\textbf{t}_{k,\ell}'|_\theta=h_k+h_\ell-|\textbf{y}_k'|_\theta=\|H_k'\|+h_\ell$.  

Since $\Sigma_k'$ is a subdiagram of $\Sigma_k$, it is a circular diagram over $M(\textbf{M}^\pazocal{L})$ such that every cell has coordinate $j_k$ and no cell is a $(\theta,t)$-cell.  As such, we may construct the reflected copy $\bar{\Sigma}_k'$.  Let $\partial\bar{\Sigma}_k'=((\bar{\textbf{x}}_k')^{-1}\bar{\textbf{p}}_{k}\bar{\textbf{y}}_k'(\bar{\textbf{q}}_{k}')^{-1})^{-1}$ with the naming indicative of the correspondence to the subpaths of $\partial\Sigma_k'$.  Observe that since $W$ is an accepted configuration, the parallel nature of the rules of $\textbf{M}_4^\pazocal{L}$ implies $\lab(\bar{\textbf{x}}_k')^{-1}$ is a suffix of $W(j_k)$.

As $\Sigma_k'$ contains no $q$-bands corresponding to the parts $(R_1^\pazocal{L}(j_k))^{-1}$ or $(R_0^\pazocal{L}(j_k))^{-1}$ of the standard base, $\bar{\Sigma}_k'$ contains no $q$-bands corresponding to the parts $Q_0^\pazocal{L}(j_k)$ or $Q_1^\pazocal{L}(j_k)$.  In particular, $\bar{\Sigma}_k'$ is exceptional, so that we may construct $\bar{\Sigma}_k'(1)$.

As in the previous constructions, let $\partial\bar{\Sigma}_k'(1)=((\bar{\textbf{x}}_k'(1))^{-1}\bar{\textbf{p}}_{k}(1)\bar{\textbf{y}}_k'(1)(\bar{\textbf{q}}_{k}'(1))^{-1})^{-1}$ where:

\begin{itemize}

\item $\lab(\bar{\textbf{x}}_k'(1))\equiv\lab(\textbf{z}_{2,1})$

\item $\lab(\bar{\textbf{p}}_{k}(1))\equiv\lab(\textbf{bot}(\pazocal{Q}_{k-1}(2,2)))$

\item $|\bar{\textbf{y}}_k'(1)|=|\textbf{y}_k'|$

\item $|\bar{\textbf{q}}_{k}'(1)|=|\textbf{q}_{k}'|$

\end{itemize}

So, $\lab(\textbf{bot}(\pazocal{Q}_{k-1}(2,2)))^{-1}\lab(\textbf{z}_{2,1})=_{G_\Omega(\textbf{M}^\pazocal{L})}\lab(\bar{\textbf{y}}_k'(1)(\bar{\textbf{q}}_{k}'(1))^{-1})$.

In particular, $v\equiv\lab(\bar{\textbf{y}}_k'(1)(\bar{\textbf{q}}_{k}'(1))^{-1})u$ is a word over $\pazocal{X}\cup\pazocal{X}^{-1}$ which represents the same element of $G_\Omega(\textbf{M}^\pazocal{L})$ as $\lab(\textbf{bot}(\pazocal{Q}_{k-1}(2,2)))^{-1}\lab(\textbf{z}_2)\lab(\textbf{top}(\pazocal{Q}_\ell(L,1)))$, so that $|v|\geq|\textbf{t}_{k,\ell}|$ by \Cref{z_2 length}.  On the other hand, \Cref{lengths 1}(c) implies:
$$|v|\leq|\bar{\textbf{y}}_k'(1)|+|\bar{\textbf{q}}_{k}'(1)|+|u|=|\textbf{y}_k'|+|\textbf{q}_{k}'|+|u|=|\textbf{t}_{k,\ell}|-|\textbf{t}_{k,\ell}'|-1+|\textbf{q}_{k}'|+|u|$$
Hence, $|u|\geq|\textbf{t}_{k,\ell}'|-|\textbf{q}_{k}'|+1=|\textbf{t}_{k,\ell}'|-\|H_k'\|+1$.

But \Cref{lengths 1}(b) implies $|\textbf{t}_{k,\ell}'|\geq|\textbf{t}_{k,\ell}'|_\theta+|\textbf{t}_{k,\ell}'|_q=\|H_k'\|+h_\ell+|\textbf{t}_{k,\ell}'|_q$, so that $$|u|\geq h_\ell+|\textbf{t}_{k,\ell}'|_q+1=h_\ell+|\textbf{t}_{k+1,\ell}|_q+3$$

\end{proof}

\begin{lemma} \label{clove special input}

$W$ is accepted by a one-machine computation of the first machine.

\end{lemma}

\begin{proof}

Suppose to the contrary that $W$ is accepted by a one-machine computation of the second machine.  As every rule of the second machine locks the `special' input sector, the admissible subword of $W$ with base $Q_0^\pazocal{L}(1)Q_1^\pazocal{L}(1)$ has empty tape word.  In particular, $|\textbf{z}_{2,2}|=2$.

But then Lemmas \ref{lengths 1}(d) and \ref{lengths 2}(a) imply $w\equiv\lab(\textbf{z}_{2,2})\lab(\textbf{top}(\pazocal{Q}_\ell(L,1)))$ itself satisfies $|w|=|\textbf{z}_{2,2}|+|\textbf{top}(\pazocal{Q}_\ell(L,1))|=h_\ell+2<h_\ell+|\textbf{t}_{k+1,\ell}|+3$, contradicting \Cref{z_22 length}.

\end{proof}

\begin{lemma} \label{no exceptional}

The reflected copy $\bar{\Sigma}_{\ell-1}$ is not exceptional.

\end{lemma}

\begin{proof}

Assuming $\bar{\Sigma}_{\ell-1}$ is exceptional, $\partial\bar{\Sigma}_{\ell-1}(1)=((\bar{\textbf{x}}_{\ell-1}(1))^{-1}\bar{\textbf{p}}_{\ell-1}(1)\bar{\textbf{y}}_{\ell-1}(1)(\bar{\textbf{q}}_{\ell-1}(1))^{-1})^{-1}$ where:

\begin{itemize}

\item $\lab(\bar{\textbf{x}}_{\ell-1}(1)^{-1})$ is the coordinate shift of $\lab(\textbf{x}_{\ell-1})$ with base $B_4^\pazocal{L}(1)$

\item $\lab(\bar{\textbf{q}}_{\ell-1}(1))\equiv\lab(\textbf{top}(\pazocal{Q}_\ell(L,1)))$

\item $|\bar{\textbf{p}}_{\ell-1}(1)|=h_{\ell-1}$

\item $|\bar{\textbf{y}}_{\ell-1}(1)|=|\textbf{y}_{\ell-1}|$

\end{itemize}

It then follows from \Cref{clove special input} and the parallel nature of the rules of the first machine that $\lab(\bar{\textbf{x}}_{\ell-1}(1)^{-1})$ is the admissible subword of $W$ with base $B_4^\pazocal{L}(1)$, and so $\lab(\bar{\textbf{x}}_{\ell-1}(1))\equiv\lab(\textbf{z}_2)$.

As a result, $v\equiv\lab(\bar{\textbf{p}}_{\ell-1}(1)\bar{\textbf{y}}_{\ell-1}(1))$ is a word over $\pazocal{X}\cup\pazocal{X}^{-1}$ which represents the same element of $G_\Omega(\textbf{M}^\pazocal{L})$ as $\lab(\textbf{z}_2)\lab(\textbf{top}(\pazocal{Q}_\ell(L,1)))$.  Hence, $w\equiv\lab(\textbf{bot}(\pazocal{Q}_{k-1}(2,2)))^{-1}v$ satisfies the hypotheses of \Cref{z_2 length}, so that $|w|\geq|\textbf{t}_{k,\ell}|$.

As $d\leq 6$, we have $\ell-1-k\geq L-19$.  A parameter choice for $L$ and \Cref{clove theta-bands} thus imply $|\textbf{t}_{k,\ell-1}|_\theta=h_k+h_{\ell-1}$.  Further, \Cref{lengths 1}(d) implies $|\textbf{t}_{k,\ell}|=|\textbf{t}_{k,\ell-1}|+|\textbf{y}_{\ell-1}|+1$, so that \Cref{lengths 1}(b) implies $|\textbf{t}_{k,\ell}|\geq h_k+h_{\ell-1}+|\textbf{y}_{\ell-1}|+1$.

But Lemmas \ref{lengths 1} and \ref{lengths 2} imply
$$|w|\leq|\textbf{bot}(\pazocal{Q}_{k-1}(2,2))|+|v|\leq h_k+|\bar{\textbf{p}}_{\ell-1}(1)|+|\bar{\textbf{y}}_{\ell-1}(1)|=h_k+h_{\ell-1}+|\textbf{y}_{\ell-1}|<|\textbf{t}_{k,\ell}|$$
yielding a contradiction.

\end{proof}

Finally, we reach the desired contradiction:

\begin{lemma} \label{no distortion disks}

A reduced minimal $h$-distortion diagram contains no disks.

\end{lemma}

\begin{proof}

By \Cref{no exceptional}, the reflected copy $\bar{\Sigma}_{\ell-1}$ cannot be exceptional.  In particular, there must exist a maximal $\theta$-band in $\bar{\Sigma}_{\ell-1}$ whose history is a rule of the second machine.  

By construction, this implies the existence of a maximal $\theta$-band $\pazocal{T}$ in $\Sigma_{\ell-1}$ whose history is a rule of the second machine.

Recall that a parameter choice for $L$ implies $(L-1)/2<\ell\leq L-7$.  So, \Cref{clove theta-bands} implies:

\begin{enumerate}

\item Every maximal $\theta$-band of $\Gamma_{\ell-1}$ crosses $\pazocal{Q}_\ell$.

\item Every maximal $\theta$-band of $\Psi_{1,\ell+1}$ that crosses $\pazocal{Q}_\ell$ also crosses $\pazocal{Q}_{\ell+1}$.

\end{enumerate}

In light of condition (2), Lemmas \ref{non-distinguished a-cells} and \ref{no distinguished clove} imply that the maximal $\theta$-bands of $\Gamma_\ell$ that cross both $\pazocal{Q}_\ell$ and $\pazocal{Q}_{\ell+1}$ form a subdiagram $\Lambda_\ell$ which is a trapezium with history $H_\ell$.  Letting $\partial\Lambda_\ell=\textbf{p}_1^{-1}\textbf{q}_1\textbf{p}_2\textbf{q}_2^{-1}$ be the standard factorization of the contour of this trapezium, then by construction:

\begin{itemize}

\item $\textbf{q}_1=\textbf{s}_\ell^{-1}$

\item $\textbf{p}_1$ is a subpath of $\textbf{top}(\pazocal{Q}_{\ell+1})$

\item $\textbf{p}_2$ is a subpath of $\textbf{bot}(\pazocal{Q}_\ell)$

\end{itemize}

The existence of $\pazocal{T}$ implies $H_\ell$ is non-empty, and so by \Cref{trapezia are computations} there exists a reduced computation $\pazocal{D}:W(j_\ell)\equiv V_0\to\dots\to V_t$ with history $H_\ell$.  In particular, there exists a letter of $H_\ell$ corresponding to a rule of the second machine.

Now, let $H_\ell'$ be the maximal (perhaps empty) prefix of $H_\ell$ consisting entirely of rules of the first machine.  By hypothesis, $H_\ell'$ is a proper prefix of $H_\ell$.  

If $H_\ell'$ is non-empty, then the subcomputation $\pazocal{D}':V_0\to\dots\to V_s$ with history $H_\ell'$ is a one-machine computation of the first machine.  So, \Cref{extend one-machine} yields a one-machine computation $\pazocal{C}':W_0\to\dots\to W_s\equiv W'$ of the first machine in the standard base with history $H_\ell'$.
By \Cref{clove special input} and the construction of \Cref{extend one-machine}, $W_0\equiv W$.  As a result, $W'$ is an accepted configuration with $\ell(W')\leq1$ such that $W'\equiv W\cdot H_\ell'$.  

Note that if $H_\ell'$ is empty, an analogous condition is given by taking $W'\equiv W$, {\frenchspacing i.e. by setting $\pazocal{C}'$ to be the empty computation}.

However, since $H_\ell'$ is a proper prefix of $H_\ell$, $W'(j_\ell)$ must be $\theta$-admissible for some $\theta\in\Theta_2$.  As a result, either:

\begin{enumerate}[label=(\roman*)]

\item $W'$ is $\theta$-admissible, in which case it has empty `special' input sector, or

\item $W'$ is not $\theta$-admissible, in which case \Cref{transposition computation not applicable} implies $W'\equiv I(u)$ for some $u\in\pazocal{L}$.

\end{enumerate}

Note that in case (ii), condition (L5) implies that the tape word of $W'$ in the `special' input sector represents the identity in $G_\Omega(\textbf{M}^\pazocal{L})$, while this condition is obviously true in case (i).

Let $\pazocal{C}''$ be the restriction of $\pazocal{C}'$ to the `special' input sector and let $\Delta_\pazocal{C}$ be the trapezium corresponding to $\pazocal{C}''$ given by \Cref{computations are trapezia}.  Then, letting $\partial\Delta_\pazocal{C}=(\textbf{p}_1')^{-1}\textbf{q}_1'\textbf{p}_2'(\textbf{q}_2')^{-1}$ be the standard factorization of the contour of this trapezium, by construction:

\begin{enumerate}[label=(\alph*)]

\item $\lab(\textbf{q}_1')$ and $\lab(\textbf{q}_2')$ are the admissible subwords of $W$ and $W'$, respectively, with base $Q_0^\pazocal{L}(1)Q_1^\pazocal{L}(1)$

\item $\textbf{p}_1'=\textbf{bot}(\pazocal{Q}_{0,\pazocal{C}})$ where $\pazocal{Q}_{0,\pazocal{C}}$ is a positive $q$-band corresponding to the part $Q_0^\pazocal{L}(1)$ of the standard base with history $H_\ell'$

\item $\textbf{p}_2'=\textbf{top}(\pazocal{Q}_{1,\pazocal{C}})$ where $\pazocal{Q}_{1,\pazocal{C}}$ is a positive $q$-band corresponding to the part $Q_1^\pazocal{L}(1)$ of the standard base with history $H_\ell'$

\end{enumerate}

Again, for uniformity we may deal with the case where $H_\ell'$ is empty by taking $\Delta_\pazocal{C}$ to be a `trapezium' filled with $0$-cells.

By the definition of the path $\textbf{z}_{2,2}$, (a) implies that $\lab(\textbf{q}_1')\equiv\lab(\textbf{z}_{2,2}^{-1})$.  

Observe that $\pazocal{Q}_\ell(L,1)$ can be viewed as the concatenation of two subbands, $\pazocal{Q}_\ell(L,1)'$ and $\pazocal{Q}_\ell(L,1)''$, where the (perhaps empty) subband $\pazocal{Q}_\ell(L,1)'$ has history $H_\ell'$.  Then, since every rule of $\textbf{M}^\pazocal{L}$ locks the $\{t(1)\}Q_0^\pazocal{L}(1)$-sector, (b) implies $\lab(\textbf{p}_1')\equiv\lab(\textbf{top}(\pazocal{Q}_\ell(L,1)'))$.

Finally, let $W''$ be the admissible word with base $Q_0^\pazocal{L}(1)Q_1^\pazocal{L}(1)$ with empty tape word and whose state letters appear in $W'$.  Then as the tape word of $W'$ in the `special' input sector represents the identity in $G_\Omega(\textbf{M}^\pazocal{L})$, $W'$ represents the same element of $G_\Omega(\textbf{M}^\pazocal{L})$ as $\lab(\textbf{q}_2')$ and satisfies $|W'|=|W'|_q=2$.

Hence, $u\equiv\lab(\textbf{p}_2')(W')^{-1}\lab(\textbf{top}(\pazocal{Q}_\ell(L,1)''))$ is a word over $\pazocal{X}\cup\pazocal{X}^{-1}$ which represents the same element of $G_\Omega(\textbf{M}^\pazocal{L})$ as $\lab(\textbf{z}_{2,2})\lab(\textbf{top}(\pazocal{Q}_\ell(L,1)))$ and, by Lemmas \ref{lengths 1} and \ref{lengths 2}, satisfies:
$$|u|\leq|\textbf{p}_2'|+|W'|+|\textbf{top}(\pazocal{Q}_\ell(L,1)'')|=\|H_\ell'\|+2+(h_\ell-\|H_\ell'\|)=h_\ell+2$$
But this contradicts \Cref{z_22 length}.

\end{proof}

\medskip


\subsection{Equivalence of length functions} \

We now study the immediate consequences of \Cref{no distortion disks}, establishing the equivalence that assures the proof that the embeddings constructed for the proof of Theorem A are indeed quasi-isometric.  

Recall that the modified length function defined in this section sets the length of the `noise' $b$-letters to $0$, and so is not equivalent to the `standard' word norm.  However, the equivalence we seek is only for the elements of the subgroup $H_\pazocal{A}$.  The next two statements, based on our study of $h$-distortion diagrams, establishes this equivalence.

\begin{lemma} \label{lengths are equal}

$|h|=\delta|h|_\pazocal{A}$.

\end{lemma}

\begin{proof}

Let $\Delta$ be a reduced minimal $h$-distortion diagram and let $\partial\Delta=\textbf{qp}$ be the standard factorization of its contour.  

By \Cref{no distortion disks}, $\Delta$ contains no disk.  As a result, \Cref{distortion diagram q} implies $\Delta$ has no $q$-band.  So, the base of any $\theta$-band of $\Delta$ must have length 0.  But then the existence of a $\theta$-band implies the existence of a quasi-rim $\theta$-band, which would then contradict \Cref{distortion diagram quasi-rim theta}.

Hence, every positive cell of $\Delta$ must be an $a$-cell.

Now, fix an $a$-cell $\pi$ in $\Delta$ and suppose an edge $\textbf{e}$ of $\partial\pi$ is an edge of $\textbf{p}$.  As $\lab(\textbf{e})\in\pazocal{A}$, \Cref{M Lambda semi-computations} implies $\lab(\partial\pi)\in\Lambda^\pazocal{A}$.  Suppose an edge of $\partial\pi$ is on the boundary of an $a$-cell $\pi'$.  By \Cref{minimal is smooth}, $\pi$ and $\pi'$ are distinct $a$-cells.  Further, \Cref{M Lambda semi-computations} again implies $\lab(\partial\pi')\in\Lambda^\pazocal{A}$.  But then $\pi$ and $\pi'$ provide a contradiction to \Cref{cancellable a-cells}.

So, every edge of $\partial\pi$ is an edge of $\partial\Delta$.  In particular, there exists a factorization $\partial\pi=\textbf{xy}$ such that $\textbf{x}$ is a subpath of $\textbf{p}$ and $\textbf{y}$ is a subpath of $\textbf{q}$.  \Cref{distortion diagram a-cell} then implies that $|\textbf{x}|_\pazocal{A},|\textbf{y}|_\pazocal{A}\leq\frac{1}{2}|\partial\pi|_\pazocal{A}$, and so $|\textbf{x}|_\pazocal{A}=|\textbf{y}|_\pazocal{A}=\frac{1}{2}|\partial\pi|_\pazocal{A}$.

Hence, as \Cref{distortion diagram A} implies that any edge of $\textbf{p}$ which is not on the boundary of an $a$-cell is adjacent to an edge of $\textbf{q}^{-1}$, it follows that $|\textbf{q}|_\pazocal{A}\geq|\textbf{p}|_\pazocal{A}$.  As $|\textbf{q}|_\theta=0$, \Cref{lengths 1}(b) then implies $|h|=|\textbf{q}|\geq\delta|\textbf{p}|_\pazocal{A}=\delta\|\textbf{p}\|=\delta|h|_\pazocal{A}$.  But by definition $|h|\leq\delta|h|_\pazocal{A}$, so that the statement follows.

\end{proof}

\begin{lemma} \label{no distortion H}

$\delta|h|_\pazocal{A}\leq|h|_\pazocal{X}\leq|h|_\pazocal{A}$.

\end{lemma}

\begin{proof}

As $\pazocal{A}\subseteq\pazocal{X}$, it follows immediately that $|h|_\pazocal{X}\leq|h|_\pazocal{A}$.

Conversely, let $w$ be a word over $\pazocal{X}\cup\pazocal{X}^{-1}$ representing $h$ in $G_\Omega(\textbf{M}^\pazocal{L})$ satisfying $\|w\|=|h|_\pazocal{X}$.  Then, \Cref{lengths are equal} implies $|w|\geq|h|=\delta|h|_\pazocal{A}$.  Letting $w\equiv u_1\dots u_k$ be a decomposition of $w$ which realizes $|w|$, then the sum of the lengths of the factors $u_i$ is at least $\delta|h|_\pazocal{A}$.  

But each factor at least one letter and at most length 1, so that $\|w\|=\sum\|u_i\|\geq\delta|h|_\pazocal{A}$.

\end{proof}

%
%
%

\medskip


\section{Proof of Theorem A: General Case} \label{Theorem 1.1 proof}


Now we are ready to complete the proof of Theorem A.  However, this proof proceeds in two cases, outlined in the next two sections, relating to the decidability of the Word Problem for the group that is being embedded.

The case outlined in this section deals with the `general case', exhibiting an embedding that satisfies conditions (1)-(4) of Theorem A for an arbitrary finitely generated recursively presented group $R$.  While this construction may be carried out for any initial group, no control is achieved regarding the decidability of the Word Problem for the resulting finitely presented group, and so this construction does not settle condition (5) in general.

With that said, this is not an issue if the Word Problem for the initial group $R$ is undecidable, as containing $R$ as a subgroup immediately implies the Word Problem for $H$ must be undecidable.  Thus, this simple construction will handle the motivating case of an initial group with undecidable Word Problem, while the proof in the next section will assume from the offset that the Word Problem for $R$ is decidable.

\medskip

\subsection{Malnormality and distortion} \

Fix a finitely a finitely generated recursively presented group $R$.  Then, using a `standard trick' (see Lemma 12.17 and Exercise 12.12 of \cite{Rotman}), there exists a presentation $\gen{Y\mid\pazocal{S}}$ of $R$ such that $|Y|<\infty$ and $\pazocal{S}$ is a recursive set of positive words in $Y$.  As cofinite sets and intersections of recursive sets are recursive, it may be assumed without loss of generality that $\pazocal{S}$ does not contain the trivial word.  Hence, $\gen{Y\mid\pazocal{S}}$ satisfies conditions (R1)-(R3) (see \Cref{sec-initial-embedding}).

Thus, we may construct the group $R_C$ with presentation $\gen{Y_C\mid\pazocal{S}_C}$ as in \Cref{sec-initial-embedding}.  By Lemmas \ref{SQ quasi} and \ref{SQ malnormal}, there exists an undistorted malnormal embedding $\varphi:R\to R_C$ given by sending the generator $a_i\in Y$ to the product $a_{1,i}\dots a_{C,i}$ of generators in $Y_C$.

We now specify an assignment for the sets pertinent to the construction of our groups $G(\textbf{M}^\pazocal{L})$, verifying the relevant hypotheses along the way.  

Here, the alphabet $\pazocal{A}$ is taken to be in bijection with the generating set $Y_C$ of the group $R_C$, with $\zeta:Y_C\to\pazocal{A}$ a fixed bijection.  We extend $\zeta$ to a bijection $\tilde{\zeta}:(Y_C\cup Y_C^{-1})^*\to(\pazocal{A}\cup\pazocal{A}^{-1})^*$ in the natural way, that is, if $w\equiv x_1^{\eps_1}\dots x_k^{\eps_k}$ for $x_i\in Y_C$ and $\eps_i\in\{\pm1\}$, then $\tilde{\zeta}(w)\equiv\zeta(x_1)^{\eps_1}\dots\zeta(x_k)^{\eps_k}$.  

With this, the language $\pazocal{L}$ is taken to be the the corresponding copy of the set of relators $\pazocal{S}_C$, {\frenchspacing i.e. $\pazocal{L}=\tilde{\zeta}(\pazocal{S}_C)$}.  Note that since $\pazocal{S}$ is assumed to be a recursive subset of $Y^*$, then $\pazocal{L}$ is similarly a recursive subset of $\pazocal{A}^*$.

Finally, $\Lambda^\pazocal{A}$ is taken to be the set of all non-trivial cyclically reduced words over $\pazocal{A}\cup\pazocal{A}^{-1}$ whose copy over $Y_C\cup Y_C^{-1}$ is a word which represents the identity in the group $R_C$, i.e. $$\Lambda^\pazocal{A}=\{w\in(\pazocal{A}\cup\pazocal{A}^{-1})^*\setminus\{1\}: w\text{ is cyclically reduced, } \tilde{\zeta}^{-1}(w)=_{R_C}1\}$$
It must be noted that this choice satisfies condition (L1) by \Cref{SQ lengths}; conditions (L2)-(L5) are immediately satisfied by construction.

The following statements illustrate the purpose of these choices.

\begin{lemma} \label{a-relations}

For any $w\in\pazocal{L}$, the relation $w=1$ holds in the group $G(\textbf{M}^\pazocal{L})$.

\end{lemma}

\begin{proof}

Lemmas \ref{one-machine language} and \ref{disk relations} imply that the words corresponding to the configurations $I(w)$ and $J(w)$ are trivial over the group $G(\textbf{M}^\pazocal{L})$.  These two words differ only by the insertion of the word $w$ in the `special' input sector, so that $w=1$ in $G(\textbf{M}^\pazocal{L})$.

\end{proof}

Identifying $\pazocal{A}$ with the corresponding subset of the tape alphabet of the `special' input sector, $\zeta$ may be identified with a map $Y_C\to G(\textbf{M}^\pazocal{L})$.  Lemma \ref{a-relations} and the theorem of von Dyck \cite{vonDyck} then imply that this map extends to a homomorphism $\phi:R_C\to G(\textbf{M}^\pazocal{L})$.  

\begin{lemma} \label{isomorphism}

The identity map on $\pazocal{X}$ extends to an isomorphism $\mu:G_\Omega(\textbf{M}^\pazocal{L})\to G(\textbf{M}^\pazocal{L})$.

\end{lemma}

\begin{proof}

By the construction of the presentations for these two groups, it suffices to show that every element of $\Omega$ represents the identity in $G(\textbf{M}^\pazocal{L})$.  What's more, by the definition of $\Omega$, it suffices to show that every element of $\pazocal{E}(\Lambda^\pazocal{A})$ represents the identity in $G(\textbf{M}^\pazocal{L})$.

Note that $\Lambda^\pazocal{A}$ consists of images under $\tilde{\zeta}$ of the words which represent the trivial element of $R_C$.  But $\tilde{\zeta}$ induces the homomorphism $\phi$, and so the words of $\Lambda^\pazocal{A}$ represent the identity in $G(\textbf{M}^\pazocal{L})$.

Now, let $w\in\pazocal{E}(\Lambda^\pazocal{A})$.  Then, there exists a semi-computation $\pazocal{S}:w\equiv w_0\to\dots\to w_t$ of $\textbf{M}^\pazocal{L}$ in the `special' input sector which $\Lambda^\pazocal{A}$-accepts $w$.  \Cref{semi-computations are semi-trapezia} then provides a semi-trapezium $\Delta$ corresponding to $\pazocal{S}$, i.e so that $\lab(\textbf{bot}(\Delta))\equiv w$ and $\lab(\textbf{top}(\Delta))\equiv w_t$.  Hence, as the sides of any semi-trapezium are labelled by identical copies of the corresponding semi-computation, $w$ and $w_t$ are conjugate in $M(\textbf{M}^\pazocal{L})$, and so are conjugate in $G(\textbf{M}^\pazocal{L})$.

But $w_t\in\Lambda^\pazocal{A}$ and so represents the identity in $G(\textbf{M}^\pazocal{L})$.  Thus, $w=1$ in $G(\textbf{M}^\pazocal{L})$.

\end{proof}

\begin{lemma} \label{embedding}

The homomorphism $\phi:R_C\to G(\textbf{M}^\pazocal{L})$ is an embedding.

\end{lemma}

\begin{proof}


Let $g\in R_C$ such that $\phi(g)=1$ and let $w$ be a word over $Y_C^{\pm1}$ which represents $g$ in $R_C$.  Then, $\tilde{w}=\tilde{\zeta}(w)$ represents $1$ in $G(\textbf{M}^\pazocal{L})$, and so also represents $1$ in $G_\Omega(\textbf{M}^\pazocal{L})$ by \Cref{isomorphism}.  Hence, there exists a reduced minimal diagram $\Delta$ over $G_\Omega(\textbf{M}^\pazocal{L})$ such that $\lab(\partial\Delta)\equiv\tilde{w}$.

By construction, $\lab(\partial\Delta)\equiv\tilde{w}$ is a word over $\pazocal{A}^{\pm1}$.  So, letting $k$ be the number of $a$-cells in $\Delta$, \Cref{all A's is Lambda} implies there exists a factorization $\tilde{w}=_{F(\pazocal{A})}\tilde{w}_1\dots\tilde{w}_k$ such that each $\tilde{w}_i$ is freely conjugate to an element of $\Lambda^\pazocal{A}$.

Letting $w_i=\tilde{\zeta}^{-1}(\tilde{w}_i)$, it follows from the definition of $\Lambda^\pazocal{A}$ that $w_i=_{R_C}1$.  But this implies $w=_{F(Y_C)}w_1\dots w_k=_{R_C}1$, so that $g=1$.

\end{proof}

Now define $\psi=\phi\circ\varphi$.  By the arguments of Section 3 and \Cref{embedding}, $\psi:R\to G(\textbf{M}^\pazocal{L})$ is an embedding.  Thus, as the canonical presentation of $G(\textbf{M}^\pazocal{L})$ is finite, it suffices to show that $\psi$ satisfies conditions (2)-(4) of Theorem A.

The next statement establishes conditions (2) and (4), essentially appealing to the transitivity of malnormal and undistorted subgroups.

\begin{lemma} \label{main-malnormal-and-q-i}

The embedding $\psi:R\to G(\textbf{M}^\pazocal{L})$ is malnormal and the restriction of $|\cdot|_{G(\textbf{M}^\pazocal{L})}$ to $R$ (identified with the image of $\psi$) is equivalent to $|\cdot|_R$.

\end{lemma}

\begin{proof}

By construction, the image of $\phi$ is the subgroup generated by the elements of $\pazocal{A}$.  So, Lemmas \ref{isomorphism} and \ref{H_A malnormal} imply that $\phi$ is a malnormal embedding.  But $\varphi$ is also a malnormal embedding by \Cref{SQ malnormal}, so that $\psi$ is malnormal by the transitivity of malnormality.

Now fix $r\in R$.  By \Cref{SQ quasi}, $|\varphi(r)|_{Y_C}=C|r|_Y$.  Moreover, as $\psi(r)=\phi(\varphi(r))\in H_\pazocal{A}$, Lemmas \ref{isomorphism} and \ref{no distortion H} yield $\delta|\psi(r)|_\pazocal{A}\leq|\psi(r)|_\pazocal{X}\leq|\psi(r)|_\pazocal{A}$.

But $\phi$ is induced by $\zeta$, and so $|\psi(r)|_\pazocal{A}=|\varphi(r)|_{Y_C}$.  Thus, $\delta C|r|_Y\leq|\psi(r)|_\pazocal{X}\leq C|r|_Y$.

\end{proof}

\medskip

\subsection{Congruence extension property} \

Thus, by \Cref{main-malnormal-and-q-i} our goal for this section is complete if we can demonstrate that $\psi$ is a CEP-embedding.  However, this argument is slightly more intricate, and indeed requires a different assignment of the sets in the construction of $G(\textbf{M}^\pazocal{L})$.

Recall that $\varphi:R\to R_C$ is given by the map which sends each letter $a_i\in Y$ to the (positive) word $A_i\equiv a_{1,i}\dots a_{C,i}$ over $Y_C$.  As such, the set of words $\pazocal{D}=\{A_1,\dots,A_m\}$ forms a basis for a free subgroup $F$ of $F(Y_C)$ with $\varphi(R)\cong\gen{\pazocal{D}\mid\pazocal{S}_C}$.

Now, let $N$ be a normal subgroup of $R$.  Then, since $R\cong \varphi(R)\cong F/\gen{\gen{\pazocal{S}_C}}^F$, there exists a normal subgroup $M\triangleleft F$ containing $\gen{\gen{\pazocal{S}_C}}^{F}$ such that $\varphi(N)\cong M/\gen{\gen{\pazocal{S}_C}}^{F}$.

As in the constructions in \Cref{sec-initial-embedding}, let $T_{M}$ be the set of non-trivial cyclically reduced words over $\pazocal{D}\cup\pazocal{D}^{-1}$ which are elements of $M$.  Note that by construction, every element of $T_{M}$ is cyclically reduced as a word over $Y_C\cup Y_C^{-1}$.  Further, as in that setting, let $L_{M}=\gen{\gen{M}}^{F(Y_C)}$.

Finally, let $\Lambda^\pazocal{A}_N$ be the set of non-trivial cyclically reduced words $w$ over $\pazocal{A}\cup\pazocal{A}^{-1}$ which satisfy $\tilde{\zeta}^{-1}(w)\in L_{M}$.

By \Cref{SQ lengths 0}, every word $w\in\Lambda^\pazocal{A}_N$ satisfies $|w|_\pazocal{A}\geq C$.  As such, $\Lambda^\pazocal{A}_N$ satisfies condition (L1).  What's more, since $L_{M}\triangleleft F(Y_C)$, it follows immediately that $\Lambda^\pazocal{A}_N$ satisfies conditions (L2)-(L4).  Lastly, setting $\pazocal{L}=\tilde{\zeta}(\pazocal{S}_C)$ as in the previous section, it follows from $\pazocal{S}_C\subseteq M$ that $\Lambda^\pazocal{A}_N$ satisfies condition (L5).

Hence, letting $\Omega_N$ be the set of cyclically reduced words over $(\pazocal{A}\cup\pazocal{A}_1\cup\pazocal{B})^{\pm1}$ which are freely conjugate to an element of $\pazocal{E}(\Lambda^\pazocal{A}_N)$, the group $G_{\Omega_N}(\textbf{M}^\pazocal{L})\cong G(\textbf{M}^\pazocal{L})/\gen{\gen{\Omega_N}}^{G(\textbf{M}^\pazocal{L})}$ satisfies the hypotheses necessary for the treatment of Sections 6-12.

Let $g\in\gen{\gen{\Omega_N}}^{G(\textbf{M}^\pazocal{L})}\cap\psi(R)$.  Then, letting $r=\psi^{-1}(g)\in R$, there exists a word $V\in F$ which represents $\varphi(r)$.  So, $W\equiv\tilde{\zeta}(V)$ is a word over $\pazocal{A}\cup\pazocal{A}^{-1}$ which represents $g$.

As $W$ represents an element of the normal subgroup $\gen{\gen{\Omega_N}}^{G(\textbf{M}^\pazocal{L})}$ of $G(\textbf{M}^\pazocal{L})$, there exists a reduced minimal diagram $\Delta$ over the disk presentation of $G_{\Omega_N}(\textbf{M}^\pazocal{L})$ with $\lab(\partial\Delta)\equiv W$.  \Cref{all A's is Lambda} then yields a factorization $W=_{F(\pazocal{A})} w_1\dots w_k$ where each $w_i$ is a word over $\pazocal{A}\cup\pazocal{A}^{-1}$ that is freely conjugate to an element of $\Lambda^\pazocal{A}_N$.  Hence, as $L_{M}\triangleleft F(Y_C)$, it follows that $V\in L_{M}$.

This implies $V\in L_{M}\cap F$, so that \Cref{SQ} implies $V\in M$.  But then the definition of $\varphi$ implies $r\in N$, so that $g\in\psi(N)$.  

Thus, $\gen{\gen{\Omega_N}}^{G(\textbf{M}^\pazocal{L})}$ is a normal subgroup of $G(\textbf{M}^\pazocal{L})$ which satisfies $\gen{\gen{\Omega_N}}^{G(\textbf{M}^\pazocal{L})}\cap\psi(R)=\psi(N)$, implying $\psi(R)\leq_{CEP}G(\textbf{M}^\pazocal{L})$.

\bigskip


\section{Proof of Theorem A: Preserving the Word Problem}


The goal of this section is to complete the proof of Theorem A by exhibiting condition (5).  As discussed in the previous section, the arguments therein suffice for the case when the Word Problem for the initial group is undecidable; what is left is to show that if the Word Problem is decidable, then the finitely presented group can be constructed to also have decidable Word Problem.

However, the arguments of the previous section will not suffice for this purpose.  Instead, we begin with an entirely new setup for the construction of the finitely presented group.

\subsection{Construction} \

Let $R$ be a finitely generated group with decidable Word problem.  Letting $X$ be a finite generating set for $R$, define $\pazocal{R}$ to be the set of all non-trivial (perhaps unreduced) words over $X\cup X^{-1}$ which represent the identity in $R$.  As the set of non-trivial words over $X\cup X^{-1}$ is a cofinite subset $(X\cup X^{-1})^*$, $\pazocal{R}$ is a recursive subset of $(X\cup X^{-1})^*$.  Note that $\gen{X\mid\pazocal{R}}$ is a presentation of $R$.

We now employ the `standard trick' referenced in \Cref{Theorem 1.1 proof}:

Let $Y=X\sqcup\bar{X}$, where $\bar{X}$ is a copy of $X$ with defining bijection $\tau:\bar{X}\to X$.  Then, define the bijection $\xi:Y\to X\cup X^{-1}$ by $\xi(x)=x$ for all $x\in X$ and $\xi(\bar{x})=\tau(\bar{x})^{-1}$ for all $\bar{x}\in\bar{X}$.  

The map $\xi$ then extends to a map $\tilde{\xi}:(Y\cup Y^{-1})^*\to(X\cup X^{-1})^*$ which restricts to a bijection $\tilde{\xi}_0:Y^*\to(X\cup X^{-1})^*$.  With this, define $\pazocal{S}=\tilde{\xi}_0^{-1}(\pazocal{R})$.  By construction, $\pazocal{S}$ is a set of (positive) words over $Y$ which does not contain the trivial word.  Moreover, as $\pazocal{R}$ is a recursive subset of $(X\cup X^{-1})^*$, $\pazocal{S}$ is a recursive subset of $Y^*$.

For convenience, we also introduce the set $\pazocal{S}_2=\{x\cdot\tau^{-1}(x)\mid x\in X\}$.  Note that for any $x\in X$, $\tilde{\xi}_0(x\cdot\tau^{-1}(x))=x\cdot x^{-1}\in\pazocal{R}$, and so $\pazocal{S}_2\subseteq\pazocal{S}$.

\begin{lemma} \label{caps in R}

Let $w$ be a non-trivial word over $Y\cup Y^{-1}$ such that $\tilde{\xi}(w)=_R1$.  Then there exists a circular diagram $\Psi_w$ over $\gen{Y\mid\pazocal{S}}$ such that:

\begin{enumerate}

\item $\lab(\partial\Psi_w)\equiv w$

\item $\text{Area}(\Psi_w)\leq\|w\|$

\item For every positive cell $\pi$ of $\Psi_w$, $\|\partial\pi\|\leq2\|w\|$

\end{enumerate}

\end{lemma}

\begin{proof}

Let $w\equiv y_1^{\eps_1}\dots y_k^{\eps_k}$ where $y_1,\dots,y_k\in Y$ and $\eps_1,\dots,\eps_k\in\{\pm1\}$.

Let $I=\{i\in\{1,\dots,k\}\mid \eps_i=-1\}$.  Perhaps passing to $w^{-1}$, we may assume that $|I|\leq\frac{1}{2}\|w\|$.  Note that if $\|w\|=1$, then $|I|=0$ so that $\|w\|-|I|=\|w\|=1$; otherwise, $\|w\|-|I|\geq\frac{1}{2}\|w\|\geq1$.

If $y_i\in X$ for $i\in I$, then $y_i\cdot\tau^{-1}(y_i)\in\pazocal{S}_2$.  In this case, we may construct a cell $\pi_i$ satisfying $\lab(\partial\pi_i)\equiv(y_i\cdot\tau^{-1}(y_i))^{-1}$.  Similarly, if $y_i\in\bar{X}$ for $i\in I$, then $\tau(y_i)\cdot y_i\in\pazocal{S}_2$, so that we may construct a cell $\pi_i$ satisfying $\lab(\partial\pi_i)\equiv(\tau(y_i)\cdot y_i)^{-1}$.

Then, there exists an annular diagram $\Psi_w'$ over $\gen{Y\mid\pazocal{S}}$ consisting of the $|I|$ cells $\pi_i$ with outer contour label $w$ and inner contour label $v^{-1}$, where $v\in Y^*$ and satisfies $\tilde{\xi}(v)\equiv\tilde{\xi}(w)$.  

In particular, $\tilde{\xi}_0(v)=_R1$, so that $V\in\pazocal{S}$.  Hence, we may paste a single cell in the middle of the annulus $\Psi_w'$ to produce a circular diagram $\Psi_w$ over $\gen{Y\mid\pazocal{S}}$ which satisfies the statement.

\end{proof}

Identifying $\xi$ with a map $Y\to\gen{X\mid\pazocal{R}}$, \Cref{caps in R} and the theorem of von Dyck imply that $\xi$ extends to a homomorphism $\gen{Y\mid\pazocal{S}}\to\gen{X\mid\pazocal{R}}$.  

Similarly, identifying the natural injection $X\to Y$ with a map $X\to\gen{Y\mid\pazocal{S}}$, this map extends to a homomorphism $\gen{X\mid\pazocal{R}}\to\gen{Y\mid\pazocal{S}}$.  Indeed, since $\xi$ restricts to the identity on $X$, these homomorphisms are inverses.  

Hence, $\gen{Y\mid\pazocal{S}}$ is a presentation of $R$ which satisfies conditions (R1)-(R3).  As such, we may define the group $R_C$ with presentation $\gen{Y_C\mid\pazocal{S}_C}$ as constructed in \Cref{sec-initial-embedding}.  The terminology of \Cref{sec-initial-embedding} is adopted for this setting.  In particular, the set $\pazocal{D}$ forms a basis for a free subgroup $F$ of $F(Y_C)$.

\begin{lemma} \label{caps in R_C}

Let $w$ be a non-trivial word over $Y_C\cup Y_C^{-1}$ which is a cyclic permutation of an element of $\gen{\gen{\pazocal{S}_C}}^F$.  Then there exists a circular diagram $\Psi_w^C$ over $\gen{Y_C\mid\pazocal{S}_C}$ such that:

\begin{enumerate}

\item $\lab(\partial\Psi_w^C)\equiv w$

\item $\text{Area}(\Psi_w^C)\leq\frac{1}{C}\|w\|$

\item For every positive cell $\pi$ of $\Psi_w^C$, $\|\partial\pi\|\leq2\|w\|$

\end{enumerate}

\end{lemma}

\begin{proof}

Since the contour label can be read as a cyclic word, we may assume without loss of generality that $w\in\gen{\gen{\pazocal{S}_C}}^F$.  Hence, $w$ is a word over $\pazocal{D}\cup\pazocal{D}^{-1}$, and so corresponds in the natural way to a non-trivial word $u$ over $Y\cup Y^{-1}$ with $\|u\|=\frac{1}{C}\|w\|$.

As $w\in\gen{\gen{\pazocal{S}_C}}^F$, it follows that $u\in\gen{\gen{\pazocal{S}}}^{F(Y)}$, so that $\tilde{\xi}(u)=_R1$.  So, \Cref{caps in R} produces a circular diagram $\Psi_u$ over $\gen{Y\mid\pazocal{S}}$ such that $\lab(\partial\Psi_u)\equiv u$, $\text{Area}(\Psi_u)\leq\|u\|$, and every positive cell $\pi$ of $\Psi_u$ satisfies $\|\partial\pi\|\leq2\|u\|$.

But then subdividing each edge of $\Psi_u$ into an $F$-subpath of length $C$ labelled by the corresponding element of $\pazocal{D}$ produces a circular diagram $\Psi_w^C$ over $\gen{Y_C\mid\pazocal{S}_C}$ satisfying the statement.

\end{proof}

\begin{lemma} \label{caps in R_C 1}

Let $w$ be a word over $Y_C\cup Y_C^{-1}$ which represents the identity in $R_C$.  Then there exists a circular diagram $\Phi_w$ over $\gen{Y_C\mid\pazocal{S}_C}$ such that:

\begin{enumerate}

\item $\lab(\partial\Phi_w)\equiv w$

\item $\text{Area}(\Phi_w)\leq\frac{1}{C}\|w\|$

\item For every positive cell $\pi$ of $\Phi_w$, $\|\partial\pi\|\leq2\|w\|$

\end{enumerate}

\end{lemma}

\begin{proof}

The proof follows induction on $\|w\|$.  For the base case $\|w\|=0$, there exists a circular diagram $\Phi_w$ with $\lab(\Phi_w)\equiv w$ consisting entirely of $0$-cells, and so satisfies the statement.

Now assume $\|w\|\geq1$.


It is easy to see that we may assume $w$ is cyclically reduced, as otherwise we may freely conjugate to a shorter word and apply the inductive hypothesis.  \Cref{SQ lengths} thus implies the existence of a factorization $w'\equiv uv$ of a cyclic permutation $w'$ of $w$ such that $u$ is a non-trivial cyclic permutation of an element of $\gen{\gen{\pazocal{S}_C}}^F$.  As such, $u$ must represent the identity in $R_C$, and so $v$ must as well.

\Cref{caps in R_C} then implies the existence of a circular diagram $\Psi_u^C$ over $\gen{Y_C\mid\pazocal{S}_C}$ such that $\lab(\partial\Psi_u^C)\equiv u$, $\text{Area}(\Psi_u^C)\leq\frac{1}{C}\|u\|$, and $\|\partial\pi\|\leq2\|u\|\leq2\|w\|$ for every positive cell $\pi$ in $\Psi_u^C$.  
Moreover, the inductive hypothesis produces a circular diagram $\Phi_v$ over $\gen{Y_C\mid\pazocal{S}_C}$ such that $\lab(\partial\Phi_v)\equiv v$, $\text{Area}(\Phi_v)\leq\frac{1}{C}\|v\|$, and $\|\partial\pi\|\leq2\|v\|\leq2\|w\|$ for every positive cell $\pi$ in $\Phi_v$.  

Hence, pasting together $\Psi_u^C$ and $\Phi_v$ (and using $0$-refinement) yields a circular diagram $\Phi_w$ satisfying the statement.



\end{proof}

Similar to the construction of \Cref{Theorem 1.1 proof}, the alphabet $\pazocal{A}$ is taken to be in bijection with the generating set $Y_C$, with $\zeta:Y_C\to\pazocal{A}$ a fixed bijection.  Then, extending $\zeta$ in the natural way to a bijection $\tilde{\zeta}:(Y_C\cup Y_C^{-1})^*\to(\pazocal{A}\cup\pazocal{A}^{-1})^*$, the language $\pazocal{L}$ is taken to be $\tilde{\zeta}(\pazocal{S}_C)$.  Again, since $\pazocal{S}$ is a recursive subset of $Y^*$, $\pazocal{L}$ is similarly a recursive subset of $\pazocal{A}^*$.

Let $\Lambda^\pazocal{A}$ be the set of all non-trivial cyclically reduced words over $\pazocal{A}\cup\pazocal{A}^{-1}$ whose copy over $Y_C\cup Y_C^{-1}$ is a word which represents the identity in $R_C$, i.e
$$\Lambda^\pazocal{A}=\{w\in(\pazocal{A}\cup\pazocal{A}^{-1})^*\setminus\{1\}: w\text{ is cyclically reduced, } \tilde{\zeta}^{-1}(w)=_{R_C}1\}$$
Then, as in \Cref{Theorem 1.1 proof}, \Cref{SQ lengths} implies $\Lambda^\pazocal{A}$ satisfies conditions (L1)-(L5).

Hence, exact analogues of Lemmas \ref{a-relations}-\ref{embedding} imply that $\zeta$ induces an embedding $\phi:R_C\to G(\textbf{M}^\pazocal{L})$, while the rest of the arguments may be repeated to show that $\psi=\phi\circ\varphi:R\to G(\textbf{M}^\pazocal{L})$ is a malnormal CEP-embedding such that the restriction of $|\cdot|_{G(\textbf{M}^\pazocal{L})}$ to $R$ is equivalent to $|\cdot|_R$.

Thus, it suffices to show that in this setting, the Word Problem for the finitely presented group $G(\textbf{M}^\pazocal{L})$ is decidable.  We do this through the concept of the \textit{Dehn function}, specifically appealing to the analysis of \Cref{sec-upper-bound}.

\subsection{Dehn functions} \

Letting $\pazocal{P}$ be the canonical (finite) presentation of $G(\textbf{M}^\pazocal{L})$, recall the following definitions:

\begin{itemize}

\item Given a word $W$ over $\pazocal{X}\cup\pazocal{X}^{-1}$ which represents the trivial element of $G(\textbf{M}^\pazocal{L})$, the \textit{area} of $W$ with respect to $\pazocal{P}$, denoted $\text{Area}_\pazocal{P}(W)$, is the minimal area of a circular diagram $\Delta$ over $\pazocal{P}$ which satisfies $\lab(\partial\Delta)\equiv W$.

\medskip

\item The \textit{Dehn function} of $\pazocal{P}$ is the function $\delta_\pazocal{P}:\N\to\N$ given by $$\delta_\pazocal{P}(n)=\max\{\text{Area}_\pazocal{P}(W):\|W\|\leq n\}$$

\end{itemize}

The Dehn function of a finite presentation was first introduced by Madlener and Otto in \cite{MadlenerOtto} as a useful invariant for studying the group.  Indeed, the Dehn function of two finite presentations of quasi-isometric groups are equivalent with respect to the asymptotic equivalence on functions $\N\to\N$ induced by the preorder $\preccurlyeq$ given by $f\preccurlyeq g$ if and only if there exists $C>0$ such that $f(n)\leq Cg(Cn)+Cn+C$ for all $n\in\N$.  As such, with respect to this equivalence, the Dehn function of a finitely presented group is invariant of the choice of finite presentation.

Among its numerous uses, the Dehn function encodes the decidability of the group's Word problem: A finitely presented group has decidable Word problem if and only if the Dehn function with respect to one of (equivalently, any of) its finite presentations of it is bounded above by (and so equivalent to) a computable function (see Theorem 2.1 of \cite{Gersten}).  

Thus, to show that $G(\textbf{M}^\pazocal{L})$ has decidable Word Problem, it suffices to find a computable function $f:\N\to\N$ such that $\delta_\pazocal{P}\leq f$.

For this, we begin by justifying the assignments of weights in \Cref{sec-weights}.

\begin{lemma} \label{disk weights}

For any disk relator $W$ for $G_\Omega(\textbf{M}^\pazocal{L})$, there exists a circular diagram $\Gamma_W$ over $\pazocal{P}$ such that $\lab(\partial\Gamma_W)\equiv W$ and $\text{Area}(\Gamma_W)\leq f_\pazocal{L}(\|W(2)\|)$.

\end{lemma}

\begin{proof}

Let $n=\|W(2)\|$.

If $W=W_{ac}$, then a single hub produces a diagram $\Gamma$ satisfying $\lab(\partial\Gamma)\equiv W$ and $\text{Area}(\Gamma)=1$.  As $n=2N+1$ in this case, $f_\pazocal{L}(n)\geq1$, so that the statement is satisfied for $\Gamma_W=\Gamma$.

Otherwise, \Cref{M time-space} yields a non-empty reduced computation $\pazocal{C}:W\equiv W_0\to\dots\to W_t\equiv W_{ac}$ of $\textbf{M}^\pazocal{L}$ accepting $W$ and satisfying $t\leq c_0\TM_\pazocal{L}(c_0n)^3+nc_0^n+c_0n+2c_0$.  So, the parameter choice $L>>c_0$ implies $t\leq h_\pazocal{L}(n)$.

Then, as in the proof of \Cref{disk relations}, \Cref{computations are trapezia} produces a trapezium $\Gamma_W'$ over the canonical presentation of $M(\textbf{M}^\pazocal{L})$ such that:

\begin{itemize}

\item $\lab(\textbf{bot}(\Gamma_W'))\equiv W$ 

\item $\lab(\textbf{top}(\Gamma_W'))\equiv W_{ac}$

\item the sides of $\Gamma_W'$ are labelled by identical copies of the history of $\pazocal{C}$.

\item $\text{Area}(\Gamma_W')\leq t\max(\|W_0\|,\dots,\|W_t\|)$

\end{itemize}

Now, identical to the construction in the proof of \Cref{disk relations}, gluing the sides of $\Gamma_W'$ together and pasting a single hub in the `center' produces a circular diagram $\Gamma_W$ over $\pazocal{P}$ with $\lab(\partial\Gamma_W)\equiv W$ and $\text{Area}(\Gamma_W)=\text{Area}(\Gamma_W')+1$.

Cref{M main difference} implies $|W_i|_a\leq 4c_0^tLN$ for all $i$, so that a parameter choice for $c_0$ yields:
$$\|W_i\|\leq 4c_0^tLN+|W_i|_q\leq4c_0^tLN+(2N+1)L\leq 7c_0^tLN$$ 
Hence, the parameter choices $c_1>>L>>N$ imply:
\begin{align*}
\text{Area}(\Gamma_W)&\leq t\max(\|W_0\|,\dots,\|W_t\|)+1\leq 7tc_0^t LN+1\leq8LN\chi(t)\leq f_\pazocal{L}(n)
\end{align*}

\end{proof}

\begin{lemma} \label{L weights}

For any $w\in\pazocal{L}$, there exists a circular diagram $\Gamma_w$ over $\pazocal{P}$ such that $\lab(\partial\Sigma_w)\equiv w$ and $\text{Area}(\Sigma_w)\leq 2f_\pazocal{L}(3\|w\|)$.

\end{lemma}

\begin{proof}

Lemmas \ref{one-machine language} and \ref{disk weights} produce two circular diagrams $\Gamma_1$ and $\Gamma_2$ over $\pazocal{P}$ such that:

\begin{itemize}

\item $\lab(\partial\Gamma_1)\equiv I(w)$ and $\text{Area}(\Gamma_1)\leq f_\pazocal{L}(\|I(w,2)\|)$

\item $\lab(\partial\Gamma_2)\equiv J(w)$ and $\text{Area}(\Gamma_2)\leq f_\pazocal{L}(\|J(w,2)\|)$

\end{itemize}

Note that $\|I(w,2)\|=\|J(w,2)\|=2\|w\|+(2N+1)$.  As $w\in\pazocal{L}$ implies $\|w\|\geq C$, we may thus assume $\|I(w,2)\|=\|J(w,2)\|\leq3\|w\|$ by the parameter choice $C>>N$.

Hence, as $I(w)$ and $J(w)$ differ only by the word $w$ (identified with its copy in the tape alphabet of the `special' input sector), gluing $\Gamma_2$ to $\Gamma_1$ along its contour produces a circular diagram $\Sigma_w$ with $\lab(\partial\Sigma_w)\equiv w$ and 
$$\text{Area}(\Sigma_w)=\text{Area}(\Gamma_1)+\text{Area}(\Gamma_2)\leq2f_\pazocal{L}(3\|w\|)$$

\end{proof}

\begin{lemma} \label{Lambda weights}

For any $w\in\Lambda^\pazocal{A}$, there exists a circular diagram $\Gamma_w$ over $\pazocal{P}$ such that $\lab(\partial\Gamma_w)\equiv w$ and $\text{Area}(\Gamma_w)\leq\|w\|f_\pazocal{L}(c_0\|w\|)$.

\end{lemma}

\begin{proof}

By the definition of $\Lambda^\pazocal{A}$, $v\equiv\tilde{\zeta}^{-1}(w)$ is a word over $Y_C\cup Y_C^{-1}$ which represents the identity in $R_C$.  So, \Cref{caps in R_C 1} produces a circular diagram $\Phi_v$ over $\gen{Y_C\mid\pazocal{S}_C}$ such that $\lab(\partial\Phi_v)\equiv v$, $\text{Area}(\Phi_v)\leq\frac{1}{C}\|v\|=\frac{1}{C}\|w\|$, and $\|\partial\pi\|\leq2\|v\|=2\|w\|$ for every positive cell $\pi$.

Let $\pi$ be a positive cell of $\Phi_v$.  Then, $u\equiv\lab(\partial\pi)\in\pazocal{S}_C^{\pm1}$, so that $\tilde{\zeta}(u)\in\pazocal{L}^{\pm1}$.  As a result, \Cref{L weights} produces a circular diagram $\Sigma_\pi$ over the presentation $\pazocal{P}$ such that $\lab(\partial\Sigma_\pi)\equiv\tilde{\zeta}(u)$ and $\text{Area}(\Sigma_\pi)\leq2f_\pazocal{L}(3\|\partial\pi\|)$.  So, noting that $f_\pazocal{L}$ is non-decreasing, then $\text{Area}(\Sigma_\pi)\leq2f_\pazocal{L}(6\|w\|)$.

Now, consider the diagram $\Gamma_w$ obtained from $\Phi_v$ by applying $\tilde{\zeta}$ to the label of each edge and replacing any positive cell $\pi$ with the circular diagram $\Sigma_\pi$.  Then, $\Gamma_w$ is a circular diagram over $\pazocal{P}$ with $\lab(\partial\Gamma)_w\equiv w$ and
$$\text{Area}(\Gamma_w)=\sum_\pi\text{Area}(\Sigma_\pi)\leq\sum_\pi2f_\pazocal{L}(6\|w\|)\leq\frac{2}{C}\|w\|f_\pazocal{L}(6\|w\|)$$
Thus, the statement follows from the parameter choices $C\geq2$ and $c_0\geq6$.

\end{proof}

\begin{lemma} \label{Omega weights}

For any $w\in\Omega$, there exists a circular diagram $\Gamma_w$ over $\pazocal{P}$ such that $\lab(\partial\Gamma_w)\equiv w$ and $\text{Area}(\Gamma_w)\leq g_\pazocal{L}(\|w\|)$.

\end{lemma}

\begin{proof}

Per the definition of $\Omega$, $w$ is a cyclically reduced word over $(\pazocal{A}\cup\pazocal{A}_1\cup\pazocal{B})^{\pm1}$ which is freely conjugate to a word $w'\in\pazocal{E}(\Lambda^\pazocal{A})$.  By \Cref{M Lambda semi-computations}, there then exists a unique semi-computation $\pazocal{S}(w'):w'\equiv w_0\to\dots\to w_t$ of $\textbf{M}^\pazocal{L}$ in the `special' input sector which $\Lambda^\pazocal{A}$-accepts $w'$.

Suppose $w'\in\Lambda^\pazocal{A}$.  Then, as $\Lambda^\pazocal{A}$ consists of cyclically reduced words, $w$ is a cyclic permutation of $w'$.  So, the statement follows from \Cref{Lambda weights} and the definition of $g_\pazocal{L}$.  

Hence, by \Cref{M Lambda semi-computations}, it suffices to assume that $w'\in\pazocal{E}_1(\Lambda^\pazocal{A}_1)$.  In particular, this implies $\pazocal{S}(w')$ is a non-empty semi-computation.  

\Cref{semi-computations are semi-trapezia} then provides a semi-trapezium $\Delta_w'$ over $M(\textbf{M}^\pazocal{L})$ in the `special' input sector such that $\lab(\textbf{bot}(\Delta_w'))\equiv w'$, $\lab(\textbf{top}(\Delta_w'))\equiv w_t$, and $\text{Area}(\Delta_w')\leq\sum\limits_{i=0}^{t-1}\|w_i\|$.

As the sides of any semi-trapezium are labelled by identical copies of the history of the corresponding semi-computation, we may then paste the sides of $\Delta_w'$ together to form an annular diagram $\Delta_w$ over the canonical presentation of $M(\textbf{M}^\pazocal{L})$ with outer contour label $w'$, inner contour label $w_t^{-1}$, and $\text{Area}(\Delta_w)=\text{Area}(\Delta_w')$.  

Since $w_t\in\Lambda^\pazocal{A}$, we may then paste the diagram $\Gamma_{w_t}$ arising from \Cref{Lambda weights} into the center of $\Delta_w$, producing a circular diagram $\Gamma_w$ over $\pazocal{P}$ with $\lab(\partial\Gamma_w)\equiv w$ and 
$$\text{Area}(\Gamma_w)=\text{Area}(\Gamma_{w_t})+\text{Area}(\Delta_w)\leq \|w_t\|f_\pazocal{L}(c_0\|w_t\|)+\sum\limits_{i=0}^{t-1}\|w_i\|$$
Let $k=\|w_t\|$.  By \Cref{semi-computation deltas}, there exist $x_1,\dots,x_k\in\pazocal{A}_1$, $\delta_1,\dots,\delta_k$, and $u_{0,i},u_{1,i},\dots,u_{k,i}\in F(\pazocal{B})$ such that $w_i\equiv u_{0,i}x_1^{\delta_1}u_{1,i}x_2^{\delta_2}\dots u_{k-1,0}x_k^{\delta_k}u_{k,i}$ for all $0\leq i\leq t-1$.  

As $w_t$ is cyclically reduced, $x_1^{\delta_1}u_{1,0}x_2^{\delta_2}\dots u_{k-1,0}x_k^{\delta_k}$ is a subword of $w$, i.e $|w|_\pazocal{A}=k$.  In particular, since $f_\pazocal{L}$ is non-decreasing, $\|w_t\|f_\pazocal{L}(c_0\|w_t\|)\leq\|w\|f_\pazocal{L}(c_0\|w\|)$.

For any $i$, \Cref{M Lambda semi-computations} implies:

\begin{enumerate}

\item $\frac{1}{2}D_\pazocal{A}(t-i-1)\leq\|u_{j-1,i}\|+\|u_{j,i}\|\leq3D_\pazocal{A}(t-i-1)$ for any $j\in\{2,\dots,k-1\}$


\item $\|u_{0,i}\|,\|u_{k,i}\|\leq D_\pazocal{A}(t-i-1)$

\end{enumerate}

If $t=1$, then this implies $\|u_{j,0}\|=0$ for all $j$, so that $\|w_0\|=|w_0|_\pazocal{A}=k=\|w\|$.  In particular, $\text{Area}(\Gamma_w)\leq\|w\|f_\pazocal{L}(c_0\|w\|)+\|w\|\leq g_\pazocal{L}(\|w\|)$.

Otherwise, $\sum\limits_{i=0}^{t-1}\|w_i\|=\sum\limits_{i=0}^{t-1}\left(k+\sum\limits_{j=0}^k\|u_{j,i}\|\right)\leq\sum\limits_{i=0}^{t-1}(k+3D_\pazocal{A}ki)\leq3D_\pazocal{A}kt^2$.  

As $k\geq C$ by the definition of $\Lambda^\pazocal{A}$, a parameter choice for $C$ implies there exists $\ell\in\{2,\dots,k-1\}$ such that $x_{\ell-1}^{\delta_{\ell-1}}u_{\ell-1,0}x_\ell^{\delta_\ell}u_{\ell,0}x_{\ell+1}^{\delta_{\ell+1}}$ is a subword of $w$.  So, as $D_\pazocal{A}$ is dependent on $C$, a parameter choice for $C$ implies: 
$$\|w\|\geq k+\|u_{\ell-1,0}\|+\|u_{\ell,0}\|\geq k+\frac{1}{2}D_\pazocal{A}(t-1)\geq k+\frac{1}{4}D_\pazocal{A}t\geq k+t$$
Hence, the parameter choice $c_0>>C$ then yields:
$$\sum_{i=0}^{t-1}\|w_i\|\leq3D_\pazocal{A}kt^2\leq 3c_0kt^2\leq c_0(k+t)^3\leq c_0\|w\|^3$$
Thus, $\text{Area}(\Gamma_w)\leq\|w\|f_\pazocal{L}(c_0\|w\|)+c_0\|w\|^3=g_\pazocal{L}(\|w\|)$.

\end{proof}

\begin{lemma} \label{Dehn}

For every $n\in\N$, $\delta_\pazocal{P}(n)\leq n\left(Kn^{12}+g_\pazocal{L}(Kn^9)+f_\pazocal{L}(Kn^3)\right)$

\end{lemma}

\begin{proof}

Let $W_0$ be a word over $\pazocal{X}\cup\pazocal{X}^{-1}$ which represents the identity in $G(\textbf{M}^\pazocal{L})$ and satisfies $\|W_0\|\leq n$.  

By the analogue of \Cref{isomorphism} in this setting, $W_0$ represents the identity in $G_\Omega(\textbf{M}^\pazocal{L})$.  As such, there exists a reduced minimal diagram $\Delta$ over the disk presentation of $G_\Omega(\textbf{M}^\pazocal{L})$ satisfying $\lab(\partial\Delta)\equiv W_0$.  So, \Cref{main upper bound} implies $\text{wt}(\Delta)\leq n\left(Kn^{12}+g_\pazocal{L}(Kn^9)+f_\pazocal{L}(Kn^3)\right)$.

Now, consider the diagram $\tilde{\Delta}$ constructed as follows:

\begin{itemize}

\item Let $\Pi$ be a disk in $\Delta$.  Then letting $\lab(\partial\Pi)\equiv W$, replace $\Pi$ with the circular diagram $\Gamma_{W}$ constructed in \Cref{disk weights}.  Note that $\text{Area}(\Gamma_W)\leq\text{wt}(\Pi)$.

\item Let $\pi$ be an $a$-cell in $\Delta$.  Then, letting $\lab(\partial\pi)\equiv w$, replace $\pi$ with the circular diagram $\Gamma_{w}$ constructed in \Cref{Omega weights}.  Note that $\text{Area}(\Gamma_w)\leq\text{wt}(\pi)$.

\end{itemize}

Then, $\tilde{\Delta}$ is a circular diagram over $\pazocal{P}$ with $\lab(\partial\Delta)\equiv W_0$ and $\text{Area}(\tilde{\Delta})\leq\text{wt}(\Delta)$.  

Hence, $\text{Area}_\pazocal{P}(W_0)\leq n\left(Kn^{12}+g_\pazocal{L}(Kn^9)+f_\pazocal{L}(Kn^3)\right)$, implying the statement.

\end{proof}

Since $f_\pazocal{L}$ and $g_\pazocal{L}$ are computable functions, \Cref{Dehn} implies $\delta_\pazocal{P}$ is bounded above by the computable function $f:\N\to\N$ given by $f(n)=n\left(Kn^{12}+g_\pazocal{L}(Kn^9)+f_\pazocal{L}(Kn^3)\right)$.  Thus, $G(\textbf{M}^\pazocal{L})$ has decidable Word problem, completing the proof of Theorem A.

\bigskip

\bibliographystyle{plain}
\bibliography{biblio}

\end{document}